\documentclass[a4paper,12pt]{amsbook}
\usepackage{amsmath,amsthm,amsfonts,amssymb,mathrsfs,graphicx,psfrag,dsfont,a4wide,mathabx}
\usepackage{minitoc,color}

\usepackage[hyperfootnotes=false]{hyperref}
\usepackage[alphabetic,initials]{amsrefs}

\makeindex

\renewcommand{\thesection}{\arabic{chapter}.\arabic{section}}
\renewcommand{\thefigure}{\arabic{chapter}.\arabic{figure}}

\theoremstyle{plain}
\newtheorem{thm}{Theorem}[chapter]

\newtheorem{prop}[thm]{Proposition}

\newtheorem{cor}[thm]{Corollary}

\newtheorem{lemma}[thm]{Lemma}

\newtheorem{question}[thm]{Question}

\theoremstyle{definition}
\newtheorem{example}[thm]{Example}
\newtheorem{exercise}[thm]{Exercise}
\newtheorem{defn}[thm]{Definition}
\newtheorem*{notation}{Notation}

\theoremstyle{remark}
\newtheorem{remark}[thm]{Remark}

\newcommand{\Aut}{\operatorname{Aut}}

\newcommand{\Br}{\operatorname{Br}}

\newcommand{\cov}{\operatorname{cov}}
\newcommand{\Crit}{\operatorname{Crit}}
\newcommand{\crit}{^{\operatorname{crit}}}

\newcommand{\defin}{\textbf}

\newcommand{\dist}{\operatorname{dist}}
\newcommand{\dotrel}{\bullet}

\newcommand{\End}{\operatorname{End}}

\newcommand{\even}{_{\operatorname{even}}}

\newcommand{\Hess}{\operatorname{Hess}}
\newcommand{\Hom}{\operatorname{Hom}}

\newcommand{\Id}{\operatorname{Id}}

\newcommand{\im}{\operatorname{im}}
\newcommand{\ind}{\operatorname{ind}}

\newcommand{\Inj}{\operatorname{Inj}}

\newcommand{\inter}{\iota}

\newcommand{\loc}{{\operatorname{loc}}}

\newcommand{\muCZ}{\mu_{\operatorname{CZ}}}
\newcommand{\OB}{_{\operatorname{OB}}}
\newcommand{\odd}{_{\operatorname{odd}}}
\newcommand{\ord}{\operatorname{ord}}

\newcommand{\reg}{{\operatorname{reg}}}

\newcommand{\std}{_{\operatorname{std}}}

\newcommand{\virdim}{\operatorname{vir-dim}}

\newcommand{\wind}{\operatorname{wind}}

\newcommand{\GL}{\operatorname{GL}}

\newcommand{\U}{\operatorname{U}}

\newcommand{\CC}{{\mathbb C}}
\newcommand{\CP}{{\mathbb C}{\mathbb P}}
\newcommand{\DD}{{\mathbb D}}

\newcommand{\NN}{{\mathbb N}}

\newcommand{\RR}{{\mathbb R}}

\newcommand{\TT}{{\mathbb T}}
\newcommand{\ZZ}{{\mathbb Z}}

\newcommand{\dD}{{\mathcal D}}

\newcommand{\jJ}{{\mathcal J}}
\newcommand{\lL}{{\mathcal L}}
\newcommand{\Lie}{\lL}
\newcommand{\mM}{{\mathcal M}}

\newcommand{\nN}{{\mathcal N}}
\newcommand{\oO}{{\mathcal O}}

\newcommand{\tT}{{\mathcal T}}
\newcommand{\uU}{{\mathcal U}}

\newcommand{\1}{\mathds{1}}
\newcommand{\p}{\partial}
\renewcommand{\dbar}{\bar{\partial}}

\definecolor{blue}{rgb}{0,0,1}
\definecolor{red}{rgb}{1,0,0}
\definecolor{green}{rgb}{0,.7,0}

\numberwithin{equation}{chapter}

\psfrag{SigmaT2}{$\Sigma = \TT^2$}
\psfrag{Sigmacrit}{$\Sigma\crit$}
\psfrag{SigmaMGamma}{$\dot{\Sigma} = \Sigma\setminus\Gamma$}
\psfrag{z}{$z$}
\psfrag{w}{$w$}
\psfrag{zeta}{$\zeta$}
\psfrag{u}{$u$}
\psfrag{W}{$W$}
\psfrag{Womega}{$(W,\omega)$}
\psfrag{gammaz}{$\gamma_z$}
\psfrag{gammaw}{$\gamma_w$}
\psfrag{gammazeta}{$\gamma_\zeta$}
\psfrag{end+}{$[0,\infty) \times M_+$}
\psfrag{end-}{$(-\infty,0] \times M_-$}
\psfrag{End+}{$([0,\infty) \times M_+, d(e^s\alpha_+))$}
\psfrag{End-}{$((-\infty,0] \times M_-, d(e^s\alpha_-))$}
\psfrag{collar-}{$([0,\epsilon) \times M_-,d(e^s\alpha_-))$}
\psfrag{collar+}{$((-\epsilon,0] \times M_+,d(e^s\alpha_+))$}
\psfrag{uk}{$u_k$}
\psfrag{uk1}{$u_{k+1}$}
\psfrag{uk2}{$u_{k+2}$}
\psfrag{v+}{$v_+$}
\psfrag{v-}{$v_-$}
\psfrag{ueps}{$u_\epsilon$}
\psfrag{ueps'}{$u_\epsilon'$}
\psfrag{What}{$(\widehat{W},J)$}
\psfrag{RtimesM+}{$(\RR \times M_+,J_+)$}
\psfrag{RtimesM-}{$(\RR \times M_-,J_-)$}
\psfrag{v0}{$v_0$}
\psfrag{v1+}{$v_1^+$}
\psfrag{v1-}{$v_1^-$}
\psfrag{v2-}{$v_2^-$}
\psfrag{v3-}{$v_3^-$}
\psfrag{M+}{$(M_+,\xi_+)$}
\psfrag{M-}{$(M_-,\xi_-)$}
\psfrag{z1hat}{$\widehat{z}_1$}
\psfrag{z2hat}{$\widehat{z}_2$}
\psfrag{z3hat}{$\widehat{z}_3$}
\psfrag{z4hat}{$\widehat{z}_4$}
\psfrag{z1check}{$\widecheck{z}_1$}
\psfrag{z2check}{$\widecheck{z}_2$}
\psfrag{z3check}{$\widecheck{z}_3$}
\psfrag{z4check}{$\widecheck{z}_4$}
\psfrag{C1hat}{$\widehat{C}_1$}
\psfrag{C2hat}{$\widehat{C}_2$}
\psfrag{C3hat}{$\widehat{C}_3$}
\psfrag{C4hat}{$\widehat{C}_4$}
\psfrag{C1check}{$\widecheck{C}_1$}
\psfrag{C2check}{$\widecheck{C}_2$}
\psfrag{C3check}{$\widecheck{C}_3$}
\psfrag{C4check}{$\widecheck{C}_4$}
\psfrag{(S,j)}{$(S,j)$}
\psfrag{u(S)}{$u(S)$}
\psfrag{Sbar}{$\overline{S}$}
\psfrag{Sprime}{$S'$}
\psfrag{d=0}{$\delta(u)=0$}
\psfrag{d>0}{$\delta(u)>0$}
\psfrag{S3=R3infty}{$S^3 = \RR^3 \cup \{\infty\}$}
\psfrag{S1xS2}{$S^1 \times S^2$}
\psfrag{S1 times}{$S^1 \times$}
\psfrag{bind}{$B$}
\psfrag{D2}{$\DD \subset \CC$}
\psfrag{homot}{$\stackrel{h}{\sim}$}
\psfrag{RtimesM}{$(\RR \times M,J_+)$}
\psfrag{M}{$(M,\xi)$}

\newcommand\CUP[1]{%
  \begingroup
  \renewcommand\thefootnote{}\footnote{#1}%
  \addtocounter{footnote}{-1}%
  \endgroup
}

\title[Contact $3$-Manifolds, Holomorphic Curves and Intersection Theory]{Contact 
$3$-Manifolds, Holomorphic Curves and Intersection Theory}
\author{Chris Wendl}
\address{Institut f\"ur Mathematik \\
Humboldt-Universit\"at zu Berlin \\
Unter den Linden 6 \\
10099 Berlin \\ 
Germany}
\email{wendl@math.hu-berlin.de}
\thanks{This material will be published by Cambridge University
Press as \textsl{Contact 3-Manifolds, Holomorphic Curves and Intersection Theory}
by Chris Wendl. This pre-publication version is
free to view and download for personal use only. 
Not for re-distribution, re-sale or use in derivative works. \copyright Chris Wendl, 2019.}

\begin{document}
\frontmatter

\begin{abstract}
This is a revision of some expository lecture notes written originally for 
a 5-hour minicourse on the intersection theory of punctured holomorphic curves and its applications
in $3$-dimensional contact topology.  The main lectures are aimed primarily
at students and require only a minimal background in holomorphic curve theory,
as the emphasis is on topological rather than analytical issues.
Some of the gaps in the analysis are then filled in by the appendices,
which include
self-contained proofs of the similarity principle and positivity of
intersections, and conclude with a ``quick reference'' for the benefit of
researchers, detailing the basic facts of Siefring's intersection theory.

Intersection theory has played a 
prominent role in the study of closed symplectic $4$-manifolds since Gromov's 
paper \cite{Gromov} on pseudoholomorphic curves, leading to 
a myriad of beautiful rigidity results that are either not accessible or 
not true in higher dimensions.  In recent years, the highly nontrivial 
extension of this theory to the case of punctured holomorphic curves,
due to Siefring \cites{Siefring:asymptotics,Siefring:intersection},
has led to similarly beautiful results about contact $3$-manifolds and their 
symplectic fillings. 
These notes begin with an overview of the closed case and an easy application 
(McDuff's characterization of symplectic ruled surfaces), and then explain the 
essentials of Siefring's intersection theory and how to use it in the real world. 
As a sample application, we discuss the classification of symplectic fillings 
of planar contact manifolds via Lefschetz fibrations \cite{Wendl:fillable}.
\end{abstract}

\maketitle

\dominitoc
\tableofcontents

\chapter*{Preface}

The main portion of this book is a lightly revised set of expository lecture notes written 
originally for a 5-hour minicourse
on the intersection theory of punctured holomorphic curves and its applications
in $3$-dimensional contact topology, which I gave
as part of the LMS Short Course \textsl{``Topology in Low Dimensions''} at
Durham University, August 26--30, 2013.  These lectures were aimed primarily
at students, and they required only a minimal background in holomorphic curve theory
since the emphasis was on topological rather than analytical issues.
The original appendices were relatively brief, their purpose being to
provide a quick survey of analytical
background material on holomorphic curves that I needed to refer to in the
lectures without assuming that students already knew it.  In revising the
manuscript for publication, 
I have taken the opportunity to add Lecture~\ref{sec:0} as a motivational
introduction to the topic of the notes, plus
two things that I felt were lacking from the
existing literature, as a result of which the appendices have become considerably more
substantial.  One (Appendix~\ref{app:positivity}) is a complete 
proof of local positivity of intersections, including just enough background
material on elliptic regularity for a student familiar with distributions and Sobolev spaces 
to consider it ``self-contained''; this notably includes a weak version of
the Micallef-White theorem, which some readers may hopefully find easier to
comprehend than the deeper result in \cite{MicallefWhite} that inspired it.
The other (Appendix~\ref{app:reference}) is a quick survey of
Siefring's intersection theory of punctured holomorphic curves, putting the 
essential facts and formulas in as compact a form as possible for the 
benefit of researchers who need a ready reference.  Most of what is in
Appendix~\ref{app:reference} also appears in Lectures~\ref{sec:3} and
~\ref{sec:4}, but the latter are written in a more pedagogical style that develops the
structure of the theory based on a few core ideas---that is presumably helpful if your
goal is to understand why the main results are true, but less so if you just
need to look up a specific formula, and Appendix~\ref{app:reference} is there
to help in the latter case.

Intersection theory has played a 
prominent role in the study of closed symplectic $4$-manifolds since Gromov's 
paper \cite{Gromov} on pseudoholomorphic curves, leading to 
a myriad of beautiful rigidity results that are either not accessible or 
not true in higher dimensions.  In the last 15 years, the highly nontrivial 
extension of this theory to the case of punctured holomorphic curves,
due to Siefring \cites{Siefring:asymptotics,Siefring:intersection},
has led to similarly beautiful results about contact $3$-manifolds and their 
symplectic fillings. 
These notes begin with an overview of the closed case and an easy application 
(McDuff's characterization of symplectic ruled surfaces), and then explain the 
essentials of Siefring's intersection theory and how to use it in the real world. 
As a sample application, Lecture~\ref{sec:5} concludes by discussing 
the classification of symplectic fillings 
of planar contact manifolds via Lefschetz fibrations \cite{Wendl:fillable}.

\subsection*{How to use these notes}

I expect a variety of audiences to find these notes useful for a variety of reasons.
Since they were written
with an audience of students in mind, I did not want to assume too much
previous knowledge of symplectic/contact geometry or holomorphic curves,
and most of the text reflects that.  On the other hand, I also expect a
certain number of readers to be experienced researchers who already know
the essentials of holomorphic curve theory---including the adjunction formula
in the closed case---but would specifically like to learn about
the intersection theory for \emph{punctured} curves.  For readers in this category,
I recommend starting with Appendix~\ref{app:reference} for an overview of the
basic facts, and then turning back to Lectures~\ref{sec:3} and~\ref{sec:4}
for details whenever necessary.  If on the other hand you are a student
and still getting to know the field of symplectic and contact topology,
you'd probably rather start from the beginning.

Or if you really want to challenge yourself, feel free to read the whole
thing backwards.

\subsection*{Acknowledgments}

I would like to thank Richard Siefring and Michael Hutchings for many
conversations over the years that have improved my understanding of
the subjects discussed in this book.
Thanks are also due to Andrew Lobb, Durham University and the London Mathematical
Society for bringing about the summer school that gave rise to the original notes.  
They were written mostly while I worked at University College London,
with partial support from a Royal Society University Research Fellowship
and a Leverhulme Research Project Grant.

\adjustmtc

\mainmatter

\setcounter{chapter}{-1}
\chapter{Motivation}
\label{sec:0}

In\CUP{This material will be published by Cambridge University
Press as \textsl{Contact 3-Manifolds, Holomorphic Curves and Intersection Theory}
by Chris Wendl. This pre-publication version is
free to view and download for personal use only. 
Not for re-distribution, re-sale or use in derivative works. \copyright Chris Wendl, 2019.}
order to illustrate briefly what these lectures are about, I'd like to
give an informal sketch of two closely related theorems from the early days of
symplectic topology.  The first is a beautiful application of the
theory of closed pseudoholomorphic curves as introduced by Gromov
in \cite{Gromov}, and its proof requires only a few basic facts from this
theory, plus some knowledge of the standard homological intersection
product from algebraic topology.  The second theorem admits a closely analogous
proof, but we will see that the intersection-theoretic
portion of the argument is difficult to make precise, because it is no longer
homological---it requires some generalization of the intersection product
in which ``cycles'' need not be closed.  One of the main
objectives of the subsequent lectures will be to make this idea precise
and demonstrate what else it can be used for.

The statements of these theorems assume familiarity with the notions of
minimal symplectic $4$-manifolds, symplectomorphisms,
symplectic submanifolds, the standard
symplectic structure on~$\RR^4$, the
sign of a transverse intersection, and the homological intersection product---some 
background on all of these topics is covered in Lectures~\ref{sec:1} and~\ref{sec:2}.

\begin{thm}
\label{thm:S2xS2}
Suppose $(M,\omega)$ is a closed, connected, minimal symplectic $4$-manifold
containing a pair of symplectic submanifolds $S_1,S_2 \subset M$ with the
following properties:
\begin{itemize}
\item Both are homeomorphic to~$S^2$;
\item Both have vanishing homological self-intersection number:
$$
[S_1] \cdot [S_1] = [S_2] \cdot [S_2] = 0.
$$
\item The set $S_1 \cap S_2 \subset M$ consists of a single transverse
and positive intersection.
\end{itemize}
Then there exists a symplectomorphism identifying $(M,\omega)$ with 
$(S^2 \times S^2,\omega_0)$ such that $S_1$ and $S_2$ are identified with
$S^2 \times \{\text{const}\}$ and $\{\text{const}\} \times S^2$
respectively, and $\omega_0$ is a product of two area forms on~$S^2$.
\end{thm}

This result says in effect that if we are given a certain type of
``local'' information about submanifolds of a closed symplectic $4$-manifold,
then this is enough to recover its global structure.  From an alternative
perspective, it says that the vast majority of closed symplectic $4$-manifolds
do not contain certain types of symplectic submanifolds.  The second result
says something similar, but now the symplectic manifold is noncompact and
the ``local'' information we are given is its structure outside of some
compact subset---the theorem is typically summarized by saying that there do
not exist any exotic symplectic $4$-manifolds that look ``standard at infinity''.

\begin{thm}
\label{thm:stdAtInfty}
Suppose $(M,\omega)$ is an open, connected, minimal symplectic $4$-manifold
with a compact subset $K \subset M$ such that $(M \setminus K,\omega)$ is
symplectomorphic to the complement of a compact subset in the standard
symplectic~$\RR^4$.  Then $(M,\omega)$ is globally symplectomorphic to the
standard symplectic~$\RR^4$.
\end{thm}

\begin{remark}
Both of these theorems appeared in less general forms in Gromov's paper 
\cite{Gromov}; see \S$2.4.A_1'$ and \S$0.3.C$ respectively.
The statements given above are attributed to both Gromov and McDuff, as they
rely on the slightly more sophisticated intersection theory of closed
holomorphic curves that was developed by McDuff within a few years after Gromov's
paper---see in particular \cite{McDuff:rationalRuled}.
Theorem~\ref{thm:stdAtInfty} can also be rephrased as the statement that
$S^3$ with its standard contact structure admits a unique minimal symplectic
filling, and we will discuss this version of the result in
Lecture~\ref{sec:5} (see in particular Corollary~\ref{cor:uniqueFillings}).
\end{remark}

Let's sketch a proof of Theorem~\ref{thm:S2xS2}.  The starting point is the
observation that since $S_1$ and $S_2$ are both \emph{symplectic} submanifolds
and their intersection is transverse and positive,
one can choose a compatible almost complex structure $J : TM \to TM$ 
on $(M,\omega)$ that preserves the tangent spaces of $S_1$ and $S_2$
(see \S\ref{sec:fibrations} for more on almost complex structures).
This makes $S_1$ and $S_2$ into images of embedded \emph{$J$-holomorphic spheres},
i.e.~smooth maps $u : S^2 \to M$ that satisfy the \emph{nonlinear Cauchy-Riemann
equation}
$$
Tu \circ i = J \circ Tu,
$$
where $i : TS^2 \to TS^2$ is the almost complex structure on $S^2$ resulting
from its standard identification with the extended complex plane $\CC \cup \{\infty\}$.
The advantage of replacing symplectic submanifolds
by $J$-holomorphic spheres is a matter of rigidity: the condition of being
a symplectic submanifold is open and thus quite flexible, i.e.~the space of
all symplectic submanifolds is unmanageably large, whereas $J$-holomorphic
spheres are solutions to an elliptic PDE, and thus tend to come in
finite-dimensional moduli spaces, which are sometimes (if we're lucky!) even
compact.  For this reason, we now consider for each $k=1,2$ the 
\emph{moduli spaces}
$$
\mM_k(J) := \left\{ u : S^2 \to M\ \big|\ Tu \circ i = J \circ Tu \text{ and }
[u] := u_*[S^2] = [S_k] \in H_2(M) \right\} \Big/ \Aut(S^2,i),
$$
where $\Aut(S^2,i)$ is the group of holomorphic automorphisms 
$\varphi : S^2 \to S^2$ of the extended
complex plane (i.e.~the M\"obius transformations), acting on the space of
$J$-holomorphic maps $u : S^2 \to M$ by $\varphi \cdot u := u \circ \varphi$.
We assign to this space the natural topology arising from $C^\infty$-convergence
of maps.  Both $\mM_1(J)$ and $\mM_2(J)$ are clearly nonempty, since they
contain equivalence classes of parametrizations of the submanifolds
$S_1$ and $S_2$ respectively.
One can now apply general results from the theory of $J$-holomorphic curves
to prove that for generic choices of the almost complex structure~$J$,
$\mM_1(J)$ and $\mM_2(J)$ are both compact smooth $2$-dimensional 
manifolds.  A quick survey of the analytical results behind this is
given in Appendix~\ref{app:closed}, and we will sketch the proof in a
somewhat more general setting in Lectures~\ref{sec:1}
and~\ref{sec:2} (see Lemmas~\ref{lemma:M0} and~\ref{lemma:M1}),
though we do not plan to get too deeply into such
analytical details in this book.  

What we will discuss in more detail
is the intersection-theoretic properties of the $J$-holomorphic spheres
in $\mM_1(J)$ and~$\mM_2(J)$.  We observe first that the hypotheses of
Theorem~\ref{thm:S2xS2} clearly imply
$$
[S_1] \cdot [S_2] = 1,
$$
as this intersection number can be computed as a signed count of transverse
intersections between $S_1$ and $S_2$, for which there is only one
intersection to count, and it is positive.
In Lecture~\ref{sec:2} and Appendix~\ref{app:positivity}, we will discuss a
standard result known as \emph{positivity of intersections}, which implies
that whenever $u : \Sigma \to M$ and $v : \Sigma' \to M$ are two 
closed $J$-holomorphic curves with non-identical images in an almost complex
$4$-manifold~$M$, their intersections are all isolated and count positively
toward the homological intersection number $[u] \cdot [v] \in \ZZ$; moreover,
the contribution of each isolated intersection is exactly $+1$ if and only if
that intersection is transverse.  This is very strong information, from which
one can deduce the following:
\begin{enumerate}
\item For each $k=1,2$ and every pair of distinct elements $u,v \in \mM_k(J)$,
the images of $u : S^2 \to M$ and $v : S^2 \to M$ are disjoint.
(This follows from the condition $[S_k] \cdot [S_k] = 0$.)
\item For every $u \in \mM_1(J)$ and $v \in \mM_2(J)$, the maps $u : S^2\to M$
and $v : S^2 \to M$ have exactly one intersection point, which is transverse
and positive.
\end{enumerate}
A related result discussed in \S\ref{sec:adjunction}, called the 
\emph{adjunction formula}, makes it possible characterize in homological
terms which $J$-holomorphic curves in an almost complex $4$-manifold
are embedded, and in this case it implies:
\begin{enumerate}
\setcounter{enumi}{2}
\item Every element of $\mM_1(J)$ or $\mM_2(J)$ is embedded.
\end{enumerate}
Finally, we will see in \S\ref{sec:foliations} that whenever
$u \in \mM_k(J)$ is an embedded $J$-holomorphic sphere in one of these
moduli spaces, the $2$-parameter family of nearby $J$-holomorphic spheres
in $\mM_k(J)$ forms a smooth foliation of the neighborhood of $u(S^2)$ in~$M$.
Combining this with the compactness of $\mM_k(J)$, it follows that the set
of points in $M$ that are contained in the images of any of the spheres
in $\mM_k(J)$ is both open and closed, thus it is everything: the holomorphic
spheres of $\mM_k(J)$ foliate~$M$.  The result is the
``coordinate grid'' depicted in 
Figure~\ref{fig:S2xS2}: starting from the two symplectically embedded spheres
$S_1,S_2 \subset M$, we obtain two smooth families of embedded $J$-holomorphic spheres
that each foliate~$M$, such that each sphere in $\mM_1(J)$ has a unique
transverse intersection with each sphere in $\mM_2(J)$.  It follows that
there is a diffeomorphism
\begin{equation}
\label{eqn:diffeoS2S2}
M \stackrel{\cong}{\longrightarrow} \mM_1(J) \times \mM_2(J),
\end{equation}
assigning to each point $p \in M$ the unique pair of holomorphic spheres
$(u,v) \in \mM_1(J) \times \mM_2(J)$ such that both have $p$ in their images.
Moreover, for each individual element of $\mM_1(J)$ parametrized by a map
$u : S^2 \to M$, there is a diffeomorphism
$$
S^2 \stackrel{\cong}{\longrightarrow} \mM_2(J)
$$
sending each $z \in S^2$ to the unique holomorphic sphere
$v \in \mM_2(J)$ that has $u(z)$ in its image; this proves that $\mM_2(J)$
has the topology of~$S^2$, and in the same manner one shows $\mM_1(J) \cong S^2$.
In summary, \eqref{eqn:diffeoS2S2} can now be interpreted as a diffeomorphism
from $M$ to $S^2 \times S^2$.  There is still a bit of work to be done in
identifying the symplectic structure $\omega$ with a product of two area
forms, but the techniques needed for this are not hard---they involve
geometric tools such as the Moser stability theorem for deformations of
symplectic forms (see e.g.~\cite{McDuffSalamon:ST3}), but no serious analysis
is required.

\begin{figure}
\psfrag{S1}{$S_1$}
\psfrag{S2}{$S_2$}
\includegraphics{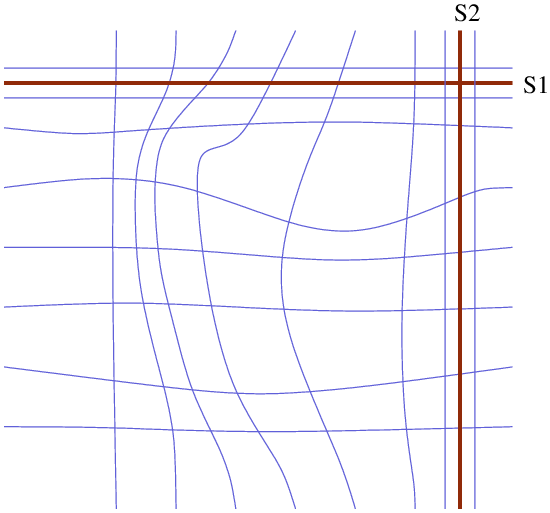}
\caption{\label{fig:S2xS2} The two symplectic submanifolds $S_1,S_2 \subset M$
generate two transverse foliations by holomorphic spheres
in the proof of Theorem~\ref{thm:S2xS2}.  The two families can be
regarded as a ``coordinate grid'' that identifies $M$ with $S^2 \times S^2$.}
\end{figure}

The original proof of Theorem~\ref{thm:stdAtInfty} used a clever ``capping''
trick to derive it from Theorem~\ref{thm:S2xS2}.  For this motivational
discussion, I would like to sketch a different proof that is conceptually
simpler, but trickier in the technical details.  

\begin{figure}
\psfrag{uw}{$f_w$}
\psfrag{vw}{$g_w$}
\psfrag{dBall}{$\p \DD_R^4$}
\psfrag{gamma1}{$\gamma_1$}
\psfrag{gamma2}{$\gamma_2$}
\includegraphics[scale=0.6]{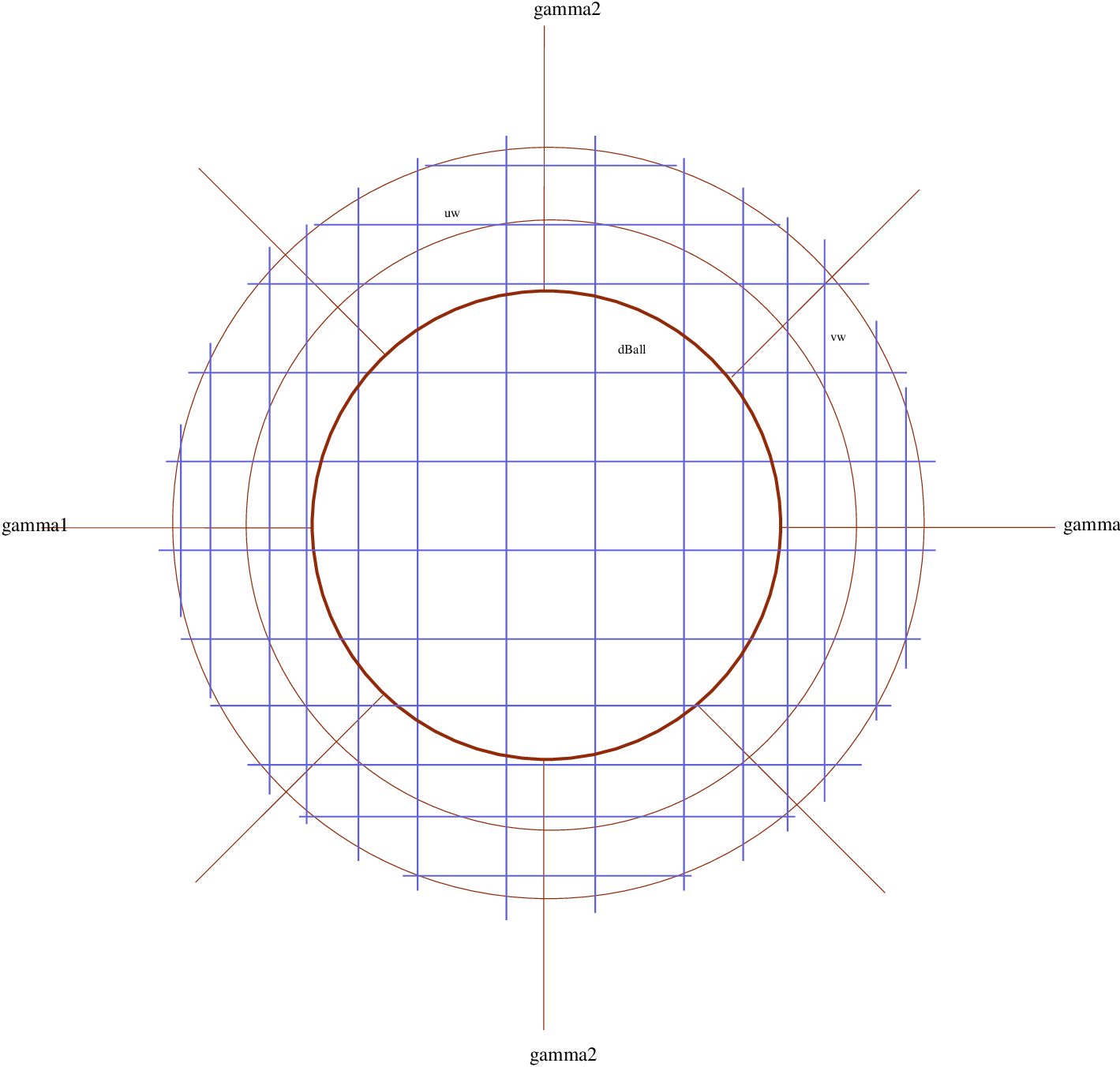}
\caption{\label{fig:stdAtInfty0} The two families of properly
  embedded holomorphic planes $f_w$ and $g_w$ form a coordinate
  grid for~$\CC^2$ and are each asymptotic on the cylindrical end
$\CC^2 \setminus \DD^4_R \cong (R,\infty) \times S^3$ to
one of two specific loops $\gamma_1,\gamma_2 \subset S^3$.}
\end{figure}

By the hypotheses of Theorem~\ref{thm:stdAtInfty}, we can decompose the open
symplectic manifold $(M,\omega)$ into two regions: one is the compact
(but otherwise completely unknown) subset $K \subset M$, and the other is
a region that we can identify with $(\RR^4 \setminus K',\omega\std)$ 
for some compact set $K' \subset \RR^4$, where $\omega\std$ denotes the
standard symplectic form on~$\RR^4$.  We would like to argue as in
Theorem~\ref{thm:S2xS2}, that is, find a nice pair of ``seed curves'' to
generate two well-behaved moduli spaces of $J$-holomorphic curves that can
then be used to form a coordinate grid identifying $M$ with~$\RR^4$.
One easy way to find such seed curves is by observing that $\RR^4$ has a
natural identification with $\CC^2$ such that the natural multiplication
by $i$ on $\CC^2$ defines a compatible almost complex structure
on $(\RR^4,\omega\std)$.  This is useful for the following reason:
$\CC^2$ contains two obvious families of holomorphic planes
\begin{equation*}
\begin{split}
f_w : \CC \to \CC^2 : z \mapsto (z,w),\qquad \text{ for $w \in \CC$},\\
g_w : \CC \to \CC^2 : z \mapsto (w,z),\qquad \text{ for $w \in \CC$},
\end{split}
\end{equation*}
all of which are properly embedded maps, with two distinct types of
asymptotic behavior.  To describe the latter, choose a large constant
$R > 0$, let $\DD_R^4 \subset \CC^2$ denote the disk of radius $R$
and identify $\CC^2 \setminus \DD_R^4$ with $(R,\infty) \times S^3$
by viewing $S^3$ as the unit sphere in $\CC^2$ and applying the diffeomorphism
$$
(R,\infty) \times S^3 \stackrel{\cong}{\longrightarrow} \CC^2 \setminus \DD_R^4 :
(r,x) \mapsto rx.
$$
Then each $f_w$ or $g_w$ maps a neighborhood of infinity into an arbitrarily small
neighborhood of the cylinder $(R,\infty) \times \gamma_1$ or
$(R,\infty) \times \gamma_2$ respectively, where we define
$$
\gamma_1 := S^1 \times \{0\} \subset S^3 \subset \CC^2, \qquad \gamma_2 := \{0\} \times S^1 \subset S^3 \subset \CC^2.
$$
A schematic picture of this asymptotic behavior and the resulting transverse
pair of holomorphic foliations of $\CC^2$ is shown in Figure~\ref{fig:stdAtInfty0}.
Informally, we will say that the planes $f_w$ are asymptotic to $\gamma_1$
and the planes $g_w$ are asymptotic to~$\gamma_2$; more precise definitions of
this terminology will appear in \S\ref{sec:punctured} when we discuss
asymptotically cylindrical maps.

\begin{figure}
\psfrag{uw}{$f_w$}
\psfrag{vw}{$g_w$}
\psfrag{K}{$K$}
\psfrag{dBall}{$\p \DD_R^4$}
\psfrag{gamma1}{$\gamma_1$}
\psfrag{gamma2}{$\gamma_2$}
\includegraphics[scale=0.6]{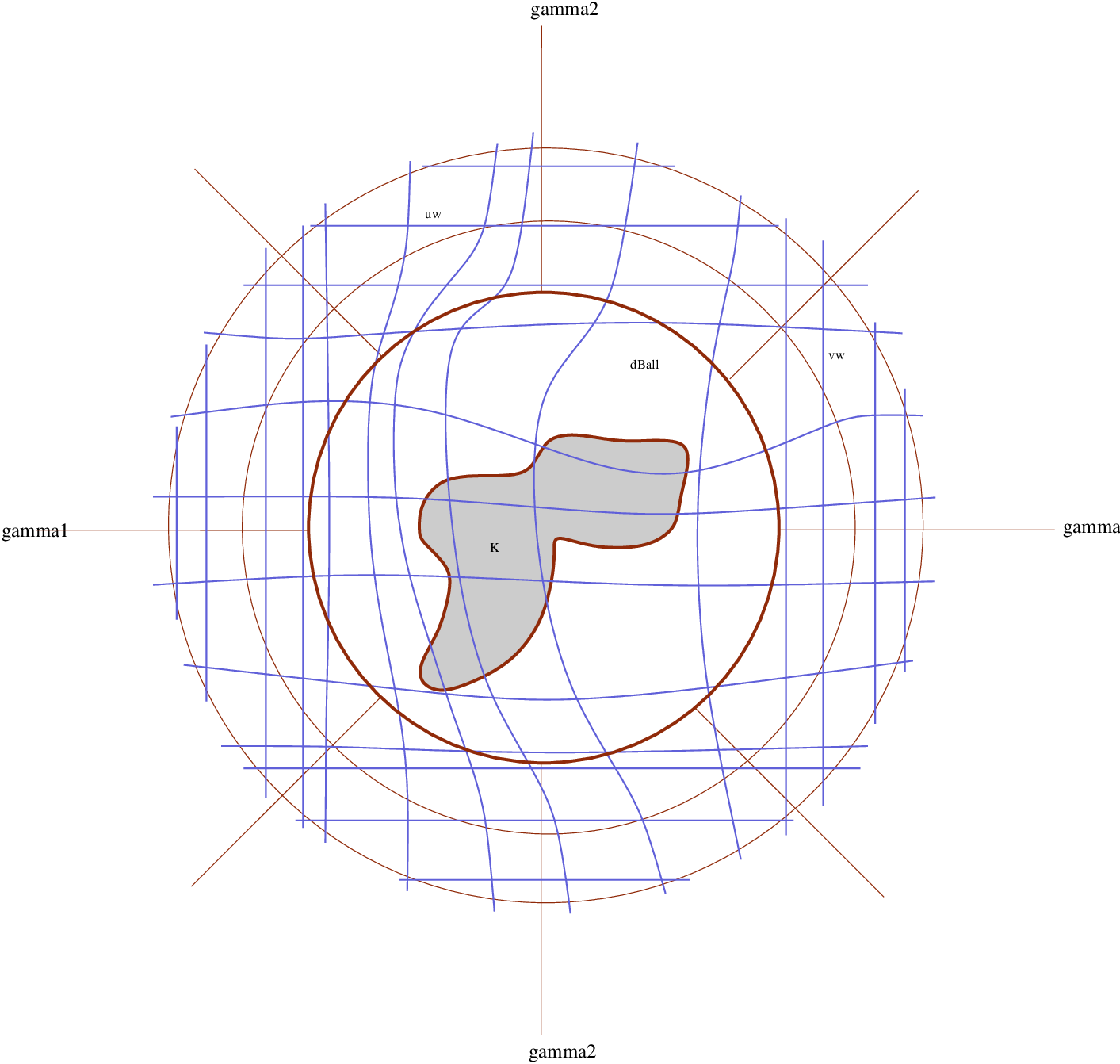}
\caption{\label{fig:stdAtInfty} The moduli spaces $\mM_1(J;\gamma_1)$
  and $\mM_2(J;\gamma_2)$ of proper $J$-holomorphic planes asymptotic to
  the loops $\gamma_1,\gamma_2 \subset S^3$ form two transverse
  foliations of~$M$ in Theorem~\ref{thm:stdAtInfty}, building a coordinate
grid that proves $M \cong \CC \times \CC = \RR^4$.}
\end{figure}

Now since $K' \subset \CC^2 = \RR^4$ is compact, $\DD^4_R$
will contain $K'$ for any $R > 0$ sufficiently large, so that we can
also regard $(M,\omega)$ as containing a copy of the region identified
above with $(R,\infty) \times S^3$.  Let us
fix such a radius and choose a compatible almost complex structure $J$
on $(M,\omega)$ that matches the standard multiplication by~$i$ on
$\CC^2 \setminus \DD^4_R \cong (R,\infty) \times S^3$.
The curves $f_w$ and $g_w$ can then be regarded as 
$J$-holomorphic planes in $M$ for every $w \in \CC$ with $|w| > R$,
and just as in Theorem~\ref{thm:S2xS2}, these two families define
elements in a pair of connected moduli spaces $\mM_1(J;\gamma_1)$ and $\mM_2(J;\gamma_2)$ of
$J$-holomorphic planes in~$M$, where we can use the loops $\gamma_1$ and
$\gamma_2$ to prescribe the asymptotic behavior of the curves in the moduli spaces.
There exists a well-developed theory of moduli spaces of
$J$-holomorphic curves with this type of asymptotic behavior, a survey of
which is given in Appendix~\ref{app:punctured}.  In the present context,
it can be applied to prove that $\mM_1(J;\gamma_1)$ and $\mM_2(J;\gamma_2)$ are both
smooth $2$-dimensional manifolds, and they are also compact except for the obvious
way in which they are not: a sequence $u_j \in \mM_k(J;\gamma_k)$ for $k \in\{1,2\}$ will
fail to have a convergent subsequence if and only if for large $j$ it is
of the form $u_j = f_{w_j} \in \mM_1(J;\gamma_1)$ or $u_j = g_{w_j} \in \mM_2(J;\gamma_2)$ for a
sequence $w_j \in \CC$ with $|w_j| \to \infty$.  This gives each of
$\mM_1(J;\gamma_1)$ and $\mM_2(J;\gamma_2)$ the topology of a compact surface with one boundary
component attached to a \emph{cylindrical end} of the form $\CC \setminus \DD_R \cong
(R,\infty) \times S^1$.

If we want to apply these two moduli spaces the same way they were used in
Theorem~\ref{thm:S2xS2}, then we need to establish the following:

\begin{lemma}
\label{lemma:motivation}
The moduli spaces $\mM_1(J;\gamma_1)$ and $\mM_2(J;\gamma_2)$ described above
have the following properties:
\begin{enumerate}
\item For each $k=1,2$ and every pair of distinct elements $u,v \in \mM_k(J;\gamma_k)$,
the images of $u : \CC \to M$ and $v : \CC \to M$ are disjoint.
\item For every $u \in \mM_1(J;\gamma_1)$ and $v \in \mM_2(J;\gamma_2)$, the maps $u : \CC\to M$
and $v : \CC \to M$ have exactly one intersection point, which is transverse
and positive.
\item Every element of $\mM_1(J;\gamma_1)$ or $\mM_2(J;\gamma_2)$ is embedded.
\end{enumerate}
\end{lemma}

Indeed, one can then argue exactly as in the proof of Theorem~\ref{thm:S2xS2}
that the two moduli spaces $\mM_1(J;\gamma_1)$ and $\mM_2(J;\gamma_2)$ form two transverse smooth foliations
of $M$ by planes, producing a coordinate grid (see Figure~\ref{fig:stdAtInfty})
that identifies $M$ with
$\CC \times \CC \cong \RR^4$.  The question I would now like to focus on is this:
why is Lemma~\ref{lemma:motivation} true?

The answer does not come from homological intersection theory, as the curves
in $\mM_1(J;\gamma_1)$ and $\mM_2(J;\gamma_2)$ are noncompact and do not represent homology
classes.  One can however use differential topological arguments to verify
the second claim in the lemma: the fact that each $f_w$ intersects each $g_{w'}$ exactly
once transversely implies via a homotopy argument that the same will be
true for any pair $u \in \mM_1(J;\gamma_1)$ and $v \in \mM_2(J;\gamma_2)$.  Indeed, 
$\mM_1(J;\gamma_1)$ and $\mM_2(J;\gamma_2)$ are each connected spaces of properly
embedded planes that are asymptotic to disjoint loops in~$S^3$, thus they
map neighborhoods of infinity to completely disjoint
regions near infinity in~$M$.  This ensures that there exist homotopies of properly
embedded maps
$$
u_\tau : \CC \to M, \qquad v_\tau : \CC \to M, \qquad \tau \in [0,1]
$$
with $u_0=u$, $u_1=f_w$, $v_0=v$ and $v_1 = g_{w'}$ such that the intersections
of $u_\tau$ with $v_\tau$ for every $\tau \in [0,1]$ are confined to 
compact subsets of both domains.  Standard arguments as in \cite{Milnor:differentiable}
then imply that $u$ and $v$ must have the same algebraic intersection count
as $f_w$ and~$g_{w'}$, which is $1$, so in light of positivity of
intersections, $u$ and $v$ can only have one intersection point and it must
be transverse.

This type of argument does not suffice to prove the other two claims
in Lemma~\ref{lemma:motivation}.  For example, suppose we would like to prove
that two distinct curves $u,v \in \mM_1(J;\gamma_1)$ must always be disjoint.
It is easy to believe this in light of the curves that we can explicitly see,
i.e.~$f_w$ and $f_{w'}$ both belong to $\mM_1(J;\gamma_1)$ for any
$w,w' \in \CC$ sufficiently large, and they are clearly disjoint if
$w \ne w'$.  To extend this to the curves that we cannot explicitly see because
they do not live entirely in the region $(R,\infty) \times S^3 \subset M$,
we would ideally like to argue via homotopy invariance, namely that if
$u_\tau$ and $v_\tau$ are two continuous families of curves in
$\mM_1(J;\gamma_1)$ with $u_0$ and $v_0$ disjoint, then $u_1$ and
$v_1$ must also be disjoint.  But here we have a problem that did not arise in the
previous paragraph: the curves $u_\tau$ and $v_\tau$ in this homotopy are
always asymptotic to \emph{the same} loop $\gamma_1 \subset S^3$,
so their images in $M$ always become arbitrarily close to each other in
the cylindrical end $(R,\infty) \times S^3$.  In this situation, there is
no way to make sure that intersections are confined to compact subsets,
and we can imagine in fact that under a homotopy, some intersections might
just escape to infinity and disappear (see Figure~\ref{fig:escape0})!

\begin{figure}
\psfrag{udotv>0}{$u \cdot v > 0$}
\psfrag{udotv=0}{$u \cdot v = 0$}
\includegraphics{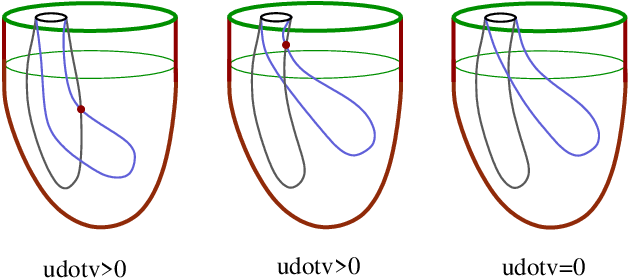}
\caption{\label{fig:escape0} The algebraic intersection count $u \cdot v \in \ZZ$ between
  two proper maps of noncompact domains can change under homotopies if the two
maps have matching asymptotic behavior.}
\end{figure}

It is a remarkable fact that in the situation under consideration, this
nightmare scenario cannot happen, and Lemma~\ref{lemma:motivation}
is indeed true.  To understand why, we will have
to explore the asymptotic behavior of noncompact $J$-holomorphic
curves much more deeply.  Still more interesting perhaps is that in
more general situations, the nightmare scenario of Figure~\ref{fig:escape0}
really can happen, but it can also be \emph{controlled}: one can
define an \emph{asymptotic contribution} that measures the possibility
for ``hidden'' intersections to emerge from infinity under small perturbations.
It turns out that just like the contribution of an isolated intersection
between two $J$-holomorphic curves, this asymptotic contribution is always
nonnegative, and adding it to the algebraic count of actual intersections
produces a meaningful homotopy-invariant intersection product.
Once this product and the corresponding generalization
of the adjunction formula have been understood, proving results like Lemma~\ref{lemma:motivation}
becomes quite easy.

The first hints of a systematic intersection theory for noncompact 
holomorphic curves appeared in Hutchings's work on embedded contact homology
\cite{Hutchings:index}, and the theory was developed in earnest a few years
later in the Ph.D.~thesis of Richard Siefring \cite{Siefring:thesis} and
his two papers \cites{Siefring:asymptotics,Siefring:intersection}.
Our primary objectives in these notes will be to explain where this theory
comes from, demonstrate how to use it, and give some examples of what it can be used for.
We'll start in Lectures~\ref{sec:1} and~\ref{sec:2} by reviewing the
intersection theory for closed holomorphic curves and discussing one
of its most famous applications, McDuff's theorem \cite{McDuff:rationalRuled}
on symplectic ruled surfaces (which is a variation on Theorem~\ref{thm:S2xS2}).
The asymptotic analysis
required for Siefring's theory is then surveyed in
Lecture~\ref{sec:3} (mostly without the proofs since these are analytically
somewhat intense), and Lecture~\ref{sec:4} uses these asymptotic
results to define the precise generalizations of the homological
intersection product and the adjunction formula that are needed for
results such as Lemma~\ref{lemma:motivation}.  
In Lecture~\ref{sec:5}, we will demonstrate how to use the theory
via a generalization of Theorem~\ref{thm:stdAtInfty}, framed in
the language of contact $3$-manifolds and their symplectic fillings.

\chapter{Closed holomorphic curves in symplectic $4$-manifolds}
\label{sec:1}

\minitoc

\vspace{12pt}

In\CUP{This material will be published by Cambridge University
Press as \textsl{Contact 3-Manifolds, Holomorphic Curves and Intersection Theory}
by Chris Wendl. This pre-publication version is
free to view and download for personal use only. 
Not for re-distribution, re-sale or use in derivative works. \copyright Chris Wendl, 2019.} 
these lectures we would like to explain some results about
symplectic $4$-manifolds with contact boundary, and some of the technical tools 
involved in proving them, notably the intersection theory of punctured
pseudoholomorphic curves.  These tools are relatively recent, but have
historical precedents that go back to the late 1980's, when the field of
symplectic topology was relatively new and many deep results about
closed symplectic $4$-manifolds were proved.  We begin with a discussion of
some of those results.

\section{Some examples of symplectic $4$-manifolds}
\label{sec:fibrations}

Suppose $M$ is a smooth manifold of even dimension $2n \ge 2$.  A
\defin{symplectic form}\index{symplectic form}\index{symplectic structure!on a manifold}\index{symplectic manifold}
on $M$ is a closed $2$-form $\omega$ that
is \defin{nondegenerate}, meaning that
$\omega(X,\cdot) \ne 0$ for every nonzero vector~$X \in TM$,
or equivalently, 
$$
\omega^n := \omega \wedge \ldots \wedge \ne 0
$$
everywhere on~$M$.
This means that $\omega^n$ is a volume form, thus it induces a natural
orientation on~$M$.   We will always assume that any symplectic manifold
$(M,\omega)$ carries the natural orientation induced by its symplectic
structure, thus we can write
$$
\omega^n > 0.
$$
We say that a submanifold $\Sigma \subset M$
is a \defin{symplectic submanifold},\index{symplectic submanifold} or is \defin{symplectically embedded},
if $\omega|_{T\Sigma}$ is also nondegenerate.

\begin{exercise}
Show that every finite-dimensional 
manifold admitting a nondegenerate $2$-form has even dimension.
\end{exercise}

There are many interesting questions one can study on a symplectic
manifold $(M,\omega)$, e.g.~one can investigate the \emph{Hamiltonian
dynamics} for a function $H : M \to \RR$, or one can study 
\emph{symplectic embedding obstructions} of one symplectic manifold into
another (see e.g.~\cites{HoferZehnder,McDuffSalamon:ST3} for more on each
of these topics).  In this lecture, we will consider the most basic
question of symplectic topology: given two closed symplectic manifolds
$(M,\omega)$ and $(M',\omega')$ of the same dimension, what properties
can permit us to conclude that they are \defin{symplectomorphic},\index{symplectomorphism}
i.e.~that there exists a diffeomorphism
$$
\varphi : M \stackrel{\cong}{\longrightarrow} M' \quad \text{ with } \quad 
\varphi^*\omega' = \omega?
$$
We shall deal with two fundamental examples of symplectic
manifolds in dimension~$4$, of which the second is a generalization of the first.

\begin{example}
\label{ex:fibration}
Suppose $\Sigma$ is a closed, connected and oriented surface, and
$\pi : M \to \Sigma$ is a smooth fibre bundle whose fibres are
also closed, connected and oriented surfaces.  The following
result of Thurston says that under a mild (and obviously necessary)
homological assumption, such fibrations always carry a
canonical deformation class of symplectic forms.

\begin{thm}[Thurston \cite{Thurston:symplectic}]
\label{thm:Thurston}
Given a fibration $\pi : M \to \Sigma$ as described above,
suppose the homology class of the fibre is not torsion in~$H_2(M)$.  
Then $M$ admits a symplectic form $\omega$ such that
all fibres are symplectic submanifolds of $(M,\omega)$.  Moreover,
the space of symplectic forms on $M$ having this property is connected.
\end{thm}

A symplectic manifold $(M,\omega)$ with a fibration whose fibres are
symplectic is called a \defin{symplectic fibration}.\index{symplectic fibration}\index{fibration!symplectic}
As a special case, if the fibres of $\pi : M \to \Sigma$ are spheres and
$\Sigma$ is a closed oriented surface, then a symplectic fibration $(M,\omega)$
over $\Sigma$ is called a \defin{symplectic ruled surface}.\index{symplectic ruled surface}\index{ruled surface}
This term is inspired by complex algebraic geometry; in particular, the word
``surface'' refers to the fact that such manifolds can also
be shown to admit complex structures, which makes them $2$-dimensional
complex manifolds, i.e.~\emph{complex surfaces}.
\end{example}

\begin{exercise}
Show that the homological condition in Theorem~\ref{thm:Thurston} is
always satisfied if the fibres are spheres.
\textsl{Hint: $A \in H_2(M)$ is a torsion class if and only if the
homological intersection number $A \cdot B \in \ZZ$ vanishes for all
$B \in H_2(M)$.  Consider the vertical subbundle $VM \subset TM \to M$,
defined as the set of all vectors in $TM$ that are tangent to fibres
of $\pi : M \to \Sigma$.
How many times (algebraically) does the zero-set of a generic section of $VM \to M$
intersect a generic fibre of $\pi : M \to \Sigma$?}
\end{exercise}

The above class of examples is a special case of the following more
general class.

\begin{example}
\label{ex:Lefschetz}
Suppose $M$ and $\Sigma$ are closed, connected, oriented, smooth manifolds of
dimensions $4$ and $2$ respectively.  A \defin{Lefschetz fibration}\index{Lefschetz fibration}\index{fibration!Lefschetz}
of $M$ over $\Sigma$ is a smooth map 
$$
\pi : M \to \Sigma
$$
with finitely many critical points $M\crit := \Crit(\pi) \subset M$ and 
critical values $\Sigma\crit := \pi(M\crit) \subset \Sigma$ such that
near each point $p \in M\crit$, there exists a complex coordinate
chart $(z_1,z_2)$ compatible with the orientation of $M$, 
and a corresponding complex coordinate $z$ on a neighborhood of $\pi(p) \in\Sigma\crit$
compatible with the orientation of $\Sigma$, in which
$\pi$ locally takes the form
\begin{equation}
\label{eqn:Morse}
\pi(z_1,z_2) = z_1^2 + z_2^2.
\end{equation}
\end{example}

\begin{remark}
\label{remark:complexStructure}
Any $2n$-dimensional manifold $M$ admits a set of complex coordinates
$(z_1,\ldots,z_n)$ near any point $p \in M$, but it is not always possible
to cover $M$ with such coordinate charts so that the transition maps are
holomorphic; this is possible if and only if $M$ also admits a
\defin{complex structure}.\index{complex structure!on a manifold}\index{complex manifold}\index{almost complex structure!integrable}
In the definition above, we have not assumed
that $M$ admits a complex structure, as the coordinates $(z_1,z_2)$ are only
required to exist locally near the finite set $M\crit$.  Note however
that any choice of complex coordinates on some domain determines an 
orientation on that domain: this follows
from the fact that under the natural identification $\RR^{2n} = \CC^n$, any
complex linear isomorphism $\CC^n \to \CC^n$, when viewed as an element of
$\GL(2n,\RR)$, has positive determinant.  In the above definition, we
are assuming that the given orientations of $M$ and $\Sigma$ always match
the orientations determined by the complex local coordinates.
\end{remark}

A Lefschetz fibration restricts to a smooth fibre bundle over the set
$\Sigma \setminus \Sigma\crit$, and the fibres of this bundle are called
the \defin{regular fibres}\index{regular fiber of a Lefschetz fibration}\index{Lefschetz fibration!regular fiber of}
of $M$; they are in general closed oriented
surfaces, and we may always assume without loss of generality that they
are connected (see Exercise~\ref{EX:connected} below).  The finitely many
\defin{singular fibres}\index{singular fiber of a Lefschetz fibration}\index{Lefschetz fibration!singular fiber of}
$\pi^{-1}(z)$ for $z \in \Sigma\crit$ are 
immersed surfaces with finitely many double points that look like the
transverse intersection of $\CC \times \{0\}$ and $\{0\} \times \CC$
in~$\CC^2$; one can see this by rewriting 
\eqref{eqn:Morse} in the coordinates $\zeta_1 := z_1 + i z_2$ and
$\zeta_2 := z_1 - i z_2$, so that the local model becomes
$\pi(\zeta_1,\zeta_2) = \zeta_1 \zeta_2$.
Each singular fibre is uniquely decomposable 
into a transversely intersecting union of
subsets that are immersed images of \emph{connected} surfaces: we
call these subsets the \defin{irreducible components},\index{Lefschetz fibration!irreducible components of a singular fiber}\index{irreducible components of a Lefschetz singular fiber}
see Figure~\ref{fig:Lefschetz}.

\begin{figure}
\includegraphics{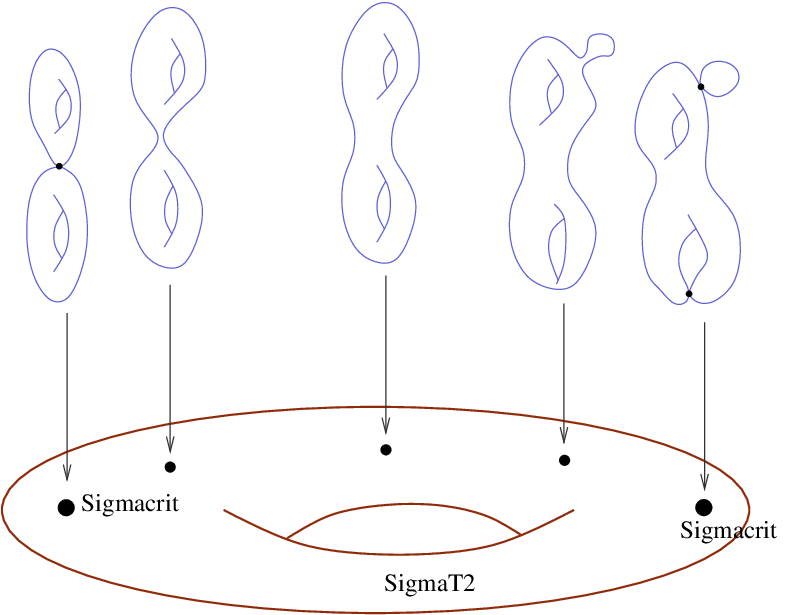}
\caption{\label{fig:Lefschetz} A Lefschetz fibration over 
$\TT^2$ with regular fibres of genus~$2$ and two singular fibres,
each of which has two irreducible components.}
\end{figure}

Thurston's theorem about symplectic structures on fibrations
was generalized to Lefschetz fibrations by Gompf.  To state the most
useful version of this result, we need to generalize the notion of a
``symplectic submanifold'' in a way that will also make sense for
singular fibres, which are not embedded submanifolds.  Since Lefschetz
critical points are defined in terms of complex local coordinates, one way
to do this is by elucidating the relationship between complex and
symplectic structures.

\begin{defn}
\label{defn:tame}
Suppose $E \to B$ is a smooth real vector bundle of even rank.  
A \defin{complex structure}\index{complex structure!on a vector bundle}
on $E \to B$ is a smooth linear bundle map 
$J : E \to E$ such that 
$J^2 = -\1$.  A \defin{symplectic structure}\index{symplectic structure!on a vector bundle}
on $E \to B$ is a smooth
antisymmetric bilinear bundle map
$\omega : E \oplus E \to \RR$ which is nondegenerate, meaning
$\omega(v,\cdot) \ne 0$ for all nonzero $v \in E$.  
We say that $\omega$ \defin{tames}\index{almost complex structure!tamed by a symplectic form}\index{tame almost complex structure}
$J$ if for all
$v \in E$ with $v \ne 0$, we have
$$
\omega(v,Jv) > 0.
$$
We say additionally that $J$ is \defin{compatible}\index{almost complex structure!compatible with a symplectic form}\index{compatible almost complex structure}
with $\omega$ if the pairing
$$
g_J(v,w) := \omega(v,Jw)
$$
is both nondegenerate and symmetric, i.e.~it defines a bundle metric.
\end{defn}
One can show that a complex or symplectic structure on a vector bundle
implies the existence of local trivializations for which all transition
maps are complex linear maps $\CC^n \to \CC^n$ or symplectic linear maps
$\RR^{2n} \to \RR^{2n}$ respectively; see \cite{McDuffSalamon:ST3} for details.
An \defin{almost complex structure}\index{almost complex structure} on a manifold $M$ is simply a complex
structure on its tangent bundle~$TM \to M$.  Here the word ``almost'' is inserted
in order to distinguish this relatively weak notion from the much more rigid notion
mentioned in Remark~\ref{remark:complexStructure}: a complex manifold carries
a natural almost complex structure (defined via multiplication by~$i$ in any
holomorphic coordinate chart), but not every almost complex structure arises
in this way from local charts, and there are many manifolds that admit
\emph{almost} complex structures but not complex structures.
One way to paraphrase Definition~\ref{defn:tame} is to say that $\omega$
tames $J$ if and only if every complex $1$-dimensional subspace of a
fibre in~$E$ is also a \emph{symplectic subspace}; similarly, if
$(M,\omega)$ is a symplectic manifold, then $\omega$ tames an almost complex
structure $J$ on $M$ if and only if every \emph{complex curve} in the 
\emph{almost complex manifold} $(M,J)$ is also a symplectic submanifold.

With this understood, suppose $\pi : M \to \Sigma$ is a Lefschetz fibration
as defined above.  We will say that a symplectic form $\omega$ on $M$ is
\defin{supported by $\pi$}\index{symplectic form!supported by a Lefschetz fibration}
if the following conditions hold:
\begin{enumerate}
\item Every fibre of $\pi|_{M\setminus M\crit} : M \setminus M\crit \to \Sigma$ is a 
symplectic submanifold;
\item On a neighborhood of $M\crit$, $\omega$ tames some almost 
complex structure $J$ that preserves the tangent spaces of the fibres.
\end{enumerate}

Gompf's generalization of Thurston's theorem can now be stated as follows.

\begin{thm}[Gompf \cite{GompfStipsicz}]
\label{thm:Gompf}
Suppose $M$ and $\Sigma$ are closed, connected and oriented manifolds of 
dimensions~$4$ and~$2$ respectively, and $\pi : M \to \Sigma$ is a
Lefschetz fibration for which the fibre represents a non-torsion class
in $H_2(M)$.  Then the space of symplectic forms on $M$ that are supported
by $\pi$ is nonempty and connected.
\end{thm}

A Lefschetz fibration $\pi : M \to \Sigma$ on a symplectic manifold $(M,\omega)$
with $\omega$ supported in the above sense is called a \defin{symplectic
Lefschetz fibration}.\index{symplectic Lefschetz fibration}\index{Lefschetz fibration!symplectic}

\begin{exercise}
\label{EX:connected}
Assuming $M$ and $\Sigma$ are closed and connected,
show that if $\pi : M \to \Sigma$ is a Lefschetz
fibration with disconnected fibers, then one can write
$\pi = \varphi \circ \pi'$ where $\varphi : \Sigma' \to \Sigma$ is a
finite covering map of degree at least~$2$ and $\pi' : M \to \Sigma'$
is a Lefschetz fibration with connected fibers.
\end{exercise}

There is a natural way to replace any smooth fibre bundle 
as in Example~\ref{ex:fibration} with a Lefschetz fibration that has
singular fibres, namely by \emph{blowing up} finitely many points.
Topologically, this can be described as follows: given $p \in M$,
choose local complex coordinates $(z_1,z_2)$ on some neighborhood
$\nN(p) \subset M$ of $p$ that are
compatible with the orientation and identify $p$ with $0 \in \CC^2$.
Let $E \to \CP^1$ denote the tautological complex line bundle,
i.e.~the bundle whose fibre over $[z_1 : z_2] \in \CP^2$ is
the complex line spanned by $(z_1,z_2) \in \CC^2$.  There is a
canonical identification of $E \setminus \CP^1$ with
$\CC^2 \setminus \{0\}$, where
$\CP^1 \subset E$ here denotes the zero-section.
Thus for some neighborhood $\nN(\CP^1) \subset E$ of $\CP^1$,
the above coordinates allow us
to identify $\nN(p) \setminus \{p\}$ with $\nN(\CP^1) \setminus \CP^1$,
and we define the (smooth, oriented) \defin{blowup}\index{blowup!of a complex manifold}\index{complex blowup}
$\widehat{M}$
of $M$ by removing $\nN(p)$ and replacing it with $\nN(\CP^1)$.
There is a natural projection
$$
\Phi : \widehat{M} \to M,
$$
such that $S := \Phi^{-1}(p)$ is a smoothly embedded $2$-sphere 
$S \cong \CP^1 \subset \widehat{M}$ (called an \defin{exceptional sphere}),\index{exceptional sphere}
whose homological self-intersection number\index{self-intersection number!homological}
satisfies
\begin{equation}
\label{eqn:exceptional}
[S] \cdot [S] = -1.
\end{equation}
The restriction of $\Phi$ to $\widehat{M} \setminus S$ is a
diffeomorphism onto $M \setminus \{p\}$.

\begin{exercise}
\label{EX:LefschetzRegular}
Show that if $\pi : M \to \Sigma$ is a Lefschetz fibration and
$p \in M \setminus M\crit$, then there exist complex local coordinates
$(z_1,z_2)$  for a neighborhood of $p$ in~$M$ and $z$ for a 
neighborhood of $\pi(p)$ in~$\Sigma$, both compatible with the
orientations, such that $\pi$ takes the form $\pi(z_1,z_2) = z_1$
near~$p$.
\end{exercise}
\begin{exercise}
\label{EX:blowupLefschetz}
Suppose $\pi : M \to \Sigma$ is a Lefschetz fibration, and $\widehat{M}$
is obtained by blowing up $M$ at a point $p \in M \setminus M\crit$,
using a complex coordinate chart as in Exercise~\ref{EX:LefschetzRegular}.
Then if $\Phi : \widehat{M} \to M$ denotes the induced projection map, 
show that $\pi \circ \Phi : \widehat{M} \to \Sigma$ is a Lefschetz fibration,
having one more critical point than $\pi : M \to \Sigma$ and containing
the exceptional sphere $\Phi^{-1}(p)$ as an irreducible component of a
singular fibre.
\end{exercise}

\begin{exercise}
Prove that the sphere $S \subset \widehat{M}$ created by blowing up $M$ at a point
satisfies \eqref{eqn:exceptional}.
\textsl{Hint: You only need to know the first Chern number of the tautological line bundle.}
\end{exercise}
\begin{exercise}
Prove that if $\widehat{M}$ is constructed by blowing up $M$ at a point, then
$\widehat{M}$ is diffeomorphic to the connected sum $M \# \overline{\CP^2}$,
where the line over $\CP^2$ indicates that it carries the opposite of its
canonical orientation (determined by the complex structure of $\CP^2$).
\textsl{Hint: Present $\CP^2$ as the union of $\CC^2$ with a ``sphere at infinity''
$\CP^1 \subset \CP^2$.  What does a tubular neighborhood of $\CP^1$ 
in $\CP^2$ look like, and what changes if you reverse the orientation?}
\end{exercise}

It is easy to prove from the above description of the blowup that if $M$ 
is a complex manifold, $\widehat{M}$ inherits a canonical complex
structure.  What is somewhat less obvious, but nonetheless true and
hopefully not so surprising by this point, is that
if $(M,\omega)$ is symplectic, then $\widehat{M}$ also inherits a
symplectic form $\widehat{\omega}$ that is canonical up to smooth
deformation through symplectic forms (see \cite{McDuffSalamon:ST3} or
\cite{Wendl:rationalRuled}*{\S 3.2}).\index{blowup!of a symplectic manifold}\index{symplectic blowup}
In this case, the resulting exceptional sphere is a
symplectic submanifold of $(\widehat{M},\widehat{\omega})$.  Conversely,
if $(M,\omega)$ is any symplectic $4$-manifold containing a
symplectically embedded exceptional sphere $S \subset M$, then one
can reverse the above operation and show that $(M,\omega)$ is
the symplectic blowup of another symplectic manifold $(M_0,\omega_0)$,
with the resulting projection $\Phi : M \to M_0$ collapsing $S$
to a point.  We say that a symplectic $4$-manifold is \defin{minimal}\index{minimal symplectic manifold}\index{symplectic manifold!minimal}
if it contains no symplectically embedded exceptional spheres, which
means it is not the blowup of any other symplectic manifold.
McDuff \cite{McDuff:rationalRuled} proved:

\begin{thm}[McDuff \cite{McDuff:rationalRuled}]
If $(M,\omega)$ is a closed symplectic $4$-manifold with
a maximal collection of pairwise disjoint exceptional spheres 
$E_1,\ldots,E_N \subset (M,\omega)$, then the symplectic manifold 
obtained from $(M,\omega)$ by ``blowing down'' 
along $E_1,\ldots,E_N$ is minimal.
\end{thm}

One can also show that if $\omega$ is supported by
a Lefschetz fibration $\pi : M \to \Sigma$, then
the symplectic form $\widehat{\omega}$ on the blowup $\widehat{M}$ can be
arranged to be supported by the Lefschetz fibration on $\widehat{M}$ arising
from Exercise~\ref{EX:blowupLefschetz};
see e.g.~\cite{Wendl:rationalRuled}*{Theorem~3.44}.

Symplectic fibrations are a rather special class of symplectic $4$-manifolds,
but the following deep theorem of Donaldson indicates that Lefschetz fibrations
are surprisingly general examples.  The theorem is actually true in all
dimensions; we will not make use of it in any concrete way in these notes, but
it is important to have as a piece of background knowledge.

\begin{thm}[Donaldson \cite{Donaldson:Lefschetz}]
Any closed symplectic manifold can be blown up finitely many times to a symplectic
manifold which admits a symplectic Lefschetz fibration over~$S^2$.
\end{thm}

\section{McDuff's characterization of symplectic ruled surfaces}
\label{sec:ruled}

If $(M,\omega)$ is a symplectic $4$-manifold with a supporting
Lefschetz fibration $\pi : M \to \Sigma$, then it admits a
$2$-dimensional symplectic submanifold $S \subset (M,\omega)$
satisfying
$$
[S] \cdot [S] = 0;
$$
indeed, $S$ can be chosen to be any regular fibre of the Lefschetz
fibration.  The following remarkable result says that if $S$ has
genus~$0$, then the converse also holds.

\begin{thm}[McDuff \cite{McDuff:rationalRuled}]
\label{thm:McDuff}
Suppose $(M,\omega)$ is a closed and connected
symplectic $4$-manifold, and
$S \subset M$ is a symplectically embedded $2$-sphere
satisfying $[S] \cdot [S] = 0$.  Then $S$ is a fibre of
a symplectic Lefschetz fibration $\pi : M \to \Sigma$
over some closed oriented surface~$\Sigma$, and
$\pi$ is a smooth symplectic fibration (i.e.~without
Lefschetz critical points) whenever
$(M\setminus S,\omega)$ is minimal.  In particular,
$(M,\omega)$ can be obtained by blowing up a symplectic ruled\index{symplectic ruled surface}
surface finitely many times.
\end{thm}

This theorem is false for surfaces $S$ with positive genus
(see Remark~\ref{remark:higherGenus} for more on this).
There is also no comparably strong result
about symplectic fibrations in dimensions greater than~$4$, as the
theory of holomorphic curves is considerably stronger in low
dimensions.  Our main goal for the rest of this lecture will be to sketch
a proof of the theorem.

The proof begins with the observation,
originally due to Gromov \cite{Gromov}, that every symplectic manifold
$(M,\omega)$ admits an almost complex structure $J$ that is \emph{compatible}
with $\omega$ in the sense of Definition~\ref{defn:tame}.
Moreover, if $S \subset (M,\omega)$ is a symplectic
submanifold, one can easily choose a compatible almost complex structure~$J$
that preserves $TS$, i.e.~it makes $S$ into an embedded $J$-complex
curve.  The main idea of the proof is then to study the entire space of
$J$-complex curves homologous to~$S$ and show that these must foliate
$M$, possibly with finitely many singularities.  

Let us define the ``space of $J$-complex curves'' more precisely.
Recall that a \defin{Riemann surface}\index{Riemann surface}\index{almost complex structure!integrable}
can be regarded as an almost 
complex\footnote{Due to a theorem of
Gauss, every almost complex structure on a manifold of real dimension~$2$
is \emph{integrable}, i.e.~it arises from an atlas of coordiate charts
with holomorphic transition maps and is thus also a complex structure
(without the ``almost'').}
manifold $(\Sigma,j)$
with\footnote{Unless otherwise noted, all dimensions mentioned in these
notes will be \emph{real} dimensions, not complex.} 
$\dim \Sigma = 2$.  Given $(\Sigma,j)$ and
an almost complex manifold $(M,J)$ of real dimension~$2n$, we say that a smooth map
$u : \Sigma \to M$ is \defin{$J$-holomorphic}, or \defin{pseudoholomorphic}\index{pseudoholomorphic curve|see {holomorphic curve}}\index{J-holomorphic curve|see {holomorphic curve}}\index{holomorphic curve}
(often abbreviated simply as ``holomorphic''),
if its tangent map is complex linear at every point, i.e.
\begin{equation}
\label{eqn:CR}
Tu \circ j \equiv J \circ Tu.
\end{equation}
This is a first order elliptic PDE: in any choice of holomorphic local
coordinates $s+it$ on a domain in $\Sigma$, \eqref{eqn:CR} is equivalent
to the nonlinear Cauchy-Riemann type equation
$$
\p_s u(s,t) + J(u(s,t))\, \p_t u(s,t) = 0.
$$
Solutions are called \defin{pseudoholomorphic curves}, where the word
``curve'' refers to the fact that their domains
are complex \emph{one}-dimensional manifolds.   They have many nice
properties, which are proved by a combination of complex function theory,
nonlinear functional analysis and elliptic regularity theory---a quick overview
of the essential properties is given in Appendix~\ref{app:curves},
and some of these will be used in the following discussion.

For any integer $g \ge 0$ and $A \in H_2(M)$, we define the \defin{moduli
space} $\mM_g^A(M,J)$ of \defin{unparametrized closed $J$-holomorphic curves}\index{moduli space!of closed holomorphic curves}
of genus~$g$ homologous to~$A$ as the space of
equivalence classes $[(\Sigma,j,u)]$, where $(\Sigma,j)$ is a closed
connected Riemann surface of genus~$g$, $u : (\Sigma,j) \to (M,J)$ is a
pseudoholomorphic map representing the homology class $[u] := u_*[\Sigma] = A$,
and we write $(\Sigma,j,u) \sim (\Sigma',j',u')$ if and only if
they are related to each other by reparametrization, i.e.~there exists
a holomorphic diffeomorphism $\varphi : (\Sigma,j) \to (\Sigma',j')$
(a \defin{biholomorphic} map)\index{biholomorphic} such that $u = u' \circ \varphi$.
We will sometimes abuse notation and abbreviate an equivalence
class $[(\Sigma,j,u)] \in \mM_g^A(M,J)$ simply as the parametrization
``$u$'' when there is no danger of confusion.
The notion of $C^\infty$-convergence defines a natural topology on
$\mM_g^A(M,J)$ such that a sequence $[(\Sigma_k,j_k,u_k)] \in \mM_g^A(M,J)$
converges to $[(\Sigma,j,u)] \in \mM_g^A(M,J)$ if and only if there exist
representatives $(\Sigma,j_k',u_k') \sim (\Sigma_k,j_k,u_k)$ for which
$$
j_k' \to j \quad\text{ and }\quad u_k' \to u
$$
uniformly with all derivatives on~$\Sigma$.  In cases where we'd prefer
not to specify the homology class, we will occasionally write
$$
\mM_g(M,J) := \coprod_{A \in H_2(M)} \mM_g^A(M,J).
$$

Observe that if $u : (\Sigma,j) \to (M,J)$ is a closed $J$-holomorphic
curve and $\varphi : (\Sigma',j') \to (\Sigma,j)$ is a holomorphic
map from another closed Riemann surface $(\Sigma',j')$, then
$u \circ \varphi : (\Sigma',j') \to (M,J)$ is also a $J$-holomorphic
curve.  If $\varphi$ is nonconstant, then holomorphicity implies that
it has degree $\deg(\varphi) \ge 1$, with equality if and only if it
is biholomorphic; in the case $k:= \deg(\varphi) > 1$, we then say that
$u'$ is a \defin{$k$-fold multiple cover} of~$u$.\index{holomorphic curve!multiply covered}\index{multiply covered holomorphic curve}
Note that in this situation, $[u'] = k[u]$, so for
instance, a curve cannot be a multiple cover if it represents a
primitive homology class.  We say that a nonconstant closed
$J$-holomorphic curve is \defin{simple}\index{holomorphic curve!simple}\index{simple holomorphic curve}
if it is not a multiple
cover of any other curve.

Returning to the specific situation of McDuff's theorem, assume $J$ is
an $\omega$-compatible almost complex structure that preserves the tangent
spaces of the symplectically embedded sphere $S \subset (M,\omega)$.
Then $\left(S,J|_{TS}\right)$ is a closed Riemann surface of genus~$0$,
and its inclusion $u_S : S \hookrightarrow M$ is an embedded 
$J$-holomorphic curve, defining an element
$$
u_S \in \mM_0^{[S]}(M,J)
$$
in the moduli space of $J$-holomorphic spheres homologous so~$S$.
A straightforward application of standard machinery now gives the following 
result, a proof of which may be found at the end of Appendix~\ref{app:closed}.

\begin{lemma}
\label{lemma:M0}
After a $C^\infty$-small perturbation of $J$ outside a neighborhood
of~$S$, the open subset $\mM_0^{[S],*}(M,J) \subset \mM_0^{[S]}(M,J)$, 
consisting of \emph{simple} $J$-holomorphic spheres homologous to $[S]$, is
a smooth oriented $2$-dimensional manifold, and it is ``compact up to
bubbling'' in the following sense.  There exists a finite set of
simple curves $\mathscr{B} \subset \mM_0(M,J)$ with positive first Chern
numbers such that if $u_k \in \mM_0^{[S],*}(M,J)$
is a sequence with no convergent subsequence in $\mM_0^{[S]}(M,J)$,
then it has a subsequence that degenerates (see Figure~\ref{fig:spheres})
to a \emph{nodal curve} 
$\{v_+,v_-\} \in \overline{\mM}_0^{[S]}(M,J)$ for some
$v_+, v_- \in \mathscr{B}$.
\end{lemma}

\begin{figure}
\includegraphics{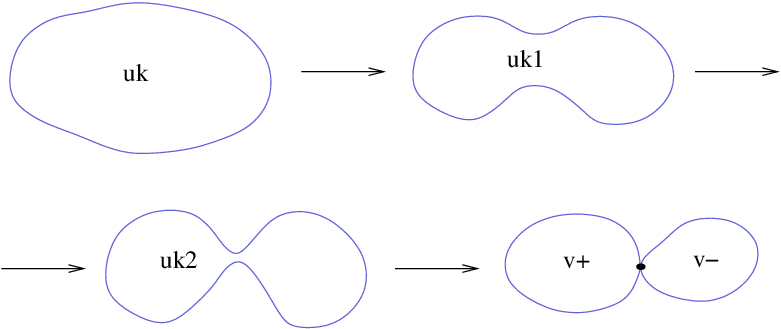}
\caption{\label{fig:spheres} A sequence of $J$-holomorphic spheres $u_k$
degenerating to a nodal curve $\{v_+,v_-\}$.}
\end{figure}

The above formulation is a bit lazy since we have not as yet given any
definition of the space $\overline{\mM}_0(M,J)$ of \emph{nodal curves}.\index{holomorphic curve!nodal|(}\index{nodal holomorphic curve|(}
More precise details of this \emph{compactification} of\index{moduli space!compactification of}\index{compactification of moduli spaces}
$\mM_0(M,J)$ may be found in Appendix~\ref{app:closed}, but for the purposes
of the present discussion, it will suffice to characterize the
degeneration of a sequence
$[(S^2,j_k,u_k)] \in \mM_0^A(M,J)$ to a
nodal curve $\{ [(S^2,j_+,v_+)],[(S^2,j_-,v_-)]\} 
\in \overline{\mM}_0^A(M,J)$ as follows.  
The nodal curve is assumed to
have the property that any choice of representatives 
$(S^2,j_\pm,v_\pm)$ comes with a distinguished intersection 
$$
v_+(z_+) = v_-(z_-),
$$
for some pair of points $z_\pm \in S^2$; this intersection is called
the \defin{node}.\index{nodes of a holomorphic curve}
Given these parametrizations,
let $C \subset S^2$ denote the equator of the sphere, separating it into the
two hemispheres
$$
S^2 = \dD_+ \cup_C \dD_-,
$$
and choose continuous surjections $\varphi_\pm : \dD_\pm \to S^2$
that map $\mathring{\dD}_\pm$ diffeomorphically to~$S^2 \setminus \{z_\pm\}$ 
and collapse $C$ to~$z_\pm$.  The map
$$
u_\infty : S^2 \to M : z \mapsto \begin{cases}
v_- \circ \varphi_-(z) & \text{ for $z \in \dD_-$,}\\
v_+ \circ \varphi_+(z) & \text{ for $z \in \dD_+$}
\end{cases}
$$
is then continuous, and smooth on $S^2 \setminus C$.  This also
defines a complex structure on $S^2 \setminus C$ by
$$
j_\infty := \begin{cases}
\varphi_-^*j_- & \text{ on $\mathring{\dD}_-$,}\\
\varphi_+^*j_+ & \text{ on $\mathring{\dD}_+$,}
\end{cases}
$$
though $j_\infty$ does not extend smoothly over~$C$.
Now the convergence
$u_k \to \{ v_+,v_- \}$ can be defined to mean that all of the
above choices can be made together with choices of representatives
$(S^2,j_k,u_k)$ such that
\begin{equation*}
\begin{split}
u_k \to u_\infty \quad &\text{ in $C^0(S^2,M)$ and $C^\infty_{\text{loc}}(S^2 \setminus C,M)$, and}\\
j_k \to j_\infty \quad &\text{ in $C^\infty_{\text{loc}}(S^2 \setminus C)$.}
\end{split}
\end{equation*}
Observe that as a result of the $C^0$-convergence, $[v_+] + [v_-] = A
\in H_2(M)$.\index{holomorphic curve!nodal|)}\index{nodal holomorphic curve|)}

Lemma~\ref{lemma:M0} relies on very general properties of
$J$-holomorphic curves that are valid in all dimensions; under a few
extra assumptions, some version
of the same result could be proved for a $2n$-dimensional
symplectic manifold $(M,\omega)$ containing a symplectically embedded
$2$-sphere $S \subset M$ with trivial normal bundle.  The following
improvement, which we will prove in Lecture~\ref{sec:2} 
(see \S\ref{sec:finishMcDuff}), is unique to dimension~$4$:

\begin{lemma}
\label{lemma:M1}
The finitely many nodal curves 
$$
\{ v_+^1,v_-^1 \}, \ldots , \{ v_+^N, v_-^N \} \in \overline{\mM}_0^{[S]}(M,J)
$$
appearing in Lemma~\ref{lemma:M0} have the following properties:
\begin{enumerate}
\item Each $v_\pm^i : S^2 \to M$ for $i=1,\ldots,N$ 
is embedded and satisfies $[v_\pm^i] \cdot [v_\pm^i] = -1$;
\item $v_+^i$ and $v_-^i$ for $i=1,\ldots,N$ intersect each other
exactly once, transversely;
\item For $i,j \in \{1,\ldots,N\}$ with $i \ne j$,
$$
v_i^+(S^2) \cap v_j^+(S^2) = v_i^+(S^2) \cap v_j^-(S^2)
= v_i^-(S^2) \cap v_j^-(S^2) = \emptyset.
$$
\end{enumerate}
Moreover, if $F \subset M$ denotes the union of all the images of these
nodal curves, then the curves
in $\mM_0^{[S]}(M,J)$ are all embedded and have pairwise
disjoint images that define a smooth foliation of some open subset
of $M \setminus F$.
\end{lemma}

With this lemma at our disposal, the proof of Theorem~\ref{thm:McDuff} 
concludes as follows: let
$$
X := \left\{ p \in M \setminus F\ \big|\ \text{$p$ is in the image of a curve
in $\mM_0^{[S]}(M,J)$} \right\}.
$$
Lemma~\ref{lemma:M1} guarantees that $X$ is an open subset of
$M \setminus F$, but by the compactness statement in
Lemma~\ref{lemma:M0}, $X$ is also a closed subset.  Since
$M \setminus F$ is connected, we conclude that the curves in
$\mM_0^{[S]}(M,J)$ fill \emph{all} of it.
Now, the compactified moduli space $\overline{\mM}_0^{[S]}(M,J)$
consists of $\mM_0^{[S]}(M,J)$ plus finitely many additional
elements in the
form of nodal curves; it has the topology of some compact 
oriented $2$-manifold $\Sigma$, and the above argument shows that
every point in $M$ is in the image of precisely one element of
$\overline{\mM}_0^{[S]}(M,J)$.  This defines a map
$$
\pi : M \to \overline{\mM}_0^{[S]}(M,J) \cong \Sigma,
$$
whose regular fibres are the images of the smoothly embedded curves
in $\mM_0^{[S]}(M,J)$, and the images of nodal curves give rise to
Lefschetz singular fibres, each with a unique critical point where
two embedded $J$-holomorphic spheres intersect transversely.  
Since all the fibres are images of
$J$-holomorphic curves and $J$ is $\omega$-tame, the fibres 
are also symplectic submanifolds.
Furthermore, the irreducible components of the singular fibres 
are exceptional spheres that are disjoint from $S$ (since the
latter is also a fibre), thus no singular fibres can exist if
$(M \setminus S,\omega)$ is minimal.

\begin{remark}
One can also prove the converse of the statement about minimality,
i.e.~if the Lefschetz fibration has no singular fibres then
$(M\setminus S,\omega)$ must be minimal.  This relies on another
theorem of McDuff \cite{McDuff:rationalRuled}, that for generic~$J$,
any exceptional sphere is homologous to a unique $J$-holomorphic
sphere, which is embedded.
A more comprehensive exposition of this topic and the more general
version of McDuff's theorem
for \emph{rational and ruled} symplectic $4$-manifolds is given in
\cite{Wendl:rationalRuled}; see also \cite{LalondeMcDuff:rationalRuled}.
\end{remark}

\section{Local foliations by holomorphic spheres}
\label{sec:foliations}

The distinctive power of holomorphic curve methods in dimension four
results from the numerical coincidence that $2 + 2 = 4$: in particular, any pair of
holomorphic curves $u \in \mM_g^A(M,J)$ and $v \in \mM_{g'}^{A'}(M,J)$ has a well-defined
homological intersection number $[u] \cdot [v] = A \cdot A' \in \ZZ$.
We will discuss this subject in earnest in \S\ref{sec:adjunction},
but before that, let us examine a slightly simpler phenomenon that is 
also distinctive to dimension~$4$ and important
for the proof of Lemma~\ref{lemma:M1}.  

Suppose $(M,J)$ is a $2n$-dimensional
almost complex manifold and $u \in \mM_0^A(M,J)$ is an
embedded $J$-holomorphic curve such that the normal bundle $N_u \to S^2$ to
any parametrization $u : S^2 \hookrightarrow M$ is trivial.  Since 
$du(z) : (T_z S^2,j) \to (T_{u(z)} M,J)$ is complex linear and injective for all
$z \in S^2$, the normal bundle naturally inherits a complex structure such that
$$
u^*TM \cong TS^2 \oplus N_u
$$
as complex vector bundles, so the first Chern numbers of these bundles satisfy
$$
c_1(u^*TM) = c_1(TS^2) + c_1(N_u) = \chi(S^2) + 0 = 2,
$$
where $c_1(u^*TM)$ is shorthand for evaluation of $c_1(u^*TM,J) \in H^2(S^2)$
on the fundamental class:
$$
c_1(u^*TM) := \langle c_1(u^*TM,J) , [S^2] \rangle = \langle u^* c_1(TM,J), [S^2] \rangle
= \langle c_1(TM,J) , u_*[S^2] \rangle =: c_1(A).
$$
If $\dim M = 4$, then triviality of $N_u$ implies that 
$u(S^2)$ is a symplectically embedded sphere with self-intersection\index{self-intersection number!homological}
number~$0$, and we saw in Lemma~\ref{lemma:M0} that in this case
$\dim \mM_0^A(M,J) = 2$.  More generally, plugging $\dim M = 2n$ and
$c_1(A) = 2$ into the virtual dimension formula \eqref{eqn:dimension}
in Appendix~\ref{app:closed} gives
$$
\virdim \mM_0^A(M,J) = 2(n-3) + 2 c_1(A) = 2n-2.
$$
This means more precisely that if $J$ is sufficiently generic, then 
the open subset of
$\mM_0^A(M,J)$ consisting only of simple curves is a smooth manifold of this
dimension, and since $u$ itself is embedded, this is true in particular
for some neighborhood of $u$ in $\mM_0^A(M,J)$.  Note also that embeddedness
of spheres in~$M$ is an open condition, so all other curves near~$u$ are also
embedded.  This observation and the dimension computation above
make the following question reasonable:

\begin{question}
\label{foliation}
Do the curves near $u$ in $\mM_0^A(M,J)$ foliate a neighborhood
of $u(S^2)$?
\end{question}

To answer this, let us choose a Riemannian metric on~$M$ and
assume there exists a smooth family of parametrizations for the curves
near~$u$ via sections of its normal bundle, 
i.e.~one can find a smooth map
\begin{equation}
\label{eqn:embedding}
\Psi : \DD^{2n-2} \times S^2 \to M : (\sigma,z) \mapsto u_\sigma(z) := \exp_{u(z)} h_\sigma(z)
\end{equation}
with $h_\sigma \in \Gamma(N_u)$ for each $\sigma \in \DD^{2n-2}$ and $h_0 \equiv 0$, 
such that the maps $u_\sigma$ parametrize curves in $\mM_0^A(M,J)$ and $u_0 = u$.
There is then a linear map $\RR^{2n-2} \to \Gamma(N_u) : X \mapsto \eta_X$
defined by
$$
\eta_X(z) = d\Psi(0,z) (X,0) = \left. \frac{d}{dt} u_{tX}(z)\right|_{t=0},
$$
and the image of this map can be identified with the tangent space
$T_u \mM_0^A(M,J)$.
Using the fact that each $u_\sigma : S^2 \to M$ satisfies a nonlinear
Cauchy-Riemann type equation, one can show that all the sections $\eta_X$
satisfy some \emph{linearized} Cauchy-Riemann type equation (cf.~Appendix~\ref{sec:CRoperators}):
in particular, for any choice of local holomorphic coordinates $s+it$
identifying a domain $\uU \subset \Sigma$ with some open set
$\Omega \subset \CC$, the local expression for $\eta_X$ in a complex
trivialization over~$\uU$ is a function $f : \Omega \to \CC^{n-1}$
satisfying a linear PDE of the form
\begin{equation}
\label{eqn:linearCR}
\p_s f(s,t) + i\, \p_t f(s,t) + A(s,t) f(s,t) = 0,
\end{equation}
for some smooth function $A(s,t)$ valued in the space
$\End_\RR(\CC^{n-1})$ of
real-linear maps on~$\CC^{n-1}$.  Except for the extra
$0$th-order term, this is the standard Cauchy-Riemann equation,
and we might therefore expect $f$ to have similar
properties to an analytic function $\Omega \to \CC^{n-1}$,
e.g.~its zeroes should be isolated unless $\eta_X \equiv 0$.
This intuition is made precise by the following consequence of
elliptic regularity theory, often called the 
\defin{similarity principle};\index{similarity principle}
a slight generalization of this result is
stated and proved in Appendix~\ref{sec:similarity}).

\begin{thm}[\textsl{similarity principle}]
\label{thm:similarity}
Suppose $\Omega \subset \CC$ is an open set, $N \in \NN$,\index{similarity principle|(}
$A : \Omega \to \End_{\RR}(\CC^N)$ is smooth, 
$f : \Omega \to \CC^N$ is a smooth function satisfying
the equation \eqref{eqn:linearCR}, and $z_0 \in \Omega$
is a point with $f(z_0) = 0$.  Then $f$ can be written
on some neighborhood $z_0 \in \uU \subset \Omega$ as
\begin{equation}
\label{eqn:similarity}
f(z) = \Phi(z) g(z), \qquad z \in \uU,
\end{equation}
for some continuous function $\Phi : \uU \to \End_{\CC}(\CC^N)$
with $\Phi(z_0) = \1$ and a holomorphic function
$g : \uU \to \CC^N$.  Moreover, if $A$ is complex linear at every point,
then $\Phi$ can be taken to be smooth.
\end{thm}
\begin{cor}
\label{cor:positiveOrder}
Given $f : \Omega \to \CC^N$ as in Theorem~\ref{thm:similarity},
$f$ is either identically zero or has only isolated zeroes.
In the latter case, if $N=1$, all zeroes of $f$ have positive order.\index{zeroes of a section!positivity of}
\end{cor}
\begin{proof}
Writing $f(z) = \Phi(z) g(z)$ as in \eqref{eqn:similarity}
for $z$ in a neighborhood $\uU$ of $z_0$, we can assume 
after shrinking $\uU$ that $\Phi(z)$ is close to $\1$ and thus 
invertible for all $z \in \uU$.  Then $f|_{\uU}$ is identically
zero if and only if $g|_{\uU}$ is, and otherwise, $g$ has an isolated
zero at~$z_0$ and thus so does~$f$.  If the latter holds and also $N=1$, 
then we can further conclude that the winding number of the loop
$$
\RR / \ZZ \to \CC \setminus \{0\} : 
\theta \mapsto g(z_0 + \epsilon e^{2\pi i \theta})
$$
for small $\epsilon > 0$ is positive, and since $\Phi$ is close
to the identity, the same is true for~$f$.
\end{proof}

The similarity principle implies that sections
$\eta_X \in T_u \mM_0^A(M,J)$ have at most finitely many zeroes in general,
but it implies much more than this in the case where $\dim M = 4$.
Indeed, $N_u \to S^2$ is in this case a complex \emph{line} bundle,
so for any section of this bundle with only isolated zeroes, the algebraic
count of the zeroes is given by the first Chern number $c_1(N_u) \in \ZZ$,
which vanishes since the bundle is trivial.  But by
Corollary~\ref{cor:positiveOrder}, the zeroes of any nontrivial section
$\eta_X \in T_u \mM_0^A(M,J)$ all count positively,\footnote{zeroes of a section!positivity of}
so it follows that
there cannot be any: $\eta_X$ is nowhere zero!  This is true for
all $X \ne 0$, and thus implies that $d\Psi(0,z) : T_{(0,z)} (\DD^2 \times S^2)
\to T_{u(z)} M$ is an isomorphism for all $z \in S^2$, hence the map
\eqref{eqn:embedding} is an embedding in some neighborhood of
$\{0\} \times S^2$, giving a positive answer to 
Question~\ref{foliation}:

\begin{prop}
\label{prop:foliation}
If $\dim M = 4$ and $u \in \mM_0^A(M,J)$ is an embedded $J$-holomorphic sphere
with trivial normal bundle, then the images of the 
curves in $\mM_0^A(M,J)$ near~$u$ foliate a neighborhood of the image of~$u$.
\qed
\end{prop}

No such general result is possible when
$\dim M > 4$, because there is no way to ``count'' the number of zeroes 
of a section of a higher rank complex vector bundle over~$S^2$.

\begin{exercise}
Suppose $L \to \Sigma$ is a complex line bundle over a closed Riemann
surface $(\Sigma,j)$, and $V \subset \Gamma(L)$ is a vector space of
sections that satisfy a real-linear Cauchy-Riemann type equation, so
in particular the similarity principle holds for sections $\eta \in V$.
Prove $\dim_\RR V \le 2 + 2 c_1(L)$.\index{similarity principle|)}
\end{exercise}

\chapter{Intersections, ruled surfaces, and contact boundaries}
\label{sec:2}

\minitoc

\vspace{12pt}

In\CUP{This material will be published by Cambridge University
Press as \textsl{Contact 3-Manifolds, Holomorphic Curves and Intersection Theory}
by Chris Wendl. This pre-publication version is
free to view and download for personal use only. 
Not for re-distribution, re-sale or use in derivative works. \copyright Chris Wendl, 2019.}
this lecture we explain the intersection theory for closed
holomorphic curves in dimension~$4$ and use it to complete
the overview from Lecture~\ref{sec:1} of
McDuff's theorem on ruled surfaces.
We will then begin discussing the generalization of these ideas to 
punctured holomorphic curves in symplectic cobordisms, and some
applications to the study of symplectic fillings.

\section{Positivity of intersections and the adjunction formula}
\label{sec:adjunction}

To complete the proof of Lemma~\ref{lemma:M1} from \S\ref{sec:ruled}, 
we must discuss the
intersection theory of $J$-holomorphic curves in dimension~$4$.  The notion of
``homological'' intersection numbers was mentioned already a few times in 
the previous lecture,
and it will be useful now to review precisely what this means.
Suppose $M$ is a closed oriented smooth $4$-manifold,
$\Sigma$ and $\Sigma'$ are closed oriented surfaces, and
$$
u : \Sigma \to M, \qquad v : \Sigma' \to M
$$
are $C^1$-smooth maps.
An intersection $u(z) = v(\zeta) = p$ is \defin{transverse}\index{transverse}\index{transversality!for intersections}
if
\begin{equation}
\label{eqn:directSum}
\im du(z) \oplus \im dv(\zeta) = T_p M,
\end{equation}
and \defin{positive} if and only if the natural orientation induced on this
direct sum by the orientations of $T_z\Sigma$ and $T_\zeta \Sigma'$
matches the orientation of $T_p M$.  Otherwise it is called \defin{negative},
and we define the \defin{local intersection index}\index{local intersection index}\index{intersection number!local|see {local intersection index}}
accordingly as
$\inter(u , z \,;\, v , \zeta) = \pm 1$.
If all intersections between $u$ and $v$ are transverse,
then they are all isolated and thus there are only finitely many, so we
can define the total \defin{intersection number}\index{intersection number!homological}
$$
[u] \cdot [v] := \sum_{u(z) = v(\zeta)} \inter(u,z \,;\, v,\zeta) \in \ZZ.
$$
The choice of notation reflects the fact that $[u] \cdot [v]$
turns out to depend only on the homology classes $[u], [v] \in H_2(M)$;
in fact, it defines a nondegenerate bilinear symmetric form
$$
H_2(M) \otimes H_2(M) \to \ZZ : [u] \otimes [v] \mapsto [u] \cdot [v].
$$
More details on this may be found e.g.~in \cite{Bredon}.

If $u$ and $v$ have an isolated but non-transverse intersection at 
$u(z) = v(\zeta) = p$, one can still define a local intersection
index $\inter(u,z \,;\, v,\zeta) \in \ZZ$ as follows.  By assumption, $z$ and
$\zeta$ each lie in the interiors of smoothly embedded closed disks
$\dD_z \subset \Sigma$ and $\dD_\zeta \subset \Sigma'$ respectively
such that
$$
u(\dD_z \setminus \{z\}) \cap v(\dD_\zeta \setminus \{\zeta\}) = \emptyset.
$$
Then one can find a $C^\infty$-small perturbation $u_\epsilon$ of $u$ such that
$u_\epsilon|_{\dD_z} \pitchfork v|_{\dD_\zeta}$ but $u_\epsilon(\p\dD_z)$ and
$v(\p\dD_\zeta)$ remain disjoint.  We set
$$
\inter(u,z \,;\, v,\zeta) := \sum_{u_\epsilon(z') = v(\zeta')} \inter(u_\epsilon,z' \,;\, v,\zeta') \in \ZZ,
$$
where the sum is restricted to pairs $(z',\zeta') \in \dD_z \times \dD_\zeta$.

\begin{exercise}
\label{EX:interWithBoondary}
Suppose $\Sigma$ and $\Sigma'$ are compact oriented surfaces with
boundary, $M$ is a smooth oriented $4$-manifold and 
$$
f_\tau : \Sigma \to M, \qquad g_\tau : \Sigma' \to M, \qquad \tau \in [0,1]
$$
are homotopies\footnote{We are not specifying the regularity of the homotopy
in this statement because it does not matter: one can use general perturbation
results as in \cite{Hirsch} to replace any continuous homotopy between two
$C^1$-smooth maps with a homotopy of class~$C^1$.  If desired, one can also
perturb all of the maps to make them smooth.} of maps with the property that for all $\tau \in [0,1]$,
$$
f_\tau(\p\Sigma) \cap g_\tau(\Sigma') = f_\tau(\Sigma) \cap g_\tau(\p\Sigma')
= \emptyset.
$$
Show that if $f_\tau$ and $g_\tau$ are of class $C^1$ and have only transverse intersections for
$\tau \in \{0,1\}$, then
\begin{equation}
\label{eqn:this}
\sum_{f_0(z) = g_0(\zeta)} \inter(f_0,z \,;\, g_0,\zeta) =
\sum_{f_1(z) = g_1(\zeta)} \inter(f_1,z \,;\, g_1,\zeta).
\end{equation}
Deduce from this that the above definition of the local intersection 
index for an isolated but non-transverse intersection is well defined and
independent of the choice of perturbation.  Then, show that
\eqref{eqn:this} also holds if the intersections for $\tau\in \{0,1\}$ are assumed
to be isolated but not necessarily transverse.
\textsl{Hint: If you have never read \cite{Milnor:differentiable},
you should.}
\end{exercise}

The following useful result is immediate from the above definition; it can
be paraphrased by saying that ``algebraically nontrivial intersections
cannot be perturbed away.''
\begin{prop}
\label{prop:posPert}
If $u : \Sigma \to M$ and $v : \Sigma' \to M$ have an isolated intersection 
$u(z) = v(\zeta)$ with $\inter(u,z \,;\, v,\zeta) \ne 0$, then for any 
neighborhood $z \in \uU_z \subset \Sigma$, any sufficiently
$C^0$-close perturbation $u_\epsilon$ of~$u$ satisfies
$u(\uU_z) \cap v(\Sigma') \ne \emptyset$.
\qed
\end{prop}

Recall next that any complex structure on a real vector space induces a
preferred orientation.  In the case where $u : (\Sigma,j) \to (M,J)$ and
$v : (\Sigma',j') \to (M,J)$ are both $J$-holomorphic curves, this means
that each space in \eqref{eqn:directSum} carries a canonical orientation
and they are automatically compatible with the direct sum, hence
$\inter(u,z \,;\, v,\zeta) = +1$.  This positivity phenomenon turns out to be true
for non-transverse intersections as well:\index{intersections!positivity of|(}\index{positivity of intersections|(}

\begin{thm}[\textsl{local positivity of intersections}]
\label{thm:positivity}
Suppose $u : (\Sigma,j) \to
(M,J)$ and $v : (\Sigma',j') \to (M,J)$ are nonconstant
pseudoholomorphic maps
with $u(z) = v(\zeta) = p \in M$ for some $z \in \Sigma$,
$\zeta \in \Sigma'$.  Then there exist neighborhoods $z \in \uU_z \subset \Sigma$
and $\zeta \in \uU_\zeta \subset \Sigma'$ such that either
$u(\uU_z) = v(\uU_\zeta)$ or 
$$
u(\uU_z \setminus \{z\}) \cap v(\uU_\zeta \setminus \{\zeta\}) = \emptyset.
$$
Moreover, in the latter case, if $\dim M = 4$ then 
$\inter(u,z \,;\, v,\zeta) \ge 1$,
with equality if and only if the intersection is transverse.
\end{thm}
A proof of this theorem is given in Appendix~\ref{app:positivity}.

To understand the global consequences of Theorem~\ref{thm:positivity}, observe
that there are certain obvious situations where a pair of closed $J$-holomorphic 
curves $u : (\Sigma,j) \to (M,J)$ and $v : (\Sigma',j') \to (M,J)$ have
infinitely many intersections, e.g.~if they represent the same curve up to
parametrization, or they are multiple covers of the same simple curve.  
In such cases, $u$ and $v$ have globally identical images, and we find
neighborhoods with $u(\uU_z) = v(\uU_\zeta)$ in Theorem~\ref{thm:positivity}.
One can show that in all other cases, the set of intersections is finite,
a phenomenon known as \defin{unique continuation}.\index{unique continuation}
Theorem~\ref{thm:positivity} then implies:

\begin{cor}[\textsl{global positivity of intersections}]
\label{cor:positivityGlobal}
If $\dim M = 4$ and $u : (\Sigma,j) \to (M,J)$ and $v : (\Sigma',j') \to (M,J)$ 
are closed connected $J$-holomorhic curves with non-identical images,
then they have finitely many intersections, and
$$
[u] \cdot [v] \ge \# \left\{ (z,\zeta) \in \Sigma \times \Sigma'\ \big|\ 
u(z) = v(\zeta) \right\},
$$
with equality if and only if all the intersections are transverse.  In
particular, $[u] \cdot [v] = 0$ if and only if 
$u(\Sigma) \cap v(\Sigma') = \emptyset$.\index{intersections!positivity of|)}\index{positivity of intersections|)}
\qed
\end{cor}

We next consider the question of how many times a single closed
$J$-holomorphic curve $u : (\Sigma,j) \to (M,J)$ intersects itself
at two distinct points in its domain, i.e.~its count of
\defin{double points}.\index{double points!of a holomorphic curve}\index{holomorphic curve!double points of}
This question obviously has no reasonable answer if $u$ is multiply
covered, so let us assume $u$ is simple, in which case it has only
finitely many double points.  We say that a point $z \in \Sigma$
is a \defin{critical point}\index{critical points of a holomorphic curve}\index{holomorphic curve!critical points of}
of $u$ if
$$
du(z) = 0.
$$

\begin{remark}
\label{remark:critical}
This usage of the term ``critical
point'' conflicts with standard terminology since typically $\dim \Sigma < \dim M$,
hence $du(z)$ can never be surjective and $u$ therefore cannot have 
any regular points, strictly speaking.  Note however that whenever
$du(z) \ne 0$, the Cauchy-Riemann equation implies that $du(z)$ is
injective.  For this reason, we will refer to points with this
property as \defin{immersed points}\index{immersed points of a holomorphic curve}\index{holomorphic curve!immersed points of}
instead of ``regular points''.
\end{remark}

A simple $J$-holomorphic curve can have critical points,
but only finitely many, and their role in intersection theory is 
dictated by the following lemma.
For an oriented surface $\Sigma$ and a symplectic manifold $(M,\omega)$,
we say that a smooth map $u : \Sigma \to M$ is \defin{symplectically
immersed}\index{symplectically immersed} if $u^*\omega > 0$.

\begin{lemma}
\label{lemma:critInj}
If $u \in \mM_g(M,J)$ is simple, then for any parametrization
$u : \Sigma \to M$ and any $z \in \Sigma$, there is a neighborhood
$z \in \uU_z \subset \Sigma$ such that $u|_{\uU_z}$ is injective.
Moreover, if $du(z)=0$, $\dim M = 4$, and $\omega_z$ is an auxiliary
choice of symplectic form defined near $u(z)$ and taming~$J$,
then there exists a positive integer $\delta(u,z) > 0$ depending only on the germ of
$u$ near~$z$, such that
$u|_{\uU_z}$ admits a $C^1$-small perturbation 
to an $\omega_z$-symplectically immersed map 
$u_\epsilon : \uU_z \to M$ that matches $u$ outside
an arbitrarily small neighborhood of~$z$ and satisfies\footnote{Notice
that each geometric double-point $u_\epsilon(\zeta_1) = u_\epsilon(\zeta_2)$ appears
twice in the summation over pairs $(\zeta_1,\zeta_2)$, hence the factor
of $1/2$ in the definition of $\delta(u,z)$, and similarly in \eqref{eqn:delta}.}
$$
\delta(u,z) = 
\frac{1}{2} \sum_{u_\epsilon(\zeta_1) = u_\epsilon(\zeta_2),\, \zeta_1 \ne \zeta_2} 
\inter(u_\epsilon , \zeta_1 \,;\, u_\epsilon, \zeta_2),
$$
where the sum is finite and ranges over pairs $(\zeta_1,\zeta_2) \in \uU_z \times \uU_z$.
\end{lemma}

A proof of this lemma is given in Appendix~\ref{app:positivity}.
It enables us to define for each simple curve $u \in \mM_g(M,J)$ the integer
\begin{equation}
\label{eqn:delta}
\delta(u) := \frac{1}{2} \sum_{u(z) = u(\zeta),\, z \ne \zeta} 
\inter(u , z \,;\, u, \zeta) + \sum_{du(z)=0} \delta(u,z) \in \ZZ,
\end{equation}
which we shall call the \defin{singularity index}\index{singularity index of a simple holomorphic curve}\index{holomorphic curve!singularity index of}
of~$u$.
The contribution $\delta(u,z) > 0$ for each critical point $z$ is
the \defin{local singularity index}\index{local singularity index}\index{holomorphic curve!local singularity index at a point} of $u$ at~$z$.

\begin{thm}
\label{thm:delta}
For any simple curve $u \in \mM_g(M,J)$ in an almost complex $4$-manifold
$(M,J)$, the integer $\delta(u)$ defined
in \eqref{eqn:delta} depends only on the genus $g$ and the homology
class $[u] \in H_2(M)$.  Moreover, $\delta(u) \ge 0$, with equality
if and only if $u$ is embedded.
\end{thm}

Note that
the second statement in Theorem~\ref{thm:delta} is an immediate consequence of
Theorem~\ref{thm:positivity} and Lemma~\ref{lemma:critInj}.  
To prove the first statement, we shall relate
$\delta(u)$ to other quantities that more obviously depend only
on $[u] \in H_2(M)$ and the genus, for instance 
the homological self-intersection number\index{self-intersection number!homological}
$$
[u] \cdot [u] \in \ZZ.
$$
To compute the latter, it suffices to compute
$[u_\epsilon] \cdot [u_\epsilon]$ for any $C^1$-small immersed
perturbation $u_\epsilon : \Sigma \to M$ of~$u$.  Choose $u_\epsilon$
to be the perturbation promised by Lemma~\ref{lemma:critInj}, so for
some auxiliary symplectic structure $\omega$ taming $J$ near the images of the
critical points of~$u$, we can assume $u_\epsilon$ is symplectically
immersed near those critical points and matches $u$ everywhere else.
Notice that by Lemma~\ref{lemma:critInj} and the definition of~$\delta(u)$,
$$
\delta(u_\epsilon) = \delta(u).
$$
Denote the normal bundle of $u_\epsilon$ by
$N_{u_\epsilon} \to \Sigma$.  Since $u_\epsilon$ is symplectically immersed
in the region where it differs from~$u$, we can deform the natural complex
structure of $u_\epsilon^*TM$ on this region to one that is tamed by
$\omega$ but also admits a splitting of complex vector bundles
$u_\epsilon^*TM \cong T\Sigma \oplus N_{u_\epsilon}$.  This modification
of the complex structure does not change $c_1(u_\epsilon^*TM)$, so we then have
\begin{equation}
\label{eqn:c1normal}
c_1([u]) = c_1(u_\epsilon^*TM) = c_1(T\Sigma) + c_1(N_{u_\epsilon})
= \chi(\Sigma) + c_1(N_{u_\epsilon}).
\end{equation}
This motivates the following notion: we define the \defin{normal Chern number}\index{normal Chern number!of a closed holomorphic curve}\index{holomorphic curve!normal Chern number of}
$c_N(u) \in \ZZ$ of any closed $J$-holomorphic curve
$u : (\Sigma,j) \to (M,J)$ to be
\begin{equation}
\label{eqn:normalChern}
c_N(u) := c_1([u]) - \chi(\Sigma).
\end{equation}
It is equal to $c_1(N_u)$ whenever $u$ is immersed, but has the advantage
of obviously depending only on $[u] \in H_2(M)$ and the topology of the
domain, so we can define it without assuming that $u$ is immersed.

The self-intersection number $[u] \cdot [u] = [u_\epsilon] \cdot [u_\epsilon]$
can now be computed by counting (with signs) the isolated intersections between
$u_\epsilon$ and a generic perturbation of the form
$$
u_\epsilon' : \Sigma \to M : z \mapsto \exp_{u_\epsilon(z)} \eta(z),
$$
where $\eta$ is a generic $C^0$-small smooth section of
$N_{u_\epsilon} \to \Sigma$, and the exponential map is defined using
any choice of Riemannian metric on~$M$.  Figure~\ref{fig:adjunction} shows
how many intersections we should
expect to see.  Any zero of $\eta$ with
order $k \in \ZZ$ will
produce an intersection of $u_\epsilon$ and $u_\epsilon'$
whose local intersection index is also~$k$, and the sum of these orders over all
zeroes of $\eta$ is $c_1(N_{u_\epsilon})$.  
Moreover, any isolated double point
$u_\epsilon(z) = u_\epsilon(\zeta)$ will produce 
\emph{two} intersections of $u_\epsilon$ and $u_\epsilon'$ with the
same local index.  These two observations produce the formula
$$
[u] \cdot [u] = 2\delta(u_\epsilon) + c_1(N_{u_\epsilon}) =
2\delta(u) + c_N(u).
$$
Since neither $[u] \cdot [u]$ nor $c_N(u)$ depends on the perturbation
$u_\epsilon$, this proves the following important result, 
known as the \defin{adjunction formula},\index{adjunction formula!for closed holomorphic curves} which implies
Theorem~\ref{thm:delta} as an immediate corollary.

\begin{thm}[adjunction formula]
\label{thm:adjunction}
For any closed, connected and simple $J$-holomorphic curve $u$
in an almost complex $4$-manifold $(M,J)$,
\begin{equation}
\label{eqn:adjunction}
[u] \cdot [u] = 2\delta(u) + c_N(u),
\end{equation}
where $c_N(u) \in \ZZ$ is the normal Chern number \eqref{eqn:normalChern},
and $\delta(u)$ is a nonnegative integer that vanishes if and only if
$u$ is embedded.
\qed
\end{thm}
\begin{cor}
\label{cor:embedded}
If $u \in \mM_g^A(M,J)$ is embedded, then every other simple curve in
$\mM_g^A(M,J)$ is also embedded.
\qed
\end{cor}

\begin{figure}
\includegraphics{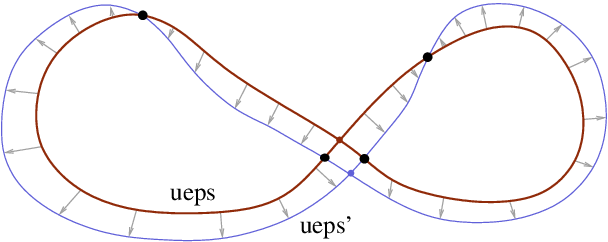}
\caption{\label{fig:adjunction} Counting the intersections of
$u_\epsilon : \Sigma \to M$ with a perturbation of the form
$u_\epsilon' = \exp_{u_\epsilon} \eta$ for some section $\eta$ of the normal bundle.}
\end{figure}

\begin{exercise}
\ 
\begin{enumerate}
\renewcommand{\labelenumi}{(\alph{enumi})}
\item
Consider the intersecting holomorphic maps $u , v : \CC \to \CC^2$ defined by
$$
u(z) = (z^3, z^5), \qquad v(z) = (z^4, z^6).
$$
Show that $u$ admits a $C^1$-small perturbation to a holomorphic
function $u_\epsilon$ such that $u_\epsilon$ and $v$ have exactly~$18$
intersections in a neighbourhood of the origin, all transverse.
\item
Try to convince yourself that the above count of 18 intersections holds
after \emph{any} generic $C^1$-small perturbation of $u$ and/or~$v$.
\item 
Show that for any neighbourhood $\uU \subset \CC$ of~$0$, the map 
$u$ admits a $C^1$-small perturbation to a holomorphic
\emph{immersion} $u_\epsilon$ such that
$$
\frac{1}{2}\# \{ (z,\zeta) \in \uU \times \uU\ |\ u_\epsilon(z) =
u_\epsilon(\zeta), \ z \ne \zeta \} = 10.
$$
\item
If you're especially ambitious, now try to convince yourself that for
\emph{any} perturbation as in part~(c) making all double points 
of $u_\epsilon$ transverse, the count of double points is the same.
\end{enumerate}
\end{exercise}

\begin{exercise}
Recall that $H_2(\CP^2)$ is generated by an embedded sphere
$\CP^1 \subset \CP^2$ with $[\CP^1] \cdot [\CP^1] = 1$.
A holomorphic curve $u : \Sigma \to \CP^2$ is said to have
\defin{degree~$d \in \NN$} if
$$
[u] = d[\CP^1].
$$
Show that all holomorphic spheres of degree~$1$ are embedded, and any other
simple holomorphic sphere in $\CP^2$ is embedded if and only if it has degree~$2$.
\end{exercise}

\section{Application to ruled surfaces}
\label{sec:finishMcDuff}

We now apply the results of the previous section to complete the proof of
Lemma~\ref{lemma:M1} from Lecture~\ref{sec:1}.  

Since $\mM_0^{[S]}(M,J)$
contains the embedded curve $u_S$ by construction,
Corollary~\ref{cor:embedded} implies that all other \emph{simple} curves
in $\mM_0^{[S]}(M,J)$ are also embedded, and we saw in \S\ref{sec:foliations}
that every embedded curve $u \in \mM_0^{[S]}(M,J)$ has a neighborhood in
$\mM_0^{[S]}(M,J)$ consisting of embeddings that foliate an open subset.
On a more global level, any two curves $u , v \in \mM_0^{[S]}(M,J)$
satisfy
$$
[u] \cdot [v] = [S] \cdot [S] = 0,
$$
thus Corollary~\ref{cor:positivityGlobal} now implies that
$u$ and $v$ are disjoint unless they are identical, hence the set of
\emph{all} simple curves in $\mM_0^{[S]}(M,J)$ foliates an open subset
of~$M$.

We must still rule out
the possibility that $\mM_0^{[S]}(M,J)$ contains a multiple cover, so arguing
by contradiction, suppose $u \in \mM_0^{[S]}(M,J)$ is a $k$-fold cover of
a simple curve $v : (\Sigma_g,j) \to (M,J)$ with genus~$g \ge 0$, for
some $k \ge 2$.  This requires the
existence of a map $\varphi : S^2 \to \Sigma_g$ of degree~$k$,
but such a map cannot
exist if $g > 0$ since $\Sigma_g$ then has a contractible universal
cover and thus $\pi_2(\Sigma_g)=0$; we conclude $g=0$.  Moreover,
the fact that the embedded sphere $S \subset M$ has trivial normal
bundle implies via the usual splitting $TM|_S = TS \oplus N_S$ that
$$
c_1([S]) = c_1(TS) + c_1(N_S) = \chi(S) = 2,
$$
so $[S] = [u] = k[v]$ implies $2 = k c_1([v])$, thus 
$k=2$ and $c_1([v]) = 1$.  Consider now the adjunction formula
\eqref{eqn:adjunction} applied to the simple curve~$v$:
$$
[v] \cdot [v] = 2\delta(v) + c_N(v) = 2\delta(v) + c_1([v]) - 2.
$$
The right hand side is an odd integer since $c_1([v]) = 1$.
However, the left hand side is $0$, as 
$0 = [S] \cdot [S] = [u] \cdot [u] = k^2 [v] \cdot [v]$,
so we have a contradiction.

Next, suppose $u_k \in \mM_0^{[S]}(M,J)$ is a sequence degenerating to a
nodal curve $\{ v_+,v_- \} \in \overline{\mM}_0^{[S]}(M,J)$, for which
Lemma~\ref{lemma:M0} guarantees that both $v_+$ and $v_-$ are
simple and satisfy $c_1([v_\pm]) > 0$.  Since $[S] = [u_k] =
[v_+] + [v_-]$ and $c_1([S]) = 2$, this implies
\begin{equation}
\label{eqn:c1vs}
c_1([v_+]) = c_1([v_-]) = 1.
\end{equation}
Since every curve $u \in \mM_0^{[S]}(M,J)$ has $c_1([u]) = c_1([S]) = 2$
and is simple, this implies that $u$ and $v_\pm$ can never have
identical images, so $[u] \cdot [v_\pm] \ge 0$ by 
positivity of intersections (Corollary~\ref{cor:positivityGlobal}).
Moreover,
$$
0 = [S] \cdot [S] = [u] \cdot \left( [v_+] + [v_-] \right) =
[u] \cdot [v_+] + [u] \cdot [v_-],
$$
where both terms at the right are nonnegative, thus both vanish
and we conclude via Corollary~\ref{cor:positivityGlobal} that
$u$ is disjoint from both $v_+$ and $v_-$.

We claim next that $v_+$ and $v_-$ cannot be the \emph{same}
curve (up to parametrization): indeed, if they are, then we
have $[S] = 2[v_+]$, and applying the adjunction formula to
$v_+$ yields the same numerical contradiction as in the case of
a multiple cover in $\mM_0^{[S]}(M,J)$.  
It follows now by Corollary~\ref{cor:positivityGlobal}
that $v_+$ and $v_-$ have finitely many intersections, all of which
count positively, and in fact
\begin{equation}
\label{eqn:intge1}
[v_+] \cdot [v_-] \ge 1
\end{equation}
since they must have at least one intersection, namely at the node.
Using $[S] = [v_+] + [v_-]$ and \eqref{eqn:c1vs}, and 
plugging in the adjunction formula 
and \eqref{eqn:c1vs} to compute $[v_\pm] \cdot [v_\pm]$, we find
\begin{equation*}
\begin{split}
0 = [S] \cdot [S] &= \left( [v_+] + [v_-] \right) \cdot 
\left( [v_+] + [v_-] \right)
= [v_+] \cdot [v_+] + [v_-] \cdot [v_-] + 2 [v_+] \cdot [v_-] \\
&= 2\delta(v_+) + c_N(v_+) + 2\delta(v_-) + c_N(v_-) + 2 [v_+] \cdot [v_-] \\
&= 2\delta(v_+) + 2\delta(v_-) + c_1([v_+]) - \chi(S^2) + c_1([v_-]) - \chi(S^2)
 + 2[v_+] \cdot [v_-] \\
&= 2 \delta(v_+) + 2\delta(v_-) + 2\left( [v_+] \cdot [v_-] - 1 \right).
\end{split}
\end{equation*}
By \eqref{eqn:intge1}, every term in this last sum is nonnegative, implying
$$
\delta(v_+) = \delta(v_-) = 0 \quad\text{ and }\quad [v_+] \cdot [v_-] = 1.
$$
Applying Corollary~\ref{cor:positivityGlobal} and Theorem~\ref{thm:adjunction},
we deduce that $v_\pm$ are each embedded and intersect each other exactly
once, transversely.  Applying the adjunction formula again to
$v_\pm$ with $c_N(v_\pm) = c_1([v_\pm]) - \chi(S^2) = -1$ then gives
$$
[v_\pm] \cdot [v_\pm] = 2\delta(v_\pm) + c_N(v_\pm) = 0 - 1 = -1,
$$
so both are $J$-holomorphic parametrizations of exceptional spheres.

Finally, we show that if $\{ v_+,v_- \}$ and $\{ v_+',v_-' \}$ are two
non-identical nodal curves arising as limits of curves
in $\mM_0^{[S]}(M,J)$, then they are disjoint.  Here ``non-identical''
can be taken to mean without loss of generality (i.e.~by reversing
the labels of $v_+$ and $v_-$ if necessary) that
$v_+$ is not equivalent to either $v_+'$ or $v_-'$ up to parametrization,
so positivity of intersections gives $[v_+] \cdot [v_\pm'] \ge 0$.
It could still happen in theory that $v_-$ is equivalent to one of
$v_+'$ or $v_-'$; say the latter, without loss of generality.
Then $[v_-] \cdot [v_-'] = -1$ by the above computation, while
$[v_+] \cdot [v_-] = [v_+] \cdot [v_-'] = 1$ and 
$[v_+'] \cdot [v_-'] = [v_+'] \cdot [v_-] = 1$, thus
\begin{equation*}
\begin{split}
0 = [S] \cdot [S] &= \left( [v_+] + [v_-] \right) \cdot \left(
[v_+'] \cdot [v_-'] \right) 
= [v_+] \cdot [v_+'] + [v_+] \cdot [v_-'] + [v_-] \cdot [v_+']
+ [v_-] \cdot [v_-'] \\
&\ge 0 + 1 + 1 - 1 = 1,
\end{split}
\end{equation*}
giving a contradiction.  The only remaining possibility is that
each of $v_\pm$ is not equivalent to each of $v_\pm'$, so their
intersections are all positive, and the expansion above implies
that they are all zero, thus both curves in $\{ v_+,v_- \}$ are
disjoint from both curves in $\{ v_+',v_-' \}$.
The proof of Lemma~\ref{lemma:M1} is now complete.

To conclude our discussion of the closed case, let us note which
properties of the intersection theory we made essential use of in
the above argument:
\begin{itemize}
\item The pairing $[u] \cdot [v]$ is \emph{homotopy invariant}.
\item The condition $[u] \cdot [v] = 0$ guarantees that two curves
$u$ and $v$ with non-identical images are \emph{disjoint}; moreover, if they
have a known intersection, then $[u] \cdot [v] = 1$ guarantees
that that intersection is \emph{transverse}.
\item There is a homotopy invariant number $\delta(u) \ge 0$ defined
for simple curves $u$, which can be computed in terms of
$[u] \cdot [u]$ and whose vanishing guarantees that $u$ is \emph{embedded}.
\end{itemize}
In order to produce a useful theory for studying contact $3$-manifolds,
we will want the intersection theory defined in the next two lectures
for \emph{punctured} holomorphic curves to have all of these same
properties.

\section{Contact manifolds, symplectic fillings and cobordisms}
\label{sec:contact}

The goal for the remainder of these lectures will be to explain a
generalization of the intersection theory described above that has
applications in $3$-dimensional contact topology.  One way to
motivate the study of contact manifolds is by considering
symplectic manifolds with boundary.

A vector field $V$ on a symplectic manifold $(M,\omega)$ is
called a \defin{Liouville vector field}\index{Liouville vector field} if it satisfies
$$
\Lie_V \omega = \omega,
$$
i.e.~its flow rescales the symplectic form exponentially.  By Cartan's
formula for the Lie derivative, this is equivalent to the condition
$$
d\lambda = \omega, \quad\text{ where $\lambda := \iota_V \omega$},
$$
and the primitive $\lambda$ is then called a \defin{Liouville form}.\index{Liouville form}
We say in this case that $\lambda$ is \defin{$\omega$-dual} to~$V$.

\begin{defn}
\label{defn:convexConcave}
Suppose $(W,\omega)$ is a symplectic manifold with boundary.
A boundary component $M \subset \p W$ is called \defin{convex}/\defin{concave}\index{convex boundary}\index{concave boundary}\index{contact-type boundary|see {convex boundary}}
if a neighborhood of $M$ admits a Liouville vector field that points
transversely outward/inward respectively at~$M$.
\end{defn}
\begin{exercise}
\label{EX:LiouvilleContact}
Suppose $M$ is an oriented hypersurface in a $2n$-dimensional symplectic 
manifold $(W,\omega)$, and $V$ is a Liouville vector field defined near~$M$, 
with $\omega$-dual Liouville form~$\lambda$.  Show that $V$ is 
positively/negatively transverse to~$M$ if and only if the restriction of
$\lambda \wedge (d\lambda)^{n-1}$ to $M$ is a positive/negative volume
form respectively.
\end{exercise}
\begin{exercise}
\label{EX:convex}
Show that in the situation of Exercise~\ref{EX:LiouvilleContact}, the
spaces of Liouville forms $\lambda$ defined near~$M \subset (W,\omega)$ such 
that $\lambda \wedge d\lambda^{n-1}|_{TM}$ is a positive or negative
volume form are convex.
\end{exercise}

Exercise~\ref{EX:LiouvilleContact} leads directly to the notion of a
contact manifold: we say that a $1$-form $\alpha$ on an oriented 
$(2n-1)$-dimensional manifold is a (positive) \defin{contact form} if\index{contact form}
\begin{equation}
\label{eqn:contact}
\alpha \wedge (d\alpha)^{n-1} > 0,
\end{equation}
and a (positive, co-oriented) \defin{contact structure}\index{contact structure}\index{contact manifold}
is any
smooth co-oriented hyperplane distribution $\xi \subset TM$ that can be defined by
$\xi = \ker\alpha$ for some contact form~$\alpha$.  Exercises~\ref{EX:LiouvilleContact}
and~\ref{EX:convex} show that whenever $M \subset \p W$ is a convex/concave 
boundary component of a symplectic manifold $(W,\omega)$, the oriented manifold 
$\pm M$ inherits a positive\footnote{We are assuming $M$ carries its canonical
orientation as a boundary component of the symplectic manifold $(W,\omega)$,
but also using the notation $-M$ to mean the same manifold with reversed
orientation---thus a positive contact structure on $-M$ is in fact a negative
contact structure on~$M$.}
contact structure, which is unique up to
deformation through families of contact structures.  Whenever $M$ is closed,
Gray's stability theorem (see e.g.~\cite{Geiges:book}) then implies that the
induced contact structure on $M$ is in fact canonical up to 
\emph{isotopy}.\footnote{Gray's stability theorem states that any
smooth $1$-parameter family of contact structures on a closed manifold arises
from an isotopy.  It is specifically true for contact \emph{structures}
and not contact \emph{forms}, and this is one good reason why we regard the
contact structure on a convex/concave boundary of a symplectic manifold as a
well-defined object, whereas the contact \emph{form} is only auxiliary data.}

\begin{exercise}
\label{EX:dalpha}
Show that up to issues of orientation, the contact condition \eqref{eqn:contact} is equivalent to the
condition that $\alpha$ is nowhere zero and $d\alpha$ restricts to a
nondegenerate $2$-form on $\xi := \ker\alpha$, i.e.~it makes
$(\xi,d\alpha) \to M$ a symplectic vector bundle.
\end{exercise}

\begin{defn}
Given two closed contact manifolds $(M_+,\xi_+)$ and $(M_-,\xi_-)$ of the
same dimension, a \defin{symplectic cobordism}\index{symplectic cobordism}
from\footnote{Certain
orientation conventions are not universally agreed upon: there is a vocal
minority of authors who would describe what we are defining here as a
``symplectic cobordism \emph{from} $(M_+,\xi_+)$ \emph{to} $(M_-,\xi_-)$.''
Whichever convention one prefers, one must be consistent about it---unlike 
topological cobordisms, the
existence of a symplectic cobordism in one direction does not imply that
one in the other direction also exists!}
$(M_-,\xi_-)$ to
$(M_+,\xi_+)$ is a compact symplectic manifold $(W,\omega)$ with
$$
\p W = -M_- \sqcup M_+,
$$
such that a neighborhood of $\p W$ admits a Liouville form $\lambda$ with
$$
\ker \left( \lambda|_{T M_\pm} \right) = \xi_\pm.
$$
If $M_- = \emptyset$, we call $(W,\omega)$ a (strong) \defin{symplectic
filling}\index{symplectic filling}
of $(M_+,\xi_+)$, and if $M_+ = \emptyset$, we say
$(W,\omega)$ is a \defin{symplectic cap}\index{symplectic cap}
for $(M_-,\xi_-)$.
\end{defn}

There are many interesting questions one can ask about contact manifolds
and the existence of symplectic fillings or cobordisms.  The strongest
results in this area are typically specific to dimension three, though
there has also been considerable recent progress in higher dimensions.
Here is a brief
sampling of known results:
\begin{enumerate}
\item Martinet \cite{Martinet} proved that every closed oriented $3$-manifold
admits a contact structure.  This result was recently extended to all dimensions
by Borman-Eliashberg-Murphy \cite{BEM}, given the obviously necessary topological
condition that an \emph{almost} contact structure exists.
\item A combination of results due to Gromov and Eliashberg
\cites{Gromov,Eliashberg:diskFilling,Eliashberg:overtwisted} implies that
any contact structure on any closed $3$-manifold $M$ is homotopic through
oriented $2$-plane fields to a contact structure $\xi$ for which $(M,\xi)$ 
admits no symplectic filling.  These are the so-called \emph{overtwisted}
contact structures.  This notion has also recently been generalized to
all dimensions in \cite{BEM}.
\item A result of Lisca \cite{Lisca:curvature}
even gives examples of closed oriented $3$-manifolds on which 
\emph{no} contact structure is symplectically fillable.  Etnyre and
Honda \cite{EtnyreHonda:noTight} later extended this to find $3$-manifolds
on which every contact structure is overtwisted.
\item In contrast to fillings, Etnyre and Honda \cite{EtnyreHonda:cobordisms}
showed that symplectic \emph{caps} do exist for any closed contact 
$3$-manifold, and in fact they come in infinitely many distinct topological
types.  The existence of caps in all higher dimensions was established only
very recently, in parallel work of Conway-Etnyre \cite{ConwayEtnyre:caps}
and Lazarev \cite{Lazarev:caps}.
\item Etnyre and Honda \cite{EtnyreHonda:cobordisms} also showed that
every closed overtwisted contact $3$-manifold admits a symplectic cobordism
to every other closed contact $3$-manifold.  The higher-dimensional
analogue of this result was recently established by
Eliash\-berg-Murphy \cite{EliashbergMurphy:cobordisms}.
\end{enumerate}

Let us state more carefully two further results along these lines that will be
discussed in Lecture~\ref{sec:5}.  We say that two
symplectic manifolds $(W,\omega)$ and $(W',\omega')$ with convex boundary
are \defin{symplectically deformation equivalent}\index{symplectic deformation equivalence}
if there is a diffeomorphism
$\varphi : W \to W'$ such that $\varphi^*\omega'$ can be deformed to
$\omega$ through a smooth $1$-parameter family of symplectic forms that are
all convex at the boundary.  The \defin{standard
contact structure}\index{standard contact structure!on $S^3$}\index{contact structure!on $S^3$}
$\xi\std$ on $S^3$ is defined by identifying $S^3$ with
the boundary of the unit ball $B^4$ with its \defin{standard symplectic form on the $4$-ball}
$$
\omega\std := \sum_{j=1}^2 d x_j \wedge dy_j
$$
and Liouville form
$$
\lambda\std := \frac{1}{2} \sum_{j=1}^2 \left( x_j\, dy_j - y_j \, d x_j \right).
$$
By this definition, $(B^4,\omega\std)$ is a symplectic filling of
$(S^3,\xi\std)$, and one can trivially produce other fillings of $(S^3,\xi\std)$
with different topological types by blowing up $(B^4,\omega\std)$ in its
interior.  This procedure however produces a fairly limited range of topological
types for manifolds $W$ with $\p W = S^3$.  Note that in terms of smooth topology,
\emph{almost anything} can have boundary $S^3$: just take any closed oriented
$4$-manifold, remove a ball and reverse the orientation.  Symplectically,
however, the situation is very different:

\begin{thm}[Gromov \cite{Gromov}]
\label{thm:fillingsS3}
Every symplectic filling of $(S^3,\xi\std)$ is symplectically deformation
equivalent to a blowup of $(B^4,\omega\std)$.
\end{thm}

Similarly, $S^1 \times S^2$ and the lens spaces $L(k,k-1)$ for $k \in \NN$\index{standard contact structure!on $S^1 \times S^2$}\index{standard contact structure!on lens spaces}
each carry standard contact structures as convex boundaries of certain\index{contact structure!on $S^1 \times S^2$}\index{contact structure!on lens spaces}
symplectic manifolds, and their fillings are also unique in the above sense:

\begin{thm}
\label{thm:moreFillings}
The contact manifolds $(S^1 \times S^2,\xi\std)$ and $(L(k,k-1),\xi\std)$
for $k \in \NN$ each have unique symplectic fillings up to deformation
equivalence and blowup.
\end{thm}

Theorem~\ref{thm:moreFillings} was proved for $S^1 \times S^2$ originally by
Eliashberg \cite{Eliashberg:diskFilling}, and the uniqueness for $L(k,k-1)$
up to diffeomorphism was proved by Lisca \cite{Lisca:fillingsLens}.
In the forms stated above, Theorems~\ref{thm:fillingsS3} and~\ref{thm:moreFillings}
are both easy applications of a more general result from \cite{Wendl:fillable},
that can be thought of as an analogue of McDuff's Theorem~\ref{thm:McDuff}
for symplectic fillings of certain contact $3$-manifolds.  This will be the
main subject of Lecture~\ref{sec:5}.

\section{Asymptotically cylindrical holomorphic curves}
\label{sec:punctured}

It is not usually useful to consider \emph{closed} holomorphic curves in 
symplectic
cobordisms---for example, the symplectic form on a cobordism could be
exact, in which case Stokes' theorem implies that all closed holomorphic
curves for a tame almost complex structure are constant.  A useful
alternative is to consider noncompact holomorphic curves with cylindrical
ends, and the proper setting for this is the noncompact \emph{completion}
of a symplectic cobordism.  The study of holomorphic curves in this setting
is a large subject known as \emph{symplectic field theory}
(see \cites{SFT,Wendl:SFT}), and we shall only touch upon a few aspects of it here.

Assume $(W,\omega)$ is a symplectic cobordism from $(M_-,\xi_-)$ to
$(M_+,\xi_+)$, with a neighborhood of $\p W$ admitting a Liouville form
$\lambda$ such that
$$
\xi_\pm = \ker \alpha_\pm, \quad\text{ where }
\alpha_\pm := \lambda|_{T M_\pm}.
$$
\begin{exercise}
\label{EX:collar}
Show that the flow from $M_\pm$ along the Liouville vector field
dual to~$\lambda$ identifies collar neighborhoods $\nN(M_\pm)
\subset W$ of $M_\pm$ with the models
\begin{equation*}
\begin{split}
(\nN(M_+),\lambda) &\cong ((-\epsilon,0] \times M_+, e^s\alpha_+), \\
(\nN(M_-),\lambda) &\cong ([0,\epsilon) \times M_-, e^s\alpha_-)
\end{split}
\end{equation*}
for sufficiently small $\epsilon > 0$, where $s$ denotes the real
coordinate in $(-\epsilon,0]$ or $[0,\epsilon)$.
\end{exercise}

For any contact manifold $(M,\xi = \ker\alpha)$, the exact symplectic manifold
$(\RR \times M,d(e^s\alpha))$ is called the \defin{symplectization}\index{symplectization of a contact manifold}
of $(M,\xi)$; one can show that its symplectomorphism type depends on $\xi$ but not
on the choice of contact form~$\alpha$.  A choice of $\alpha$ does however
determine a distinguished vector field that spans the characteristic line fields
of the hypersurfaces $\{s\} \times M$: we define the \defin{Reeb vector field}\index{Reeb vector field!of a contact form}
to be the unique vector field $R_\alpha$ on $M$ satisfying
$$
d\alpha(R_\alpha,\cdot) \equiv 0 \quad\text{ and }\quad \alpha(R_\alpha) \equiv 1.
$$

\begin{figure}
\includegraphics{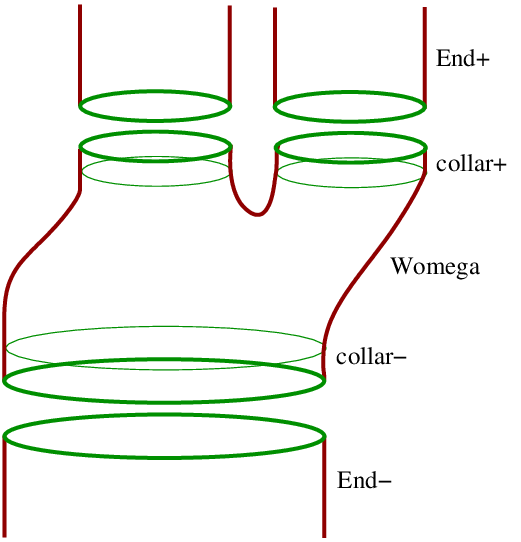}
\caption{\label{fig:completion} The completion of a symplectic cobordism
is constructed by attaching half-symplectizations to form cylindrical ends.}
\end{figure}

The \defin{symplectic completion}\index{completion of a symplectic cobordism}\index{symplectic completion|see {completion of a symplectic cobordism}}
of the cobordism $(W,\omega)$ is defined by attaching halves of symplectizations
along the collar neighborhoods from Exercise~\ref{EX:collar}, producing the
noncompact symplectic manifold (see Figure~\ref{fig:completion}).
\begin{equation}
\label{eqn:completion}
(\widehat{W},\widehat{\omega}) := \big((-\infty,0] \times M_- , d(e^s\alpha_-)\big) 
\cup_{M_-} (W,\omega) \cup_{M_+}
\big( [0,\infty) \times M_+ , d(e^s\alpha_+) \big).
\end{equation}
Informally, the symplectization of $(M,\xi)$ can also be thought of as the 
completion of a trivial symplectic cobordism from $(M,\xi)$ to itself.

Given a choice of contact form $\alpha$ for $\xi$,
$(\RR \times M, d(e^s\alpha))$
carries a special class $\jJ(\alpha)$ of compatible\index{almost complex structure!compatible with a contact form}\index{compatible almost complex structure}
almost complex structures $J$, defined by the conditions
\begin{itemize}
\item $J(\p_s) = R_\alpha$;
\item $J(\xi) = \xi$ and $J|_\xi$ is compatible with $d\alpha|_\xi$;
\item $J$ is invariant under the translation action $(s,p) \mapsto (s + c,p)$
for all $c \in \RR$.
\end{itemize}
For any $J \in \jJ(\alpha)$, a periodic orbit $x : \RR \to M$ of $R_\alpha$
with period $T > 0$ gives rise to a $J$-holomorphic cylinder
$$
u : \RR \times S^1 \to \RR \times M : (s,t) \mapsto (Ts,x(Tt)).
$$
Such curves are referred to as \defin{orbit cylinders}\index{orbit cylinder}\index{trivial cylinder|see {orbit cylinder}}
(sometimes also \defin{trivial cylinders}), and they serve as
asymptotic models for the more general class of holomorphic curves that we
now wish to consider.

Indeed, on the completion $(\widehat{W},\widehat{\omega})$ as defined above,
let $\jJ(\omega,\alpha_+,\alpha_-)$ denote the space of 
almost complex structures that are compatible with $\omega$ on~$W$ and
belong to $\jJ(\alpha_\pm)$ on
$[0,\infty) \times M_+$ and $(-\infty,0] \times M_-$ respectively.  A choice of 
$J \in \jJ(\omega,\alpha_+,\alpha_-)$ makes $(\widehat{W},J)$ into an
\defin{almost complex manifold with cylindrical ends}.\index{cylindrical ends!of an almost complex manifold}\index{almost complex manifold with cylindrical ends}
A Riemann surface with cylindrical ends can likewise be constructed
by introducing \emph{punctures} into a closed Riemann surface.\index{cylindrical ends!of a Riemann surface}\index{Riemann surface!punctured}
Namely, suppose $(\Sigma,j)$ is closed, and
$\Gamma \subset \Sigma$ is a finite set partitioned into two subsets
$\Gamma = \Gamma^+ \sqcup \Gamma^-$, which we will call the \defin{positive}\index{positive punctures of a holomorphic curve}
and \defin{negative punctures},\index{negative punctures of a holomorphic curve}
writing the resulting punctured surface as
$$
\dot{\Sigma} := \Sigma \setminus \Gamma.
$$
Near each $z \in \Gamma^\pm$, one can identify a closed neighborhood
$\dD_z \subset \Sigma$ of $z$ biholomorphically with the standard unit
disk $(\DD,i)$ such that $z$ is identified with the origin, 
and then identify $\DD \setminus \{0\}$  in turn
with a half-cylinder via the biholomorphic map
\begin{equation*}
\begin{split}
[0,\infty) \times S^1 \to \DD \setminus \{0\} : (s,t) \mapsto e^{-2\pi (s+it)},
\quad & \text{ for $z \in \Gamma^+$,}\\
(-\infty,0] \times S^1 \to \DD \setminus \{0\} : (s,t) \mapsto e^{2\pi (s+it)},
\quad & \text{ for $z \in \Gamma^-$.}
\end{split}
\end{equation*}
We will refer to this identification as a choice of \defin{cylindrical
coordinates}\index{cylindrical coordinates on a punctured Riemann surface}
near $z \in \Gamma^\pm$.  Making such a choice for all
punctures, this determines a decomposition
\begin{equation}
\label{eqn:cylindricalEnds}
\dot{\Sigma} = \big((-\infty,0] \times C_-\big) \cup_{C_-} \Sigma_0 \cup_{C_+}
\big([0,\infty) \times C_+\big)
\end{equation}
analogous to \eqref{eqn:completion}, where 
$\Sigma_0 := \Sigma \setminus \bigcup_{z \in \Gamma} \mathring{\dD}_z$ 
can be regarded
as a cobordism with $\p\Sigma_0 = -C_- \sqcup C_+$ between two disjoint unions
of circles $C_\pm$, and the complex structure on the cylindrical ends
is always the standard one, i.e.~with $i\p_s = \p_t$ in cylindrical
coordinates $(s,t)$.

We say that a smooth map $u : \dot{\Sigma} \to \widehat{W}$ is
(positively or negatively)
\defin{asymptotic} at $z \in \Gamma^\pm$ to a $T$-periodic orbit\index{asymptotic Reeb orbit}\index{Reeb orbit!asymptotic}
$x : \RR \to M_\pm$ of $R_{\alpha_\pm}$ if there exists a choice of
cylindrical coordinates as above in which $u$ near~$z$ takes the form
$$
u(s,t) = \exp_{(Ts,x(Tt))} h(s,t) \in \RR \times M_\pm \quad
\text{ for $|s|$ large},
$$
where the exponential map is defined
with respect to a translation-invariant choice of Riemannian metric
on $\RR \times M_\pm$, and $h(s,t)$ is a vector field along the
orbit cylinder that decays to~$0$ with all derivatives as $s \to \pm\infty$.
We say that $u : \dot{\Sigma} \to M$ is \defin{asymptotically
cylindrical}\index{asymptotically cylindrical map}\index{holomorphic curve!asymptotically cylindrical}
if it is positively/negatively asymptotic to some closed
Reeb orbit in $M_+$ or $M_-$ respectively at each of its positive/negative
punctures; see Figure~\ref{fig:asympCyl}.

Observe that the completion $\widehat{W}$ admits a natural compactification
as a compact topological manifold with boundary:
$$
\overline{W} := \big([-\infty,0] \times M_-\big) \cup_{M_-} W \cup_{M_+}
\big([0,\infty] \times M_+\big).
$$
In the same way, the decomposition \eqref{eqn:cylindricalEnds} allows us to
define the \defin{circle compactification}\index{circle compactification of a punctured Riemann surface}\index{Riemann surface!circle compactification of}
$\overline{\Sigma}$ of
$\dot{\Sigma}$, a compact topological $2$-manifold with boundary 
whose interior is identified with~$\dot{\Sigma}$.  It follows then from
the definition above that any asymptotically cylindrical map
$u : \dot{\Sigma} \to \widehat{W}$ extends naturally to a continuous map
$$
\bar{u} : \overline{\Sigma} \to \overline{W}
$$
which takes each component of $\p\overline{\Sigma}$ to a closed Reeb
orbit in $\{\pm\infty\} \times M_\pm$.

\begin{figure}
\includegraphics{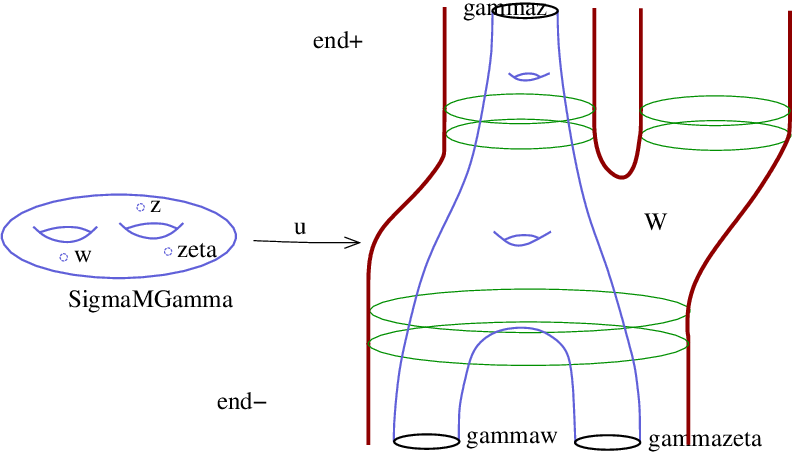}
\caption{\label{fig:asympCyl} An asymptotically cylindrical
map $u : \dot{\Sigma} \to \widehat{W}$ of a punctured surface
of genus~$2$ into a completed cobordism, with
one positive puncture $z \in \Sigma$ asymptotic to a Reeb orbit $\gamma_z$ in
$M_+$, and two positive punctures $w,\zeta \in \Sigma$ asymptotic to
Reeb orbits $\gamma_w$ and $\gamma_\zeta$ in $M_-$.}
\end{figure}

If the set of punctures is nonempty, then an asymptotically cylindrical map
$u : \dot{\Sigma} \to \widehat{W}$ does not represent a homology class 
in $H_2(\widehat{W})$, but one can
use the compactifications described above to assign it a
\emph{relative} homology class.  Concretely, let
$$
\widebar{\boldsymbol{\gamma}}^\pm \subset \{\pm\infty\} \times M_\pm \subset \p \widebar{W}
$$
denote the union of all of the images of the positive/negative asymptotic
orbits of~$u$; topologically, this is a disjoint union of embedded circles.  
The \defin{relative homology class}\index{relative homology class}\index{asymptotically cylindrical map!relative homology class of}
of $u$ is then defined as
$$
[u] := \widebar{u}_*[\widebar{\Sigma}] \in H_2(\widebar{W},\widebar{\boldsymbol{\gamma}}^+ \cup \widebar{\boldsymbol{\gamma}}^-),
$$
where $[\widebar{\Sigma}] \in H_2(\widebar{\Sigma},\p\widebar{\Sigma})$ denotes
the relative fundamental class of~$\widebar{\Sigma}$.
The long exact sequence of the pair $(\widebar{W},\widebar{\boldsymbol{\gamma}}^+ \cup \widebar{\boldsymbol{\gamma}}^-)$
implies that any two asymptotically cylindrical maps having the same asymptotic orbits
with the same multiplicities have relative homology classes that differ by a
unique \emph{absolute} homology class, that is, a class in the image of the
natural map $H_2(\widebar{W}) \to H_2(\widebar{W},\widebar{\boldsymbol{\gamma}}^+ \cup \widebar{\boldsymbol{\gamma}}^-)$;
note that the latter is injective since $H_2(\widebar{\boldsymbol{\gamma}}^+ \cup \widebar{\boldsymbol{\gamma}}^-) = 0$.
In most situations, it is convenient to apply the obvious deformation
retraction $\widebar{W} \to W$ and regard $[u]$ as an element of
$H_2(W,\widebar{\boldsymbol{\gamma}}^+ \cup \widebar{\boldsymbol{\gamma}}^-)$,
with $\widebar{\boldsymbol{\gamma}}^\pm$ now regarded as submanifolds of
$M_\pm \subset \p W$.  In the special case where $(\widehat{W},\widehat{\omega})$
is just the symplectization of a single contact manifold $(M,\xi)$, we take
this one step further and retract $\RR \times M$ to $\{0\} \times M$, so that
$[u]$ lives naturally in
$H_2(M,\widebar{\boldsymbol{\gamma}}^+ \cup \widebar{\boldsymbol{\gamma}}^-)$.
In the following we will state definitions assuming that $(\widehat{W},\widehat{\omega})$
is a completion of a nontrivial cobordism instead of a symplectization, but
one can make obvious modifications to accommodate the latter case.

Moduli spaces of punctured $J$-holomorphic
curves are now defined as follows.
Choose finite ordered sets of closed Reeb orbits
$$
\boldsymbol{\gamma^+} = (\gamma_1^+,\ldots,\gamma_{r_+}^+) \text{ in $M_+$}
\qquad\text{ and }\qquad
\boldsymbol{\gamma^-} = (\gamma_1^-,\ldots,\gamma_{r_-}^-) \text{ in $M_-$,}
$$
and a relative homology class
$A \in H_2(W,\widebar{\boldsymbol{\gamma}}^+ \cup \widebar{\boldsymbol{\gamma}}^-)$,
where $\widebar{\boldsymbol{\gamma}}^\pm$ denotes the union of the images of
the orbits $\gamma_1^\pm,\ldots,\gamma_{r_\pm}^\pm$.  We then define
$$
\mM_g^A(\widehat{W},J ; \boldsymbol{\gamma}^+, \boldsymbol{\gamma}^- )
:= \left\{ (\Sigma,j,\Gamma^+,\Gamma^-,u) \right\} \Big/ \sim,
$$
where
\begin{itemize}
\item $(\Sigma,j)$ is a closed connected Riemann surface of genus~$g$;
\item $\Gamma^\pm = (z_1^\pm,\ldots,z_{r_\pm}^\pm)$ are disjoint finite ordered
sets of pairwise distinct points in $\Sigma$, defining a punctured surface
$\dot{\Sigma} := \Sigma \setminus (\Gamma^+ \cup \Gamma^-)$;
\item The map $u : (\dot{\Sigma},j) \to (\widehat{W},J)$
is $J$-holomorphic, asymptotic to
$\gamma_i^\pm$ at $z_i^\pm \in \Gamma^\pm$ for $i=1,\ldots,r_\pm$, and
represents the relative homology class~$A$;
\item
Two such tuples are considered equivalent if they are related by a
biholomorphic map that preserves the sets of positive and negative punctures,
along with their orderings.
\end{itemize}
We shall denote unions of these
spaces over all possible choices of data by
\begin{equation*}
\begin{split}
\mM_g(\widehat{W},J ; \boldsymbol{\gamma}^+, \boldsymbol{\gamma}^-) &:=
\coprod_{A \in H_2(W,\widebar{\boldsymbol{\gamma}}^+ \cup \widebar{\boldsymbol{\gamma}}^-)}
\mM_g^A(\widehat{W},J ; \boldsymbol{\gamma}^+, \boldsymbol{\gamma}^-),\\
\mM_{g,r_+,r_-}(\widehat{W},J) &:= \coprod_{|\boldsymbol{\gamma}^\pm| = r_\pm}
\mM_g(\widehat{W},J ; \boldsymbol{\gamma}^+, \boldsymbol{\gamma}^-),\\
\mM_g(\widehat{W},J) &:= \coprod_{r_+, r_- \ge 0} \mM_{g,r_+,r_-}(\widehat{W},J).
\end{split}
\end{equation*}
A topology on $\mM_g(\widehat{W},J)$ can be defined by saying that
a sequence $[(\Sigma_k,j_k,\Gamma^+_k,\Gamma^-_k,u_k)]$ converges to
an element $[(\Sigma,j,\Gamma^+,\Gamma^-,u)]$ if there exist representatives
$(\Sigma,j_k',\Gamma^+,\Gamma^-,u_k') \sim (\Sigma_k,j_k,\Gamma^+_k,\Gamma^-_k,
u_k)$ such that
$$
j_k \to j \text{ in $C^\infty(\Sigma)$}, \qquad
u_k \to u \text{ in $C^\infty_{\text{loc}}(\dot{\Sigma},\widehat{W})$}, 
\quad\text{ and }\quad
\bar{u}_k \to \bar{u} \text{ in $C^0(\overline{\Sigma},\overline{W})$}.
$$

Our goal for the next pair of lectures will be to write down
generalizations of the homological intersection number and the adjunction
formula for curves in $\mM_g(\widehat{W},J)$.  These will be instrumental
in the proof of Theorems~\ref{thm:fillingsS3} and~\ref{thm:moreFillings}.

\chapter{Asymptotics of punctured holomorphic curves}
\label{sec:3}

\minitoc

\vspace{12pt}

If\CUP{This material will be published by Cambridge University
Press as \textsl{Contact 3-Manifolds, Holomorphic Curves and Intersection Theory}
by Chris Wendl. This pre-publication version is
free to view and download for personal use only. 
Not for re-distribution, re-sale or use in derivative works. \copyright Chris Wendl, 2019.}
$u_1 \in \mM_g(\widehat{W},J;\boldsymbol{\gamma}_1^+,\boldsymbol{\gamma}_1^-)$ and
$u_2 \in \mM_g(\widehat{W},J;\boldsymbol{\gamma}_2^+,\boldsymbol{\gamma}_2^-)$ are two
asymptotically cylindrical holomorphic curves in
a $4$-dimensional completed symplectic cobordism, it remains true as
in the closed case that intersections of $u_1$ with $u_2$ are isolated
and positive unless both curves have identical images (i.e.~they
cover the same simple curve up to parametrization).  Since the domains
are no longer compact, however, it is not obvious whether the number
of intersections is still finite.  If it is finite, then one can
define an algebraic intersection number
$$
u_1 \cdot u_2 \in \ZZ
$$
which is guaranteed to be nonnegative, and strictly positive unless the
two curves are disjoint.  Such a number is not very useful though unless
it is \emph{homotopy invariant}, i.e.~we would like to know that for
any family $u_s \in 
\mM_g(\widehat{W},J;\boldsymbol{\gamma}_1^+,\boldsymbol{\gamma}_1^-)$ that depends
continuously (with respect to the topology of the moduli space) on
a parameter $s \in [0,1]$, we have $u_0 \cdot u_2 = u_1 \cdot u_2$.
This turns out to be \emph{false} in general, as the noncompactness of
the domains can allow intersections to escape to infinity and disappear
under homotopies (see Figure~\ref{fig:escape}).  
It is a very powerful fact, first suggested by Hofer and then worked
out in detail by Siefring \cites{Siefring:thesis,Siefring:asymptotics,
Siefring:intersection}, that this phenomenon can be controlled: one can
define for any two distinct punctured holomorphic curves a count of
\emph{virtual} intersections that are ``hidden at infinity,'' such that
the sum of this number with $u_1 \cdot u_2$ is homotopy invariant.
We will define this precisely in the next lecture and explain some
applications in Lecture~\ref{sec:5}.  As a preliminary step, it is
necessary to gain a fairly precise understanding of the asymptotic
behavior of punctured holomorphic curves, so that will be the topic for
this lecture.

\begin{remark}
While all results in this and the next lecture are stated in the setting of
symplectizations of contact manifolds and (completed) symplectic cobordisms
between them, they are valid in somewhat greater generality: they continue
to hold namely whenever contact forms are replaced by stable Hamiltonian
structures, so long as one can still assume that all closed Reeb orbits are
nondegenerate (or Morse-Bott---see the footnote attached to
Theorem~\ref{thm:star}).  The main results are restated in this more general form
in Appendix~\ref{app:reference}.
\end{remark}

\section{Holomorphic half-cylinders as gradient-flow lines}
\label{sec:Morse}

Historically, the study of punctured holomorphic curves arose from
an analogy with Floer's interpretation of Morse theory as the study of
gradient-flow lines of a Morse function (see e.g.~\cites{Salamon:Floer,AudinDamian}).
In Morse theory, one considers a manifold $M$ with a smooth function
$f : M \to \RR$, which is called a \defin{Morse function}\index{Morse function}
if its
Hessian at every critical point $p \in \Crit(f)$
$$
\Hess_p := \nabla df(p) : T_p M \times T_p M \to \RR
$$
is nondegenerate; here $\nabla$ denotes the covariant derivative for
any choice of connection on~$M$, but the Hessian does not depend on
this choice since $df(p) = 0$.  Recall that the Hessian is automatically
a symmetric bilinear map, and if we choose a Riemannian metric $g$ with
Levi-Civita connection $\nabla$ and consider instead the covariant
derivative of the gradient, we can then identify $\Hess_p$ with the linear map
$$
A_p := \nabla (\nabla f)(p) : T_p M \to T_p M,
$$
which is symmetric with respect to the inner product defined by~$g$.
One way of proving the classical Morse inequalities on $M$ is by defining a
homology theory with a chain complex generated by critical points in
$\Crit(f)$, and a differential defined by counting isolated solutions
to the gradient-flow problem
$$
\mM(p_+,p_-) := \left\{ x : \RR \to M \ \Big|\ \text{$\dot{x} = \nabla f(x)$
and $\lim_{s \to \pm\infty} x(s) = p_\pm$} \right\},
$$
for $p_\pm \in \Crit(f)$.  In particular, one can show that the resulting
homology theory is isomorphic to the usual singular homology $H_*(M)$,
thus giving relations between the topology of~$M$ and the set of critical
points of~$f$, see e.g~\cites{Schwarz:Morse,AudinDamian}.

\begin{figure}
\includegraphics{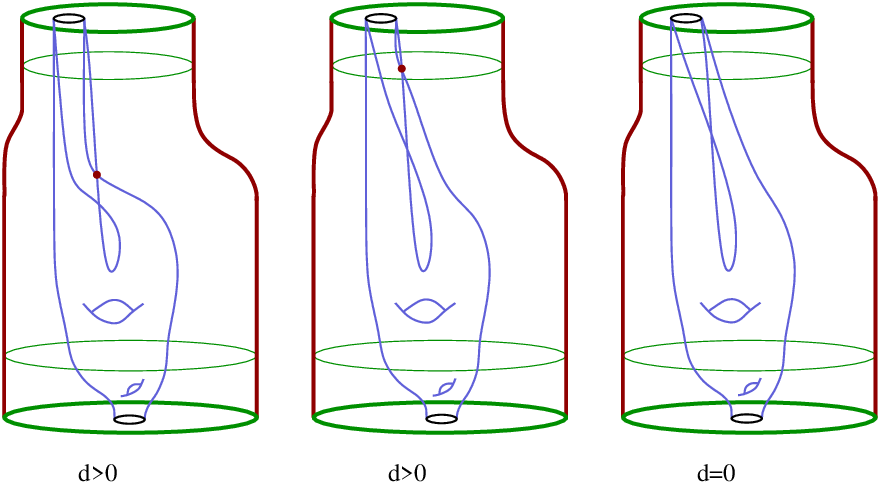}
\caption{\label{fig:escape} The condition of two asymptotically cylindrical
holomorphic curves (or two ends of the same curve) being \emph{disjoint} is
not homotopy invariant, as intersections can escape to infinity if the
two asymptotic Reeb orbits coincide.}
\end{figure}

Since the Hessian $A_p = \nabla(\nabla f)(p)$ is symmetric, 
its eigenvectors in $T_p M$ are
orthogonal and its eigenvalues are real.  Another way of expressing the
Morse condition is to say that $0 \not\in \sigma(A_p)$ 
for all $p \in \Crit(f)$, and the Morse index of~$p$ is then the algebraic
count of negative eigenvalues in $\sigma(A_p)$.
It turns out that the spectrum $\sigma(A_p)$ also controls the asymptotic 
behavior of gradient-flow lines approaching~$p$:
the following result from the theory of ordinary differential equations 
makes this statement precise.

\begin{prop}
\label{prop:MorseAsymp}
Suppose $f : M \to \RR$ is a Morse function on a Riemannian manifold
$(M,g)$, and $x \in \mM(p_+,p_-)$ is a gradient-flow line between two
critical points $p_+,p_- \in \Crit(f)$.  Let $h_\pm(s) \in T_{p_\pm} M$
denote the unique smooth functions defined for $s$ sufficiently close
to $\pm\infty$ by
$$
x(s) = \exp_{p_\pm} h_\pm(s).
$$
Then there exist unique nontrivial eigenvectors $v_\pm \in T_{p_\pm} M$ 
of $A_{p_\pm}$ with
$$
A_{p_\pm} v_\pm = \lambda_\pm v_\pm, \qquad
\lambda_+ < 0 \text{ and } \lambda_- > 0,
$$
such that $h_+(s)$ and $h_-(s)$ satisfy the exponential decay formula
$$
h_\pm(s) = e^{\lambda_\pm s} (v_\pm + r_\pm(s))
\quad\text{ for $s$ near $\pm\infty$,}
$$
where $r_\pm(s) \in T_{p_\pm} M$ are functions satisfying 
$r_\pm(s) \to 0$ as $s \to \pm\infty$.
\end{prop}
\begin{exercise}
\label{EX:workhorse}
Try to prove the following lemma in the background of 
Proposition~\ref{prop:MorseAsymp}: suppose $S$ is a real symmetric
$n$-by-$n$ matrix, $A(s)$ is a smooth matrix-valued function with $A(s) \to S$ as
$s \to \infty$ and $v(s) \in \RR^n$ is a smooth function that is defined
for large~$s$, satisfies the linear ODE $\dot{v}(s) - A(s) v(s) = 0$ and
decays to~$0$ as $s \to \infty$.  Then $v(s)$ satisfies
$$
v(s) = e^{\lambda s}(v_+ + r(s))
$$
for a unique eigenvector $v_+$ of $S$ with $S v_+ = \lambda v_+$ and 
$\lambda < 0$, and a function $r(s)$ with $r(s) \to 0$ as $s \to \infty$.
\end{exercise}

One consequence of Proposition~\ref{prop:MorseAsymp} is that the direction
of approach of a gradient-flow line to a nondegenerate critical point is
always determined by an eigenvector of the Hessian.  We will not discuss
this result any further here, but it will serve as motivation for some
similar results about asymptotics of $J$-holomorphic half-cylinders,
which can be proved using methods of elliptic regularity theory.

To see what this discussion has to do with holomorphic curves,
consider a contact manifold $(M,\xi)$ with contact form $\alpha$ and
translation-invariant almost complex structure $J \in \jJ(\alpha)$ on
the symplectization $(\RR \times M, d(e^s\alpha))$.  Denote the
positive/negative half-cylinders by
$$
Z_+ := [0,\infty) \times S^1, \qquad
Z_- := (-\infty,0] \times S^1
$$
with their standard complex structures defined by $i\p_s = \p_t$ in the
coordinates $(s,t)$.  We defined in \S\ref{sec:punctured} what it means
for a $J$-holomorphic half-cylinder $u : (Z_\pm,i) \to (\RR \times M,J)$
to be asymptotic to a closed Reeb orbit.  We claim that such half-cylinders
can be regarded in a loose sense as \emph{gradient-flow lines} of a functional
on $C^\infty(S^1,M)$ whose critical points are closed Reeb orbits.
To see this, 
let $\pi_\alpha : TM \to \xi$ denote the projection along the Reeb vector field.
Then the nonlinear Cauchy-Riemann equation $\p_s u + J(u) \, \p_t u = 0$
satisfied by a map $u = (f,v) : Z_\pm \to \RR \times M$ is 
equivalent to the three equations
\begin{equation}
\label{eqn:tripartite}
\begin{split}
\p_s f - \alpha(\p_t v) &= 0, \\
\p_t f + \alpha(\p_s v) &= 0, \\
\pi_\alpha \p_s v + J\,  \pi_\alpha \p_t v &= 0.
\end{split}
\end{equation}
Consider the \defin{contact action functional}\index{contact action functional}
$$
\Phi_\alpha : C^\infty(S^1,M) \to \RR : \gamma \mapsto \int_{S^1} \gamma^*\alpha.
$$
\begin{exercise}
Show that for any smooth $1$-parameter family of loops
$\gamma_s : S^1 \to M$ with $\gamma := \gamma_0$ and
$\eta := \p_s \gamma_s|_{s=0} \in \Gamma(\gamma^*TM)$,
$$
d\Phi_\alpha(\gamma) \eta := 
\left. \frac{d}{ds} \Phi_\alpha(\gamma_s)\right|_{s=0} =
\int_{S^1} d\alpha(\eta(t), \dot{\gamma}(t)) \, dt.
$$
Deduce that $\gamma \in C^\infty(S^1,M)$ is a critical point of
$\Phi_\alpha$ if and only if $\dot{\gamma}(t) \in \ker d\alpha$ for
all $t$, meaning $\dot{\gamma}$ is everywhere proportional
to~$R_\alpha$.
\end{exercise}

Observe that $\Phi_\alpha$ has a very large symmetry group: it is
independent of the choice of parametrization for a loop $\gamma : S^1 \to M$,
and correspondingly, $d\Phi_\alpha(\gamma) \eta$ vanishes for any
variation $\eta$ in the direction of the Reeb vector field.  Since the main
point of this discussion however is to study asymptotic approach to
Reeb orbits, we can limit our attention to loops that are $C^\infty$-close
to Reeb orbits: such loops are always immersions transverse to $\xi$,
and all nearby loops are obtained (up to parametrization) via perturbations
along~$\xi$.  We shall therefore consider $d\Phi_\alpha(\gamma)$ restricted
to sections of $\gamma^*\xi$.  Define an
$L^2$-inner product on $\Gamma(\gamma^*\xi)$ by
\begin{equation}
\label{eqn:L2}
\langle \eta_1 , \eta_2 \rangle_{L^2} := \int_{S^1} 
d\alpha(\eta_1(t) , J \eta_2(t))  \, dt;
\end{equation}
this is nondegenerate and symmetric since
$J|_\xi$ is compatible with $d\alpha|_\xi$.  Now for any
$\gamma \in C^\infty(S^1,M)$ and $\eta \in \Gamma(\gamma^*\xi)$, we have
$$
d\Phi_\alpha(\gamma) \eta = \langle -J \pi_\alpha \dot{\gamma} , 
\eta \rangle_{L^2},
$$
thus we can sensibly define $\nabla^\xi \Phi_\alpha(\gamma) := 
- J \pi_\alpha \dot{\gamma}$ and interpret the third equation in
\eqref{eqn:tripartite} as a gradient flow equation for the family of loops
$v(s) := v(s,\cdot) \in C^\infty(S^1,M)$,
\begin{equation}
\label{eqn:gradFlow}
\pi_\alpha \p_s v(s) = \nabla^\xi \Phi_\alpha(v(s)).
\end{equation}
This interpretation is mostly formal, as equations like \eqref{eqn:gradFlow}
typically do not yield well-defined flows on infinite-dimensional 
Fr\'echet manifolds
such as $C^\infty(S^1,M)$; in reality, one must study these equations as
PDEs rather than ODEs and use elliptic theory to obtain results, but the
gradient-flow interpretation provides something of a blueprint indicating
what results one should try to prove.

For example, it is now reasonable to
expect that the asymptotic behavior of solutions to \eqref{eqn:tripartite} 
might be controlled by the spectrum of some symmetric operator interpreted
as the ``Hessian'' of~$\Phi_\alpha$.  We deduce the form of this operator
as follows.  Assume $\gamma : S^1 \to M$ parametrizes a Reeb orbit with
period $T > 0$ such that $\alpha(\dot{\gamma}(t)) = T$ for all~$t$.
Suppose $\gamma_s$ is a smooth $1$-parameter family of loops with
$\gamma_0 = \gamma$ and $\p_s\gamma_s|_{s=0} =: \eta \in \Gamma(\gamma^*\xi)$.
Then choosing any symmetric connection $\nabla$ on~$M$, the Hessian
of $\Phi_\alpha$ at $\gamma$ should map $\eta$ to the covariant derivative
of $\nabla^\xi \Phi_\alpha$ in the direction~$\eta$: a computation gives
\begin{equation}
\label{eqn:Hessian}
\nabla \left(\nabla^\xi\Phi_\alpha\right)(\gamma) \eta := 
\left. \nabla_s\left( \nabla^\xi \Phi_\alpha\right)(\gamma_s)\right|_{s=0} = 
\left.\nabla_s \left( -J \pi_\alpha \dot{\gamma}_s \right)\right|_{s=0} = 
-J (\nabla_t \eta - T \nabla_\eta R_\alpha).
\end{equation}
Note that since $\nabla^\xi\Phi_\alpha(\gamma) = 0$, this expression is
independent of the choice of connection.
This motivates the following definition.

\begin{defn}
\label{defn:operator}
Given a Reeb orbit $\gamma : S^1 \to M$ parametrized so that
$\alpha(\dot{\gamma}) \equiv T > 0$ is constant, the \defin{asymptotic
operator}\index{asymptotic operator}\index{Reeb orbit!asymptotic operator of}
associated to $\gamma$ is
$$
\mathbf{A}_\gamma : \Gamma(\gamma^*\xi) \to \Gamma(\gamma^*\xi) :
\eta \mapsto -J (\nabla_t \eta - T \nabla_\eta R_\alpha).
$$
\end{defn}
\begin{exercise}
Fill in the gaps in the computation \eqref{eqn:Hessian}.
\end{exercise}

Let $H^1(\gamma^*\xi)$ denote the Sobolev space of sections\index{Sobolev spaces}
$S^1 \to \gamma^*\xi$ of class $L^2$ that have weak derivatives also
of class~$L^2$.  The operator $\mathbf{A}_\gamma$ then extends to a
continuous linear map $H^1(\gamma^*\xi) \to L^2(\gamma^*\xi)$.
By a similar argument as with the usual Hessian of a smooth function on a
finite-dimensional manifold, one can show that $\mathbf{A}_\gamma$ is
always symmetric with respect to the $L^2$-inner product \eqref{eqn:L2},
and in fact:
\begin{prop}[\cite{HWZ:props2}*{\S 3}]
\label{prop:selfAdjoint}
For every Reeb orbit $\gamma$, the asymptotic operator $\mathbf{A}_\gamma$
determines an unbounded self-adjoint operator on $L^2(\gamma^*\xi)$ with
dense domain $H^1(\gamma^*\xi)$.  Its spectrum $\sigma(\mathbf{A}_\gamma)$
consists of real eigenvalues that accumulate at $-\infty$ and $+\infty$,
and nowhere else.
\end{prop}

The natural analogue of the Morse condition for $\Phi_\alpha$ is now
the following.

\begin{defn}
\label{defn:nondegenerate}
A Reeb orbit $\gamma$ is called \defin{nondegenerate}\index{Reeb orbit!nondegenerate}\index{nondegenerate Reeb orbit}
if 
$\ker \mathbf{A}_\gamma = \{0\}$.
\end{defn}
\begin{exercise}
\label{EX:nondegenerate}
Show that for any contact form $\alpha$, the flow $\varphi_{R_\alpha}^t$ of
the Reeb vector field preserves $\alpha$ for all~$t$, so in particular,
it preserves $\xi = \ker \alpha$ and the symplectic bundle structure
$d\alpha|_{\xi}$.  Then show that a
Reeb orbit $\gamma : S^1 \to M$ of period $T > 0$ is nondegenerate
if and only if
$$
d \varphi_{R_\alpha}^T|_{\xi_{\gamma(0)}} : \xi_{\gamma(0)} \to
\xi_{\gamma(0)}
$$
does not have $1$ as an eigenvalue.  Deduce from this that nondegenerate
Reeb orbits are (up to parametrization) always isolated in~$C^\infty(S^1,M)$.
\end{exercise}

\section{Asymptotic formulas for cylidrical ends}
\label{sec:asymptotics}

We shall now state some asymptotic results analogous
to Proposition~\ref{prop:MorseAsymp}, but for holomorphic curves instead
of gradient-flow lines.  In the form presented here, these results are
due to Siefring \cites{Siefring:thesis,Siefring:asymptotics}, 
and they are generalizations
and improvements of earlier results of Hofer-Wysocki-Zehnder \cites{HWZ:props1,HWZ:props4},
Kriener \cite{Kriener} and Mora~\cite{Mora}.  The proofs are lengthy and
technical,
so we will omit them, but the results should hopefully be believable via
the analogy with Morse theory discussed above.

The basic workhorse result of this subject is
an asymptotic analogue of the similarity principle (Theorem~\ref{thm:similarity}),
in the spirit of Exercise~\ref{EX:workhorse}.  To state this,
recall that for any closed Reeb orbit $\gamma : S^1 \to M$ on a
$(2n+1$)-dimensional contact manifold $(M,\xi = \ker\alpha)$,
one can find a \emph{unitary} trivialization of the bundle
$\gamma^*\xi \to S^1$, identifying $d\alpha|_\xi$ and $J|_\xi$ with the
standard symplectic and complex structures on $\RR^{2n} = \CC^n$.
If $J_0 : \RR^{2n} \to \RR^{2n}$ denotes the standard complex structure, the
asymptotic operator $\mathbf{A}_\gamma : \Gamma(\gamma^*\xi) \to 
\Gamma(\gamma^*\xi)$ is then identified with a first-order differential
operator
\begin{equation}
\label{eqn:modelA}
\mathbf{A} := -J_0 \frac{d}{dt} - S : C^\infty(S^1,\RR^{2n}) \to C^\infty(S^1,\RR^{2n}),
\end{equation}
where $S : S^1 \to \End(\RR^{2n})$ is a smooth loop of real
$2n$-by-$2n$ matrices, and symmetry of $\mathbf{A}$ with respect
to the standard $L^2$-inner product translates into the condition that
$S(t)$ is a symmetric matrix for all~$t$.  The following statement and the
two that follow it should each be interpreted as \emph{two} closely
related statements, one with plus signs and the other with minus signs.

\begin{thm}
\label{thm:asympSimilarity}
Suppose $S : Z_\pm \to \End(\RR^{2n})$ is a smooth family of $2n$-by-$2n$
matrices satisfying\index{similarity principle!asymptotic analogue}
$$
S(s,t) \mapsto S(t) \quad \text{ uniformly in~$t$ as $s \to \pm\infty$},
$$
where $S : S^1 \to \End(\RR^{2n})$ is a smooth family of 
\emph{symmetric} matrices such that the asymptotic operator
$\mathbf{A}$ defined in \eqref{eqn:modelA} has trivial kernel.
Suppose further that $f : Z_\pm \to \RR^{2n}$ is a smooth function
that is not identically zero and satisfies
\begin{equation}
\label{eqn:CRasymptotic}
\p_s f(s,t) + J_0\, \p_t f(s,t) + S(s,t) f(s,t) = 0, \quad\text{ and }\quad
f(s,\cdot) \to 0 \text{ uniformly as $s \to \pm\infty$}.
\end{equation}
Then there exists a unique nontrivial eigenfunction 
$v_\lambda \in C^\infty(S^1,\RR^{2n})$ of $\mathbf{A}$ with
$$
\mathbf{A} v_\lambda = \lambda v_\lambda, \qquad \pm \lambda < 0,
$$
and a function $r(s,t) \in \RR^{2n}$ satisfying $r(s,\cdot) \to 0$
uniformly as $s \to \pm \infty$, such that for
sufficiently large $|s|$,
\begin{equation}
\label{eqn:asympSimilarity}
f(s,t) = e^{\lambda s} \left[ v_\lambda(t) + r(s,t) \right].
\end{equation}
\end{thm}

Now assume $\gamma : S^1 \to M$ is a nondegenerate $T$-periodic Reeb orbit in
$(M,\xi = \ker\alpha)$, parametrized so that $\alpha(\dot{\gamma}) \equiv T$.
Nondegeneracy implies that the asymptotic operator $\mathbf{A}_\gamma$
has trivial kernel.  Fixing $J \in \jJ(\alpha)$, 
recall that in \S\ref{sec:punctured}, we defined a $J$-holomorphic
half-cylinder $u : Z_+ \to \RR \times M$ or $u : Z_- \to \RR \times M$
to be (positively or negatively) asymptotic to $\gamma$
if, after a possible reparametrization near infinity,
\begin{equation}
\label{eqn:uh}
u(s,t) = \exp_{(Ts,\gamma(t))} h(s,t)  \quad \text{ for $|s|$ large},
\end{equation}
where the exponential map is assumed translation-invariant and
$h(s,t)$ is a vector field along the orbit cylinder with
$h(s,\cdot) \to 0$ in $C^\infty(S^1)$ as $s \to \pm\infty$.  In particular,
as $|s| \to \infty$, $u(s,t)$ becomes $C^\infty$-close to the orbit cylinder
$(s,t) \mapsto (Ts,\gamma(t))$, which is 
an immersion with normal bundle equivalent
to $\gamma^*\xi$.  After a further
reparametrization of $Z_\pm$, we can then arrange for \eqref{eqn:uh} to
hold for a unique section
$$
h(s,t) \in \xi_{\gamma(t)},
$$
which we will call the \defin{asymptotic representative}\index{asymptotic representative of a punctured holomorphic curve}\index{holomorphic curve!asymptotic representative at a puncture}
of~$u$.
Note that the uniqueness of $h$ depends on our choice of parametrization
$\gamma : S^1 \to M$ for the Reeb orbit; different choices will change
$h$ by a shift in the $t$-coordinate.

The relation \eqref{eqn:uh} is a special case of the following
general scenario.  We have an almost complex manifold $(W,J)$ with two immersed $J$-holomorphic curves
$v : (\Sigma,j) \to (W,J)$ and $u : (\Sigma',j') \to (W,J)$, together with
a (not necessarily holomorphic) ``reparametrization''
diffeomorphism $\varphi : \Sigma \to \Sigma'$ and a section $h$
of the normal bundle $N_v \to \Sigma$ to $v$ such that
$$
u \circ \varphi = \exp_v h, \quad\text{ or equivalently }\quad
u = \exp_{v \circ \psi} \eta,
$$
where we define $\psi := \varphi^{-1}$ and $\eta := h \circ \psi$, the
latter being a section of the induced bundle $\psi^*N_v \to \Sigma'$.
It turns out (see Proposition~\ref{prop:pushoff} in Appendix~\ref{sec:normalPushoffLemma})
that in this situation, one can always view
$\eta$ as a solution of a linear Cauchy-Riemann type equation,
hence its local behavior is governed
by the similarity principle---or in the asymptotic setting,
by Theorem~\ref{thm:asympSimilarity} above.
In the present context, this idea can
be used to prove:

\begin{thm}
\label{thm:asymptotics}
Suppose $u : Z_\pm \to \RR \times M$ is a $J$-holomorphic half-cylinder
positively/negatively asymptotic
to the nondegenerate Reeb orbit $\gamma : S^1 \to M$, and let
$h_u(s,t) \in \xi_{\gamma(t)}$ denote its asymptotic representative.
Then if $h_u$ is not identically zero,
there exists a unique nontrivial eigenfunction $f_\lambda$ of
$\mathbf{A}_\gamma$ with
$$
\mathbf{A}_\gamma f_\lambda = \lambda f_\lambda, \qquad \pm \lambda < 0,
$$
and a section $r(s,t) \in \xi_{\gamma(t)}$ satisfying
$r(s,\cdot) \to 0$ uniformly as $s \to \pm \infty$,
such that for sufficiently large $|s|$,
$$
h_u(s,t) = e^{\lambda s} \left[ f_\lambda(t) + r(s,t) \right].
$$
\end{thm}

In the situation of Theorem~\ref{thm:asymptotics}, we will say that
$u(s,t)$ approaches the Reeb orbit $\gamma$ along the
\defin{asymptotic eigenfunction}\index{asymptotic eigenfunction}\index{holomorphic curve!asymptotic eigenfunction at a puncture}
$f_\lambda$ and with 
\defin{decay rate}~$|\lambda|$.\index{holomorphic curve!decay rate at a puncture}
Observe that this theorem can be viewed as
describing the asymptotic approach of \emph{two} $J$-holomorphic half-cylinders
to each other, namely $u(s,t)$ and the orbit cylinder $(Ts,\gamma(t))$.
A similar result holds for \emph{any} two
curves approaching the same orbit, and one can then establish a lower bound 
on the resulting ``relative'' decay rate.  For our purposes,
this result can be expressed most conveniently as follows.

\begin{thm}
\label{thm:asymptoticsRelative}
Suppose $u , v : Z_\pm \to \RR \times M$ are two $J$-holomorphic half-cylinders,
both positively/negatively asymptotic to the nondegenerate Reeb orbit 
$\gamma : S^1 \to M$,
with asymptotic representatives $h_u$ and $h_v$, asymptotic
eigenfunctions $f_u$, $f_v$ and decay rates
$|\lambda_u|$, $|\lambda_v|$ respectively.
Then if $h_u - h_v$ is not identically zero, it satisfies
$$
h_u(s,t) - h_v(s,t) = e^{\lambda s} \left[ f_\lambda(t) + r(s,t) \right]
$$
for a unique nontrivial eigenfunction $f_\lambda$ of $\mathbf{A}_\gamma$ with
$$
\mathbf{A}_\gamma f_\lambda = \lambda f_\lambda, \qquad \pm \lambda < 0,
$$
and a section $r(s,t) \in \xi_{\gamma(t)}$ satisfying
$r(s,\cdot) \to 0$ uniformly as $s \to \pm\infty$.
Moreover:
\begin{itemize}
\item If $f_u = f_v$, then $|\lambda| > |\lambda_u| = |\lambda_v|$.
\item Otherwise, $|\lambda| = \min\{ |\lambda_u| , |\lambda_v| \}$.
\end{itemize}
\end{thm}

We say in the situation of Theorem~\ref{thm:asymptoticsRelative} that
$u$ and $v$ approach each other along the \defin{relative asymptotic 
eigenfunction}\index{relative asymptotic eigenfunction}\index{holomorphic curve!relative asymptotic eigenfunction of two punctures}\index{asymptotic eigenfunction!relative}
$f_\lambda$ with \defin{relative decay rate} $|\lambda|$.\index{relative decay rate}\index{holomorphic curve!relative decay rate of two punctures}

Observe that if $u$ and $v$ are two asymptotically cylindrical curves with
a pair of ends for which $h_u - h_v \equiv 0$ in Theorem~\ref{thm:asymptoticsRelative}, 
then standard unique continuation arguments imply that $u$ and $v$ have identical images,
i.e.~they both cover the same simple curve.  In all other cases, the asymptotic
formula provides a neighborhood of infinity on which $h_u - h_v$ must be nowhere
zero, so $u$ and $v$ have no intersections near infinity.  If $u$ and $v$
are asymptotic to different \emph{covers} of the same orbit, then one can
argue in the same way by replacing each with suitable covers
$$
\tilde{u}(s,t) := u(ks,kt), \qquad \tilde{v}(s,t) := v(\ell s, \ell t)
$$
whose asymptotic Reeb orbits match.
In this way, one can deduce the following important consequence, which was
not previously obvious:

\begin{cor}
\label{cor:finite}
If $u : (\dot{\Sigma},j) \to (\widehat{W},J)$ and $v : (\dot{\Sigma}',j') \to
(\widehat{W},J)$ are two asymptotically cylindrical $J$-holomorphic curves
with non-identical images, then they have at most finitely many intersections.
\end{cor}

Similarly:

\begin{cor}
\label{cor:branch}
If $u : (\dot{\Sigma},j) \to (\widehat{W},J)$ is an asymptotically cylindrical
$J$-holomorphic curve which is simple, then it is embedded on some
neighborhood of the punctures.
\end{cor}
\begin{proof}
If $u$ has two ends asymptotic to covers of the same orbit, we deduce as in
Corollary~\ref{cor:finite} that their images are either identical or
disjoint near infinity, and the former is excluded via unique continuation
arguments if $u$ is simple.
There could still be double points near a single end asymptotic to a
multiply covered Reeb orbit, i.e.~suppose $Z_\pm \subset \dot{\Sigma}$ is
an end on which $u|_{Z_\pm}$ is asymptotic to
$$
\gamma(t) = \gamma_0(kt),
$$
where $k \ge 2$ is an integer and $\gamma_0 : S^1 \to M$ is an
\emph{embedded} Reeb orbit.  Then writing $u(s,t) = 
\exp_{(Ts,\gamma_0(kt))} h(s,t)$ on $Z_\pm$ as
in Theorem~\ref{thm:asymptotics},
the reparametrizations $u_j(s,t) := u(s,t + j/k)$ for
$j=1,\ldots,k-1$ are each also $J$-holomorphic half-cylinders
asymptotic to $\gamma$, with asymptotic representatives $h_j(s,t)
:= h(s,t + j/k)$, and Theorem~\ref{thm:asymptoticsRelative} implies
that $h_j - h$ is either identically zero or nowhere zero
near infinity for $j=1,\ldots,k-1$.
The former is again excluded via unique continuation if $u$ is simple.
\end{proof}

\section{Winding of asymptotic eigenfunctions}
\label{sec:winding}

When $\dim M = 3$, the asymptotic eigenfunctions in the above discussion
are nowhere vanishing sections of complex line bundles $\gamma^*\xi \to S^1$, so
they have well-defined winding numbers relative to any choice of
trivialization.  This defines the notion of the \emph{asymptotic winding}
of a holomorphic curve as it approaches an orbit.  It is extremely useful
to observe that these winding numbers come with \emph{a priori} bounds.

\begin{thm}[\cite{HWZ:props2}]
\label{thm:winding}
Suppose $S : S^1 \to \End(\RR^2)$ is a smooth loop of symmetric
$2$-by-$2$ matrices and $\mathbf{A} : C^\infty(S^1,\RR^2) \to 
C^\infty(S^1,\RR^2)$ denotes the model asymptotic operator
$$
\mathbf{A} = -J_0 \frac{d}{dt} - S,
$$
with spectrum $\sigma(\mathbf{A}) \subset \RR$.  Then there exists a well-defined
integer-valued function
$$
\wind : \sigma(\mathbf{A}) \to \ZZ
$$
determined by $\wind(\lambda) := \wind(v_\lambda)$, where 
$v_\lambda \in C^\infty(S^1,\RR^2)$ is any nontrivial eigenfunction with
eigenvalue~$\lambda$.  Moreover, this function is monotone increasing and attains
every value in $\ZZ$ exactly twice (counting multiplicity of eigenvalues).
\end{thm}

\begin{exercise}
Verify Theorem~\ref{thm:winding} for the special case where
$S(t)$ is a constant multiple of the identity matrix.
\textsl{(The general case can be derived from this using perturbation
  theory for self-adjoint operators; see \cite{HWZ:props2}*{Lemma~3.6}
  or \cite{Wendl:SFT}*{Chapter~3}.)}
\end{exercise}

Given a closed Reeb orbit $\gamma$ in a contact $3$-manifold
$(M,\xi = \ker\alpha)$, one can now choose a trivialization
$$
\tau : \gamma^*\xi \to S^1 \times \RR^2
$$
and define
$$
\wind^\tau : \sigma(\mathbf{A}_\gamma) \to \ZZ
$$
by $\wind^\tau(\lambda) := \wind(f)$ where $f : S^1 \to \RR^2$ is the expression
via $\tau$ of any nontrivial eigenfunction $f_\lambda \in \Gamma(\gamma^*\xi)$
with $\mathbf{A}_\gamma f_\lambda = \lambda f_\lambda$.  It follows 
immediately from Theorem~\ref{thm:winding} that $\wind^\tau$ is a monotone
surjective function attaining all values exactly twice.  Since eigenvalues
of $\mathbf{A}_\gamma$ do not accumulate except at $\pm\infty$, 
we can then define the integers:\index{Reeb orbit!extremal winding numbers of}\index{extremal winding numbers of a Reeb orbit}
\begin{equation}
\label{eqn:alphas}
\begin{split}
\alpha_+^\tau(\gamma) &:= \min\left\{ \wind^\tau(\lambda)\ |\ 
\lambda \in \sigma(\mathbf{A}_\gamma) \cap (0,\infty) \right\},\\
\alpha_-^\tau(\gamma) &:= \max\left\{ \wind^\tau(\lambda)\ |\ 
\lambda \in \sigma(\mathbf{A}_\gamma) \cap (-\infty,0) \right\},\\
p(\gamma) &:= \alpha_+^\tau(\gamma) - \alpha_-^\tau(\gamma).
\end{split}
\end{equation}
As implied by this choice of notation, $\alpha_\pm^\tau(\gamma)$ each
depend on the choice of trivialization $\tau$, but $p(\gamma)$ does not.
If $\gamma$ is nondegenerate, hence $0 \not\in\sigma(\mathbf{A}_\gamma)$,
it follows from Theorem~\ref{thm:winding} that $p(\gamma)$ is either $0$
or~$1$: we shall say accordingly that $\gamma$ is \defin{even}\index{Reeb orbit!even/odd}\index{even Reeb orbit|see {parity of a Reeb orbit}}\index{odd Reeb orbit|see {parity of a Reeb orbit}}
or \defin{odd} respectively, and call $p(\gamma)$ the \defin{parity} of~$\gamma$.\index{Reeb orbit!parity of}\index{parity of a Reeb orbit}

The winding invariants we've just defined have an important relation with
another integer associated to nondegenerate Reeb orbits, namely the
\defin{Conley-Zehnder index}\index{Conley-Zehnder index}\index{Reeb orbit!Conley-Zehnder index of}\index{nondegenerate Reeb orbit!Conley-Zehnder index of}
$$
\muCZ^\tau(\gamma) \in \ZZ,
$$
a Maslov-type index that was originally introduced in the study of
Hamiltonian systems (see \cites{ConleyZehnder:index,SalamonZehnder:Morse})
and can also be defined for nondegenerate Reeb orbits in any dimension.
It can be thought of
as a measurement of the degree of ``twisting'' (relative to~$\tau$) 
of the nearby Reeb flow around~$\gamma$.  We refer to \cite{HWZ:props2}*{\S 3}
or \cite{Wendl:SFT}*{Chapter~3}
for further details on $\muCZ^\tau$; for our purposes in the $3$-dimensional
case, the following result from \cite{HWZ:props2}*{\S 3} can just as well
be taken as a \emph{definition}:

\begin{prop}
\label{prop:CZwinding}
For any nondegenerate Reeb orbit $\gamma : S^1 \to M$ in $(M,\xi=\ker\alpha)$ 
with a trivialization $\tau$ of $\gamma^*\xi$,
$$
\muCZ^\tau(\gamma) = 2\alpha^\tau_-(\gamma) + p(\gamma) =
2\alpha^\tau_+(\gamma) - p(\gamma).
$$
\end{prop}

\begin{exercise}
\label{EX:gcd}
To any closed Reeb orbit of period $T > 0$ parametrized by a loop
$\gamma : S^1 \to M$ with $\dot{\gamma} \equiv T \cdot R_\alpha(\gamma)$,
one can associate a Reeb orbit of period $kT$ for each $k \in \NN$,
parametrized by
$$
\gamma^k : S^1 \to M : t \mapsto \gamma(kt).
$$
We say $\gamma^k$ is the \defin{$k$-fold cover} of~$\gamma$,\index{Reeb orbit!multiply covered}\index{multiply covered Reeb orbit}
and it is
\defin{multiply covered} if $k \ge 2$.  We say $\gamma$ is
\defin{simply covered}\index{Reeb orbit!simply covered}\index{simply covered Reeb orbit}
if it is not the $k$-fold cover of another Reeb orbit
for any $k \ge 2$.
\begin{enumerate}
\renewcommand{\labelenumi}{(\alph{enumi})}
\item 
Given a Reeb orbit $\gamma$, check that the $k$-fold cover of each eigenfunction
of $\mathbf{A}_{\gamma}$ is
an eigenfunction of~$\mathbf{A}_{\gamma^k}$.  Assuming $\tau$ is the pullback under
$S^1 \to S^1 : t \mapsto kt$ of a trivialization of $\gamma^*\xi \to S^1$,
deduce from Theorem~\ref{thm:winding} that a nontrivial eigenfunction 
$f$ of $\mathbf{A}_{\gamma^k}$ is a
$k$-fold cover if and only if $\wind^\tau(f)$ is divisible by~$k$.
\item
Under the same assumptions, show that for any
nontrivial eigenfunction $f$ of $\mathbf{A}_{\gamma^k}$,
$$
\cov(f) := \max \{ m \in \NN\ |\ \text{$f$ is an $m$-fold cover} \}  
= \gcd(k , \wind^\tau(f)).
$$
\item
Show that if $\gamma$ is a Reeb orbit that has even Conley-Zehnder index,
then so does every multiple cover $\gamma^k$ of $\gamma$.
\end{enumerate}
\end{exercise}

\section{Local foliations and the normal Chern number}
\label{sec:puncturedFoliations}

We now address a generalization of the question considered in
\S\ref{sec:foliations}: if $(\widehat{W},\omega)$ is a completed
symplectic cobordism of dimension~$4$, what conditions can guarantee that a
$2$-parameter family of embedded punctured holomorphic curves in
$\widehat{W}$ will form a foliation?  There are several
issues here that do not arise in the closed case: for example, if
$$
u : \dot{\Sigma} = \Sigma \setminus (\Gamma^+ \cup \Gamma^-) \to \widehat{W}
$$
is embedded, it is not guaranteed
in general that all nearby curves $u_\epsilon : \dot{\Sigma} \to \widehat{W}$ 
are also embedded, e.g.~$u$ may have multiple ends asymptotic to
the same Reeb orbit, allowing $u_\epsilon$ to have double points near 
that orbit which
escape to infinity as $u_\epsilon \to u$.  We will address this issue in
the next lecture and ignore it for now, as we must first deal with the
more basic question of how to count zeroes of sections on the normal
bundle $N_u \to \dot{\Sigma}$.  Indeed, let us consider as in
\S\ref{sec:foliations} a $1$-parameter family of 
$J$-holomorphic curves 
$u_\sigma$ near~$u_0 := u$, presented as
$\exp_u \eta_\sigma$ for sections $\eta_\sigma \in \Gamma(N_u)$.
One can then show that
\begin{equation}
\label{eqn:perturbation}
\eta := \left. \frac{\p}{\p \sigma} u_\sigma\right|_{\sigma=0} \in
\Gamma(N_u)
\end{equation}
satisfies a linear Cauchy-Riemann
type equation.  We would like to know when such sections are guaranteed
to be nowhere zero.  Write the positive and negative contact boundary
components of the cobordism $(W,\omega)$ as
$$
\p(W,\omega) = (-M_-,\xi_-) \sqcup (M_+,\xi_+).
$$
Since $u$ is always transverse to the
contact bundles $\xi_\pm$ near infinity, one can identify
$N_u$ with $u^*\xi_\pm$ on the cylindrical ends.
By the similarity principle, zeroes of $\eta$ are
isolated and positive, but the total algebraic count of them is not a
homotopy invariant since they may escape to infinity under homotopies;
in fact, there could in theory be infinitely many.  It turns out however
that on any cylindrical end $Z_\pm \subset \dot{\Sigma}$ near a puncture
$z \in \Gamma^\pm$ where $u$ is
asymptotic to an orbit $\gamma_z$, the relevant linear Cauchy-Riemann 
type equation has the same form as in Theorem~\ref{thm:asympSimilarity}, 
with $\mathbf{A}_{\gamma_z}$ as the relevant asymptotic operator.  The theorem
thus implies that $\eta$ is nowhere zero near each puncture~$z$, and it has a 
well-defined \defin{asymptotic winding}\index{asymptotic winding of a section}
relative to any choice of trivialiation
$\tau$ of $\gamma_z^*\xi_\pm$,
$$
\wind^\tau(\eta;z) \in \ZZ,
$$
defined simply as $\wind^\tau(v_\lambda)$ where 
$v_\lambda \in \Gamma(\gamma_z^*\xi_\pm)$
is the asymptotic eigenfunction appearing in \eqref{eqn:asympSimilarity}.
This implies that $\eta^{-1}(0) \subset \dot{\Sigma}$ is finite, 
so we can define the algebraic count of zeroes
\begin{equation}
\label{eqn:Z}
Z(\eta) := \sum_{z \in \eta^{-1}(0)} \ord(\eta;z) \in \ZZ,
\end{equation}
where $\ord(\eta;z)$ denotes the order of each zero, and the similarity 
principle
guarantees that $Z(\eta) \ge 0$, with equality if and only if $\eta$ is nowhere
zero.  This number is still not homotopy invariant, because zeroes can still escape
to infinity under homotopies.  However, the crucial observation is that we can 
\emph{keep track} of this phenomenon via the asymptotic winding numbers: by
Theorem~\ref{thm:winding}, $\wind^\tau(\eta;z)$ satisfies the \emph{a priori} bounds
\begin{equation}
\label{eqn:windingBounds}
\begin{split}
\wind^\tau(\eta;z) &\le \alpha^\tau_-(\gamma_z), \quad \text{ if $z \in \Gamma^+$,}\\
\wind^\tau(\eta;z) &\ge \alpha^\tau_+(\gamma_z), \quad \text{ if $z \in \Gamma^-$.}
\end{split}
\end{equation}
This motivates the definition of the \defin{asymptotic defect}\index{asymptotic defect}
of $\eta$, as the integer
\begin{equation}
\label{eqn:Zinfty}
Z_\infty(\eta) := \sum_{z \in \Gamma^+} \left[ \alpha^\tau_-(\gamma_z) - \wind^\tau(\eta;z) \right]
+ \sum_{z \in \Gamma^-} \left[ \wind^\tau(\eta;z) - \alpha^\tau_+(\gamma_z) \right],
\end{equation}
where the trivializations $\tau$ of $\gamma_z^*\xi_\pm$ can be chosen arbitrarily since each
difference $\alpha_\mp^\tau(\gamma_z) - \wind^\tau(\eta;z)$ does not depend on this choice.
By construction, any $\eta \in \Gamma(N_u)$ satisfying a Cauchy-Riemann type equation as described
above now has both $Z(\eta) \ge 0$ and $Z_\infty(\eta) \ge 0$, and their sum turns out to
give the closest thing possible to a homotopy invariant count of zeroes:

\begin{prop}
\label{prop:totalZ}
For any section $\eta \in \Gamma(N_u)$ with only finitely many zeroes,
the sum $Z(\eta) + Z_\infty(\eta)$ depends only on the bundle $N_u$ and the asymptotic operators
$\mathbf{A}_z$ for $z \in \Gamma$, not on~$\eta$.  In particular, this gives an upper bound
on the algebraic count of zeroes of any section $\eta$ appearing in \eqref{eqn:perturbation}.
\end{prop}

This result motivates the interpretation of $Z_\infty(\eta)$ as a count of \emph{virtual} or
``hidden zeroes at infinity.''\index{zeroes of a section!hidden at infinity}\index{hidden at infinity!zeroes of a section}
We will prove Proposition~\ref{prop:totalZ} by defining another quantity that is manifestly
homotopy invariant and happens to equal $Z(\eta) + Z_\infty(\eta)$: this will be a
generalization of the \emph{normal Chern number}, which we defined for closed holomorphic
curves in \S\ref{sec:adjunction}.  

We must first define the notion of a \emph{relative}
first Chern number for complex vector bundles over punctured surfaces.
Suppose first that $E \to \dot{\Sigma}$ is a complex line bundle, and $\tau$ denotes a\index{asymptotic trivializations|(}
choice of trivializations for $E$ over the cylindrical ends of $\dot{\Sigma}$,
i.e.~over small neighborhoods of each puncture.  Such
trivializations always exist since complex vector bundles over $S^1$ are always trivial.
In fact, $E \to \dot{\Sigma}$ is globally trivializable if the set of punctures
is nonempty, because $\dot{\Sigma}$ is then retractable
to its $1$-skeleton---nonetheless, a given set of trivializations $\tau$
over the ends may or may not be globally extendable over the rest of~$\dot{\Sigma}$.
An obstruction to such extensions is given by
the \defin{relative first Chern number}\index{relative first Chern number}
of $E$ with respect to~$\tau$:
we define it as an algebraic count of zeroes,
$$
c_1^\tau(E) := Z(\eta) \in \ZZ
$$
where $Z(\eta)$ is defined as in \eqref{eqn:Z} for a section
$\eta \in \Gamma(E)$ with finitely many zeroes, and we assume that $\eta$
is \emph{constant and nonzero near infinity}
with respect to~$\tau$.  It follows by standard arguments as in \cite{Milnor:differentiable}
that $c_1^\tau(E)$ does not depend on the choice $\eta$: the point is that any two
such choices are homotopic through sections that are nonzero near infinity, so zeroes
stay within a compact subset under the homotopy.  Observe that in the special case
where $\dot{\Sigma} = \Sigma$ is a closed surface without punctures, there is no
choice of asymptotic trivialization $\tau$ to be made and the above definition
matches the usual first Chern number~$c_1(E)$.  When there are punctures, $c_1^\tau(E)$
depends on the choice~$\tau$.

For a higher rank complex vector bundle $E \to \dot{\Sigma}$ with a trivialization $\tau$
near infinity, $c_1^\tau(E)$ can be defined by assuming the following two axioms:
\begin{enumerate}
\item $c_1^{\tau_1 \oplus \tau_2}(E_1 \oplus E_2) = c_1^{\tau_1}(E_1) + c_1^{\tau_2}(E_2)$;
\item $c_1^\tau(E) = c_1^{\tau'}(E')$ whenever $E$ and $E'$ admit a complex bundle isomorphism
identifying $\tau$ with~$\tau'$.
\end{enumerate}
The following exercise shows that this is a reasonable definition.
\begin{exercise}
\label{EX:relativeChern}
Show that for any complex vector bundle $E$ over a punctured Riemann surface 
$\dot{\Sigma}$ of rank~$n$ with an
asymptotic trivialization $\tau$, there exist complex line bundles
$E_1,\ldots,E_n \to \dot{\Sigma}$ with asymptotic trivializations $\tau_1,\ldots,\tau_n$
such that
$$
(E,\tau) \cong (E_1 \oplus \ldots \oplus E_n , \tau_1 \oplus \ldots \oplus \tau_n),
$$
and if $E_1',\ldots,E_n'$ and $\tau_1',\ldots,\tau_n'$ are another $n$-tuple of line
bundles and asymptotic trivializations with this property, then
$$
c_1^{\tau_1}(E_1) + \ldots + c_1^{\tau_2}(E_2) = c_1^{\tau_1'}(E_1') + \ldots
+ c_1^{\tau_n'}(E_n').
$$
\end{exercise}

From now on, let $\tau$ denote a fixed arbitrary choice of trivializations of the\index{asymptotic trivializations|)}
bundles $\gamma^*{\xi_\pm}$ for all Reeb orbits~$\gamma$; several
things in the calculations below will depend on this choice, but the
most important expressions typically will not.  Since the normal
bundle $N_u$ matches $\xi_\pm$ near infinity, $\tau$ determines an asymptotic
trivialization of $N_u$, allowing us to define the relative first Chern
number $c_1^\tau(N_u)$.  More generally, if $u : \dot{\Sigma} \to \widehat{W}$ is
\emph{any} asymptotically cylindrical map, not necessarily immersed, then
it is still immersed and transverse to $\xi_\pm$ near infinity, so $\tau$
also determines an asymptotic trivialization of the rank~$2$ complex
vector bundle $(u^*T\widehat{W},J) \to \dot{\Sigma}$, by observing that the
first factor in the splitting
$$
T(\RR \times M_\pm) = (\RR \oplus \RR R_{\alpha_\pm}) \oplus \xi_\pm
$$
carries a canonical complex trivialization.  We shall denote the resulting
relative first Chern number for $u^*T\widehat{W}$ by $c_1^\tau(u^*T\widehat{W})$.

\begin{exercise}
Show that if $u : (\dot{\Sigma},j) \to (\widehat{W},J)$ is an asymptotically 
cylindrical and immersed $J$-holomorphic curve, with complex normal bundle
$N_u \to \dot{\Sigma}$, then
\begin{equation}
\label{eqn:c1relation}
c_1^\tau(u^*T\widehat{W}) = \chi(\dot{\Sigma}) + c_1^\tau(N_u).
\end{equation}
\end{exercise}

\begin{defn}
\label{defn:normalChern}
For any asymptotically cylindrical $J$-holomorphic curve
$u : (\dot{\Sigma},j) \to (\widehat{W},J)$ 
asymptotic to Reeb orbits $\gamma_z$ in $M_\pm$ at its punctures 
$z \in \Gamma^\pm$, we define the \defin{normal Chern number}\index{normal Chern number!of a punctured holomorphic curve}\index{holomorphic curve!normal Chern number of}
of $u$ to be the integer
$$
c_N(u) := c_1^\tau(u^*T\widehat{W}) - \chi(\dot{\Sigma}) +
\sum_{z \in \Gamma^+} \alpha^\tau_-(\gamma_z) - \sum_{z \in \Gamma^-}
\alpha^\tau_+(\gamma_z).
$$
\end{defn}
\begin{exercise}
Show that the definition of $c_N(u)$ above is independent of the
choice of trivializations~$\tau$.
\end{exercise}

The normal Chern number $c_N(u)$ clearly depends only on the homotopy class
of $u$ as an asymptotically cylindrical map, together with 
the properties of its asymptotic Reeb orbits.  When $u$ is immersed,
we can rewrite it via \eqref{eqn:c1relation} as
\begin{equation}
\label{eqn:cNrelation}
c_N(u) = c_1^\tau(N_u) +
\sum_{z \in \Gamma^+} \alpha^\tau_-(\gamma_z) - \sum_{z \in \Gamma^-}
\alpha^\tau_+(\gamma_z).
\end{equation}
Proposition~\ref{prop:totalZ} then follows immediately from:

\begin{thm}
\label{thm:totalZ}
Suppose $u : (\dot{\Sigma},j) \to (\widehat{W},J)$ is an immersed\index{normal Chern number!as a count of zeroes}
asymptotically cylindrical $J$-holomorphic curve, and
$\eta \in \Gamma(N_u)$ is a smooth section of its normal bundle
with at most finitely many zeroes.  Then
$$
Z(\eta) + Z_\infty(\eta) = c_N(u).
$$
\end{thm}

In the situation of interest, we already know that both
$Z(\eta)$ and $Z_\infty(\eta)$ are nonnegative, so this yields:

\begin{cor}
\label{cor:totalZ}
If $u : (\dot{\Sigma},j) \to (\widehat{W},J)$ is an immersed
asymptotically cylindrical $J$-holomorphic curve and
$\eta \in \Gamma(N_u)$ is a section of its normal bundle
describing nearby $J$-holomorphic curves as in \eqref{eqn:perturbation},
then
$$
Z(\eta) \le c_N(u);
$$
in particular, if $c_N(u) = 0$ then every such section is zero free.
\end{cor}

\begin{proof}[Proof of Theorem~\ref{thm:totalZ}]
Let $\tau_0$ denote the unique choice of asymptotic trivialization
of $N_u$ such that
$$
\wind^{\tau_0}(\eta ; z) = 0 \quad\text{ for all $z \in \Gamma$}.
$$
Note that if $u$ has multiple ends approaching the same orbit
$\gamma$ in $M_\pm$, this choice may require non-isomorphic trivializations
of $\gamma^*\xi_\pm$ for different ends, but this will pose no difficulty
in the following.  For this choice, we have
$$
Z(\eta) = c_1^{\tau_0}(N_u),
$$
thus using \eqref{eqn:cNrelation} and the definition \eqref{eqn:Zinfty} of
$Z_\infty(\eta)$,
\begin{equation*}
\begin{split}
Z(\eta) + Z_\infty(\eta) &= c_1^{\tau_0}(N_u) + 
\sum_{z \in \Gamma^+} \alpha^{\tau_0}_-(\gamma_z) - \sum_{z \in \Gamma^-}
\alpha^{\tau_0}_+(\gamma_z) \\
&= c_N(u).
\end{split}
\end{equation*}
\end{proof}

Corollary~\ref{cor:totalZ} tells us that in order to find
$2$-dimensional families of embedded $J$-holomorphic curves
that locally form foliations, one should restrict attention to
curves satisfying $c_N(u) = 0$.  To see what kinds of curves satisfy
this condition, recall (see Appendix~\ref{app:punctured}) that
a general $J$-holomorphic curve $u : \dot{\Sigma} \to \widehat{W}$
in a $2n$-dimensional cobordism $\widehat{W}$, with
positive/negative punctures $z \in \Gamma := \Gamma^+ \cup
\Gamma^-$ asymptotic
to nondegenerate Reeb orbits $\gamma_z$, is defined to have \defin{index}\index{index!of a punctured holomorphic curve}\index{holomorphic curve!index of}
$$
\ind(u) = (n-3) \chi(\dot{\Sigma}) + 2 c_1^\tau(u^*T\widehat{W}) +
\sum_{z \in \Gamma^+} \muCZ^\tau(\gamma_z) -
\sum_{z \in \Gamma^-} \muCZ^\tau(\gamma_z).
$$
As usual, all dependence on the trivialization $\tau$ in terms on
the right hand side cancels out in the sum.
This index is the \emph{virtual dimension} of the moduli space of all
curves homotopic to~$u$, and for generic~$J$, the open subset of
simple curves in this space is a smooth manifold of this dimension.
Let us restrict to the case $\dim \widehat{W} = 4$, so $n=2$,
and let $g$ denote the genus of~$\Sigma$, hence
\begin{equation}
\label{eqn:index4}
\ind(u) = - \chi(\dot{\Sigma}) + 2 c_1^\tau(u^*T\widehat{W}) +
\sum_{z \in \Gamma^+} \muCZ^\tau(\gamma_z) -
\sum_{z \in \Gamma^-} \muCZ^\tau(\gamma_z).
\end{equation}
and
\begin{equation}
\label{eqn:chi}
\chi(\dot{\Sigma}) = 2 - 2g - \#\Gamma.
\end{equation}
There is also a natural partition of $\Gamma$ into the
\emph{even} and \emph{odd} punctures
$$
\Gamma = \Gamma\even \cup \Gamma\odd,
$$
defined via the parity of the corresponding orbit as defined in
\S\ref{sec:winding}, or equivalently, the parity of the
Conley-Zehnder index.\footnote{Note that while the
Conley-Zehnder index $\muCZ^\tau(\gamma) \in \ZZ$ generally depends
on a choice of trivialization $\tau$ of the contact bundle
along $\gamma$, different choices of trivialization change the
index by multiples of~$2$, thus the odd/even parity is independent
of this choice.}
Now combining \eqref{eqn:index4}, \eqref{eqn:chi}, 
Definition~\ref{defn:normalChern} and the Conley-Zehnder/winding
relations of Proposition~\ref{prop:CZwinding}, we have
\begin{equation}
\label{eqn:2cN}
\begin{split}
2 c_N(u) &= 2 c_1^\tau(u^*T\widehat{W}) - 2 \chi(\dot{\Sigma}) +
\sum_{z \in \Gamma^+} 2 \alpha^\tau_-(\gamma_z) - \sum_{z \in \Gamma^-} 2 \alpha^\tau_+(\gamma_z) \\
&= 2 c_1^\tau(u^*T\widehat{W}) - \chi(\dot{\Sigma}) - (2 - 2g - \#\Gamma) +
\sum_{z \in \Gamma^+} \left[ \muCZ^\tau(\gamma_z) - p(\gamma_z) \right] \\
& \qquad - \sum_{z \in \Gamma^-} \left[ \muCZ^\tau(\gamma_z) + p(\gamma_z) \right] \\
&= \ind(u) - 2 + 2g + \#\Gamma - \#\Gamma\odd \\
&= \ind(u) - 2 + 2g + \#\Gamma\even.
\end{split}
\end{equation}

Since we are interested in $2$-dimensional families of curves,
assume $\ind(u) = 2$.  Then the right hand side of \eqref{eqn:2cN} is nonnegative,
and vanishes if and only if $g = \#\Gamma\even = 0$, i.e.~$\dot{\Sigma}$ is a 
punctured \emph{sphere} and all asymptotic orbits
have \emph{odd} Conley-Zehnder index.  This leads to the following result.
We state it for now with an extra assumption (condition~(iv) below)
in order to avoid the possibility of extra intersections emerging 
from infinity---this can be relaxed using the technology introduced in the
next lecture, but the weaker result will also suffice for our application
in Lecture~\ref{sec:5}.

\begin{thm}
\label{thm:puncturedFoliations}
Suppose $u : (\dot{\Sigma},j) \to (\widehat{W},J)$ is an embedded asymptotically
cylindrical $J$-holomorphic curve such that:
\begin{enumerate}
\item[(i)] $\ind(u) = 2$;
\item[(ii)] $\dot{\Sigma}$ has genus~$0$;
\item[(iii)] all asymptotic orbits of $u$ have odd Conley-Zehnder index;
\item[(iv)] all the punctures are asymptotic to distinct Reeb orbits, all of them
simply covered.
\end{enumerate}
Then some neighborhood of $u$ in the moduli space $\mM_0(\widehat{W},J)$ is a smooth
$2$-dimensional manifold consisting of pairwise disjoint embedded curves that
foliate a neighborhood of $u(\dot{\Sigma})$ in~$\widehat{W}$.
\qed
\end{thm}

\chapter{Intersection theory for punctured holomorphic curves}
\label{sec:4}

\minitoc

\vspace{12pt}

We\CUP{This material will be published by Cambridge University
Press as \textsl{Contact 3-Manifolds, Holomorphic Curves and Intersection Theory}
by Chris Wendl. This pre-publication version is
free to view and download for personal use only. 
Not for re-distribution, re-sale or use in derivative works. \copyright Chris Wendl, 2019.}
are now ready to explain the intersection theory introduced by Siefring
\cite{Siefring:intersection} for asymptotically cylindrical holomorphic
curves in $4$-dimensional completed symplectic cobordisms.  The theory follows a pattern
that we saw in our discussion of the normal Chern number in \S\ref{sec:puncturedFoliations}: 
the obvious geometrically
meaningful quantities such as $u \cdot v$ (counting intersections
between $u$ and $v$) and $\delta(u)$ (counting double points and
critical points of~$u$) can
be defined, and are nonnegative, but they are not homotopy invariant
since intersections may sometimes escape to infinity.  In each case, however,
one can add a nonnegative count of ``hidden intersections at infinity,''
defined in terms of asymptotic winding numbers, so that the sum is homotopy invariant.

\section{Statement of the main results}
\label{sec:statement}

Throughout this lecture, we assume $(W,\omega)$ is a \emph{four-dimensional} 
symplectic cobordism with $\p(W,\omega) = (-M_-,\xi_- = \ker\alpha_-) 
\sqcup (M_+,\xi_+ = \ker\alpha_+)$, $(\widehat{W},\omega)$ is its completion
and $J \in \jJ(\omega,\alpha_+,\alpha_-)$.
For two asymptotically
cylindrical maps $u : \dot{\Sigma} \to \widehat{W}$ and
$v : \dot{\Sigma}' \to \widehat{W}$ with at most finitely many intersections,
we define the algebraic intersection number
$$
u \cdot v := \sum_{u(z) = v(\zeta)} \inter(u,z \,;\, v,\zeta) \in \ZZ,
$$
and similarly, if $u$ has at most finitely many double points and
critical points, then it has a well-defined \emph{singularity index}
$$
\delta(u) := \frac{1}{2} \sum_{u(z) = u(\zeta),\, z \ne \zeta} 
\inter(u , z \,;\, u, \zeta) + \sum_{du(z)=0} \delta(u,z) \in \ZZ,
$$
i.e.~the sum of the local intersection indices for all double points with the
local singularity index at each critical point (cf.~Lemma~\ref{lemma:critInj}).
If $u$ and $v$ are
both asymptotically cylindrical $J$-holomorphic curves, then we saw
in Corollaries~\ref{cor:finite} and~\ref{cor:branch} that $u \cdot v$ is
well defined if $u$ and $v$ have non-identical images, and
$\delta(u)$ is also well defined if $u$ is simple.  Moreover,
the usual results on positivity of intersections
(Appendix~\ref{app:positivity}) then imply
$$
u \cdot v \ge 0,
$$
with equality if and only if $u$ and $v$ are disjoint, and
$$
\delta(u) \ge 0,
$$
with equality if and only if $u$ is embedded.  So far, all of this is the
same as in the closed case, but the crucial difference here is that neither
$u \cdot v$ nor $\delta(u)$ is invariant under homotopies, which makes
them harder to control in general.  For example, there is no reasonable
definition of ``$u \cdot u$'' since trying to count intersections of $u$
with a small perturbation of itself (as one does in the closed case)
may give a number that depends on the perturbation.
The situation is saved by the
following results from \cite{Siefring:intersection}.

\begin{thm}
\label{thm:star}
For any two asymptotically cylindrical maps $u : \dot{\Sigma} \to \widehat{W}$\index{intersection number!of asymptotically cylindrical maps ($*$-pairing)}
and $v : \dot{\Sigma}' \to \widehat{W}$ with
nondegenerate\footnote{Theorems~\ref{thm:star} 
and~\ref{thm:adjunctionPunctured} both also hold under the more
general assumption that all asymptotic orbits belong to
Morse-Bott families, as long as one imposes the restriction that
asymptotic orbits of curves are not allowed to change under
homotopies. (This assumption is vacuous in the nondegenerate case
since nondegenerate orbits are isolated.)  One can also generalize
the theory further to allow homotopies with moving asymptotic orbits,
in which case additional nonnegative counts of ``hidden'' intersections
must be introduced; see \cite{Wendl:automatic}*{\S 4.1} and \cite{SiefringWendl}.}
asymptotic orbits, there exists a pairing
$$
u * v \in \ZZ
$$
with the following properties:
\begin{enumerate}
\item $u * v$ depends only on the homotopy classes of $u$ and $v$ as
asymptotically cylindrical maps;
\item If $u : (\dot{\Sigma},j) \to (\widehat{W},J)$ and $v : (\dot{\Sigma},j') \to (\widehat{W},J)$ 
are $J$-holomorphic curves with non-identical images, then
$$
u * v = u \cdot v + \inter_\infty(u,v),
$$
where $\inter_\infty(u,v)$ is a nonnegative integer interpreted as the
count of ``hidden intersections at infinity.''\index{hidden at infinity!intersections}\index{intersections!hidden at infinity}
Moreover, there exists a perturbation 
$J_\epsilon \in \jJ(\omega,\alpha_+,\alpha_-)$ which is
$C^\infty$-close to $J$, and a pair of asymptotically cylindrical
$J_\epsilon$-holomorphic curves
$u_\epsilon : (\dot{\Sigma},j_\epsilon) \to (\widehat{W},J_\epsilon)$ and
$v_\epsilon : (\dot{\Sigma},j_\epsilon') \to (\widehat{W},J_\epsilon)$
close to $u$ and $v$ in their respective moduli spaces, such that
$$
u_\epsilon \cdot v_\epsilon = u * v.
$$
\end{enumerate}
\end{thm}

The last statement in the above theorem, involving the perturbations
$u_\epsilon$ and $v_\epsilon$, helps us interpret $u * v$ as the
count of intersections between \emph{generic} curves homotopic
to $u$ and~$v$.  That particular detail is not proved in
\cite{Siefring:intersection}, nor anywhere else in the literature---it has
the status of a ``folk theorem,'' meaning that at least a few
experts would be able to prove it as an exercise, but have not written down the
details in any public forum.  The proof involves Fredholm theory on
exponentially weighted Sobolev spaces, as explained e.g.~in
\cites{HWZ:props3,Wendl:automatic}, and we will not prove it here either,
but have included the statement mainly for the sake of intuition.
It is not needed for any of the most important applications
of Theorem~\ref{thm:star}, such as:

\begin{cor}
\label{cor:disjoint}
If $u$ and $v$ are $J$-holomorphic curves 
satisfying $u * v = 0$, then any two $J$-holomorphic curves that have
non-identical images and are homotopic to $u$ and $v$ respectively are
disjoint.
\end{cor}

In order to write down the punctured version of the adjunction formula,
we must introduce a little bit more notation.  Suppose
$\gamma : S^1 \to M$ is a Reeb orbit in a contact $3$-manifold 
$(M,\xi=\ker\alpha)$, and $k \in \NN$.  This gives rise to the
$k$-fold covered Reeb orbit
$$
\gamma^k : S^1 \to M : t \mapsto \gamma(kt),
$$
and we define the \defin{covering multiplicity}\index{Reeb orbit!covering multiplicity of}\index{covering multiplicity!of a Reeb orbit}
$\cov(\gamma)$ of a 
general Reeb orbit $\gamma$ as the largest $k \in \NN$ such that
$\gamma = \gamma_0^k$ for some other Reeb orbit~$\gamma_0$.
Similarly, if $f \in \Gamma(\gamma^*\xi)$ is an eigenfunction of
$\mathbf{A}_\gamma$ with eigenvalue $\lambda \in \RR$, then the
$k$-fold cover
$$
f^k \in \Gamma((\gamma^k)^*\xi), \qquad f^k(t) := f(kt)
$$
is an eigenfunction of $\mathbf{A}_{\gamma_k}$ with eigenvalue
$k\lambda$, and for any Reeb orbit $\gamma$ and nontrivial eigenfunction
$f$ of $\mathbf{A}_{\gamma}$, we define $\cov(f) \in \NN$ to be the largest
integer~$k$ such that $f$ is a $k$-fold cover of an eigenfunction for
a Reeb orbit covered by~$\gamma$.  Observe that, in general,
$1 \le \cov(f) \le \cov(\gamma)$, and $\cov(f)$ always divides $\cov(\gamma)$.
Note also that any trivialization $\tau$ of $\gamma^*\xi$ naturally
determines a trivialization of $(\gamma^k)^*\xi$, which we shall denote
by~$\tau^k$.

\begin{remark}
\label{remark:covInfty}
Exercise~\ref{EX:gcd} implies that if
$\gamma : S^1 \to M$ is a simply covered (i.e.~embedded) Reeb orbit
in a contact $3$-manifold $(M,\xi=\ker\alpha)$ and $\tau$ is a trivialization
of $\gamma^*\xi$, then for any
$k \in \NN$ and a nontrivial eigenfunction $f$ of $\mathbf{A}_{\gamma^k}$
with $\wind^{\tau^k}(f) > 0$,
$\cov(f)$ depends only on $k$ and $\wind^{\tau^k}(f)$, in fact:
$$
\cov(f) = \gcd\big( k , \wind^{\tau^k}(f)\big).
$$
\end{remark}

We now associate to any Reeb orbit $\gamma$ in a contact $3$-manifold
$(M,\xi=\ker\alpha)$ the \defin{spectral covering numbers}\index{spectral covering number}\index{Reeb orbit!spectral covering number of}\index{covering multiplicity!of an asymptotic eigenfunction}
$$
\bar{\sigma}_\pm(\gamma) := \cov(f_\pm) \in \NN,
$$
where $f_\pm \in \Gamma(\gamma^*\xi)$ is any choice of eigenfunction
of $\mathbf{A}_\gamma$ with $\wind^\tau(f_\pm) = \alpha^\tau_\pm(\gamma)$.
Remark~\ref{remark:covInfty} implies that $\bar{\sigma}_\pm(\gamma)$
does not depend on this choice.  Finally, if $u : \dot{\Sigma}
\to \widehat{W}$ is an asymptotically cylindrical map with
punctures $z \in \Gamma^\pm$ asymptotic to orbits
$\gamma_z$ in $M_\pm$, we define the \defin{total spectral covering
number}\index{spectral covering number}\index{holomorphic curve!spectral covering number of}
of $u$ by
$$
\bar{\sigma}(u) := \sum_{z \in \Gamma^+} \bar{\sigma}_-(\gamma_z) +
\sum_{z \in \Gamma^-} \bar{\sigma}_+(\gamma_z).
$$
Observe that $\bar{\sigma}(u)$ really does not depend on the map~$u$,
but only on its sets of positive and negative asymptotic orbits.
It is a positive integer in general, and we have
$$
\bar{\sigma}(u) - \#\Gamma \ge 0,
$$
with equality if and only if all of the so-called ``extremal'' eigenfunctions
at the asymptotic orbits of $u$ are simply covered.  This is true in
particular whenever all asymptotic orbits of $u$ are simply covered.

The next statement is the punctured generalization of the adjunction\index{adjunction formula!for punctured holomorphic curves|(}
formula (Theorem~\ref{thm:adjunction}): it relates $u * u$ to $\delta(u)$, the\index{self-intersection number!of asymptotically cylindrical maps}
spectral covering number $\bar{\sigma}(u)$, and 
our generalization of the normal Chern number $c_N(u)$
from \S\ref{sec:puncturedFoliations} (see Definition~\ref{defn:normalChern}).

\begin{thm}
\label{thm:adjunctionPunctured}
If $u : (\dot{\Sigma},j) \to (\widehat{W},J)$ is an asymptotically cylindrical
and simple $J$-holomorphic curve with punctures $\Gamma \subset \Sigma$, 
then there exists an integer
$$
\delta_\infty(u) \ge 0,
$$
interpreted as the count of ``hidden double points at infinity,''\index{hidden at infinity!double points}\index{double points!hidden at infinity}
such that
\begin{equation}
\label{eqn:adjunctionPunctured}
u * u = 2\left[ \delta(u) + \delta_\infty(u) \right] + c_N(u) + 
\left[ \bar{\sigma}(u) - \#\Gamma \right].
\end{equation}
In particular, $\delta(u) + \delta_\infty(u)$ depends only
on the homotopy class of $u$ as an asymptotically cylindrical map.
Moreover, there exists a perturbation 
$J_\epsilon \in \jJ(\omega,\alpha_+,\alpha_-)$ which is
$C^\infty$-close to $J$, and a $J_\epsilon$-holomorphic curve
$u_\epsilon : (\dot{\Sigma},j_\epsilon) \to (\widehat{W},J_\epsilon)$
close to $u$ in the moduli space, such that
$\delta_\infty(u_\epsilon) = 0.$
\end{thm}
\begin{cor}
\label{cor:embeddedPunctured}
If $u \in \mM_g(\widehat{W},J)$ is simple and satisfies $\delta(u) = \delta_\infty(u) = 0$, 
then every simple curve in the same connected component of $\mM_g(\widehat{W},J)$ 
is embedded.
\end{cor}

\begin{remark}
It is important to notice the lack of the words ``and only if'' in Corollary~\ref{cor:embeddedPunctured}:
an embedded curve $u$ always has $\delta(u) = 0$ but may in general have
$\delta_\infty(u) > 0$, in which case it could be homotopic to a simple curve
with critical or double points.
\end{remark}

The remainder of this lecture will be concerned with the definitions
of $u * v$, $\inter_\infty(u,v)$ and $\delta_\infty(u)$, and
the proofs of Theorems~\ref{thm:star} and~\ref{thm:adjunctionPunctured}.

\begin{remark}
\label{remark:notation}
The reader should be aware of a few notational differences between these
notes and the original source \cite{Siefring:intersection}.  One relatively
harmless difference is in the appearance of the adunction formula (Equation~\eqref{eqn:adjunctionPunctured}
above vs.~\cite{Siefring:intersection}*{Equation~(2-5)}), as Siefring does
not define or mention the normal Chern number, but writes an expression that
is equivalent due to \eqref{eqn:2cN}.  A more serious difference of conventions
appears in the formulas we will use to define $u * v$ and $\delta(u)$ below,
e.g.~\eqref{eqn:bigOmega} and \eqref{eqn:bigSelfOmega} contain ``$\pm$'' and
``$\mp$'' symbols that do not appear in the equivalent formulas
in \cite{Siefring:intersection}.  The reason is that alternate
versions of these numbers need to be defined for asymptotic orbits that appear
at positive or negative ends; Siefring handles this issue with a
notational shortcut, formally viewing Reeb orbits that occur at negative ends
as orbits with \emph{negative covering multiplicity}.  In these notes,
covering multiplicities are always positive.
\end{remark}

\section{Relative intersection numbers and the $*$-pairing}
\label{sec:star}

For the remainder of this lecture, fix a choice
of trivializations of the bundles $\gamma^*\xi_\pm \to S^1$ for every
\emph{simply covered} Reeb orbit $\gamma : S^1 \to M_\pm$.  Wherever
a trivialization along a multiply covered orbit $\gamma^k$ is needed,
we will use the one induced on $(\gamma^k)^*\xi_\pm \to S^1$ by our
chosen trivialization of $\gamma^*\xi_\pm$, and denote this choice
as usual by~$\tau$.

For two asymptotically cylindrical maps $u : \dot{\Sigma} \to \widehat{W}$
and $v : \dot{\Sigma}' \to \widehat{W}$, we define the
\defin{relative intersection number}\index{relative intersection number}
$$
u \dotrel_\tau v := u \cdot v^\tau \in \ZZ,
$$
where $v^\tau : \dot{\Sigma}' \to \widehat{W}$ denotes any $C^0$-small
perturbation of $v$ such that $u$ and $v^\tau$ have at most finitely many
intersections and $v^\tau$ is ``pushed off'' near $\pm\infty$ in directions
determined by~$\tau$, i.e.~if $v$ approaches the orbit $\gamma : S^1 \to M_\pm$ 
asymptotically at a puncture $z$, then $v^\tau$ at the same puncture approaches 
a loop of the form $\exp_{\gamma(t)} \epsilon \eta(t)$, where
$\epsilon > 0$ is small and $\eta \in \Gamma(\gamma^*\xi_\pm)$ satisfies 
$\wind^\tau(\eta) = 0$.  Since $v^\tau$ asymptotically approaches loops that
may (without loss of generality) be assumed disjoint from the asymptotic orbits
of~$u$, it follows from Exercise~\ref{EX:interWithBoondary}
that this definition is independent of the choice of perturbation, and it
only depends on the homotopy classes of $u$ and $v$ (as asymptotically
cylindrical maps) plus the trivializations~$\tau$.  The dependence on $\tau$
indicates that $u \dotrel_\tau v$ is not a very meaningful number on its own,
so it will not be an object of primary study for us,
but like the relative first Chern numbers in 
\S\ref{sec:puncturedFoliations}, it will provide a useful tool for 
organizing information.

\begin{exercise}
Show that $u \dotrel_\tau v = v \dotrel_\tau u$.
\end{exercise}

Suppose $u : \dot{\Sigma} \to \widehat{W}$ and
$v : \dot{\Sigma}' \to \widehat{W}$ are asymptotically cylindrical and have
finitely many intersections, so $u \cdot v$ is well defined.  Then 
$u \dotrel_\tau v$ can be computed with the perturbation $v^\tau$ assumed to
be nontrivial only in some neighborhood of infinity where $u$ and $v$ are
disjoint, so that $u \cdot v^\tau$ counts the intersections of $u$ with $v$,
plus some additional intersections that appear in a neighborhood of
infinity when $v$ is perturbed to~$v^\tau$.  We shall denote this count
of additional intersections near infinity by 
$\inter_\infty^\tau(u,v) \in \ZZ$, so we can write
$$
u \dotrel_\tau v = u\cdot v + \inter_\infty^\tau(u,v)
$$
whenever $u \cdot v$ is well defined.  

The number $\inter_\infty^\tau(u,v)$ also depends on $\tau$ and is thus
not meaningful on its own, but it is useful to observe that it can be
computed in terms 
of relative asymptotic winding numbers---this observation will lead us
to the natural definitions of the much more meaningful quantities 
$u * v$ and~$\inter_\infty(u,v)$, which do not depend on~$\tau$.
To see this, denote the punctures of
$u$ and $v$ by $\Gamma_u = \Gamma^+_u \cup \Gamma^-_u$ and
$\Gamma_v = \Gamma^+_v \cup \Gamma^-_v$ respectively, and for any
$z \in \Gamma_u$ or $\Gamma_v$, denote the corresponding asymptotic
orbit of $u$ or $v$ by~$\gamma_z^{k_z}$, where we assume $\gamma_z$
is a \emph{simply covered} orbit and $k_z \in \NN$ is the covering
multiplicity.  A contribution to $\inter_\infty^\tau(u,v)$ may
come from any pair of punctures $(z,\zeta) \in \Gamma^\pm_u \times \Gamma^\pm_v$,
so we shall denote this contribution
by $\inter_\infty^\tau(u,z \,;\, v,\zeta)$ and write
\begin{equation}
\label{eqn:breakdown}
\inter^\tau_\infty(u,v) = \sum_{(z,\zeta) \in \Gamma^\pm_u \times \Gamma^\pm_v}
\inter^\tau_\infty(u,z \,;\, v,\zeta).
\end{equation} 
\begin{remark}
\label{remark:signs}
In \eqref{eqn:breakdown} and several other expressions in this lecture,
the summation should be understood as a sum of
two summations, one with $\pm = +$ and the other with $\pm = -$.
\end{remark}
If $\gamma_z \ne \gamma_\zeta$, then $u$ and $v^\tau$ have no intersections
in neighborhoods of these particular punctures, implying
$$
\inter^\tau_\infty(u,z \,;\, v,\zeta) = 0 
\quad\text{ if $\gamma_z \ne \gamma_\zeta$}.
$$
Now assume $\gamma := \gamma_z = \gamma_\zeta$, and
let $T > 0$ denote the period of $\gamma$.  We shall parametrize punctured 
neighborhoods of $z$ and $\zeta$ by half-cylinders $Z_\pm$ and
consider the resulting maps
$$
u(s,t), v(s,t) \in \RR \times M_\pm,
$$
defined for $|s|$ sufficiently large and
asymptotic to $\gamma^{k_z}$ and $\gamma^{k_\zeta}$ respectively.
We first consider the special case where both asymptotic orbits
have the same covering multiplicity, so let
$$
k := k_z = k_\zeta.
$$
Asymptotic approach to $\gamma^k$ means we can write
$$
u(s,t) = \exp_{(kTs, \gamma(kt))} h_u(s,t), \qquad
v(s,t) = \exp_{(kTs, \gamma(kt))} h_v(s,t),
$$
for sections $h_u$ and $h_v$ of $\xi_\pm$ along the
orbit cylinder such that both decay uniformly to~$0$
as $s \to \pm\infty$.  The assumption that $u$ and
$v$ have no intersections near infinity implies
moreover that for some $s_0 > 0$, each of the sections
$$
(s,t) \mapsto h_u(s,t + j/k) - h_v(s,t), \qquad j=0,\ldots,k-1
$$
has no zeroes in the region $|s| \ge s_0$.
The perturbation $v^\tau$ can now be defined as
$$
v^\tau(s,t) = \exp_{(kTs, \gamma(kt))} \left[ h_v(s,t) +
\epsilon \eta(s,t) \right],
$$
where $\epsilon > 0$ is small and $\eta(s,t) \in \xi_\pm$ can be 
assumed to vanish for $|s| \le s_0$ and to satisfy
$\eta(s,t) \to \eta_\infty(kt)$ as $s \to \pm\infty$, with
$\eta_\infty \in \Gamma(\gamma^*\xi_\pm)$ a nowhere vanishing section
satisfying $\wind^\tau(\eta_\infty) = 0$.
Intersections of $v^\tau$ with $u$ in the region $|s| \ge s_0$
are now in one-to-one correspondence with solutions of
the equation
$$
F_j(s,t) := h_u(s,t + j/k) - h_v(s,t) - \epsilon \eta(s,t) = 0,
$$
for arbitrary values of $j \in \{0,\ldots,k-1\}$.  Notice that
$F_j$ admits a continuous extension to $s=\pm\infty$ with
$F_j(\pm\infty,t) = -\epsilon\eta_\infty(kt)$.
Since $\wind^\tau(\eta_\infty) = 0$ and $\epsilon > 0$ is small,
the algebraic count of zeroes of $F_j$ on the region
$\{ |s| \ge s_0 \}$ is thus
$$
\pm \left[ \wind^\tau \left(F_j(\pm\infty,\cdot)\right) - 
\wind^\tau \left( F_j(\pm s_0,\cdot) \right) \right] =
\mp\wind^\tau\left( h_u(\pm s_0,\cdot + j/k) - h_v(\pm s_0, \cdot) \right),
$$
i.e.~it is the \emph{relative asymptotic winding number} of $v$ about 
the reparametrization $u(s,t+j/k)$, with respect to the trivialization~$\tau$.
Summing this over all such reparametrizations gives
\begin{equation}
\label{eqn:wind2}
\sum_{j=0}^{k-1} \mp \wind^\tau\big( h_u(s,\cdot + j/k) - h_v(s,\cdot) \big),
\end{equation}
where the parameter $s$ can be chosen to be any number sufficiently close
to~$\pm\infty$.
If $k_z \ne k_\zeta$, then the above computation is valid for the covers
$u^{k_\zeta}(s,t) := u(k_\zeta s, k_\zeta t)$ and 
$v^{k_z}(s,t) := v(k_z s, k_z t)$, both
asymptotic to $\gamma^{k_z k_\zeta}$, and \eqref{eqn:wind2}
must then be divided by $k_z k_\zeta$ to compute 
$\inter^\tau_\infty(u,z \,;\, v,\zeta)$.

\begin{remark}
\label{remark:braid}
The computation above can be interpreted in terms of braids: namely,
if $u$ and $v$ have at most finitely many intersections,
then their asymptotic behavior at a pair of punctures with matching asymptotic
orbit (up to multiplicity) determines
up to isotopy
a pair of (perhaps multiply covered) disjoint connected braids, whose
linking number with each other is (up to a sign)
$\inter^\tau_\infty(u,z\,;\,v,\zeta)$.  It appears in this form
in the work of Hutchings on embedded contact homology; see
Appendix~\ref{sec:ECH} for further discussion.
\end{remark}

The discussion thus far has been valid for any pair of asymptotically 
cylindrical maps.  If we now assume $u$ and $v$ are also $J$-holomorphic,
then Theorem~\ref{thm:asymptoticsRelative} expresses the summands
in \eqref{eqn:wind2} as winding numbers of certain relative asymptotic
eigenfunctions for $\gamma^{k_z k_\zeta}$, and these winding numbers
satisfy \emph{a priori} bounds due to Theorem~\ref{thm:winding}.
Specifically, assume $u(s,t)$ and $v(s,t)$ approach their respective
covers of $\gamma$ along asymptotic eigenfunctions $f_u$ and $f_v$
with decay rates $|\lambda_u|$ and $|\lambda_v|$ respectively,
so by \eqref{eqn:alphas} we have
$$
\mp \wind^\tau(f_u) \ge \mp \alpha^\tau_\mp(\gamma^{k_z}), \qquad
\mp \wind^\tau(f_v) \ge \mp \alpha^\tau_\mp(\gamma^{k_\zeta}).
$$
Then the covers $u^{k_\zeta}(s,t)$ and $v^{k_z}(s,t)$ approach
$\gamma^{k_z k_\zeta}$ along asymptotic eigenfunctions
$f_u^{k_\zeta}$ and $f_v^{k_z}$ with decay rates
$k_\zeta |\lambda_u|$ and $k_z |\lambda_v|$
respectively, and the winding is bounded by
$$
\mp \wind^\tau\big(f_u^{k_\zeta}\big) \ge \mp k_\zeta \alpha^\tau_\mp(\gamma^{k_z}), \qquad
\mp \wind^\tau\big(f_v^{k_z}\big) \ge \mp k_z \alpha^\tau_\mp(\gamma^{k_\zeta}).
$$
Theorem~\ref{thm:asymptoticsRelative} now implies that the relative
decay rate controlling the approach of $v(s,t)$ to any of the
reparametrizations $u(s,t + j/k)$ is at least the minimum of
$k_\zeta |\lambda_u|$ and $k_z |\lambda_v|$, thus the corresponding
winding number is similarly bounded due to Theorem~\ref{thm:winding}.
We conclude that
each of the summands in \eqref{eqn:wind2} is bounded from below by
the integer $\Omega^\tau_\pm(\gamma^{k_z},\gamma^{k_\zeta})$, where for
any $k , m \in \NN$ we define
\begin{equation}
\label{eqn:bigOmega}
\Omega^\tau_\pm(\gamma^k, \gamma^m) :=
\min\left\{ \mp k \alpha^\tau_\mp(\gamma^m), 
\mp m \alpha^\tau_\mp(\gamma^k) \right\}.
\end{equation}
Adding the summands in \eqref{eqn:wind2} for $j=0,\ldots,k_z k_\zeta - 1$
and then dividing by the combinatorial factor $k_z k_\zeta$ produces the
bound
$$
\inter^\tau_\infty(u,z \,;\, v,\zeta) \ge 
\Omega^\tau_\pm\big(\gamma_z^{k_z},\gamma_\zeta^{k_\zeta}\big) \quad
\text{ if $\gamma_z = \gamma_\zeta$}.
$$
If we extend the
definition of $\Omega^\tau_\pm$ by setting
$$
\Omega^\tau_\pm(\gamma_1^k , \gamma_2^m) := 0 \quad
\text{ whenever $\gamma_1 \ne \gamma_2$},
$$
then a universal lower bound for $\inter_\infty^\tau(u,v)$ can now be
written in terms of asymptotic winding numbers as
\begin{equation}
\label{eqn:iLowerBound}
\inter_\infty^\tau(u,v) \ge \sum_{(z,\zeta) \in \Gamma_u^\pm \times \Gamma_v^\pm}
\Omega^\tau_\pm\big(\gamma_z^{k_z}, \gamma_\zeta^{k_\zeta}\big).
\end{equation}

\begin{defn}
\label{defn:iInfty}
For any asymptotically cylindrical maps $u : \dot{\Sigma} \to \widehat{W}$\index{hidden at infinity!intersections|(}\index{intersections!hidden at infinity|(}
and $v : \dot{\Sigma}' \to \widehat{W}$ with finitely many intersections,\index{asymptotic contribution!to the $*$-pairing|(}\index{intersection number!asymptotic contribution to|(}
define
$$
\inter_\infty(u,v) := \inter^\tau_\infty(u,v) - 
\sum_{(z,\zeta) \in \Gamma_u^\pm \times \Gamma_v^\pm}
\Omega^\tau_\pm\big(\gamma_z^{k_z}, \gamma_\zeta^{k_\zeta}\big).
$$
Similarly, for \emph{any} asymptotically cylindrical maps $u$ and $v$
(not necessarily with finitely many intersections), we can define\index{intersection number!of asymptotically cylindrical maps ($*$-pairing)}
$$
u * v := u \dotrel_\tau v - 
\sum_{(z,\zeta) \in \Gamma_u^\pm \times \Gamma_v^\pm}
\Omega^\tau_\pm\big(\gamma_z^{k_z}, \gamma_\zeta^{k_\zeta}\big).
$$
When it is well defined, $\inter_\infty(u,v)$ is sometimes called the
\defin{asymptotic contribution} to $u * v$.
\end{defn}
\begin{exercise}
Check that neither of the above definitions depends on the choice
of trivializations~$\tau$.
\end{exercise}

Definitions involving $\Omega^\tau_\pm(\gamma^k,\gamma^m)$ may seem
not very enlightening at first, and they are seldom used in practice
for computations, but it's useful to keep in mind what these terms
mean: they are \emph{theoretical bounds} on the possible relative 
asymptotic winding of ends of $u$ around (all possible reparametrizations of)
ends of~$v$.  We will say that a given winding number is \defin{extremal}\index{extremal winding}\index{asymptotic eigenfunction!extremal winding of}
whenever it achieves the corresponding theoretical bound.  
We conclude, for example:

\begin{thm}[\textsl{asymptotic positivity of intersections}]
\label{thm:asymptoticPositivity}
If $u$ and $v$ are asymptotically cylindrical $J$-holomorphic
curves with non-identical images, then $\inter_\infty(u,v) \ge 0$, with
equality if and only if for all pairs of ends of $u$ and $v$ respectively\index{positivity of intersections!asymptotic}\index{asymptotic positivity of intersections}
asymptotic to covers of\index{hidden at infinity!intersections}\index{intersections!hidden at infinity|)}\index{asymptotic contribution!to the $*$-pairing|)}\index{intersection number!asymptotic contribution to|)}
the same Reeb orbit, all of the resulting relative asymptotic eigenfunctions have
extremal winding.
\qed
\end{thm}

It is also immediate from the above definition that
$u * v$ is homotopy invariant and equals $u \cdot v + \inter_\infty(u,v)$
whenever $u \cdot v$ is well defined.
This completes the proof of 
Theorem~\ref{thm:star}, except for the claim that one can always achieve
$u \cdot v = u * v$ after a perturbation of the data.  This can be proved
by observing that the subset of $\mM_{g}(\widehat{W},J) \times \mM_{g'}(\widehat{W},J)$ consisting
of pairs $(u,v)$ with $\inter_\infty(u,v) > 0$ consists precisely of those
pairs that share an asymptotic orbit at which some relative asymptotic
winding number is not extremal.  By Theorem~\ref{thm:winding}, this means
that the relative decay rate of some pair of ends approaching the same
orbit is an eigenvalue other than the one closest to~$0$.  One can
then use Fredholm theory with exponential weights 
(cf.~\cites{HWZ:props3,Wendl:automatic,Hryniewicz:fast})
to show that the moduli 
space of pairs of curves satisfying this relative decay condition has
strictly smaller Fredholm index than the usual moduli space, thus for
generic data, it is a submanifold of positive codimension.

\section{Adjunction formulas, relative and absolute}
\label{sec:relAdjunction}

In order to generalize the adjunction formula, we begin by computing
$u \dotrel_\tau u$ for an immersed simple $J$-holomorphic curve
$u : (\dot{\Sigma},j) \to (\widehat{W},J)$ with asymptotic orbits
$$
\{ \gamma_z^{k_z} \}_{z \in \Gamma^\pm},
$$
where the notation is chosen as in the previous section so that
$\gamma_z$ is always a
simply covered orbit and $k_z \in \NN$ is the corresponding covering
multiplicity.  Choose a section
$\eta$ of the normal bundle $N_u \to \dot{\Sigma}$ with finitely many
zeroes and such that, on each cylindrical neighborhood $Z_\pm \subset
\dot{\Sigma}$ of a puncture $z \in \Gamma^\pm$,
$$
\eta(s,t) \to \eta_\infty(k_z t) \quad \text{ uniformly in~$t$ as $s \to \pm\infty$,}
$$
for some nonzero $\eta_\infty \in \Gamma(\gamma_z^*\xi_\pm)$ satisfying
$\wind^\tau(\eta_\infty) = 0$.  We can also assume the zeroes of
$\eta$ are disjoint from all points $z \in \dot{\Sigma}$ at which
$u(z) = u(\zeta)$ for some $\zeta \ne z$.  Then $u \dotrel_\tau u =
u \cdot u^\tau$, where
$$
u^\tau(z) = \exp_{u(z)} \epsilon \eta(z)
$$
for some $\epsilon > 0$ small.  As we saw in \S\ref{sec:adjunction}, there
are two obvious sources of intersections between $u$ and $u^\tau$:
\begin{enumerate}
\item Each transverse double point $u(z) = u(\zeta)$ with $z \ne \zeta$
contributes two transverse positive intersections, one near~$z$ and one near~$\zeta$.
More generally, the algebraic count of intersections contributed by each
isolated double point is twice its local intersection index.
\item Each zero $\eta(z)=0$ contributes an intersection at~$z$, with local
intersection index equal to the order of the zero.  The
algebraic count of these zeroes is the relative first Chern number
$c_1^\tau(N_u) \in \ZZ$.
\end{enumerate}
Unlike in the closed case, 
there are now two additional sources of intersections.
As we saw in the previous section, if $z, \zeta \in \Gamma^\pm$ are
two distinct punctures with $\gamma_z = \gamma_\zeta$, then perturbing
$u$ to $u^\tau$ will cause $\inter^\tau_\infty(u,z \,;\, u,\zeta) \in \ZZ$ additional
intersections of $u$ and $u^\tau$ to appear near infinity along the
corresponding half-cylinders, and this number is also bounded below by
$\Omega^\tau_\pm\big(\gamma_z^{k_z},\gamma_\zeta^{k_\zeta}\big)$, defined in
\eqref{eqn:bigOmega} in terms of winding numbers of asymptotic
eigenfunctions.  Additionally, near any $z \in \Gamma^\pm$ with
$k_z > 1$, $u$ may intersect different parametrizations of $u^\tau$
near infinity.  To see this, we can again parametrize a neighborhood of
$z$ in $\dot{\Sigma}$ with the half-cylinder $Z_\pm$ and write
$$
u(s,t) = \exp_{(kTs , \gamma(kt))} h(s,t)
$$
for large $|s|$,
where $k := k_z$, $\gamma := \gamma_z$, $T > 0$ is the period of~$\gamma$
and $h(s,t) \in \xi_\pm$ decays to~$0$ as $s \to \pm\infty$.  Since
$u$ is simple, it has no double points in some neighborhood of
infinity, which means that for some $s_0 > 0$, we have
$$
h(s,t) \ne h(s, t + j/k) \quad\text{ for all $|s| \ge s_0$, $t \in S^1$,
$j \in \{1,\ldots,k-1\}$}.
$$
The perturbation $u^\tau$ on this neighborhood may be assumed to take the form
$$
u^\tau(s,t) = \exp_{(kTs , \gamma(kt))} \left[ h(s,t) + \epsilon 
\eta_\infty(kt)\right],
$$
for some $\epsilon > 0$ small, where again $\wind^\tau(\eta_\infty) = 0$.
Thus intersections of $u$ with $u^\tau$ on the region $\{ |s| \ge s_0 \}$
correspond to solutions of
$$
F_j(s,t) := h(s,t + j/k) - h(s,t) - \epsilon\eta_\infty(kt) = 0
$$
for arbitrary values of $j=0,\ldots,k-1$.  For $j=0$, this equation has
no solutions.  For $j=1,\ldots,k-1$, we observe that $F_j$ extends 
continuously to $s=\pm\infty$ with $F_j(\pm\infty,t) = -\epsilon\eta_\infty(kt)$
and obtain the count of solutions
$$
\pm\left[ \wind^\tau\left( F_j(\pm\infty,\cdot)\right) - 
\wind^\tau\left( F_j(\pm s_0,\cdot) \right) \right]
= \mp \wind^\tau \left( h(\pm s_0,\cdot + j/k) - h(\pm s_0,\cdot) \right).
$$
The count of additional intersections of $u$ with $u^\tau$ in a
neighborhood of $z$ is therefore
\begin{equation}
\label{eqn:iuz}
\inter^\tau_\infty(u , z) := \mp \sum_{j=1}^{k_z - 1} 
 \wind^\tau \left( h(s,\cdot + j/k) - h(s,\cdot) \right),
\end{equation}
where $s$ is any number sufficiently close to $\pm\infty$,
and we can then write the total count of asymptotic contributions to
$u \dotrel_\tau u$ as
$$
\inter^\tau_\infty(u) := \sum_{z,\zeta \in \Gamma^\pm,\,
z \ne \zeta} \inter^\tau_\infty(u,z \,;\, u,\zeta) +
\sum_{z \in \Gamma^\pm} \inter^\tau_\infty(u,z).
$$
This yields the computation
$$
u \dotrel_\tau u = 2\delta(u) + c^\tau_1(N_u) + \inter^\tau_\infty(u),
$$
and since $c_1^\tau(N_u) = c_1^\tau(u^*T\widehat{W}) - \chi(\dot{\Sigma})$, we
deduce from this a relation that is valid for any (not necessarily
immersed) simple and asymptotically cylindrical $J$-holomorphic curve,
called the \defin{relative adjunction formula}\index{relative adjunction formula}\index{adjunction formula!relative}\index{self-intersection number!relative}
\begin{equation}
\label{eqn:relAdjunction}
u \dotrel_\tau u = 2\delta(u) + c^\tau_1(u^*T\widehat{W}) -
\chi(\dot{\Sigma}) + \inter^\tau_\infty(u).
\end{equation}
This version of the adjunction formula first appeared in \cite{Hutchings:index}.

\begin{remark}
\label{remark:braids}
As with $\inter^\tau_\infty(u,z \,;\, v,\zeta)$ (cf.~Remark~\ref{remark:braid}), 
$\inter^\tau_\infty(u,z)$ can be given
a braid-theoretic interpretation: it is (up to a sign) the writhe of
the braid defined by identifying a
neighborhood of the framed loop $\gamma_z$ with $S^1 \times \DD$
and projecting the embedded loop $u(s,\cdot)$ to $M_\pm$ for any
$s$ close to~$\pm\infty$; see Appendix~\ref{sec:ECH}.
\end{remark}

As we did with $\inter^\tau_\infty(u,v)$ in the previous section,  it will be
useful to derive a theoretical bound on $\inter^\tau_\infty(u)$.
We already have $\inter^\tau_\infty(u,z \,;\, u,\zeta) \ge 
\Omega^\tau_\pm\big(\gamma_z^{k_z} , \gamma_\zeta^{k_\zeta}\big)$, and
must deduce a similar bound for $\inter^\tau_\infty(u,z)$.
Let $\gamma := \gamma_z$ and $k := k_z$, and write $u(s,t) =
\exp_{(Ts,\gamma(kt))} h(s,t)$ as usual, and for $j=1,\ldots,k-1$, let
$$
u_j(s,t) := u(s,t + j/k) = \exp_{(Ts,\gamma(kt))} h_j(s,t), \quad
\text{ where } h_j(s,t) := h(s, t + j/k).
$$
By theorem~\ref{thm:asymptotics},
$h(s,t)$ is controlled as $s \to \pm\infty$ by some eigenfunction
$f$ of $\mathbf{A}_{\gamma^k}$ with eigenvalue $\lambda$, and 
by \eqref{eqn:alphas},
$$
\mp \wind^\tau(f) \ge \mp \alpha^\tau_\mp(\gamma^k).
$$
The reparametrizations $h_j(s,t)$ are similarly controlled by
reparametrized eigenfunctions
$$
f_j(t) := f(t + j/k)
$$
with $\wind^\tau(f_j) = \wind^\tau(f)$ and the same eigenvalue, and
the relative 
decay rates controlling $h_j - h$ are then at least $|\lambda|$ due to
Theorem~\ref{thm:asymptoticsRelative}, implying (via
Theorem~\ref{thm:winding}) a corresponding bound on the relative
winding terms in \eqref{eqn:iuz}, thus
$$
\inter^\tau_\infty(u,z) \ge \mp (k-1) \wind^\tau(f) \ge \mp
(k-1) \alpha^\tau_\mp(\gamma^k).
$$

The bound established above is only a first attempt, as we will see in 
a moment that a stricter bound may hold in general.  
If $\wind^\tau(f)$ is not extremal,
i.e.~$\mp \wind^\tau(f) \ge \mp\alpha^\tau_\mp(\gamma^k) + 1$, the above
computation gives
\begin{equation}
\label{eqn:badBound}
\inter^\tau_\infty(u,z) \ge \mp (k-1) \alpha^\tau_\mp(\gamma^k) + k-1.
\end{equation}

Alternatively, suppose $\wind^\tau(f)$ is extremal, hence equal to 
$\alpha^\tau_\mp(\gamma^k)$,
and let $m = \cov(f)$, so $k = m \ell$ for some $\ell \in \NN$ and
$$
f(t) = g^m(t) := g(mt)
$$
for some eigenfunction $g$ of $\mathbf{A}_{\gamma^\ell}$ which is simply
covered.  It follows that for $j=1,\ldots,k-1$, 
$f_j \equiv f$ if and only if $j \in \ell \ZZ$.  When $j$ is not divisible
by~$\ell$, Theorem~\ref{thm:asymptoticsRelative} now gives a relative
decay rate equal to $|\lambda|$ and thus relative winding equal
to $\wind^\tau(f)$, so adding up these terms for the $m(\ell-1)$ values
of~$j$ not in $\ell\ZZ$ contributes
\begin{equation}
\label{eqn:notDivisible}
\mp m(\ell - 1) \alpha^\tau_\mp(\gamma^k)
\end{equation}
to $\inter^\tau_\infty(u,z)$.

For $j = 1,\ldots,m-1$, we claim that the
asymptotic winding of $h_{j\ell} - h$ is stricter than the bound
established above, i.e.~for large $|s|$,
\begin{equation}
\label{eqn:divisible}
\mp \wind^\tau\left( h_{j\ell}(s,\cdot) - h(s,\cdot)\right) \ge
\mp \alpha^\tau_\mp(\gamma^k) + 1.
\end{equation}
By Theorem~\ref{thm:asymptoticsRelative}, there is a nontrivial
eigenfunction $\varphi_j \in \Gamma((\gamma^k)^*\xi_\pm)$ of
$\mathbf{A}_{\gamma^k}$ with eigenvalue $\lambda'$ such that
$$
h_{j\ell}(s,t) - h(s,t) = e^{\lambda' s}\left[ \varphi_j(t) + r'(s,t) \right],
$$
for large $|s|$, with $r'(s,\cdot) \to 0$ uniformly as $s \to \pm\infty$.
Now if the claim is false, then $\wind^\tau(\varphi_j) = \alpha^\tau_\mp(\gamma^k)
= \wind^\tau(f)$.  Since $f$ is an $m$-fold cover, this means
$\wind^\tau(\varphi_j)$ is divisible by~$m$, and Remark~\ref{remark:covInfty}
then implies that $\varphi_j$ is also an $m$-fold cover, thus
\begin{equation}
\label{eqn:msymmetric}
\varphi_j(t + 1/m) = \varphi_j(t) \quad\text{ for all $t \in S^1$.}
\end{equation}
But observe:
\begin{equation*}
\begin{split}
0 &= \sum_{r=0}^{m-1} \left[ h\left(s,t + \frac{j+r}{m}\right) - h\left(s, t + \frac{r}{m}\right)\right]
= \sum_{r=0}^{m-1} \left[ h_{j\ell}\left(s,t + \frac{r}{m} \right) - h\left(s,t + \frac{r}{m} \right) \right] \\
&= \sum_{r=0}^{m-1} e^{\lambda' s}\left[ \varphi_j\left(t + \frac{r}{m}\right) + r'\left(s,t + \frac{r}{m}\right) \right] ,
\end{split}
\end{equation*}
implying
$$
\sum_{r=0}^{m-1} \varphi_j(t + r / m) = 0 \quad\text{ for all $t \in S^1$.}
$$
Since $\varphi_j$ is not identically zero, this contradicts \eqref{eqn:msymmetric} and thus
proves the claim.  

Adding to \eqref{eqn:notDivisible} the $m-1$ terms bounded by \eqref{eqn:divisible}, we
conclude
$$
\inter^\tau_\infty(u,z) \ge \mp m (\ell-1) \alpha^\tau_\mp(\gamma^k) + (m-1)
\left[ \mp \alpha^\tau_\mp(\gamma^k) + 1 \right] \\
= \mp (k-1) \alpha^\tau_\mp(\gamma^k)  + (m-1).
$$
This bound is weaker than \eqref{eqn:badBound}, but the latter is valid only
when $\wind^\tau(f)$ is non-extremal, thus the former is the strongest possible
bound in general.
Recall that the covering multiplicity $m = \cov(f)$ is precisely what we denoted by
$\bar{\sigma}_\mp(\gamma^k)$ in \S\ref{sec:statement}.
To summarize, we now define for any simply covered orbit $\gamma$ and $k \in \NN$,
\begin{equation}
\label{eqn:bigSelfOmega}
\Omega^\tau_\pm(\gamma^k) := \mp (k-1) \alpha^\tau_\mp(\gamma^k) + 
\left[ \bar{\sigma}_\mp(\gamma^k) - 1 \right].
\end{equation}
The above computation then implies
\begin{equation}
\label{eqn:iBound}
\inter^\tau_\infty(u,z) \ge \Omega^\tau_\pm\big(\gamma_z^{k_z}\big).
\end{equation}

\begin{defn}
\label{defn:deltaInfty}
For any asymptotically cylindrical map $u : \dot{\Sigma} \to \widehat{W}$
that is embedded outside some compact set, we define the
\defin{asymptotic contribution} to the singularity index by\index{asymptotic contribution!to the singularity index}\index{singularity index of a simple holomorphic curve!asymptotic contribution to}
$$
\delta_\infty(u) := \frac{1}{2} \left[ \inter^\tau_\infty(u) -
\sum_{z,\zeta \in \Gamma^\pm,\, z \ne \zeta} 
\Omega^\tau_\pm\big(\gamma_z^{k_z} , \gamma_\zeta^{k_\zeta} \big) -
\sum_{z \in \Gamma^\pm} \Omega^\tau_\pm\big(\gamma_z^{k_z}\big) \right].
$$
\end{defn}	
\begin{exercise}
Check that the above definition does not depend on the trivializations~$\tau$.
Then try to convince yourself that it's an integer, not a half-integer.
\end{exercise}

Like Theorem~\ref{thm:asymptoticPositivity} in the previous section, the following
is now immediate from the computation above:

\begin{thm}
\label{thm:sing}
If $u$ is an asymptotically cylindrical and simple $J$-holomorphic\index{asymptotic contribution!to the singularity index}\index{singularity index of a simple holomorphic curve!asymptotic contribution to}
curve, then $\delta_\infty(u) \ge 0$, with
equality if and only if:
\begin{enumerate}
\item For all pairs of ends asymptotic to covers of the same Reeb 
orbit, the resulting relative asymptotic eigenfunctions have
extremal winding;
\item For all ends asymptotic to multiply covered Reeb orbits, the
relative asymptotic eigenfunctions controlling the approach of distinct
branches to each other have extremal winding.
\end{enumerate}
\end{thm}

The proof of the absolute adjunction formula in Theorem~\ref{thm:adjunctionPunctured}
now consists only of plugging in the relevant definitions and computing.

\begin{exercise}
\label{EX:acalculation}
Show that for any simply covered Reeb orbit $\gamma$ and $k \in \NN$,
$$
\Omega^\tau_\pm(\gamma^k) - \Omega^\tau_\pm(\gamma^k,\gamma^k) \mp \alpha^\tau_\mp(\gamma^k)
= \bar{\sigma}_\mp(\gamma^k) - 1.
$$
\end{exercise}

\begin{proof}[Proof of Theorem~\ref{thm:adjunctionPunctured}]
Plugging the relative adjunction formula \eqref{eqn:relAdjunction} into the
definition of $u * u$ (Definition~\ref{defn:iInfty}) gives\index{self-intersection number!of asymptotically cylindrical maps}
\begin{equation*}
\begin{split}
u * u &= u \dotrel_\tau u - \sum_{(z,\zeta) \in \Gamma^\pm \times \Gamma^\pm}
\Omega^\tau_\pm\big(\gamma_z^{k_z} , \gamma_z^{k_z} \big) \\
&= 2\delta(u) + c_1^\tau(u^*T\widehat{W}) - \chi(\dot{\Sigma}) + \inter^\tau_\infty(u)
- \sum_{(z,\zeta) \in \Gamma^\pm \times \Gamma^\pm}
\Omega^\tau_\pm\big(\gamma_z^{k_z} , \gamma_z^{k_z} \big).
\end{split}
\end{equation*}
Now replacing $c_1^\tau(u^*T\widehat{W}) - \chi(\dot{\Sigma})$ with
$c_N(u)$ plus some extra terms from Definition~\ref{defn:normalChern}, and
$\inter^\tau_\infty(u)$ with $2\delta_\infty(u)$ plus extra terms from
Definition~\ref{defn:deltaInfty}, all terms of the form
$\Omega^\tau_\pm\big(\gamma_z^{k_z},\gamma_\zeta^{k_\zeta}\big)$ with
$z \ne \zeta$ cancel and the above becomes
$$
u * u = 2\left[ \delta(u) + \delta_\infty(u)\right] + c_N(u)
 + \sum_{z \in \Gamma^\pm} \left[ \Omega^\tau_\pm\big(\gamma_z^{k_z}\big)
 - \Omega^\tau_\pm\big(\gamma_z^{k_z} , \gamma_z^{k_z}\big) \mp 
 \alpha^\tau_\mp\big(\gamma_z^{k_z}\big) \right].
$$
The result then follows from Exercise~\ref{EX:acalculation}.\index{adjunction formula!for punctured holomorphic curves|)}
\end{proof}

\begin{exercise}
\label{EX:funReeb}
Assume $\gamma : S^1 \to M$ is a nondegenerate Reeb orbit in a contact 
$3$-manifold $(M,\xi=\ker\alpha)$, and given $J \in \jJ(\alpha)$, 
let $u_\gamma : \RR \times S^1 \to \RR \times M$
denote the associated $J$-holomorphic orbit cylinder.\index{orbit cylinder}  
\begin{enumerate}
\renewcommand{\labelenumi}{(\alph{enumi})}
\item 
Show that $c_N(u_\gamma) = -p(\gamma)$,
where $p(\gamma) \in \{0,1\}$ is the parity of the Conley-Zehnder 
index of~$\gamma$.  
\item
Show that $u_\gamma * u_\gamma = - \cov(\gamma) \cdot p(\gamma)$.
\item
Deduce from part~(b) that if $u^k$ denotes a $k$-fold cover of a given
asymptotically cylindrical $J$-holomorphic curve $u$, it is \emph{not}
generally true that $u^k * v^\ell = k\ell (u * v)$.\\
\textsl{Remark: One can show however that in general,
$u^k * v^\ell \ge k\ell (u * v)$, cf.~Proposition~\ref{prop:coverRelation}.}
\item
Use the adjunction formula to show the following: if $\gamma$ is a multiple cover
of a Reeb orbit with even Conley-Zehnder index, and $J'$ is an arbitrary
almost complex structure on $\RR \times M$ that is compatible with $d(e^s\alpha)$
and belongs to $\jJ(\alpha)$ outside a compact subset,
then there is no simple
$J'$-holomorphic curve homotopic to $u_\gamma$ through asymptotically cylindrical maps.
\end{enumerate}
\end{exercise}

\chapter{Symplectic fillings of planar contact $3$-manifolds}
\label{sec:5}

\minitoc

\vspace{12pt}

In\CUP{This material will be published by Cambridge University
Press as \textsl{Contact 3-Manifolds, Holomorphic Curves and Intersection Theory}
by Chris Wendl. This pre-publication version is
free to view and download for personal use only. 
Not for re-distribution, re-sale or use in derivative works. \copyright Chris Wendl, 2019.}
this lecture, we will explain an application of the intersection theory
of punctured holomorphic curves to the problem of classifying symplectic
fillings of contact $3$-manifolds.  The main result is stated in
\S\ref{sec:pony} as Theorem~\ref{thm:pony}, and it may be seen as an
analogue of McDuff's Theorem~\ref{thm:McDuff} in a slightly different
context---indeed, the structure of the proof is very similar, but the
technical details are a bit more intricate and require the machinery
developed in Lecture~\ref{sec:4}. Before stating the
theorem and sketching its proof, we review some topological facts about
Lefschetz fibrations, open books, and symplectic fillings.

\section{Open books and Lefschetz fibrations}
\label{sec:Gompf}

As we saw in Lecture~\ref{sec:1}, symplectic forms on $4$-manifolds can be
characterized topologically (up to deformation) via Lefschetz fibrations.
The natural analogue of a Lefschetz fibration for a contact manifold
is an \defin{open book decomposition}.\index{open book decomposition}
If $M$ is a closed oriented
$3$-manifold, an open book is a pair $(B,\pi)$, where $B \subset M$ is an
oriented link and
$$
\pi : M \setminus B \to S^1
$$
is a fibration such that some neighborhood $\nN(\gamma) \subset M$ of each
connected component $\gamma \subset B$ admits an identification with
$S^1 \times \DD$ in which $\pi$ takes the form
$$
\pi|_{\nN(\gamma) \setminus \gamma} : S^1 \times (\DD \setminus \{0\}) \to S^1 :
(\theta,(r,\phi)) \mapsto \phi.
$$
Here $(r,\phi)$ denote polar coordinates on the disk $\DD$, with the angle normalized
to take values in $S^1 = \RR / \ZZ$.  We call $B$ the \defin{binding}\index{binding of an open book}\index{open book decomposition!binding of}
of the open
book, and the fibres $\pi^{-1}(\phi) \subset M$ are its \defin{pages};\index{pages of an open book}\index{open book decomposition!pages of}
these are
open surfaces whose closures are compact oriented surfaces with oriented boundary
equal to~$B$.  Figure~\ref{fig:OBDs} shows simple examples on $S^3$
and $S^1 \times S^2$.

A contact structure $\xi$ on $M$ is said to be \defin{supported}\index{open book decomposition!supporting a contact structure}\index{contact structure!supported by an open book}
by the open book
$\pi : M \setminus B \to S^1$ if one can write $\xi = \ker\alpha$ some contact
form $\alpha$ with
$$
\alpha|_{TB} > 0 \quad\text{ and } \quad d\alpha|_{\text{pages}} > 0.
$$
Equivalently, one can require that the components of $B$ are closed Reeb
orbits with respect to~$\alpha$, and everywhere else the Reeb vector field
is positively transverse to the pages.  This definition is due to Giroux, and
contact forms that satisfy these conditions 
are sometimes called \defin{Giroux forms}.\index{Giroux form}

The following contact analogue of Theorems~\ref{thm:Thurston} and~\ref{thm:Gompf}
is a translation into modern language of a classical result of Thurston and Winkelnkemper:

\begin{thm}[Thurston-Winkelnkemper \cite{ThurstonWinkelnkemper}]
\label{thm:ThurstonWinkelnkemper}
Every open book on a closed oriented $3$-manifold supports a unique contact
structure up to isotopy.\index{open book decomposition!supporting a contact structure}\index{contact structure!supported by an open book}
\end{thm}

\begin{figure}
\includegraphics{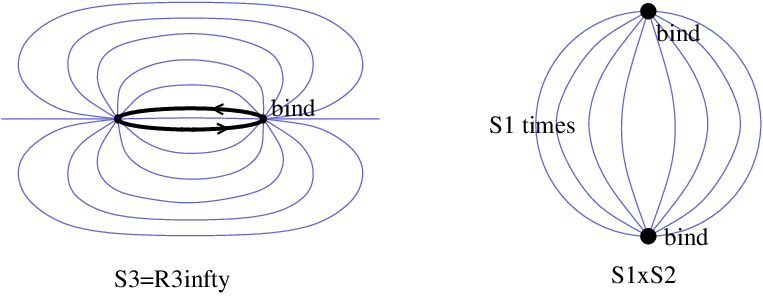}
\caption{\label{fig:OBDs} 
Simple open book decompositions on $S^3$ and $S^1 \times S^2$, with pages
diffeomorphic to the plane and the cylinder respectively.}
\end{figure}

A much deeper result known as the \emph{Giroux correspondence}
\cite{Giroux:ICM} asserts that the set of contact structures up
to isotopy on any closed $3$-manifold admits a natural bijection to the
set of open books up to a topological operation called \emph{positive 
stabilization}.  We will not need this fact in the discussion below, but it
is worth mentioning since it
has had a major impact on the modern field of contact topology; see
e.g.~\cite{Etnyre:lectures} for more on this subject.

In order to discuss symplectic fillings, we will also need to consider a
more general class of Lefschetz fibrations, in which both the base and
fibre are allowed to have boundary.  Specifically, assume $W$ is a compact
oriented $4$-manifold with boundary and corners, where the boundary consists of two
smooth faces
$$
\p W = \p_h W \cup \p_v W,
$$
the \defin{horizontal}\index{horizontal boundary of a Lefschetz fibration}\index{Lefschetz fibration!horizontal boundary of}
and \defin{vertical boundary}\index{vertical boundary of a Lefschetz fibration}\index{Lefschetz fibration!vertical boundary of}
respectively, which
intersect each other at a corner of codimension~$2$.  Given a compact
oriented surface $\Sigma$ with nonempty boundary, we define a
\defin{bordered Lefschetz fibration}\index{bordered Lefschetz fibration}\index{Lefschetz fibration!bordered}
of $W$ over $\Sigma$ to be a smooth map
$$
\Pi : W \to \Sigma
$$
with finitely many \emph{interior} critical points 
$W\crit := \Crit(\Pi) \subset \mathring{W}$ and critical values
$\Sigma\crit := \Pi(W\crit) \subset \mathring{\Sigma}$ such that:
\begin{enumerate}
\item As in Example~\ref{ex:Lefschetz}, critical points take the form 
$\Pi(z_1,z_2) = z_1^2 + z_2^2$ in complex local coordindates compatible with
the orientations;
\item The fibres have nonempty boundary;
\item $\Pi^{-1}(\p\Sigma) = \p_v W$ and 
$$
\Pi|_{\p_v W} : \p_v W \to \p\Sigma
$$
is a smooth fibration;
\item $\p_h W = \bigcup_{z \in \Sigma} \p\left( \Pi^{-1}(z) \right)$, and
$$
\Pi|_{\p_h W} : \p_h W \to \Sigma
$$
is also a smooth fibration.
\end{enumerate}

In the following, we assume the base is the closed unit disk
(see Figure~\ref{fig:LefschetzBordered}),
$$
\Sigma := \DD \subset \CC.
$$
In this case, the vertical boundary is a connected fibration of some
compact oriented surface with boundary over $\p\DD = S^1$,
$$
\pi := \Pi|_{\p_v W} : \p_v W \to S^1,
$$
and the horizontal boundary is a disjoint union of circle bundles
over~$\DD$; since bundles over $\DD$ are trivial,
the connected components of $\p_h W$ can then be identified with 
$S^1 \times \DD$ such that $\pi$ on the corner
$\p_h W \cap \p_v W = \p(\p_h W) = \coprod (S^1 \times \p\DD)$
takes the form $\pi(\theta,\phi) = \phi$.  This means that after
smoothing the corners of $\p W$, the latter inherits from
$\Pi : W \to \DD$ an open book decomposition $\pi : \p W \setminus B \to S^1$
uniquely up to isotopy, with $\p_h W$ regarded as a tubular neighborhood of
the binding $B := \coprod (S^1 \times \{0\})$.

\begin{figure}
\includegraphics{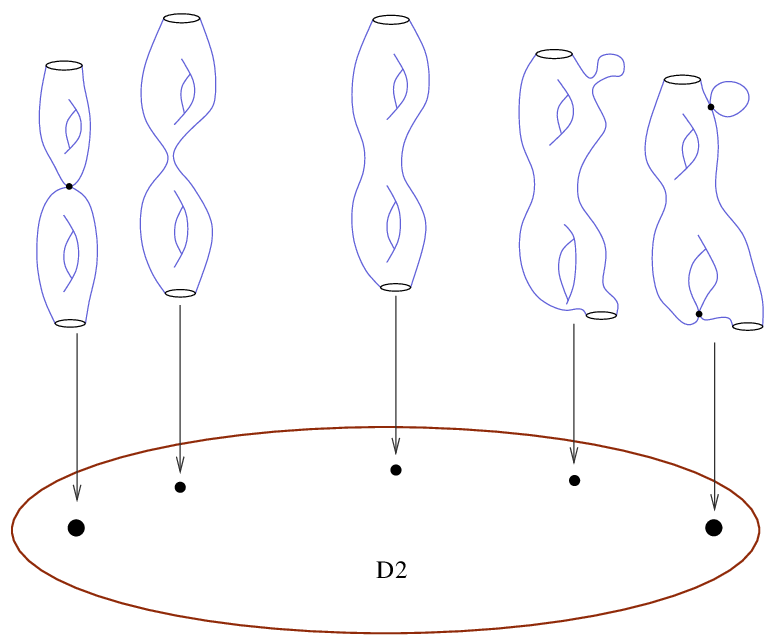}
\caption{\label{fig:LefschetzBordered}
A bordered Lefschetz fibration over the unit disk $\DD \subset \CC$,
where the regular fibres have genus~$2$ and two boundary components, 
and there are two singular fibres, each with two irreducible components.
The boundary inherits an open book with pages of genus~$2$ and two
binding components.}
\end{figure}

Recall that for any surface fibration $F \hookrightarrow M \to S^1$
that is trivial near the boundaries of the fibres,
the parallel transport (with respect to any connection) along a
full traversal of the loop $S^1$ determines (uniquely up to isotopy)
a diffeomorphism $\varphi : F \to F$ that is trivial near~$\p F$;
we call this the \defin{monodromy}\index{monodromy|(}\index{Lefschetz fibration!monodromy|(}
of the fibration.  One can thus
define the monodromy of a Lefschetz fibration along any loop 
containing no critical values---in particular, the monodromy along
$\p\DD$ is also called the \emph{monodromy of the open book}\index{open book decomposition!monodromy of|(}
induced at the boundary.  

It is a basic fact about the topology of Lefschetz fibrations that
the monodromy along a loop can always be expressed in terms of
\emph{positive Dehn twists}; see e.g.~\cite{GompfStipsicz}.
For our purposes, the relevant version of this statement is the
following.  Let $z_0 = 1 \in \p\DD$ and denote the fibre at $z_0$ by
$F := \Pi^{-1}(z_0)$. Pick a set of smooth paths
$$
\gamma_z : [0,1] \to \DD, \quad \text{ for each $z \in \DD\crit$},
$$
from $\gamma_z(0) = z_0$ to $\gamma_z(1) = z$, intersecting each
other only at~$z_0$.  Then for each $z \in \DD\crit$ and
$p \in W\crit \cap \Pi^{-1}(z)$, there is a unique isotopy class of smoothly
embedded circles
$$
S^1 \cong C_p \subset F
$$
that can be collapsed to~$p$ under parallel transport along 
$\gamma_{z}$; this is called the \defin{vanishing cycle}\index{vanishing cycle}\index{Lefschetz fibration!vanishing cycle}
of~$p$.  We then have:

\begin{prop}
\label{prop:monodromy}
If $\Pi : W \to \DD$ is a bordered Lefschetz fibration, then the
monodromy $F \to F$ of the induced open book at the boundary is a 
composition of positive Dehn twists along the vanishing cycles
$C_p \subset F$ for each critical point $p \in W\crit$.
\end{prop}

\begin{example}
\label{ex:cylinders}
Suppose $\Pi : W \to \DD$ has regular fibre $F \cong [-1,1] \times S^1$
and exactly $k \ge 0$ singular fibres, each consisting of two disks
connected along a critical point (see Figure~\ref{fig:cylinders}, left).
The resulting open book on $\p W$ then has pages diffeomorphic to
$\RR \times S^1$ and monodromy $\delta^k$, where $\delta$ denotes the
positive Dehn twist along the separating curve $\{0\} \times S^1$, which
generates the mapping class group of $\RR \times S^1$.
If we blow up $W$ at a regular point in the interior, 
then by Exercise~\ref{EX:blowupLefschetz}
we obtain a new bordered Lefschetz fibration with one additional singular 
fibre consisting of an annulus connected to an exceptional sphere 
(Figure~\ref{fig:cylinders}, right).  This blowup operation obviously
does not change the open book on $\p W$, which is consistent
with Proposition~\ref{prop:monodromy} since the additional
Dehn twist introduced by the extra critical point is along a \emph{contractible}\index{Lefschetz fibration!monodromy|)}
vanishing cycle, and is therefore isotopic to the identity.\index{monodromy|)}\index{open book decomposition!monodromy of|)}
\end{example}

\begin{figure}
\includegraphics{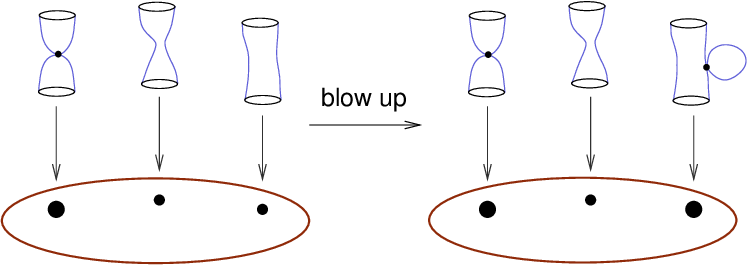}
\caption{\label{fig:cylinders} Two bordered Lefschetz fibrations with regular fibre
diffeomorphic to the annulus. The picture at the right is obtained from the one at
the left by blowing up at a regular point.}
\end{figure}

Let us denote the contact manifolds supported by the open books on $\p W$
in Example~\ref{ex:cylinders} by $(M_k,\xi_k)$.  It is not too hard to say
precisely what these contact manifolds are: topologically, we have
$M_0 \cong S^1 \times S^2$, $M_1 \cong S^3$, and $M_k$ for $k \ge 2$ is the
lens space $L(k,k-1)$.  All of these carry \emph{standard} contact structures
that can be defined as follows.  We defined $(S^3,\xi\std)$ already in
\S\ref{sec:contact} by identifying $S^3$ with the boundary of the unit ball in
$\RR^4$ with coordinates $(x_1,y_1,x_2,y_2)$ and writing $\xi\std = 
\ker\left( \lambda\std|_{T S^3} \right) \subset TS^3$, where
$$
\lambda\std := \frac{1}{2} \sum_{j=1}^2 \left( x_j\, dy_j - y_j \, d x_j \right).
$$
Under the natural identification $\RR^4 \to \CC^2 : (x_1,y_1,x_2,y_2) \mapsto
(x_1 + iy_1, x_2 + i y_2)$, this contact structure is invariant under the action
of $\U(2)$, thus the standard contact structure $\xi\std$ on any lens space
$L(p,q)$ for two coprime integers $p > q \ge 1$ can be defined via the quotient
$$
\left(L(p,q), \xi\std\right) := (S^3,\xi\std) \big/ G_{p,q},
$$
where $G_{p,q} \subset \U(2)$ denotes the subgroup
$$
G_{p,q} := \left\{ \begin{pmatrix} \zeta & 0 \\ 0 & \zeta^q \end{pmatrix} \in \U(2)\ 
\bigg| \ \zeta^p = 1 \right\}.
$$
Finally, on $S^1 \times S^2$, we use the coordinates $(\eta,\theta,\phi)$, where
$\eta \in S^1 = \RR /\ZZ$ and $(\theta,\phi) \in [0,\pi] \times (\RR / 2\pi\ZZ)$
are the natural spherical coordinates on~$S^2$, and write
$$
\xi\std = \ker\left[ f(\theta)\, d\eta + g(\theta)\, d\phi  \right]
$$
for a suitably chosen loop $(f,g) : \RR / \pi\ZZ \to \RR^2 \setminus \{0\}$ that is
based at the point $(f(0),g(0)) = (1,0)$ and winds exactly once counterclockwise around
the origin.  Any two choices of $(f,g)$ that make the above expression a smooth
contact form on $S^1 \times S^2$ and have the stated winding property produce isotopic
contact structures; see e.g.~\cite{Geiges:book}.
The following can now be verified by constructing explicit open books that support these
contact structures and then computing the monodromy.

\begin{prop}
\label{prop:examples}
There are contactomorphisms $(M_0,\xi_0) \cong (S^1 \times S^2,\xi\std)$,
$(M_1,\xi_1) \cong (S^3,\xi\std)$, and $(M_k,\xi_k) \cong (L(k,k-1),\xi\std))$
for each $k \ge 2$.
\end{prop}

We will say that a symplectic form $\omega$ on $W$ is 
\defin{supported by}\index{symplectic form!supported by a bordered Lefschetz fibration|(}
a bordered Lefschetz fibration $\Pi : W \to \DD$
if the following conditions hold:
\begin{enumerate}
\item Every fibre of $\Pi|_{W\setminus W\crit} : W \setminus W\crit 
\to \DD$ is a symplectic submanifold;
\item On a neighborhood of $W\crit$, $\omega$ tames some almost 
complex structure $J$ that preserves the tangent spaces of the fibres;
\item On a neighborhood of $\p W$, $\omega = d\lambda$ for some
$1$-form $\lambda$ such that $\lambda|_{T(\p_h W)}$ and 
$\lambda|_{T(\p_v W)}$ are
each contact forms, and the induced Reeb vector field on $\p_h W$ is
tangent to the fibres (in the positive direction).
\end{enumerate}

Observe that for the contact form $\lambda$ on the smooth faces
of $\p W$ in the above definition, $d\lambda = \omega$ is
necessarily positive on the pages of the induced open book, and
$\lambda$ is also positive on the binding in~$\p_h W$, so that
$\lambda|_{\p W}$ satisfies a variation on the conditions for
a Giroux form.
The natural analogue of Theorem~\ref{thm:Gompf} in this context
is the following:

\begin{thm}
\label{thm:supported}
On any bordered Lefschetz fibration $\Pi : W \to \DD$, the
space of supported symplectic forms is nonempty and connected,
and the corner of $\p W$ can be smoothed so that $(W,\omega)$
becomes (canonically up to symplectic deformation)
a symplectic filling of the contact structure supported by the
induced open book at the boundary.
\end{thm}

We note one additional detail about this construction:
a symplectic form $\omega$ on $W$ may sometimes be exact since
$\p W \ne \emptyset$, but the condition of $\omega$ being
positive on fibres imposes contraints that may make this
impossible.  In particular, $\omega$ can never be exact if
any singular fibre of $\Pi : W \to \DD$
contains an irreducible component that is closed---this 
would violate Stokes' theorem.  We say that a bordered
Lefschetz fibration is \defin{allowable}\index{Lefschetz fibration!allowable}\index{bordered Lefschetz fibration!allowable}
if no such components
exist, which is equivalent to saying that all the vanishing
cycles are homologically nontrivial.
For example, the Lefschetz fibration in Figure~\ref{fig:LefschetzBordered}
is not allowable, due to the presence of a closed irreducible
component in the singular fibre at the right, but one can
show that this component is an exceptional sphere, thus an
allowable Lefschetz fibration could be produced by blowing it down
(cf.~Exercise~\ref{EX:blowupLefschetz}).

It turns out that if
$\Pi : W \to \DD$ is allowable, one can always construct $\omega$ so
that it is not only exact but also arises from a \emph{Weinstein
structure}, a much more rigid notion of a symplectic filling.
We will not discuss Weinstein and Stein fillings any further 
here (see \cites{Etnyre:convexity,OzbagciStipsicz,CieliebakEliashberg}),
except to mention the following related result:

\addtocounter{thm}{-1}
\renewcommand{\thethm}{\thechapter.\arabic{thm}$'$}
\begin{thm}
\label{thm:supportedStein}
If $\Pi : W \to \DD$ is an allowable bordered Lefschetz
fibration, then $(W,\omega)$ in Theorem~\ref{thm:supported}
can be arranged to be a Weinstein filling of the\index{Stein filling}\index{Weinstein filling}\index{symplectic filling!Stein}\index{symplectic filling!Weinstein}
contact manifold $(\p W,\xi)$ supported by the open book
induced at the boundary. In particular, $(\p W,\xi)$ is
Stein fillable.
\end{thm}
\renewcommand{\thethm}{\thechapter.\arabic{thm}}

Theorems~\ref{thm:supported} and~\ref{thm:supportedStein} can be
found in a variety of forms in the literature but are usually not
stated quite so precisely as we have stated them here---complete proofs
of our versions (including also cases where $\Sigma \ne \DD$) may 
be found in \cite{LisiVanhornWendl1}.\index{symplectic form!supported by a bordered Lefschetz fibration|)}

\section{A classification theorem for symplectic fillings}
\label{sec:pony}

An open book decomposition of a $3$-manifold is called \defin{planar}\index{open book decomposition!planar}\index{planar open book}
if its pages have genus~$0$, i.e.~they are punctured spheres.  We then call
$(M,\xi)$ a \defin{planar contact manifold}\index{contact structure!planar}\index{contact manifold!planar}\index{planar contact manifold}
if $M$ admits a planar open book
supporting~$\xi$.  The planar contact manifolds play something of a special
role in $3$-dimensional contact topology, similar to the role of
rational and ruled surfaces among symplectic $4$-manifolds (see \cites{McDuff:rationalRuled,Wendl:rationalRuled}).
It is not always easy to recognize whether a
given contact structure is planar or not, but many results in either direction or
known: Etnyre \cite{Etnyre:planar} showed for instance that all \emph{overtwisted}
contact structures are planar, and by an obstruction established in the
same paper, the standard contact structures on unit cotangent bundles of
oriented surfaces with positive genus are never planar.  As we saw in
Proposition~\ref{prop:examples}, the standard contact structures on $S^3$,
$S^1 \times S^2$ and $L(k,k-1)$ for $k \ge 2$ are all planar, as are all
contact structures that arise on boundaries of bordered Lefschetz
fibrations with genus~$0$ fibres.

For an arbitrary contact $3$-manifold $(M,\xi)$, the problem of classifying
all of its symplectic fillings is often hopeless---many examples are
known for instance which admit infinite (but not necessarily exhaustive) 
lists of
pairwise non-homeomorphic or non-diffeomorphic Stein fillings \cites{Smith:torusFibrations,
OzbagciStipsicz:infinitely,AkhmedovEtnyreMarkSmith}.  On the other hand,
many of the earliest results on this question gave finite classifications,
and sometimes even \emph{uniqueness} (up to certain obvious ambiguities) of symplectic
fillings, e.g.~for $S^3$ \cites{Gromov, Eliashberg:diskFilling}, 
$S^1 \times S^2$ \cite{Eliashberg:diskFilling},
the unit cotangent bundle of $S^2$ \cite{Hind:RP3}, and lens 
spaces \cite{McDuff:rationalRuled,Hind:Lens,Lisca:fillingsLens}.  Most of
these finiteness results can now be deduced from the theorem stated below.

We will say that a symplectic filling $(W,\omega)$ of a contact
$3$-manifold $(M,\xi)$ admits a \defin{symplectic Lefschetz
fibration over $\DD$}\index{symplectic filling!Lefschetz fibrations on}\index{Lefschetz fibration!of a symplectic filling}\index{bordered Lefschetz fibration!of a symplectic filling}
if there exists a bordered Lefschetz fibration
$\Pi : E \to \DD$ with a supported symplectic form $\omega_E$ such that,
after smoothing the corners on $\p E$, $(E,\omega_E)$ is symplectomorphic
to $(W,\omega)$.  Whenever this is the case, the
Lefschetz fibration determines a supporting open book on $(M,\xi)$
uniquely up to isotopy.

\begin{thm}[\cite{Wendl:fillable}]
\label{thm:pony}
Suppose $(W,\omega)$ is a symplectic filling of a contact $3$-manifold\index{Lefschetz fibration!as filling of an open book}\index{open book decomposition!filled by a Lefschetz fibration}
$(M,\xi)$ which is supported by a planar open book $\pi : M \setminus B
\to S^1$.  Then $(W,\omega)$ admits a symplectic Lefschetz fibration
over~$\DD$, such that the induced open book at the boundary is isotopic
to $\pi : M \setminus B \to S^1$.  Moreover, the Lefschetz fibration
is allowable if and only if $(W,\omega)$ is minimal.
\end{thm}

One can say slightly more: \cite{Wendl:fillable} shows in fact that the
isotopy class of the Lefschetz fibration produced on $(W,\omega)$ depends
only on the deformation class of the symplectic structure, hence the
problem of classifying fillings up to deformation reduces to the
problem of classifying Lefschetz fibrations that fill a given open book.
In some cases, this provides an immediate uniqueness result, for instance:

\begin{cor}
\label{cor:uniqueFillings}
The symplectic fillings of $(S^3,\xi\std)$, $(S^1 \times S^2,\xi\std)$ and
$(L(k,k-1),\xi\std)$ for $k \ge 2$ are unique up to symplectic deformation
equivalence and blowup.
\end{cor}
\begin{proof}
By Proposition~\ref{prop:examples}, the contact manifolds in question all
admit supporting open books with cylindrical pages and monodromy equal
to $\delta^k$ for some $k \ge 0$, where $\delta$ is the positive Dehn
twist that generates the mapping class group of $\RR \times S^1$.
The only allowable Lefschetz fibration that produces such an open book
at the boundary is the one with fibre $[-1,1] \times S^1$ and exactly
$k$ singular fibres of the form pictured in Figure~\ref{fig:cylinders}
at the left.  Theorem~\ref{thm:pony} then implies that all the minimal 
symplectic fillings in question are supported by Lefschetz fibrations
of this type, which determines their symplectic structures up to
deformation equivalence via Theorem~\ref{thm:supported}.
\end{proof}

Further uniqueness results along these lines have been obtained in papers
by Plamenevskaya and Van Horn-Morris \cite{PlamenevskayaVanHorn}, and
Kaloti and Li \cite{KalotiLi:Stein}, each by studying
the factorizations of mapping classes on planar surfaces into products of
positive Dehn twists and then applying Theorem~\ref{thm:pony}.  In a
slightly different direction, Wand \cite{Wand:planar} used
Theorem~\ref{thm:pony} to establish a new obstruction for a contact
$3$-manifold to be planar.

Theorem~\ref{thm:supportedStein} implies another quite general consequence
of the above result:

\begin{cor}
\label{cor:strongNotStein}
Every symplectic filling of a planar contact manifold is deformation
equivalent to a blowup of a Stein filling.  In particular, any contact
manifold that is both planar and symplectically fillable is also\index{Stein filling}\index{Weinstein filling}\index{symplectic filling!Stein}\index{symplectic filling!Weinstein}
Stein fillable.
\end{cor}

Ghiggini \cite{Ghiggini:strongNotStein} gave examples of 
contact $3$-manifolds that are symplectically but not Stein fillable,
hence Corollary~\ref{cor:strongNotStein} implies that Ghiggini's examples
cannot be planar.  Similarly, Wand 
\cite{Wand:factorizations} and Baker, Etnyre and Van Horn-Morris
\cite{BakerEtnyreVanhorn} have given examples of Stein fillable contact
manifolds with (necessarily non-planar) supporting open books that cannot 
arise from boundaries of Lefschetz fibrations.

\begin{remark}
Theorem~\ref{thm:pony} and Corollary~\ref{cor:strongNotStein} 
can be generalized to allow Lefschetz fibrations over 
arbitrary compact oriented surfaces with boundary \cite{LisiVanhornWendl2}.
In this form, they apply to a larger class of contact manifolds, including
many that are not planar; a prototypical example of this is the uniqueness
(proved originally in \cite{Wendl:fillable}) of strong fillings of the $3$-torus,
whose tight contact structures are never planar.  Generalizing in a
different direction, \cite{NiederkruegerWendl} shows that both results
also remain valid (but only specifically for
\emph{planar} contact manifolds) if the symplectic filling
condition on $(M,\xi)$ is weakened to the existence of a compact symplectic manifold
$(W,\omega)$ with $\p W = M$ and $\omega|_\xi > 0$.
Such objects are known as \defin{weak symplectic fillings}\index{weak symplectic filling}\index{symplectic filling!weak}
of $(M,\xi)$, and they have been extensively studied, but at present,
the planar contact manifolds are the only class for which any meaningful
classification of weak fillings is known to be feasible.\footnote{The loophole in this statement
  is that by a frequently used lemma of Eliashberg \cite{Eliashberg:contactProperties}*{Prop.~3.1},
  weak fillings for which $\omega$ is exact near the boundary can always
  be deformed to strong fillings, thus whenever
  $M$ happens to be a rational holomogy $3$-sphere, the classification of weak
  fillings is the same as that of strong fillings.  This is, however,
  an essentially topological phenomenon that has little to do with contact geometry.}
\end{remark}

\section{Sketch of the proof}

As in our proof of McDuff's result on ruled surfaces, the main idea for
Theorem~\ref{thm:pony} is to consider a moduli space of holomorphic
curves whose intersection theory is sufficiently well behaved
to view them as fibres of a Lefschetz fibration.  The first step is
thus to define the moduli space and show that it is nonempty.  This rests
on a construction known as the \emph{holomorphic open book}; the following
result was first stated in \cite{ACH}, and two proofs later appeared in
independent work of Abbas \cite{Abbas:openbook} and the author 
\cite{Wendl:openbook}.

\begin{thm}[Holomorphic open book construction]
\label{thm:openbook}
Suppose $(M,\xi)$ is supported by a planar open book $\pi : M \setminus B 
\to S^1$.  Then there exists a Giroux form $\alpha$ and an almost complex\index{open book decomposition!holomorphic|(}\index{holomorphic open book|(}
structure $J_+ \in \jJ(\alpha)$ such that:
\begin{enumerate}
\item
Each connected component $\gamma \subset B$ is a nondegenerate Reeb orbit
with $\muCZ^\tau(\gamma) = 1$, where $\tau$ is any trivialization in which
the pages approach~$\gamma$ with winding number~$0$.
\item
Each page $P \subset M$ lifts to an embedded asymptotically cylindrical 
$J_+$-holomorphic curve $u_P : \dot{\Sigma} \to \RR \times M$ with all
punctures positive, and $\ind(u_P) = 2$.
\end{enumerate}
\end{thm}

Observe that pages of an open book come always in $1$-parameter families, but when we 
lift them to the symplectization $\RR \times M$, an additional parameter 
appears due to the translation-invariance.  Thus Theorem~\ref{thm:openbook} produces
a $2$-dimensional moduli space 
$$
\mM\OB^+ \subset \mM_0(\RR\times M,J_+)
$$
of $J_+$-holomorphic pages; it is diffeomorphic to $\RR \times S^1$ and admits a
free action by $\RR$-translations, so that
$$
\mM\OB^+ \big/ \RR \cong S^1.
$$
The curves in $\mM\OB^+$ have the ``correct'' index, in the sense that the actual 
and virtual dimensions match.  In \cite{Wendl:openbook}, it is shown in fact that 
for suitable (non-generic!)
choices of data, an open book with pages of \emph{any} genus $g \ge 0$ admits a
$2$-parameter family of pseudoholomorphic lifts, but they have index $2 - 2g$,
which is the correct virtual dimension only when $g=0$.  This is why
Theorem~\ref{thm:pony} fails in general for open books of positive genus
(cf.~Remark~\ref{remark:higherGenus}).\index{open book decomposition!holomorphic|)}\index{holomorphic open book|)}

Now suppose $(W,\omega)$ is a symplectic filling of $(M,\xi)$, where the latter is supported
by a planar open book.  By modifying $\omega$ near $\p W$, we can assume without loss of
generality (possibly after rescaling $\omega$) that it takes the form $d(e^s\alpha)$ in a collar
neighborhood of the boundary, where $\alpha$ is the contact form provided by
Theorem~\ref{thm:openbook}.  Let $(\widehat{W},\widehat{\omega})$ denote the resulting
symplectic completion, and choose an $\widehat{\omega}$-compatible almost complex structure
$J$ which is generic in~$W$ and matches $J_+$ (from Theorem~\ref{thm:openbook}) on 
$[0,\infty) \times M$.
Since the $J_+$-holomorphic pages in $\RR \times M$ have no negative punctures, each can
be assumed to lie in $[0,\infty) \times M$ after some $\RR$-translation, so these
give rise to a $2$-dimensional family of embedded $J$-holomorphic curves living in the
cylindrical end of~$\widehat{W}$,
which we shall refer to henceforward as the \emph{$J$-holomorphic pages in~$\widehat{W}$}.
Let
$$
\mM\OB \subset \mM_0(\widehat{W},J)
$$
denote the connected component of the moduli space $\mM_0(\widehat{W},J)$ 
that contains these $J$-holomorphic pages, and let $\overline{\mM}\OB$ denote its closure in the
compactified moduli space $\overline{\mM}_0(\widehat{W},J)$ (see Appendix~\ref{app:punctured}).
Theorem~\ref{thm:pony} can be deduced from the following:

\begin{prop}
\label{prop:foliationM}
The compactified moduli space $\overline{\mM}\OB$ is diffeomorphic to the $2$-disk, and the
elements of $\overline{\mM}\OB$ can be described as follows.
\begin{itemize}
\item The smooth curves in $\mM\OB$ are all embedded and pairwise disjoint, and they
foliate $\widehat{W}$ outside a finite union of properly embedded surfaces.
\item There is a natural identification
$$
\p\overline{\mM}\OB = \mM\OB^+ \big/ \RR,
$$
where each $\RR$-equivalence class of $J_+$-holomorphic pages $u_P \in \mM\OB^+$ in $\RR \times M$
is identified with a holomorphic building that has empty main level and a single upper level
consisting of~$u_P$.\footnote{It is standard to define the space of holomorphic buildings
$\overline{\mM}_g(\widehat{W},J)$ such that two buildings are considered equivalent if they
differ only by an $\RR$-translation of one of the symplectization levels; see \cite{SFTcompactness}.}
\item There are at most finitely many elements of
$\overline{\mM}\OB \setminus \mM\OB$ in the interior of $\overline{\mM}\OB$, and they are
pairwise disjoint nodal curves in $\widehat{W}$, each having exactly two connected components,
which are embedded and intersect each other exactly once, transversely.  Any such component that
is closed also has homological self-intersection number~$-1$.
\end{itemize}
Every point in $\widehat{W}$ lies in the image of a unique (possibly nodal) curve
in the interior of~$\overline{\mM}\OB$.
\end{prop}

To prove this, notice first that all curves in $\mM\OB$ are guaranteed to be simple,
since their asymptotic orbits are all distinct and simply covered.  By
Theorem~\ref{thm:sing}, all $u \in \mM\OB$ then satisfy $\delta_\infty(u) = 0$,
as double points can only be hidden at infinity if there exist multiply covered
asymptotic orbits or two distinct punctures asymptotic to coinciding orbits.
Since $\delta(u) + \delta_\infty(u)$ is homotopy invariant (Theorem~\ref{thm:adjunctionPunctured}),
and the $J$-holomorphic pages $u_P$ are embedded and thus satisfy $\delta(u_P) = 0$,
we conclude
$$
\delta(u) = \delta_\infty(u) = 0 \quad\text{ for all $u \in \mM\OB$},
$$
hence all curves in $\mM\OB$ are embedded.
Theorem~\ref{thm:puncturedFoliations} then implies slightly more: since the curves
in $\mM\OB$ also have index~$2$ and genus~$0$ and all their asymptotic orbits
have odd Conley-Zehnder index, we have:
\begin{lemma}
\label{lemma:puncturedFoliations}
For each $u \in \mM\OB$, there is a neighborhood $\uU \subset \mM\OB$ such that
the curves in $\uU$ are all embedded and their images foliate a neighborhood
of the image of $u$ in $\widehat{W}$.
\end{lemma}

The self-intersection number $u * u$ for any curve $u \in \mM\OB$ can now be\index{self-intersection number!of asymptotically cylindrical maps}
computed easily from Siefring's adjunction
formula \eqref{eqn:adjunctionPunctured}: since all asymptotic orbits are simply 
covered, the spectral covering term $\bar{\sigma}(u) - \#\Gamma$ vanishes,
and so does $c_N(u)$ due to formula \eqref{eqn:2cN}, thus
$$
u * u = 2\left[ \delta(u) + \delta_\infty(u) \right] + c_N(u) + 
\left[ \bar{\sigma}(u) - \#\Gamma \right] = 0.
$$
This result can alternatively be deduced as an immediate corollary of
Lemma~\ref{lemma:int=0} below.
By Theorem~\ref{thm:star}, we conclude:

\begin{lemma}
\label{lemma:selfint=0}
Any two distinct curves in $\mM\OB$ are disjoint.
\end{lemma}
Combining that with Lemma~\ref{lemma:puncturedFoliations}, it follows that the
curves in $\mM\OB$ form a smooth foliation of some open subset of~$\widehat{W}$.

\begin{lemma}
\label{lemma:orbitCylinder}
Assume $u_P \in \mM^+\OB$ is a $J_+$-holomorphic page in $\RR \times M$, and
$u_\gamma : \RR \times S^1 \to \RR \times M$ is the orbit cylinder for an embedded\index{orbit cylinder}
Reeb orbit $\gamma$ that is a component of the binding~$B$.  Then
$$
u_P * u_\gamma = 0.
$$
\end{lemma}
\begin{proof}
The page $P \subset M$ is always disjoint from the binding $B \subset M$, thus
$u_P \cdot u_\gamma = 0$, so it only remains to show that $\inter_\infty(u_P,u_\gamma)=0$.
By Theorem~\ref{thm:asymptoticPositivity}, this is true if and only if the asymptotic
eigenfunction controlling the approach of the relevant end of $u_P$ to $\gamma$ has
extremal winding.  Let $\tau$ denote the trivialization of $\gamma^*\xi$ in which
pages approach~$\gamma$ with winding number~$0$, hence by construction,
the winding (relative to~$\tau$) of the eigenfunction in question is~$0$.
Combining Theorem~\ref{thm:openbook} with Proposition~\ref{prop:CZwinding}, we also have
$$
1 = \muCZ^\tau(\gamma) = 2\alpha^\tau_-(\gamma) + p(\gamma),
$$
hence the extremal winding is also $\alpha^\tau_-(\gamma) = 0$, and this implies
$\inter_\infty(u_P,u_\gamma) = 0$ as claimed.
\end{proof}

\begin{lemma}
\label{lemma:int=0}
Assume $u_P \in \mM^+\OB$ is a $J_+$-holomorphic page in $\RR \times M$ and
$v : \dot{\Sigma}' \to \RR \times M$ is any $J_+$-holomorphic curve whose positive
ends are all asymptotic to embedded Reeb orbits in the binding~$B$.  Then
$u_P * v = 0$.
\end{lemma}
\begin{proof}
The argument depends only on the following facts:
\begin{enumerate}
\item $u_P$ has no negative ends;
\item By Lemma~\ref{lemma:orbitCylinder}, $u_P * u_\gamma = 0$ for all orbits
$\gamma$ that appear at positive ends of~$v$, where $u_\gamma$ denotes the orbit 
cylinder over~$\gamma$.
\end{enumerate}
Figure~\ref{fig:OBDhomotopy} shows a homotopy through asymptotically cylindrical maps
for two curves satisfying the above conditions ($u_P$ has positive genus in the picture,
which has no impact on the argument). After a homotopy, we may assume namely that
$u_P$ lives entirely in $[0,\infty) \times M$, while the portion of $v$ living in
$[0,\infty) \times M$ is simply a disjoint union of orbit cylinders $u_\gamma$
for which Lemma~\ref{lemma:orbitCylinder} implies $u * u_\gamma = 0$.  Using the
homotopy invariance\footnote{Note that
the homotopy in our proof of Lemma~\ref{lemma:int=0} is \emph{not} a homotopy
through $J$-holomorphic curves, but only through asymptotically cylindrical maps.}
of the $*$-pairing, we conclude that $u * v$ equals a sum of terms
of the form $u * u_\gamma$, all of which vanish.
\end{proof}

\begin{figure}
\includegraphics{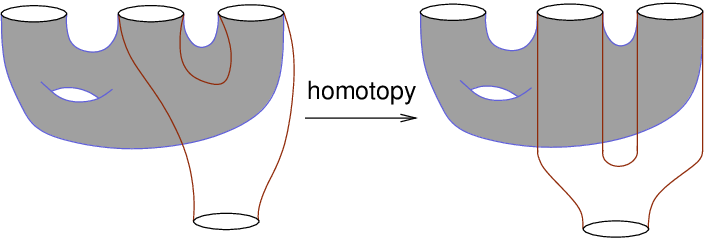}
\caption{\label{fig:OBDhomotopy} A homotopy through asymptotically cylindrical
maps for the Proof of Lemma~\ref{lemma:int=0}.}
\end{figure}

\begin{lemma}[cf.~\cite{Siefring:intersection}*{Theorem~5.21}]
\label{lemma:none}
Other than the $J_+$-holomorphic pages $u_P \in \mM^+\OB$ and
the orbit cylinders over embedded orbits in~$B$, there exist
no $J_+$-holomorphic curves in $\RR \times M$ that are asymptotic to embedded orbits
in~$B$ at all their positive punctures.
\end{lemma}
\begin{proof}
Any such curve $v : \dot{\Sigma}' \to \RR \times M$ must intersect one of the pages
$u_P$, as these foliate $\RR \times (M \setminus B)$, hence
$$
u_P * v \ge u_P \cdot v > 0,
$$
and this contradicts Lemma~\ref{lemma:int=0}.
\end{proof}

We are now in a position to justify the description of the compactification
$\overline{\mM}\OB$ given in Proposition~\ref{prop:foliationM}.

\begin{lemma}
\label{lemma:noBuildings}
Suppose $u_k \in \mM\OB$ is a sequence convergent to a holomorphic building
with at least one nontrivial upper level.  Then the main level of the limit is
empty, and its upper level is a $J_+$-holomorphic page in~$\mM^+\OB$.
\end{lemma}
\begin{proof}
If the lemma is false, then we obtain a holomorphic building whose top level
contains a $J_+$-holomorphic curve in $\RR \times M$ that is not
in $\mM^+\OB$ but has all its positive punctures asymptotic to orbits in
the binding~$B$ (see Figure~\ref{fig:compactnessFilling}).  
This is impossible by Lemma~\ref{lemma:none}.
\end{proof}

\begin{figure}
\includegraphics{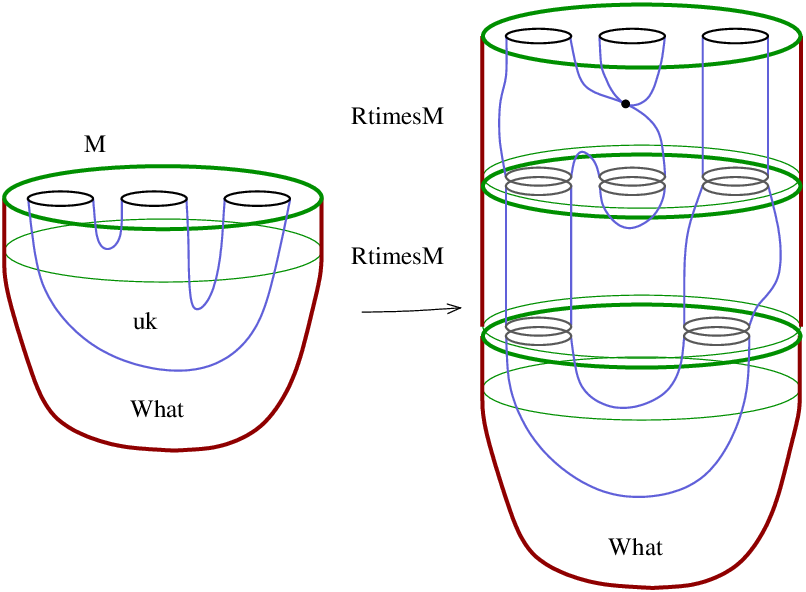}
\caption{\label{fig:compactnessFilling} A hypothetical degeneration of a sequence
$u_k \in \mM\OB$ to a holomorphic building in~$\overline{\mM}\OB$. This scenario
is ruled out by Lemma~\ref{lemma:none}, which says that the two curves in the
top level that are not orbit cylinders cannot exist.}
\end{figure}

We can now identify $\p\overline{\mM}\OB$ with $\mM^+\OB / \RR \cong S^1$ as
described in Proposition~\ref{prop:foliationM}, and the above lemma says that
all other elements of $\overline{\mM}\OB \setminus \mM\OB$ must be buildings
with no upper levels, i.e.~nodal 
$J$-holomorphic curves in~$\widehat{W}$.  The components of these nodal curves
have only positive ends (if any), all asymptotic to distinct simply covered
orbits in the binding~$B$.  These orbits all have
have Conley-Zehnder index~$1$ relative to the canonical trivialization.  Now
\eqref{eqn:indexPunctured} gives the index of such a curve
$v : \Sigma' \setminus \Gamma' \to \widehat{W}$ as
$$
\ind(v) = -\chi(\Sigma' \setminus \Gamma') + 2 c_1^\tau(v^*T\widehat{W}) +
\sum_{z \in \Gamma'} \muCZ^\tau(\gamma_z) =
-\chi(\Sigma') + 2 c_1^\tau(v^*T\widehat{W}).
$$
This matches the index formula for the closed case closely enough that one can
now repeat the compactness argument in the proof of Lemma~\ref{lemma:M0} (see
Appendix~\ref{app:closed}) more or less verbatim,\footnote{There are two main 
differences from the closed case: first, it is trivial to prove that non-closed
curves in $\mathscr{B}$ are simple, since their asymptotic orbits are distinct
and simply covered, while for closed components one must apply the same argument
as before.  Second, proving compactness (and hence finiteness) of $\mathscr{B}$
requires first ruling out holomorphic buildings with nontrivial upper levels---this
works the same way as in Lemma~\ref{lemma:noBuildings}.}
thus proving:

\begin{lemma}
\label{lemma:M0punctured}
There exists a finite set of simple curves $\mathscr{B} \subset
\mM_0(\widehat{W},J)$, with index~$0$, such that every nodal curve in
$\overline{\mM}\OB$ has exactly two components $v_+, v_- \in \mathscr{B}$.
\qed
\end{lemma}

To finish the proof, we must study the intersection-theoretic properties of
the components of nodal curves $\{ v_+,v_- \} \in \overline{\mM}\OB$.
Given such a curve as limit of a sequence $u_k \in \mM\OB$, we have
\begin{equation}
\label{eqn:sumInt}
0 = u_k * u_k = v_+ * v_+ + v_- * v_- + 2(v_+ * v_-).
\end{equation}
\begin{exercise}
\label{EX:sumInt}
Verify \eqref{eqn:sumInt}, using the definition of the $*$-pairing from
Lecture~\ref{sec:4}.
\end{exercise}

Observe that $v_+$ and $v_-$ cannot be the same curve up to parametrization,
as they are required to have distinct sets of asymptotic orbits.  This implies
that they have at least one isolated intersection, so by
Theorem~\ref{thm:star},
\begin{equation}
\label{eqn:atLeast1}
v_+ * v_- \ge v_+ \cdot v_- \ge 1.
\end{equation}
Since $\ind(v_\pm) = 0$ and all asymptotic orbits of $v_\pm$ have
odd Conley-Zehnder index, \eqref{eqn:2cN} gives
$$
c_N(v_\pm) = -1.
$$
Now applying Siefring's adjunction formula (Theorem~\ref{thm:adjunctionPunctured}), 
the spectral covering numbers
$\bar{\sigma}(v_\pm)$ are each equal to the number of punctures since all orbits are 
simply covered, so these terms vanish from the adjunction formula and we have
$$
v_\pm * v_\pm = 2\left[ \delta(v_\pm) + \delta_\infty(v_\pm)\right] +
c_N(v_\pm) =  2\left[ \delta(v_\pm) + \delta_\infty(v_\pm)\right] - 1.
$$
Combining this with \eqref{eqn:sumInt} gives
$$
0 = 2\sum_{\pm} \left[ \delta(v_\pm) + \delta_\infty(v_\pm)\right] +
2 (v_+ * v_- - 1),
$$
so in light of \eqref{eqn:atLeast1}, we have
$$
\delta(v_\pm) = \delta_\infty(v_\pm) = 0, \qquad
v_+ * v_- = v_+ \cdot v_- = 1,
\quad \text{ and } \quad
v_\pm * v_\pm = -1,
$$
implying that $v_\pm$ are both embedded and intersect each other
exactly once, transversely.  Moreover, if either component is closed,
then its homological self-intersection number is now 
$v_\pm \cdot v_\pm = v_\pm * v_\pm = -1$, hence it is a $J$-holomorphic
exceptional sphere; see Figure~\ref{fig:nodal}.  In the same manner,
one can show that the nodal curves in $\overline{\mM}\OB$ are all fully 
disjoint from each other and from the smooth curves in~$\mM\OB$.

\begin{figure}
\includegraphics{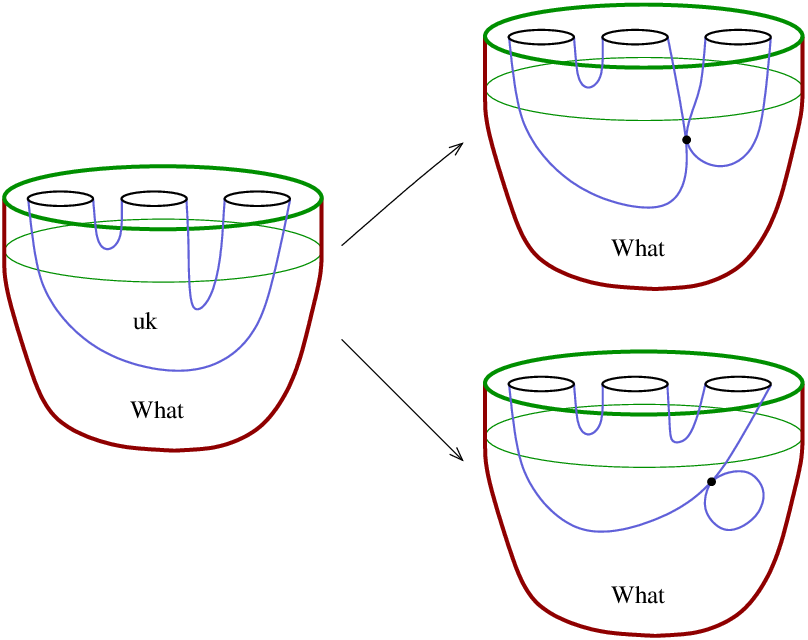}
\caption{\label{fig:nodal} Two possible degenerations of a sequence
$u_k \in \mM\OB$ to nodal curves in $\overline{\mM}\OB$.  The second
scenario includes a $J$-holomorphic exceptional sphere and is thus only 
possible if $(W,\omega)$ is not minimal.}
\end{figure}

Arguing as in the proof of Theorem~\ref{thm:McDuff}, let
$F \subset \widehat{W}$ denote the union of the images of the
curves in $\mathscr{B}$, which are finitely many properly embedded
surfaces.  Then let
$$
X := \left\{ p \in \widehat{W} \setminus F\ \Big|\ \text{$p$ 
is in the image of a curve in $\mM\OB$} \right\}.
$$
The lemmas proved above imply that $X$ is an open and closed
subset of $\widehat{W} \setminus F$, thus $X = \widehat{W}
\setminus F$, and we see that every point in $\widehat{W}$ is in
the image of a unique (possibly nodal) curve in $\overline{\mM}\OB$,
giving a surjective map
$$
\Pi : \widehat{W} \to \overline{\mM}\OB \setminus \p\overline{\mM}\OB.
$$
Since the $J_+$-holomorphic pages $[u_P] \in \mM^+\OB / \RR = 
\p\overline{\mM}\OB$ also foliate $M \setminus B$ under the projection
$\RR \times M \to M$, we can extend $\Pi$ to the natural compactification
$\overline{W} := \widehat{W} \cup (\{\infty\} \times M)$ as a surjective
map
$$
\Pi : \overline{W} \setminus B \to \overline{\mM}\OB,
$$
whose smooth fibres are the compact symplectically embedded surfaces with boundary
obtained as the images of maps $\bar{u} : \overline{\Sigma} \to \overline{W}$
for $u \in \mM\OB$, and we are treating $B$ as a submanifold of
$\{\infty\} \times M = \p\overline{W}$.  
There is still a small amount of work to be done in
identifying the above construction with something that one can regard as a
smooth symplectic Lefschetz fibration; details (in a more general setting) 
may be found in \cite{LisiVanhornWendl2}.

One detail in Proposition~\ref{prop:foliationM} has not yet been
verified: we've seen that $\mM\OB$ is an
oriented $2$-dimensional manifold, compactified by adding finitely many
interior points (the nodal curves) and the boundary
$\p\overline{\mM}\OB = \mM\OB^+ / \RR \cong S^1$, hence
$\overline{\mM}\OB$ is a compact oriented surface with one boundary
component, but we claim in fact that it is a \emph{disk}.  To see this, choose
a smooth loop $\gamma : S^1 \to M$ near a binding component that meets 
every page of the open book exactly once transversely.  Viewing $\gamma$ as
a loop in $\{\infty\} \times M = \p\overline{W}$, the loop
$$
\Pi \circ \gamma : S^1 \to \overline{\mM}\OB
$$
then parametrizes $\p\overline{\mM}\OB$.  Now, $\gamma$ is obviously not
contractible in $M \setminus B$, but we can easily assume it is contractible
in $\overline{W} \setminus B$: indeed, $\gamma$ can be chosen contractible
in~$M$, and then translating downward from $\{\infty\} \times M$ to a level
$\{s \} \times M \subset \widehat{W}$ for $s \in [0,\infty)$ gives a
contractible loop in~$\widehat{W}$.  Composing this contraction with~$\Pi$,
we conclude
$$
[\p\overline{\mM}\OB] = 0 \in \pi_1(\overline{\mM}\OB),
$$
hence $\overline{\mM}\OB \cong \DD$.


\appendix

\renewcommand{\thesection}{\Alph{chapter}.\arabic{section}}
\renewcommand{\thefigure}{\Alph{chapter}.\arabic{figure}}

\chapter{Properties of pseudoholomorphic curves}
\label{app:curves}

\minitoc
\vspace{12pt}

In\CUP{This material will be published by Cambridge University
Press as \textsl{Contact 3-Manifolds, Holomorphic Curves and Intersection Theory}
by Chris Wendl. This pre-publication version is
free to view and download for personal use only. 
Not for re-distribution, re-sale or use in derivative works. \copyright Chris Wendl, 2019.}
this appendix we will summarize (without proofs) the essential global analytical
results about pseudoholomorphic curves that are used in various places in
these lectures.  The first section covers results on closed holomorphic
curves that are needed in Lectures~1 and~2, and
\S\ref{app:punctured} then states the generalizations of these results
to punctured curves in completed symplectic cobordisms.  For more details
on each, we refer to \cite{McDuffSalamon:Jhol2} or \cite{Wendl:lectures}
for the closed case and \cite{Wendl:SFT} for the punctured case.

\section{The closed case}
\label{app:closed}

Given a closed symplectic manifold $(M,\omega)$ with a 
compatible\footnote{The vast majority of the results we will state here can 
also be generalized for almost complex structures that are \emph{tamed} by
$\omega$ but not necessarily compatible.  Such generalizations become much
less straightforward whenever asymptotic analysis is involved, thus in the
punctured case, it is best always to assume $J$ is compatible and not just
tame.} almost
complex structure $J$, we defined in \S\ref{sec:ruled} the moduli
space $\mM_g^A(M,J)$ of (equivalence classes up to parametrization of)
$J$-holomorphic curves with genus $g \ge 0$ homologous to $A \in H_2(M)$.
We shall now summarize the main analytical properties of
this space and use them to prove Lemma~\ref{lemma:M0}.

The \defin{virtual dimension}\index{moduli space!virtual dimension of}\index{virtual dimension}
of $\mM_g^A(M,J)$, also
sometimes called the \defin{index}\index{holomorphic curve!index of}\index{index!of a closed holomorphic curve}
of a curve $u \in \mM_g^A(M,J)$ and denoted
by $\ind(u) \in \ZZ$, is defined to be the integer
\begin{equation}
\label{eqn:dimension}
\virdim \mM_g^A(M,J) := (n-3)(2 - 2g) + 2 c_1(A),
\end{equation}
where $c_1(A)$ is shorthand for the evaluation of the first Chern class
$c_1(TM,J) \in H^2(M)$ on the homology class~$A$.  This definition of
$\virdim \mM_g^A(M,J)$ is justified by Theorem~\ref{thm:dimension} below.

Recall that closed $J$-holomorphic curves $u : (\Sigma,j) \to (M,J)$ are
always either \defin{simple}\index{holomorphic curve!simple|(}\index{somewhere injective|see {simple holomorphic curve}}\index{simple holomorphic curve|(}
or \defin{multiply covered},\index{holomorphic curve!multiply covered}\index{multiply covered holomorphic curve}
where the
latter means $u = v \circ \varphi$ for some closed $J$-holomorphic curve
$v : (\Sigma',j') \to (M,J)$ and holomorphic map
$\varphi : (\Sigma,j) \to (\Sigma',j')$ of degree $\deg(\varphi) > 1$.
By a slight abuse of terminology (see Remark~\ref{remark:critical}), we
refer to a point $z \in \Sigma$ in the domain of a $J$-holomorphic curve
$u : \Sigma \to M$ as a \defin{critical point}\index{critical points of a holomorphic curve}\index{holomorphic curve!critical points of}
if $du(z) = 0$; the alternative
is that $du(z) : T_z \Sigma \to T_{u(z)}M$ is injective, in which case we
call $z$ an \defin{immersed point}.\index{immersed points of a holomorphic curve}\index{holomorphic curve!immersed points of}
Combining general topological arguments with the local properties of 
$J$-holomorphic curves (e.g.~Theorem~\ref{thm:representation} in
Appendix~\ref{app:positivity}), one can show:
\begin{thm}
\label{thm:simple}
Every nonconstant, closed and connected $J$-holomorphic curve 
$u : (\Sigma,j) \to (M,J)$ has at most finitely many critical points.
Moreover, if $u$ is simple, then it also has at most finitely many
double points, hence it is embedded outside of some finite subset of~$\Sigma$.
\end{thm}
The following related result is sometimes referred to as the
\defin{unique continuation}\index{unique continuation}
principle:
\begin{thm}
\label{thm:uniqueContin}
If $u$ and $v$ are two closed $J$-holomorphic curves that are both
simple, then they are either equivalent up to parametrization or have at
most finitely many intersections.
\end{thm}

The \defin{automorphism group}\index{automorphism group!of a closed holomorphic curve}
of a triple $(\Sigma,j,u)$ representing
an element of $\mM_g^A(M,J)$ is defined as
$$
\Aut(\Sigma,j,u) = \left\{ \varphi : (\Sigma,j) \to (\Sigma,j)
\text{ biholomorphic} \ |\ u = u \circ \varphi \right\}.
$$
This group is always finite if $u : \Sigma \to M$ is not constant, and
Theorem~\ref{thm:simple} implies that it is trivial whenever $u$ is simple.\index{holomorphic curve!simple|)}\index{simple holomorphic curve|)}

The following result is dependent on a definition of the
term \defin{Fredholm regular},\index{Fredholm regular|(}\index{holomorphic curve!regular|see {Fredholm regular}}\index{transversality!for holomorphic curves|see {Fredholm regular}}
which is rather technical and therefore we will not give it---this is obviously
a terrible thing to do, 
but hopefully Theorems~\ref{thm:generic} and~\ref{thm:automatic} below
will make up for it.  The proofs of these results depend on the regularity theory of
elliptic PDEs; see \cite{McDuffSalamon:Jhol2} or \cite{Wendl:lectures} for details.

\begin{thm}
\label{thm:dimension}
The subset of $\mM_g^A(M,J)$ consisting of all curves that are Fredholm regular and
have trivial automorphism groups is \emph{open}, and moreover, it naturally
admits the structure of a smooth oriented finite-dimensional manifold,\index{moduli space!virtual dimension of}\index{virtual dimension}
with dimension equal to $\virdim \mM_g^A(M,J)$.
\end{thm}

Recall that for any topological space $X$, a subset $Y \subset X$ is said to be
\defin{comeager}\index{comeager}\index{Baire set|see {comeager}}\index{second category|see {comeager}}\index{Baire category theorem}
if it contains a countable intersection of open dense
sets.\footnote{It is common among symplectic topologists to say that comeager
subsets are ``Baire sets'' or are ``of second category,'' but this seems
to be slightly inconsistent with the standard usage of these terms in other
fields.}  If $X$ is a complete metric space, then the Baire category theorem
implies that every comeager subset of~$X$ is also dense.

\begin{thm}
\label{thm:generic}
Suppose $(M,\omega)$ is a closed symplectic manifold, $\uU \subset M$ is an
open subset, and $J_0$ is an $\omega$-compatible 
almost complex structure on~$M$.
Let $\jJ(\uU,J_0)$ denote the space of all smooth 
$\omega$-compatible almost complex
structures $J$ on~$M$ such that $J \equiv J_0$ on $M\setminus \uU$, and assign
to $\jJ(\uU,J_0)$ the natural $C^\infty$-topology.  Then
there exists comeager subset $\jJ^\reg(\uU,J_0) \subset \jJ(\uU,J_0)$ such that
for all $J \in \jJ^\reg(\uU,J_0)$, every simple curve $u \in \mM_g^A(M,J)$ that
intersects~$\uU$ is Fredholm regular.
\end{thm}

Results such as Theorem~\ref{thm:generic} that hold for all data in some
comeager subset are often said to hold for \defin{generic}\index{generic}
data, so one
can summarize the two theorems above by saying that the moduli space of
simple $J$-holomorphic curves is a smooth manifold of the ``correct''
dimension for ``generic''~$J$.  This fact is true even for moduli spaces
with $\virdim \mM_g^A(M,J) < 0$, implying that in such spaces, no Fredholm
regular curves exist:

\begin{cor}
\label{cor:nonnegative}
For generic $\omega$-compatible almost complex structures $J$ in a closed
symplectic manifold $(M,\omega)$, every simple $J$-holomorphic curve
$u$ satisfies $\ind(u) \ge 0$.
\end{cor}

Theorem~\ref{thm:generic} is a ``transversality'' result, i.e.~it 
follows from an infinite-dimensional version (the \emph{Sard-Smale
theorem}) of the standard fact
from differential topology that any two submanifolds 
intersect each other transversely after a generic perturbation.
Occasionally, one also needs transversality results for non-generic data.
Such results exist---they follow from the Riemann-Roch formula in certain
fortunate situations---but their utility is typically limited to 
dimension~$4$ and genus~$0$.  The following theorem of
Hofer-Lizan-Sikorav \cite{HoferLizanSikorav} is closely related to the
question of local foliations discussed in \S\ref{sec:foliations}.

\begin{thm}
\label{thm:automatic}
If $\dim M = 4$ and $J$ is any almost complex structure on~$M$,
then every immersed $J$-holomorphic curve $u \in \mM_g^A(M,J)$ with
$\ind(u) > 2g-2$ is Fredholm regular.\index{Fredholm regular|)}\index{transversality!automatic}\index{automatic transversality!for closed holomorphic curves}
\end{thm}

The moduli space $\mM_g^A(M,J)$ is not generally compact, but if
$M$ is closed and $J$ is compatible with a symplectic form $\omega$, then
it has a natural \emph{compactification}.  The \defin{energy}\index{energy!of a closed holomorphic curve}\index{holomorphic curve!energy of}
of a curve
$u \in \mM_g(M,J)$ can be defined as
$$
E(u) = \int_{\Sigma} u^*\omega
$$
for any parametrization $u : \Sigma \to M$; the taming condition implies
that $E(u) \ge 0$ for all $J$-holomorphic curves, with equality if and
only if the curve is constant.  Observe that $E(u)$ only depends on
$[u] \in H_2(M)$.

The moduli space of \defin{stable nodal $J$-holomorphic curves of
arithmetic genus~$g$}\index{holomorphic curve!nodal}\index{nodal holomorphic curve}\index{moduli space!compactification of}\index{compactification of moduli spaces}
homologous to $A \in H_2(M)$ is defined as
$$
\overline{\mM}_g^A(M,J) := \left\{ (S,j,u,\Delta) \right\} \Big/ \sim,
$$
where:
\begin{itemize}
\item $(S,j)$ is a (possibly disconnected) closed Riemann surface;
\item The set of \defin{nodes},\index{nodes of a holomorphic curve}
$\Delta \subset S$, is a finite unordered 
set of pairwise distinct points organized into pairs
$$
\Delta = \left\{ \{ \widehat{z}_1,\widecheck{z}_1\},\ldots,\{\widehat{z}_r,\widecheck{z}_r\} \right\}
$$
such that the singular surface
$$
\widehat{S} := S \Big/ \widehat{z}_j \sim \widecheck{z}_j \text{ for $j=1,\ldots,r$}
$$
is homeomorphic to a (possibly singular) fibre of some Lefschetz
fibration with regular fibres of genus~$g$;\index{arithmetic genus}\index{singular fiber of a Lefschetz fibration}\index{Lefschetz fibration!singular fiber of}
\item $u : (S,j) \to (M,J)$ is a pseudoholomorphic map with $[u] = A$
that descends to the quotient $\widehat{S} = S / \sim$ as a continuous
map $\widehat{S} \to M$;
\item Every connected component of $S \setminus\Delta$ on which $u$ is constant
has negative Euler characteristic;
\item We write $(S,j,u,\Delta) \sim (S',j',u',\Delta')$ if there is a 
biholomorphic map $\varphi : (S,j) \to (S',j')$ such that 
$u = u'\circ \varphi$ and $\varphi$ maps
pairs in $\Delta$ to pairs in~$\Delta'$.
\end{itemize}
The condition on the Euler characteristics of constant components is called\index{stability!of a nodal holomorphic curve}
\defin{stability}---its effect is to exclude certain ambiguities that would
otherwise cause non-uniqueness of limits for the natural topology on
$\overline{\mM}_g^A(M,J)$.
Assuming $A \ne 0$ so that elements of $\mM_g^A(M,J)$ are never constant,
there is a natural inclusion $\mM_g^A(M,J) \subset \overline{\mM}_g^A(M,J)$
defined by setting $\Delta := \emptyset$ for any $[(\Sigma,j,u)] \in
\mM_g^A(M,J)$.  We denote the union over all $A \in H_2(M)$ 
by~$\overline{\mM}_g(M,J)$.

\begin{thm}[\textsl{Gromov's compactness theorem}]
\label{thm:compactness}
For each $A \in H_2(M)$, $g \ge 0$ and each $\omega$-compatible almost complex\index{Gromov's compactness theorem|(}
structure $J$ on
a closed symplectic manifold $(M,\omega)$, $\overline{\mM}_g^A(M,J)$ admits a
natural topology as a compact metrizable space.  Moreover,
any sequence $u_k \in \mM_g(M,J)$ of curves satisfying a uniform energy bound
$E(u_k) \le C$
has a subsequence convergent to an element of~$\overline{\mM}_g(M,J)$.
\end{thm}
\begin{remark}
The second statement in the above theorem does not impose any direct
restriction on the homology classes $[u_k] \in H_2(M)$, but it 
\emph{implies} the existence of a subsequence with constant homology.
Observe that the required energy bound is automatic
if all $u_k$ represent a fixed homology class.
\end{remark}

\begin{figure}
\includegraphics{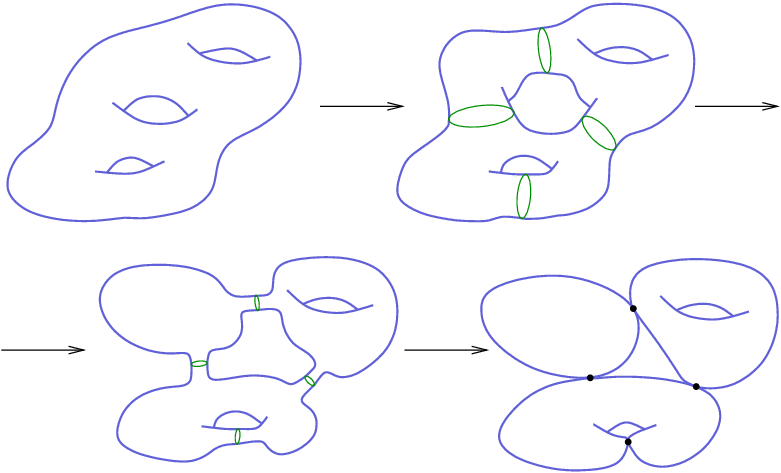}
\caption{\label{fig:compactness} A sequence of genus~$3$ holomorphic curves
degenerating to a nodal curve of arithmetic genus~$3$, with four nodes, 
two connected components of genus~$0$
and one of genus~$1$.}
\end{figure}

It will not be necessary for our purposes to give a complete definition of the
topology of $\overline{\mM}_g^A(M,J)$, but we can describe the convergence of a 
sequence of smooth curves $[(\Sigma_k,j_k,u_k)] \in \mM_g^A(M,J)$ to a
nodal curve $[(S,j,u,\Delta)] \in \overline{\mM}_g^A(M,J)$ as follows.  
(See Figure~\ref{fig:compactness} for an example.)
Let $S'$ denote the
compact topological $2$-manifold with boundary 
(Figure~\ref{fig:compactness6}, lower left)
obtained from $S$
by replacing each point $z \in \Delta \subset S$ with the circle
$$
C_z := T_z S \big/ \RR_+,
$$
where $\RR_+ := (0,\infty)$ acts on $T_z S$ by scalar multiplication.
The smooth structure of $S \setminus \Delta$ does not have an obviously
canonical extension over $S'$, but each boundary component $C_z \subset \p S'$
inherits from the conformal structure of $(S,j)$ a natural class of
preferred diffeomorphisms to $S^1$.  Now 
since the points in $\Delta$ come in pairs $\{ \widehat{z},\widecheck{z}\}$, we
can make a choice of preferred orientation-reversing diffeomorphisms
$C_{\widehat{z}} \to C_{\widecheck{z}}$ for each such pair and glue
corresponding boundary components of $S'$ to define a closed surface
(Figure~\ref{fig:compactness6}, lower right),
$$
\overline{S} := S' \big/ C_{\widehat{z}} \sim C_{\widecheck{z}}.
$$
This is naturally a closed topological $2$-manifold and it also carries a
smooth structure and a complex structure on $\overline{S} \setminus C
= S \setminus \Delta$, where
$$
C := \bigcup_{z \in \Delta} C_z \subset \overline{S}.
$$
Since $u(\widehat{z}) = u(\widecheck{z})$ for each pair
$\{ \widehat{z},\widecheck{z}\} \subset \Delta$, $u$ extends from $\overline{S}
\setminus C$ to a continuous map
$$
\bar{u} : \overline{S} \to M.
$$

\begin{figure}
\includegraphics[scale=0.9]{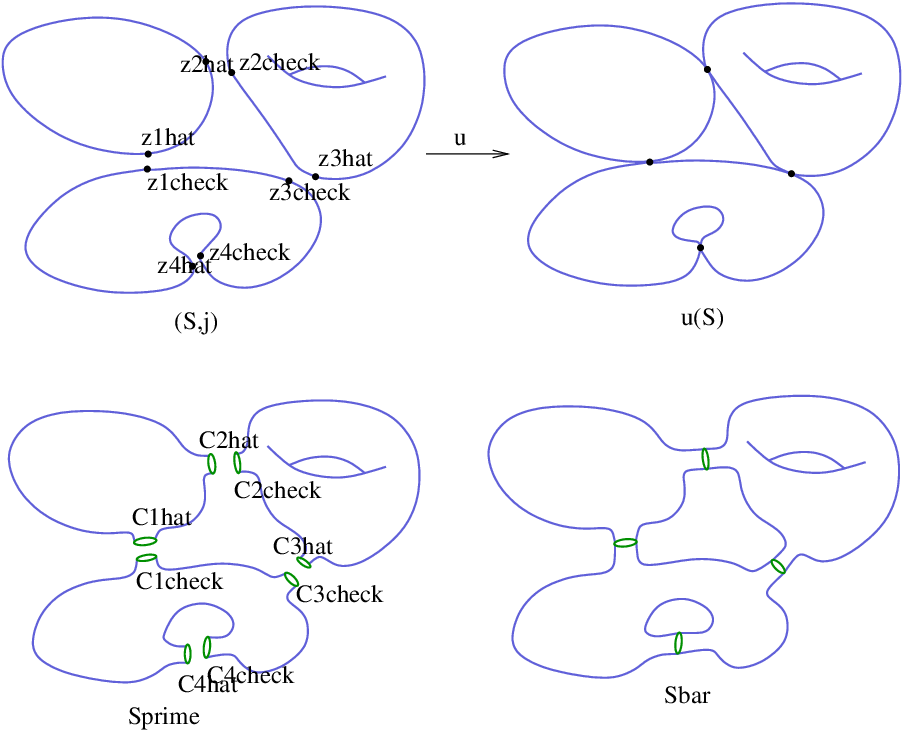}
\caption{\label{fig:compactness6} Four ways of viewing the nodal holomorphic 
of Figure~\ref{fig:compactness}.  At the upper left, we see the
disconnected Riemann surface $(S,j)$ with 
nodal pairs $\{\widehat{z}_i,\widecheck{z}_i\}$ for $i=1,2,3,4$.  To the right of 
this is a possible picture of the image of the nodal curve, with nodal pairs 
always mapped to identical points.  The bottom right shows the surface 
$S'$ with boundary, obtained from $S$ by replacing the points 
$\widehat{z}_i,\widecheck{z}_i$ with circles $\widehat{C}_i,\widecheck{C}_i$.  
Gluing these pairs of circles together gives the
closed connected surface $\overline{S}$ at the bottom right, whose genus is by
definition the arithmetic genus of the nodal curve.}
\end{figure}

The convergence $[(\Sigma_k,j_k,u_k)] \to [(S,j,u,\Delta)]$ can now be defined
to mean that for sufficiently large~$k$ there exist homeomorphisms
$$
\varphi_k : \overline{S} \to \Sigma_k
$$
whose restrictions to $\overline{S} \setminus C$ are
smooth and have smooth inverses, such that\index{Gromov's compactness theorem|)}
$$
\varphi_k^*j_k \to j \text{ in $C^\infty_{\text{loc}}(\overline{S} \setminus C)$,} \quad
u_k \circ \varphi_k \to u \text{ in $C^\infty_{\text{loc}}(\overline{S} \setminus C,M)$},
\quad\text{ and }\quad
u_k \circ \varphi_k \to \bar{u} \text{ in $C^0(\overline{S},M)$}.
$$

The analytical toolbox is now complete enough to 
fill in the following gap from \S\ref{sec:ruled}.

\begin{proof}[Proof of Lemma~\ref{lemma:M0}]
By construction, $\mM_0^{[S]}(M,J)$ contains an embedded curve
$u_S$, defined as the inclusion of~$S$.  The almost
complex structure $J$ cannot be assumed ``generic'' in the sense of
Theorem~\ref{thm:generic} since we chose it specifically to have the
property of preserving~$TS$.  We claim however that $u_S$ is nonetheless
Fredholm regular due to Theorem~\ref{thm:automatic}.  Indeed,
it has trivial normal bundle $N_S \to S^2$ since $[S] \cdot [S] = 0$, so the
natural splitting of complex vector bundles
$$
(u_S^*TM,J) = (TS^2,j) \oplus (N_S,J)
$$
implies
$$
c_1([S]) := c_1(u_S^*TM) = c_1(TS^2) + c_1(N_S) = \chi(S^2) + 0 = 2.
$$
Plugging $c_1([S]) = 2$ and $n=2$ into the index formula \eqref{eqn:dimension} 
now gives
$$
\ind(u_S) = -2 + 2 c_1([S]) = 2,
$$
hence $\virdim \mM_0^{[S]}(M,J) = 2$.  Since $u_S$ also is immersed, it
now satisfies the hypotheses of Theorem~\ref{thm:automatic}, so Fredholm
regularity follows.

To achieve smoothness near the rest of the simple curves in $\mM_0^{[S]}(M,J)$,
it suffices to choose a generic perturbation $J'$ of $J$ on the open
subset $M \setminus S$.  Indeed, for any such $J'$, assuming $J' = J$
along $S$ ensures that $u_S$ is also $J'$-holomorphic, so
unique continuation (Theorem~\ref{thm:uniqueContin}) then implies that
no other $J'$-holomorphic curve in $M$ can be contained entirely in~$S$ unless 
it is a multiple cover of~$u_S$.  In particular, $u_S$ itself is the only
such curve that is either simple or homologous to $[S]$.  It follows then by 
Theorem~\ref{thm:generic} that every other simple curve in
$\mM_0(M,J')$ is also Fredholm regular, so by Theorem~\ref{thm:dimension},
the subset $\mM_0^{[S],*}(M,J') \subset \mM_0^{[S]}(M,J')$ of simple curves
is an oriented $2$-dimensional manifold.
To simplify the notation, we relabel $J := J'$ from now on.

By Gromov's compactness theorem (Theorem~\ref{thm:compactness}), any
sequence $u_k \in \mM_0^{[S],*}(M,J)$ with no convergent subsequence in
$\mM_0^{[S]}(M,J)$ converges to a nodal curve with arithmetic genus~$0$.
The genus condition implies that its connected components are all
spheres, so we can regard the nodal curve simply as a finite set of 
$J$-holomorphic spheres $v_1,\ldots,v_N \in \mM_0(M,J)$
with $N \ge 2$, satisfying the condition
\begin{equation}
\label{eqn:homologySum}
[v_1] + \ldots + [v_N] = [S].
\end{equation}
These spheres cannot at first be assumed to be simple, but for each
$j=1,\ldots,N$, there is a simple curve $w_j \in \mM_{g_j}(M,J)$ and
an integer $k_j \in \NN$ such that $v_j$ is a $k_j$-fold cover of~$w_j$;
here we adopt the convention $w_j = v_j$ if $k_j=1$.  If $k_j > 1$,
then $v_j$ factors through a holomorphic map $S^2 \to \Sigma_{g_j}$
of degree~$k_j$, where $\Sigma_{g_j}$ is a closed connected 
surface with genus~$g_j$; but no such map exists if $g_j > 0$ since
the universal cover of $\Sigma_{g_j}$ is then contractible, implying
$\pi_2(\Sigma_{g_j}) = 0$, so we conclude that each $w_j$ has genus~$0$.
Now since all simple $J$-holomorphic curves in $M$ are Fredholm
regular, Corollary~\ref{cor:nonnegative} and the index formula
\eqref{eqn:dimension} give
$$
\ind(w_j) = -2 + 2 c_1([w_j]) \ge 0,
$$
hence $c_1([w_j]) \ge 1$.  Since $c_1([S]) = 2$, \eqref{eqn:homologySum}
now gives
\begin{equation}
\label{eqn:c1sum}
k_1 + \cdots + k_N \le k_1 c_1([w_1]) + \ldots + k_N c_1([w_N]) = 2,
\end{equation}
thus $N=2$ and $k_1 = k_2 = c_1([v_1]) = c_1([v_2]) = 1$.  We conclude
that the nodal curve has exactly two components, both simple, and
since $[v_1] + [v_2] = [S]$, they satisfy the uniform energy bound
\begin{equation}
\label{eqn:uniformBound}
E(v_j) = \langle [\omega] , [v_j] \rangle \le \langle [\omega],
[v_1] \rangle + \langle [\omega], [v_2] \rangle = \langle [\omega] , [S] \rangle
\end{equation}
for $j=1,2$.

Finally, we claim that the set of simple curves $v \in \mM_0(M,J)$ with
$c_1([v]) = 1$ is finite.  By Theorems~\ref{thm:dimension}
and~\ref{thm:generic}, this set is a $0$-dimensional manifold, i.e.~a discrete
set, so finiteness will follow if we can show that it is compact.
This follows essentially by a repeat of the argument above; note that
Gromov's compactness theorem is applicable due to the energy bound
\eqref{eqn:uniformBound}.  Now if a sequence of
such curves converges to a nodal curve with more than one component, then 
it produces an inequality like \eqref{eqn:c1sum} but with~$1$ on the
right hand side, which gives a contradiction.  The only remaining possibility
is that a sequence $v_k$ of curves with $c_1([v_k])=1$ converges to a
smooth but multiply covered curve $v$, but this is immediately excluded
since $c_1([v]) = 1$, so $[v]$ is a primitive homology class.
\end{proof}

\begin{remark}
\label{remark:higherGenus}
Let us see what goes wrong if one tries to prove an analogue of McDuff's 
theorem about ruled surfaces under the assumption of a symplectically
embedded
surface $S \subset (M,\omega)$ with $[S] \cdot [S] = 0$ and
genus $g > 0$.  One can still construct an 
embedded  $J$-holomorphic curve $u_S \in \mM_g^{[S]}(M,J)$, and since its
normal bundle $N_S \to S$ is necessarily trivial, the splitting 
$u^*TM = TS \oplus N_S$ now gives $c_1([S]) = \chi(S) = 2 - 2g$, so
\eqref{eqn:dimension} now gives
$$
\virdim \mM_g^{[S]}(M,J) = -(2 - 2g) + 2 c_1([S]) = 2 - 2g.
$$
This answer is desirable when $g=0$ because $2$ is the right number of
dimensions to foliate a $4$-manifold by holomorphic curves---but if
$g > 0$, one cannot hope to find a $2$-parameter family of holomorphic
curves homologous to~$[S]$, and in fact the curves should disappear
entirely after a generic perturbation if $g > 1$.  The failure of the proof
is thus attributable essentially to the Riemann-Roch formula, from which the
dimension formula \eqref{eqn:dimension} is derived.  It is more than a
failure of technology, however, as the theorem is false when $g > 0$.
\end{remark}

\section{Curves with punctures}
\label{app:punctured}

A general reference for the contents of this section is \cite{Wendl:SFT}.

Assume $(W,\omega)$ is a $2n$-dimensional symplectic cobordism with
$$
\p(W,\omega) = (-M_-,\xi_- = \ker\alpha_-) \sqcup
(M_+,\xi_+ = \ker\alpha_+),
$$
$(\widehat{W},\widehat{\omega})$ denotes its completion and
$J \in \jJ(\omega,\alpha_+,\alpha_-)$; see \S\ref{sec:punctured}
for the relevant definitions.  Consider an asymptotically
cylindrical $J$-holomorphic curve $u : (\dot{\Sigma} = \Sigma\setminus\Gamma,j) 
\to (\widehat{W},J)$
asymptotic to nondegenerate\footnote{Most of this discussion can also be
generalized to allow Reeb orbits in Morse-Bott nondegenerate families, though
the index formula becomes more complicated (see e.g.~\cites{Bourgeois:thesis,Wendl:automatic}).
In general, the linearized Cauchy-Riemann operator is not Fredholm (and thus the
moduli space is not well behaved) unless some nondegeneracy condition is imposed
on the ends.} Reeb orbits $\gamma_z$ in $M_\pm$ at its
positive/negative punctures $z \in \Gamma^\pm \subset \Sigma$.  The index
formula for $u$ can be expressed in terms of the Conley-Zehnder indices of 
its asymptotic orbits, but this requires a choice of normal trivialization
along each orbit.  We shall therefore fix an arbitrary choice of 
trivialization of $\gamma^*\xi_\pm$ for every Reeb orbit $\gamma$ in $M_\pm$,
and denote this choice collectively by~$\tau$.  The Conley-Zehnder index
of $\gamma$ relative to~$\tau$ will then be denoted by $\muCZ^\tau(\gamma)$,
and we write the \defin{index}\index{index!of a punctured holomorphic curve}\index{holomorphic curve!index of}
of $u$ as
\begin{equation}
\label{eqn:indexPunctured}
\ind(u) := (n-3) \chi(\dot{\Sigma}) + 2 c_1^\tau(u^*T\widehat{W}) +
\sum_{z \in \Gamma^+} \muCZ^\tau(\gamma_z) - \sum_{z \in \Gamma^-}
\muCZ^\tau(\gamma_z),
\end{equation}
where $c_1^\tau(u^*T\widehat{W})$ denotes the \emph{relative first Chern
number} of the complex vector bundle $(u^*T\widehat{W},J) \to \dot{\Sigma}$;
cf.~\S\ref{sec:puncturedFoliations}.  One can check that the sum on the
right hand side of \eqref{eqn:indexPunctured} does not depend on the
choice~$\tau$.  As with closed curves, $\ind(u)$ is also called the
\defin{virtual dimension}\index{virtual dimension}\index{moduli space!virtual dimension of}
of the connected component of
$\mM_g(\widehat{W},J)$ containing~$u$; one can show in fact that it only
depends on the Reeb orbits, the genus, and the relative homology
class of~$u$.

A curve $u : (\dot{\Sigma},j) \to (\widehat{W},J)$ in $\mM_g(\widehat{W},J)$ 
is \defin{multiply covered}\index{multiply covered holomorphic curve|(}\index{holomorphic curve!multiply covered|(}
whenever it can be written as $u = v \circ \varphi$ 
for some $v : (\dot{\Sigma}',j') \to (\widehat{W},J)$ in 
$\mM_{g'}(\widehat{W},J)$ and a holomorphic map
$$
\varphi : (\Sigma,j) \to (\Sigma',j') \quad\text{ with } \quad
\varphi(\dot{\Sigma}) = \dot{\Sigma}',
$$
having degree $\deg(\varphi) > 1$.  The \defin{automorphism
group}\index{automorphism group!of a punctured holomorphic curve}
$\Aut(\Sigma,j,\Gamma,u)$ can be defined similarly as the group
of biholomorphic maps $\varphi : (\Sigma,j) \to (\Sigma,j)$ that fix each
point in $\Gamma$ and satisfy $u = u \circ \varphi$.  If $u$ is not
multiply covered, it is called \defin{simple},\index{simple holomorphic curve|(}\index{holomorphic curve!simple|(}
and then it necessarily
has trivial automorphism group.
A straightforward combination of standard arguments for the closed
case (e.g.~\cite{McDuffSalamon:Jhol2}*{Prop.~2.5.1}) with Siefring's
relative asymptotic formula (Theorem~\ref{thm:asymptoticsRelative}) proves:

\begin{thm}
\label{thm:uniqueContinPunctured}
Theorems~\ref{thm:simple} and~\ref{thm:uniqueContin} also hold\index{multiply covered holomorphic curve|)}\index{holomorphic curve!multiply covered|)}
for asymptotically cylindrical $J$-holomorphic curves in~$\widehat{W}$.\index{holomorphic curve!simple|)}\index{simple holomorphic curve|)}
\end{thm}

A proof of the following generalization of Theorem~\ref{thm:dimension} is
sketched in \cite{Wendl:automatic}*{Theorem~0}:

\begin{thm}
\label{thm:theorem0}
The subset of $\mM_g(\widehat{W},J)$ consisting of all Fredholm regular
curves with trivial automorphism group is open and admits the structure
of a smooth finite-dimensional manifold, whose dimension near any
given curve $u \in \mM_g(\widehat{W},J)$ is $\ind(u)$.\index{Fredholm regular|(}
\end{thm}
We have intentionally omitted the word ``oriented'' from
Theorem~\ref{thm:theorem0}, as the question of orientations is somewhat
subtler here than in the closed case; see \cite{BourgeoisMohnke} or
\cite{Wendl:SFT}*{Chapter~11}.  Theorem~\ref{thm:generic} generalizes as follows:
\begin{thm}
\label{thm:genericPunctured}
Assume $\uU \subset \widehat{W}$ is an
open subset with compact closure, fix $J_0 \in \jJ(\omega,\alpha_+,\alpha_-)$, 
and define
$$
\jJ(\uU,J_0) := \left\{ J \in \jJ(\omega,\alpha_+,\alpha_-) \ \Big|\ \text{$J
\equiv J_0$ on $\widehat{W} \setminus \uU$} \right\}
$$
with its natural $C^\infty$-topology.
Then there exists a comeager subset $\jJ^\reg(\uU,J_0) \subset \jJ(\uU,J_0)$ such that
for all $J \in \jJ^\reg(\uU,J_0)$, every simple curve $u \in \mM_g(\widehat{W},J)$ that
intersects~$\uU$ is Fredholm regular.
\end{thm}
There is a further variation on the theme of ``generic transversality'' 
that only makes sense in the translation-invariant setting of a symplectization:
a perturbation of
a translation-invariant structure $J \in \jJ(\alpha)$ on $\RR \times M$ that is generic in the
sense of Theorem~\ref{thm:genericPunctured} cannot generally be
assumed translation-invariant, but Dragnev \cite{Dragnev} (see also the
appendix of \cite{Bourgeois:homotopy} or \cite{Wendl:SFT}*{Chapter~8}) proved:
\begin{thm}
\label{thm:Dragnev}
Suppose $(M,\xi = \ker\alpha)$ is a closed contact manifold, 
$\uU \subset M$ is an open subset and $J_0 \in \jJ(\alpha)$, and denote
$$
\jJ(\uU,J_0) := \left\{ J \in \jJ(\alpha)\ |\ \text{$J \equiv J_0$ on
$\RR \times (M \setminus \uU)$} \right\}.
$$
Then there exists a comeager subset $\jJ^\reg(\uU,J_0) \subset \jJ(\uU,J_0)$ such that
for all $J \in \jJ^\reg(\uU,J_0)$, every simple curve $u \in \mM_g(\RR\times M,J)$ that
intersects~$\RR \times \uU$ is Fredholm regular.
\end{thm}
Observe that in the symplectization, the translation-invariance of $J \in \jJ(\alpha)$
turns any curve $u \in \mM_g(\RR\times M,J)$ that isn't a cover of an orbit cylinder
into a $1$-parameter family, so Theorem~\ref{thm:Dragnev} implies a slightly
different analogue of Corollary~\ref{cor:nonnegative}:
\begin{cor}
\label{cor:positive}
For generic $J \in \jJ(\alpha)$ on the symplectization of a closed contact manifold
$(M,\xi=\ker\alpha)$, every simple $J$-holomorphic curve that is not an orbit cylinder
satisfies $\ind(u) \ge 1$.
\end{cor}

The punctured generalization of our previous ``automatic'' transversality result
(Theorem~\ref{thm:automatic}) is again valid only in dimension~$4$, and is
most easily stated in terms of the normal Chern number (see \S\ref{sec:puncturedFoliations}):

\begin{thm}[\cite{Wendl:automatic}*{Theorem~1}]
\label{thm:automaticPunctured}
If $\dim \widehat{W} = 4$ and $J \in \jJ(\omega,\alpha_+,\alpha_-)$,
then every immersed $J$-holomorphic curve $u \in \mM_g^A(\widehat{W},J)$ with
$\ind(u) > c_N(u)$ is Fredholm regular.\index{Fredholm regular|)}\index{automatic transversality!for punctured holomorphic curves}\index{transversality!automatic}
\end{thm}

Before stating the generalization of Gromov's compactness theorem,\index{SFT compactness theorem|(}
we must define the \defin{energy}\index{energy!of a punctured holomorphic curve}\index{holomorphic curve!energy of}\index{Hofer energy|see {energy of a punctured holomorphic curve}}
of a curve $u \in \mM_g(\widehat{W},J)$.
The obvious definition (by integrating $u^*\widehat{\omega}$) is not quite
the right one, as for instance orbit cylinders $u_\gamma(s,t) = (Ts,\gamma(t))$
in the symplectization $(\RR \times M,d(e^s\alpha))$ satisfy
$$
\int_{\RR\times S^1} u_\gamma^*d(e^s \alpha) = \infty.
$$
Instead, denote
$$
\tT := \left\{ \varphi : \RR \to (-1,1) \text{ smooth} 
\ \big|\ \text{$\varphi'(s) > 0$ for all
$s \in \RR$ and $\varphi(s) = s$ near $s=0$} \right\},
$$
and observe that for every $\varphi \in \tT$, the $2$-form on $\widehat{W}$
defined by
$$
\omega_\varphi := \begin{cases}
\omega & \text{ on $W$,}\\
d\left(e^{\varphi(s)}\alpha_+\right) & \text{ on $[0,\infty) \times M_+$}, \\
d\left(e^{\varphi(s)}\alpha_-\right) & \text{ on $(-\infty,0] \times M_-$}
\end{cases}
$$
is symplectic, and any $J \in \jJ(\omega,\alpha_+,\alpha_-)$ is
$\omega_\varphi$-compatible.  We then define
\begin{equation}
\label{eqn:HoferEnergy}
E(u) := \sup_{\varphi \in \tT} \int_{\dot{\Sigma}} u^*\omega_\varphi
\end{equation}
for any parametrization $u : \dot{\Sigma} \to \widehat{W}$ of a
curve in~$\mM_g(\widehat{W},J)$.

\begin{figure}
\includegraphics[width=5.5in]{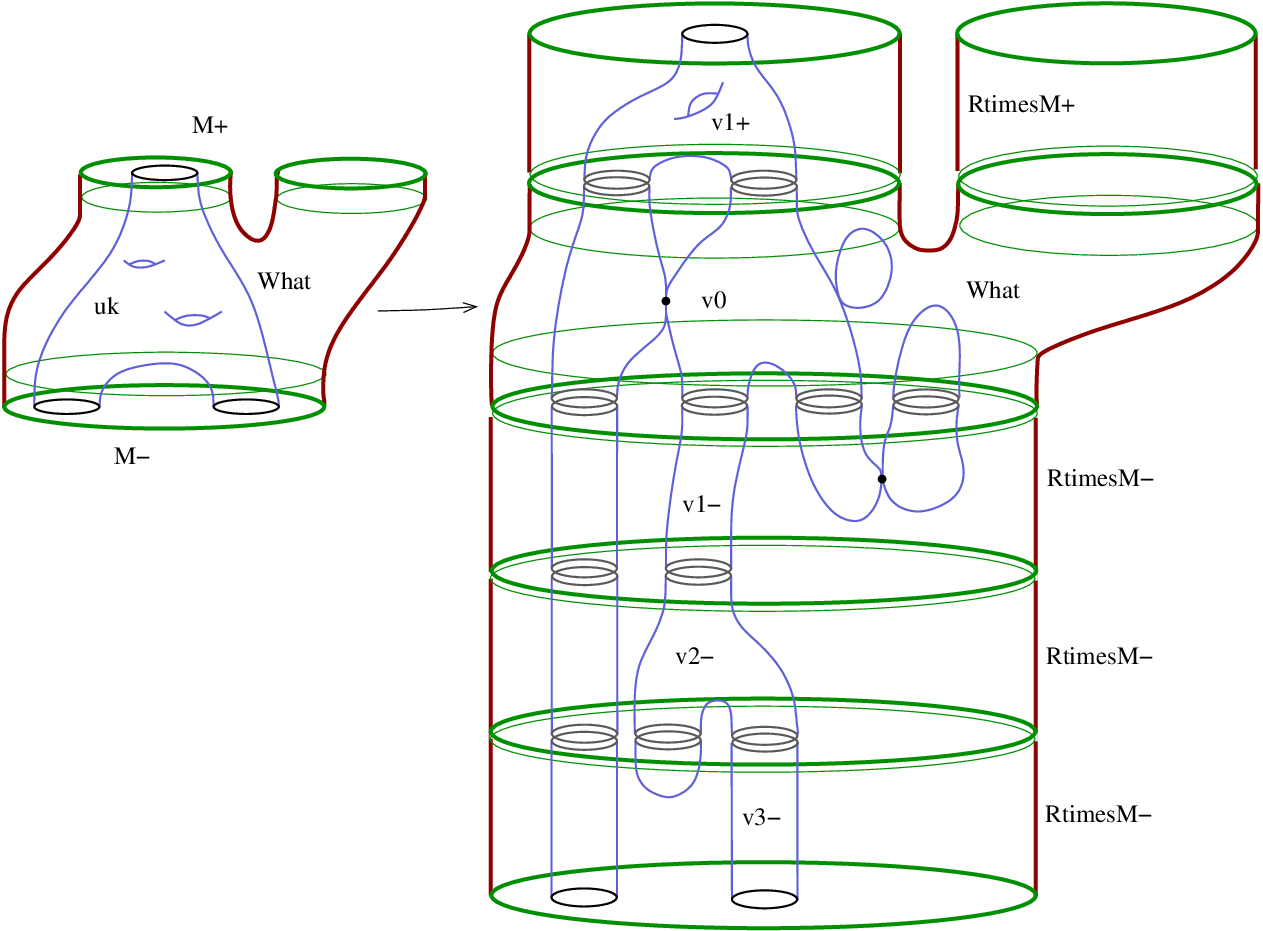}
\caption{\label{fig:SFT} Degeneration of a sequence $u_k$ of
punctured holomorphic curves 
with genus~$2$, one positive end and two negative ends in a symplectic cobordism.
The limiting holomorphic building $(v_1^+,v_0,v_1^-,v_2^-,v_3^-)$ in this example
has one upper level, a main level and three lower levels, each of which is a
(possibly disconnected) punctured nodal holomorphic curve.
The building has arithmetic genus~$2$ and the same numbers of positive and negative
ends as~$u_k$.}
\end{figure}

The natural compactification of $\mM_g(\widehat{W},J)$ is the space\index{compactification of moduli spaces}\index{moduli space!compactification of}\index{holomorphic building!stable}
$\overline{\mM}_g(\widehat{W},J)$ of \defin{stable $J$-holomorphic
buildings} 
$$
(v_{N_+}^+,\ldots,v_{1}^+,v_0,v_1^-,\ldots,v_{N_-}^-),
$$
which have $N_+ \ge 0$ \defin{upper levels}, $N_- \ge 0$ \defin{lower levels}
and exactly one \defin{main level}.\index{holomorphic building!levels}
Each of the levels is a
(possibly disconnected) asymptotically cylindrical nodal curve
that is stable in the sense defined in \S\ref{app:closed}, where
\begin{itemize}
\item $v_i^+$ for $i = 1,\ldots,N_+$ live in $\RR \times M_+$ and are
$J_+$-holomorphic, with 
$$
J_+ := J|_{[0,\infty) \times M_+} \in \jJ(\alpha_+) ;
$$
\item $v_0$ lives in $\widehat{W}$ and is $J$-holomorphic;
\item $v_i^-$ for $i = 1,\ldots,N_-$ live in $\RR \times M_-$ and are
$J_-$-holomorphic, with 
$$
J_- := J|_{(-\infty,0] \times M_-} \in \jJ(\alpha_-).
$$
\end{itemize}
The levels also connect to each other, meaning that the data of a building
includes a bijection between the positive ends of each level and the negative
ends of the level above it such that matching ends are asymptotic to the same
Reeb orbit---orbits that appear in this way are not considered asymptotic orbits
of the building itself, but are sometimes called \defin{breaking orbits}
(see Figure~\ref{fig:SFT}).\index{breaking orbit}\index{Reeb orbit!breaking}
The \defin{arithmetic genus}\index{arithmetic genus} $g \ge 0$ can be characterized
by the following condition: if $\widehat{S}$ denotes the space obtained 
from the domains of all the levels by filling in all punctures and then 
identifying any two nodal points that 
belong to the same node and any two punctures between levels that are
matched by the aforementioned bijection, then
$\widehat{S}$ is homeomorphic to a (possibly singular) fiber of some
Lefschetz fibration with closed regular fibers of genus~$g$.\index{singular fiber of a Lefschetz fibration}\index{Lefschetz fibration!singular fiber of}
Equivalence of holomorphic buildings is defined via the obvious notion of
biholomorphic equivalence (preserving nodes and matching punctures), with
the additional feature that upper and lower levels may be translated freely,
i.e.~two levels that are identical after an $\RR$-translation of an upper or
lower level are considered equivalent.
Finally, the stability 
condition is enhanced with the stipulation that none of the 
levels $v_i^\pm$ may consist exclusively of orbit cylinders without 
any nodes; this is necessary in order to make sure that the natural topology
of $\overline{\mM}_g(\widehat{W},J)$ is Hausdorff.\index{orbit cylinder}\index{stability!of a holomorphic building}

The natural inclusion
$$
\mM_g(\widehat{W},J) \hookrightarrow \overline{\mM}_g(\widehat{W},J)
$$
regards any smooth curve $u \in \mM_g(\widehat{W},J)$ as a building
that has no upper or lower levels and no nodes.

\begin{thm}
\label{thm:SFTcompactness}
For every $g \ge 0$ and every $J \in \jJ(\omega,\alpha_+,\alpha_-)$,
$\overline{\mM}_g(\widehat{W},J)$ admits a natural topology as a metrizable
space, and its connected components are compact.  Moreover,
any sequence $u_k \in \mM_g(\widehat{W},J)$ of curves satisfying a
uniform energy bound $E(u_k) \le C$ in the sense of \eqref{eqn:HoferEnergy}
has a subsequence convergent to an element of~$\overline{\mM}_g(\widehat{W},J)$.
\end{thm}

A small modification is appropriate in the case where $(\widehat{W},J)$
is the completion of a \emph{trivial} symplectic cobordism, i.e.~an 
$\RR$-invariant symplectization $\RR \times M$.  In this case the levels are
still ordered, but there is no distinguished main level, nor a distinction
between ``upper'' and ``lower'' levels, and the notion of equivalence allows
$\RR$-translations in all levels---the latter means in particular that
$\overline{\mM}_g(\RR\times M,J)$ is not a compactification of
$\mM_g(\RR \times M,J)$, but rather of $\mM_g(\RR \times M,J) \big/ \RR$.
For full details on these matters, including a precise definition of the
notion of convergence to a holomorphic building, we refer to
\cite{SFTcompactness}.\index{SFT compactness theorem|)}

\chapter{Local Positivity of intersections}
\label{app:positivity}

\minitoc
\vspace{12pt}

In\CUP{This material will be published by Cambridge University
Press as \textsl{Contact 3-Manifolds, Holomorphic Curves and Intersection Theory}
by Chris Wendl. This pre-publication version is
free to view and download for personal use only. 
Not for re-distribution, re-sale or use in derivative works. \copyright Chris Wendl, 2019.}
this appendix we explain the local results in the background of the
standard theorems of \S\ref{sec:adjunction} on positivity of intersections
and the adjunction formula.  Readers wishing to understand the geometric
picture without worrying about the analytical details may read the statement
of Theorem~\ref{thm:representation} in \S\ref{sec:representation} and then
skip ahead to
\S\ref{sec:intSec}, which proves positivity of intersections using
the local representation formula of Theorem~\ref{thm:representation}
as a black box.  The main tool in proving the latter is the similarity
principle, which is explained (along with the necessary background on
elliptic regularity) in \S\ref{sec:similarity}.

Since all important results in this appendix are local, we will mostly
discuss functions defined on the domains
$$
\DD := \left\{ z \in \CC\ \big|\ |z| \le 1 \right\} \quad\text{ and }\quad
\DD_\rho := \left\{ z \in \CC\ \big|\ |z| \le \rho \right\}
$$
for $\rho > 0$.

\section{Regularity and the similarity principle}
\label{sec:similarity}

The similarity principle can be thought of as a linearized version of
positivity of intersections: it gives a local description of
solutions to linear Cauchy-Riemann type equations near their zeroes,
proving in particular that they qualitatively resemble complex-analytic
functions.  The proof given in this section is more or less self-contained---it 
requires some understanding of the theory of distributions and Sobolev spaces, but avoids using the
harder aspects of elliptic regularity theory such as the Calder\'on-Zygmund
inequality.  It is based in large part on arguments that were explained
to the author by Jean-Claude Sikorav.

\subsection{Linear Cauchy-Riemann type operators}
\label{sec:CRoperators}

Linear Cauchy-Riemann equations on vector bundles arise naturally from 
infinitessimal perturbations of $J$-holomorphic curves.

\begin{defn}
Suppose $(\Sigma,j)$ is a Riemann surface, $E \to \Sigma$ is a smooth
complex vector bundle, and $F \to \Sigma$ denotes the complex vector bundle
$$
F := \overline{\Hom}_\CC(T\Sigma,E)
$$
whose sections are the complex-antilinear bundle maps $T\Sigma \to E$.
A (smooth) \defin{linear Cauchy-Riemann type
operator}\index{Cauchy-Riemann type operator}
is a first-order real-linear partial differential operator
$\mathbf{D} : \Gamma(E) \to \Gamma(F)$ that satisfies the Leibniz rule
$$
\mathbf{D}(f\eta) = (\dbar f) \eta + f \mathbf{D}\eta \quad\text{ for all }\quad
\eta \in \Gamma(E), \ f \in C^\infty(\Sigma,\RR),
$$
where $\dbar f \in \Omega^{0,1}(T\Sigma)$ denotes the complex-valued
$1$-form $df + i\, df \circ j$.
\end{defn}

\begin{remark}
\label{remark:notComplex}
If $\mathbf{D} : \Gamma(E) \to \Gamma(F)$ in the above definition is also\index{Cauchy-Riemann type operator!complex linear}
complex linear, then it satisfies a complex version of the Leibniz rule, namely
$$
\mathbf{D}(f\eta) = (\dbar f) \eta + f \mathbf{D}\eta \quad\text{ for all }\quad
\eta \in \Gamma(E), \ f \in C^\infty(\Sigma,\CC).
$$
It is important however to allow the possibility that $\mathbf{D} : \Gamma(E) \to \Gamma(F)$
is only real and not complex linear, even though $E$ and $F$ both carry
complex structures.  Unless one restricts attention to complex manifolds with
\emph{integrable} complex structures, most of the linearized Cauchy-Riemann operators that
arise in the context of $J$-holomorphic curve theory are not complex linear.
\end{remark}

\begin{remark}
\label{remark:CRaffine}
It is easy to check that if $\mathbf{D} : \Gamma(E) \to \Gamma(F)$ is a
Cauchy-Riemann type operator and $A : E \to F$ is a smooth real-linear
bundle map, then $\mathbf{D} + A$ is also a Cauchy-Riemann type operator.
Moreover, for any two Cauchy-Riemann type operators $\mathbf{D}$ and
$\mathbf{D}'$ on~$E$, the map $\mathbf{D}' - \mathbf{D} : \Gamma(E) \to \Gamma(F)$
is $C^\infty$-linear and thus arises from a smooth bundle map $A : E \to F$,
meaning $\mathbf{D}' = \mathbf{D} + A$.  This proves that the space of all
Cauchy-Riemann type operators on $E$ is an affine space over $\Gamma(\Hom_\RR(E,F))$.
\end{remark}

Given an open subset $\uU \subset \Sigma$ with a holomorphic coordinate
$z = s + it : \uU \to \CC$ identifying $(\uU,j)$ with $(\DD,i)$ and a complex trivialization of
$E|_\uU$, there is a naturally induced trivialization of $F|_\uU$ such that
if $\eta \in \Gamma(E|_\uU)$ is represented by the function $f : \DD \to \CC^n$,
then the same function also represents the section
$\xi \in \Gamma(F|_\uU)$ given by $\xi(X) = d\bar{z}(X)\, \eta(p)$ for
$p \in \uU$ and $X \in T_p\Sigma$.  These choices identify
the spaces of sections of $E$ and $F$ over $\uU$ with $C^\infty(\DD,\CC^n)$
such that the map
$$
\p_s + i\p_t : C^\infty(\DD,\CC^n) \to C^\infty(\DD,\CC^n)
$$
represents a linear Cauchy-Riemann type operator on~$E|_\uU$.  It follows 
via Remark~\ref{remark:CRaffine} that
\emph{every} Cauchy-Riemann type operator $\mathbf{D}: \Gamma(E|_\uU) \to \Gamma(F|_\uU)$
is in this way identified with a map of the form
\begin{equation}
\label{eqn:localCR}
\p_s + i \p_t + A : C^\infty(\DD,\CC^n) \to C^\infty(\DD,\CC^n)
\end{equation}
for some smooth function $A : \DD \to \End_\RR(\CC^n)$.  With this local
picture understood, it will sometimes also be useful to consider Cauchy-Riemann
type operators on complex vector bundles that are not equipped with a
smooth structure, e.g.~pullbacks of smooth bundles along non-smooth
(but differentiable) maps.  In general, one says that a vector bundle
is \defin{of class $C^k$} if it is equipped with an atlas of local trivializations
whose transition maps are all of class~$C^k$.  One can then speak of sections
of class $C^m$ for any $m \le k$ but not for $m > k$; the former notion makes
sense due to the fact that for $m \le k$, the product of a $C^k$-smooth
function with a $C^m$-smooth function is also of class~$C^m$.  In the following,
we shall allow non-smooth vector bundles $E \to \Sigma$ but continue to 
assume that the base is a \emph{smooth} Riemann surface, i.e.~the almost
complex structure $j$ on $\Sigma$ is smooth, so that holomorphic local
coordinate charts on $\Sigma$ are automatically also smooth.  For this reason,
$F := \overline{\Hom}_\CC(T\Sigma,E)$ always inherits from $E$ and $\Sigma$ an
atlas of local trivializations with the same regularity as~$E$.  If
$E$ is of class~$C^k$, then the notion of a differential operator from $E$
to $F$ of order $r \in \NN$ makes sense as long as $r \le k$.

\begin{defn}
Suppose $(\Sigma,j)$ is a Riemann surface, $E \to \Sigma$ is a complex
vector bundle of class $C^k$ for some $k \in \NN \cup \{\infty\}$,
$F = \overline{\Hom}_\CC(T\Sigma,E)$, and $m \le k-1$ is a nonnegative integer.  
A \defin{linear Cauchy-Riemann type operator of class $C^m$}\index{Cauchy-Riemann type operator!of class $C^m$}
is a first-order real-linear partial differential operator $\mathbf{D}$ from $E$ to $F$
such that under arbitrary choices of local holomorphic coordinates and
trivializations as described in the previous paragraph, $\mathbf{D}$ locally
takes the form $\p_s + i\p_t + A$ for some $A \in C^m(\DD,\End_\RR(\CC^n))$.
\end{defn}
\begin{remark}
\label{remark:onlyk-1}
The condition $m \le k-1$ is required in the above definition since
the transformation of the zeroth-order term in a Cauchy-Riemann type operator
under a transition map depends in general on the first derivative of the
transition map, i.e.~if the latter is only of class $C^k$, then the
condition $A \in C^{k-1}$ is coordinate-invariant but $A \in C^k$ would not be.
The same remark applies to connections on a bundle of class~$C^k$.
\end{remark}
\begin{remark}
\label{remark:CRLp}
For functions of Sobolev class $W^{k,p}$, there is also a well-defined\index{Sobolev spaces}
continuous product pairing $C^k \times W^{k,p} \to W^{k,p}$ due to the
fact that products of continuous functions with $L^p$-functions are also in~$L^p$.
As a consequence, one can also speak of Cauchy-Riemann
type operators of class $W^{k,p}$ whenever the bundle is of class~$C^{k+1}$.
\end{remark}

\begin{example}
If $E \to \Sigma$ is endowed with a holomorphic vector bundle structure, then\index{holomorphic vector bundle}\index{Cauchy-Riemann type operator!complex linear}
it carries a canonical (complex-)linear Cauchy-Riemann type operator
$\mathbf{D} : \Gamma(E) \to \Gamma(F)$ such that the local holomorphic functions
$\eta \in \Gamma(E|_\uU)$ on open sets $\uU \subset \Sigma$ are precisely those
which satisfy $\mathbf{D}\eta = 0$.  This operator takes the form
$\p_s + i\p_t$ with respect to any choice of local holomorphic coordinates
and holomorphic trivialization, and the holomorphicity of the transition maps
guarantees that this definition does not depend on any choices.
\end{example}

\begin{example}
For any connection $\nabla$ on $E \to \Sigma$, $\mathbf{D}\eta := \nabla \eta
+ i\, \nabla \eta \circ j$ defines a linear Cauchy-Riemann type operator.
\end{example}

\begin{example}
If $u : (\Sigma,j) \to (M,J)$ is a $J$-holomorphic curve, then linearizing
the nonlinear operator $\dbar_J(u) := du + J \circ du \circ j$ along a smooth
family of maps $\{u_\sigma : \Sigma \to M\}_{\sigma \in (-\epsilon,\epsilon)}$ with $u_0 = u$ and $\eta := 
\left.\p_\sigma u_\sigma\right|_{\sigma=0} \in \Gamma(u^*TM)$ gives rise to a
linear Cauchy-Riemann type operator of the form
\begin{equation*}
\begin{split}
\Gamma(u^*TM) &\stackrel{\mathbf{D}_u}{\to} \Gamma(\overline{\Hom}_\CC(T\Sigma,u^*TM)), \\
\eta &\mapsto \nabla \eta + J(u) \circ \nabla\eta \circ j + (\nabla_\eta J) \circ Tu \circ j,
\end{split}
\end{equation*}
where $\nabla$ is an arbitrary choice of symmetric connection on~$M$.
\end{example}

\subsection{Elliptic regularity}

In the following we will consider $\CC^n$-valued functions of one
complex variable $z = s + it$, for which we denote the standard local models
of Cauchy-Riemann and anti-Cauchy-Riemann type operators by
$$
\dbar := \p_s + i\p_t, \qquad \p := \p_s - i \p_t.
$$

\begin{notation}[Sobolev spaces]
In this appendix, the Sobolev space of functions $f : \DD \to \CC^n$ 
admitting weak derivatives\index{Sobolev spaces}
of class $L^p$ up to order$k \ge 0$ is denoted by~$W^{k,p}(\DD)$,
and for the case $p=2$ we abbreviate the Hilbert spaces
$H^k(\DD) := W^{k,2}(\DD)$.
The larger vector spaces $W^{k,p}_\loc(\DD)$ and $H^k_\loc(\DD)$ consist of all 
functions on $\DD$ whose restrictions to compact subsets of the interior
of~$\DD$ are of class $W^{k,p}$ or $H^k$ respectively.  An important special
case of this is $L^1_\loc(\DD) = W^{0,1}_\loc(\DD)$, the space of all 
locally integrable functions on~$\DD$.
We write the space of smooth compactly supported functions on the interior
of $\DD$ as~$C_0^\infty(\DD)$.
\end{notation}

Since $\dbar$ and $\p$ are first-order differential operators with constant
coefficients, they define bounded linear maps
$$
\dbar, \p : W^{k,p}(\DD) \to W^{k-1,p}(\DD)
$$
for each $k \in \NN$ and $p \in [1,\infty]$.
We will need to use the ``easy'' ($p=2$) case of the following nontrivial 
fact from elliptic regularity theory.

\begin{prop}
\label{prop:regularity}
For each $p \in (1,\infty)$, the operator $\dbar : W^{1,p}(\DD) \to L^p(\DD)$
is surjective and admits a bounded right inverse $T : L^p(\DD) \to W^{1,p}(\DD)$.
\end{prop}
\begin{proof}[Sketch of the proof for $p=2$]
The locally integrable function $K : \CC \to \CC$ defined almost everywhere
by
$$
K(z) := \frac{1}{2\pi z}
$$
is a fundamental solution for the equation $\dbar u = f$, meaning it satisfies
$\dbar K = \delta$ in the sense of distributions, so in particular,
one can use convolution to associate to any $f \in C_0^\infty(\CC)$ a
function $u := K * f \in C^\infty(\CC)$ satisfying $\dbar u = f$.  Since $C_0^\infty(\DD) \subset
C_0^\infty(\CC)$ is dense in $L^p(\DD)$, the desired right inverse
$T : L^p(\DD) \to W^{1,p}(\DD)$ can
be defined as the unique bounded linear extension of $C_0^\infty(\DD) \to C^\infty(\DD) :
f \mapsto \left.(K * f)\right|_\DD$
after establishing an estimate of the form
$$
\| K * f \|_{W^{1,p}(\DD)} \le c \| f \|_{L^p} \quad\text{ for all }\quad
f \in C_0^\infty(\DD).
$$
This is equivalent to three estimates,
\begin{equation}
\label{eqn:3estimates}
\| K * f\|_{L^p(\DD)} \le c \| f\|_{L^p}, \quad
\| \p(K * f) \|_{L^p(\DD)} \le c \| f \|_{L^p}, \quad
\| \dbar(K * f) \|_{L^p(\DD)} \le c \| f \|_{L^p},
\end{equation}
each again for $f \in C_0^\infty(\DD)$.  The third of these is immediate
since $\dbar (K * f) = f$.  The first estimate is a minor variation on
the standard \emph{Young's inequality} for convolutions (see e.g.~\cite{LiebLoss}*{\S 4.2}),
and admits a similar proof based on the H\"older inequality and Fubini's
theorem---the crucial assumptions here are only that
$K$ is locally integrable and $\DD \subset \CC$ is bounded.  The hard part
in general is the second estimate, though in the case $p=2$,
a straightforward argument is possible using the Fourier transform.

The idea is to interpret both sides of the equation $\dbar K = \delta$ as 
tempered distributions on $\CC$, which then have well-defined Fourier transforms
in the sense of distributions.  Expressing these Fourier transforms as
functions of a variable $\zeta \in \CC$, the Fourier transform $\widehat{K}(\zeta)$
of $K(z)$ gets multiplied by $2\pi i \zeta$ to produce the Fourier transform of
$\dbar K(z)$, so $\dbar K = \delta$ implies
\begin{equation}
\label{eqn:FourierK}
2\pi i \zeta \widehat{K}(\zeta) = \widehat{\delta}(\zeta) = 1.
\end{equation}
If $f \in C_0^\infty(\CC)$ and we define another function on $\CC$ by
$u := K * f$, then $u$ also defines a tempered distribution, whose Fourier
transform $\widehat{u}$ then satisfies
$$
\widehat{u} = \widehat{K * f} = \widehat{K} \widehat{f},
$$
so by \eqref{eqn:FourierK} we have $2\pi i \zeta \widehat{u}(\zeta) = \widehat{f}(\zeta)$.
Denoting the Lebesgue measure on $\CC$ for functions of $\zeta \in \CC$ by
$d\mu(\zeta)$, Plancherel's theorem now implies
\begin{equation*}
\begin{split}
\| \p(K * f) \|_{L^2(\DD)}^2 &\le \| \p u \|_{L^2(\CC)}^2 = 
\int_\CC \big| \widehat{\p u}(\zeta) \big|^2 \, d\mu(\zeta) =
\int_\CC \left| 2\pi i \widebar{\zeta} \widehat{u}(\zeta) \right|^2 \, d\mu(\zeta) \\
&= \int_\CC \left| \frac{\widebar{\zeta}}{\zeta} 2\pi i \zeta \widehat{u}(\zeta) \right|^2 \, d\mu(\zeta)
= \int_\CC \big| \widehat{f}(\zeta) \big|^2 \, d\mu(\zeta) =
\big\| \widehat{f} \big\|_{L^2(\CC)}^2 = \| f \|_{L^2(\CC)}^2,
\end{split}
\end{equation*}
and the last expression is the same as $\| f \|_{L^2(\DD)}^2$ if we assume
$f \in C_0^\infty(\DD)$, hence the remaining estimate is proven.

The case $p \ne 2$ requires totally different arguments, which begin by
writing $\p (K * f) = \p K * f$ as a principal value integral
$$
\p (K * f)(z) = -\frac{1}{\pi} \lim_{\epsilon \to 0^+} \int_{|\zeta - z| \ge \epsilon}
\frac{f(\zeta)}{(z - \zeta)^2} \, d\mu(\zeta),
$$
in which the right hand side can be interpreted as the convolution of 
a distribution $\p K$ with a smooth function~$f$.
The limit in this expression is necessary because in contrast to $1 / z$, the function
$1 / z^2$ that arises by differentiating $K(z)$ is not locally integrable on~$\CC$,
and for this reason, simple convolution inequalities do not apply.
Estimates in $L^p$ for transformations given by singular integrals of this type are the
subject of a much harder analytical result, the Calder\'on-Zygmund inequality.
Details on this and the rest of the argument sketched above may be found
in \cite{Wendl:lectures}*{Chapter~2}; we shall not present them here
since the $p \ne 2$ case, while important for the general theory of
pseudoholomorphic curves, is not needed in our discussion of intersection
theory.
\end{proof}

\begin{remark}
A closely related result is the existence of an estimate
\begin{equation}
\label{eqn:basicEstimate}
\| u \|_{W^{1,p}} \le c \| \dbar u \|_{L^p} \quad \text{ for all }\quad
u \in C_0^\infty(\DD).
\end{equation}
This can be derived from Proposition~\ref{prop:regularity} using a bit of 
extra knowledge about the fundamental solution $K(z) = 1 / 2\pi z$, but
in the case $p=2$ it also admits the following simple proof
borrowed from Sikorav \cite{Sikorav}.  The bound on $\| u \|_{W^{1,p}}$ is again 
equivalent to three bounds, namely on $\| u \|_{L^p}$, $\| \p u \|_{L^p}$
and $| \dbar u \|_{L^p}$, where the third is immediate.  Since $u$ is assumed
to be smooth with compact support, the first bound follows from a standard
Sobolev estimate, the Poincar\'e inequality (see e.g.~\cite{AdamsFournier}*{\S 6.30}).
To achieve the second bound, it is convenient to write $z = s + it \in \CC$ and
consider the complex partial derivative operators
$$
\p_z := \frac{\p}{\p z} = \frac{1}{2} \p, \qquad 
\p_{\bar{z}} := \frac{\p}{\p \bar{z}} = \frac{1}{2} \dbar
$$
along with the corresponding complex-valued $1$-forms
$$
dz = ds + i\, dt, \qquad d\bar{z} = ds - i\, dt.
$$
For any smooth compactly supported function $u : \CC \to \CC$, we can now
write
$$
du = \p_z u\, dz + \p_{\bar{z}} u \, d\bar{z},
$$
and the complex-valued
$1$-form $u \, d\bar{u}$ has compact support in $\CC$, so applying Stokes'
theorem to $d ( u\, d\bar{u}) = du \wedge d\bar{u}$ on a sufficiently large
disk $\DD_R \subset \CC$ gives
\begin{equation*}
\begin{split}
0 &= \int_{\p\DD_R} u \, d\bar{u} = \int_{\DD_R} du \wedge d\bar{u} =
\int_{\DD_R} (\p_z u\, dz + \p_{\bar{z}} u \, d\bar{z} ) \wedge
( \p_z \bar{u}\, dz + \p_{\bar{z}} \bar{u} \, d\bar{z}) \\
&= \frac{1}{4} \int_{\DD_R} \left( |\p u|^2 - |\dbar u|^2 \right)\, dz \wedge d\bar{z},
\end{split}
\end{equation*}
proving $\| \p u \|_{L^2} = \| \dbar u \|_{L^2}$.

Note that by applying \eqref{eqn:basicEstimate} to derivatives $\p^\alpha u$
with arbitrary multi-indices $\alpha$ and using the fact that $\p^\alpha$
commutes with~$\dbar$, one obtains the easy generalization
$$
\| u \|_{W^{k,p}} \le c \| \dbar u \|_{W^{k-1,p}} \quad \text{ for all }\quad
u \in C_0^\infty(\DD)
$$
for every $k \in \NN$.   By density, this extends to
\begin{equation}
\label{eqn:basicEstimate2}
\| u \|_{W^{k,p}} \le c \| \dbar u \|_{W^{k-1,p}} \quad \text{ for all }\quad
u \in W^{k,p}_0(\DD),
\end{equation}
where $W^{k,p}_0(\DD) \subset W^{k,p}(\DD)$ denotes the closed subspace
defined as the $W^{k,p}$-closure of $C_0^\infty(\DD)$.
\end{remark}

Note that if $p > 2$ and $f \in L^p(\DD)$, then the statement of Proposition~\ref{prop:regularity}
gives a solution $u := Tf \in W^{1,p}(\DD)$ to the equation $\dbar u = f$,
and $u$ is then \emph{continuous} by the Sobolev embedding theorem.
We will need this fact for certain applications, but since we did not prove 
the $p > 2$ case of Proposition~\ref{prop:regularity}, the continuity of
solutions to $\dbar u = f \in L^p$ for $p > 2$ needs to be proved separately.
This turns out to be not so hard.

\begin{prop}
\label{prop:LpToC0}
Let $T : L^2(\DD) \to H^1(\DD)$ denote the bounded right inverse of
$\dbar : H^1(\DD) \to L^2(\DD)$ provided by Proposition~\ref{prop:regularity},
defined as an extension of the convolution
operator $f \mapsto K*f$.  Then for $p \in (2,\infty)$,
$T$ sends any $f \in L^p(\DD) \subset L^2(\DD)$
into $C^0(\DD)$, and it restricts to a bounded linear operator
$$
T : L^p(\DD) \to C^0(\DD).
$$
\end{prop}
\begin{proof}
Observe that the fundamental solution $K(z) = 1 / 2\pi z$ belongs to
$L^q_\loc(\CC)$ whenever $1 \le q < 2$, and in fact the $L^q$-norm of
$K$ on the unit disk $\DD(z) \subset \CC$ about a point $z \in \CC$ satisfies
$$
\| K \|_{L^q(\DD(z))} \le C
$$
for some constant $C > 0$ that depends on $q$ but not on~$z$.
In particular if $p > 2$,
this is true for $q \in (1,2)$ such that $1/q + 1/p = 1$.  Now if
$f \in C_0^\infty(\DD)$, H\"older's inequality implies that for
every $z \in \CC$,
\begin{equation*}
\begin{split}
| K * f(z) | &= \left| \int_{\CC} K(z - \zeta) f(\zeta) \, d\mu(\zeta) \right|
\le \int_\DD | K(z - \zeta) | \cdot |f(\zeta)| \, d\mu(\zeta) \\
&\le \| K(z - \cdot) \|_{L^q(\DD)} \cdot \| f \|_{L^p(\DD)} \le C \| f \|_{L^p(\DD)},
\end{split}
\end{equation*}
hence $\| Tf \|_{L^\infty} \le C \| f \|_{L^p}$.  By standard results on
convolutions of smooth functions with distributions (see e.g.~\cite{LiebLoss}*{\S 6.13}),
$K*f$ is a smooth function for each
$f \in C_0^\infty(\DD)$, thus the map $f \mapsto K*f$ extends to a bounded
linear map from $L^p(\DD)$ to the $L^\infty$-closure of the space of bounded
smooth functions, which is $C^0(\DD)$.  Since $L^p(\DD)$ embeds continuously
into $L^2(\DD)$, this extension is necessarily the same as
$T : L^2(\DD) \to H^1(\DD)$ for all $f \in L^p(\DD)$.
\end{proof}

Here is the first of several applications of these estimates.

\begin{prop}
\label{prop:regularity2}
If $g \in H^k_\loc(\DD)$ for some integer $k \ge 0$, then every weak solution $f \in L^1_\loc(\DD)$ to
the equation $\dbar f = g$ is also in $H^{k+1}_\loc(\DD)$.  In particular,
$f$ is smooth whenever $g$ is smooth.
\end{prop}
\begin{proof}[Sketch of the proof]
One starts by showing that if $k \ge 1$ and
$f$ is already known to be in $H^1(\DD_r)$ for some $r > 0$,
then $f$ will also in $H^{k+1}(\DD_{r'})$ for every $r' \in (0,r)$,
and it satisfies an estimate of the form
\begin{equation}
\label{eqn:smallerDomain}
\| f \|_{H^{k+1}(\DD_{r'})} \le c \| f \|_{H^1(\DD_r)} + c \| g \|_{H^k(\DD_r)}.
\end{equation}
To show for instance that $f$ is in $H^2(\DD_{r'})$, it suffices to show that
both of the partial derivatives $\p_s f$ and $\p_t f$ are in $H^1(\DD_{r'})$,
and for this purpose one can approximate them by difference quotients, e.g.
$$
D^h_s f(s,t) := \frac{f(s + h,t) - f(s,t)}{h}
$$
for $h \in \RR \setminus \{0\}$ close enough to $0$ so that this definition
makes sense on a neighborhood of $\DD_{r'}$.  These difference quotients are
automatically of class $H^1$ since $f$ is, and the main task is to show that
they satisfy a uniform $H^1$-bound on $\DD_{r'}$ as $h \to 0$, as the
Banach-Alaoglu theorem then implies that they converge weakly to a function
in $H^1$ as $h \to 0$, implying that $\p_s f$ (or $\p_t f$ respectively) is
indeed of class~$H^1$.  The required uniform bound comes from the basic
elliptic estimate \eqref{eqn:basicEstimate2}: to apply it, one chooses a
smooth function $\beta : \DD_r \to [0,1]$ that equals $1$
on $\DD_{r'}$ and has compact support in the interior of~$\DD_r$, so that
$\beta D^h_s f$ is now of class $H^1_0$ on~$\DD_r$, thus
$\| D^h_s f \|_{H^1(\DD_{r'})} \le \| \beta D^h_s f \|_{H^1(\DD_r)}$ is bounded
in terms of $\| \dbar(\beta D^h_s f) \|_{L^2(\DD_r)}$.  This in turn 
can be bounded in terms of $\| D^h_s f \|_{L^2}$ and
$\| D^h_s g \|_{L^2}$, which are both uniformly bounded as $h \to 0$ because
$f$ and $g$ are both of class~$H^1$.  If $k > 1$, then one can now repeat
this argument with the knowledge that $f$ is of class $H^2$ on~$\DD_{r'}$,
and continue repeating it on smaller disks at each step until $f$ is shown
to be of class~$H^{k+1}$, with the estimate \eqref{eqn:smallerDomain} as a
quantitative expression of this fact.  This argument shows in particular
that if $f$ is of class $H^1_\loc$ and $\dbar f$ is of class $H^k_\loc$,
then $f$ is also of class~$H^{k+1}_\loc$.

Before weakening the hypothesis on $f$ further, it is useful to notice
that the previous paragraph makes possible a generalization of
Proposition~\ref{prop:regularity}: for every $k \in \NN$, the operator
$\dbar : H^k(\DD) \to H^{k-1}(\DD)$ admits a bounded right inverse
$$
T_k : H^{k-1}(\DD) \to H^k(\DD).
$$
The proof of this is by induction on~$k$, with Proposition~\ref{prop:regularity}
as the case $k=0$.  If one fixes some $R > 1$ and assumes that a right inverse
$T_{k-1} : H^{k-2}(\DD_R) \to H^{k-1}(\DD_R)$ of $\dbar : H^{k-1}(\DD_R) \to
H^{k-2}(\DD_R)$ exists, then a right inverse $T_k : H^{k-1}(\DD) \to H^k(\DD)$
for $\dbar : H^{k}(\DD) \to H^{k-1}(\DD)$
can be defined in the form
$$
T_k f := \left.T_{k-1}\widetilde{f}\right|_{\DD}
$$
for $f \in H^{k-1}(\DD) \subset H^{k-2}(\DD)$, where
$$
H^{k-1}(\DD) \to H^{k-1}(\DD_R) : f \mapsto \widetilde{f}
$$
is any choice of bounded linear extension
operator, i.e.~satisfying $\widetilde{f}|_{\DD} = f$.  The reason this defines
a bounded operator $H^{k-1}(\DD) \to H^k(\DD)$ is that if
$u = T_k f$, then $u$ is the restriction to a smaller disk $\DD \subset \DD_R$
of a function $T_{k-1}\widetilde{f} \in H^{k-1}(\DD_R)$ which satisfies
$\dbar T_{k-1}\widetilde{f} = \widetilde{f} \in H^{k-1}(\DD_R)$, thus the
previous paragraph implies that $u$ is also in $H^k(\DD)$, and 
\eqref{eqn:smallerDomain} produces the required estimate on $\| u \|_{H^k(\DD)}$.

Finally, if $f \in L^1(\DD_r)$ and $\dbar f = g \in H^k(\DD_r)$, one can now
take the bounded right inverse $T_{k+1} : H^k(\DD_r) \to H^{k+1}(\DD_r)$
and consider the function $h := f - T_{k+1}g$, which is in $L^1(\DD_r)$ and
is a weak solution to the equation $\dbar h = 0$.  The real and imaginary
parts of $h$ are then weak solutions to the Laplace equation, and by convolution
with an approximate identity, one can approximate them in $L^1(\DD_r)$ by
smooth solutions to the Laplace equation.  The latter are characterized by the
mean value property (see \cite{Evans}*{\S 2.2.3}), which behaves well
under $L^1$-convergence, implying that the real and imaginary parts of $h$ 
also satisfy the mean value property and are therefore smooth.  In
particular, $h$ then belongs to $H^{k+1}_\loc(\DD_r)$, and therefore so does
$f = h + T_{k+1}g$.
\end{proof}

\begin{cor}
\label{cor:CRreg}
Suppose $E$ is a complex vector bundle of class $C^{k+1}$ over a Riemann surface~$\Sigma$, and
$\mathbf{D} : \Gamma(E) \to \Gamma(\overline{\Hom}_\CC(T\Sigma,E))$ is a 
linear Cauchy-Riemann type operator of class~$C^k$.
Then every weak solution of class $L^2_\loc$ to the equation 
$\mathbf{D}\eta = 0$ is of class~$H^{k+1}_\loc$.  In particlar, if
the bundle $E$ and operator $\mathbf{D}$ are smooth, then all weak solutions of
class $L^2_\loc$ are smooth.
\end{cor}
\begin{proof}
Locally, a weak solution to $\mathbf{D}\eta = 0$ of class $L^2_\loc$ can be represented by a
$\CC^n$-valued function $f \in L^2(\DD)$ satisfying $(\dbar + A)f = 0$ for 
some function $A : \DD \to \End_\RR(\CC^n)$ of class~$C^k$.
In particular $A$ is continuous, so $-Af$ is of class~$L^2$, and the equation $\dbar f = -Af$
thus implies via Prop.~\ref{prop:regularity2} that $f$ is of class~$H^1_\loc$. 
If $A$ is also of class~$C^1$, it follows that $-Af$ is in $H^1_\loc$, and
another application of Prop.~\ref{prop:regularity2} implies
$f \in H^2_\loc$.  Repeat until $A$ runs out of derivatives.
\end{proof}

At one point in \S\ref{sec:intSec} we will need a nonlinear analogue of
the above result, which applies to $J$-holomorphic curves in an almost
complex manifold $(M,J)$.  This justifies the fact that we only consider
\emph{smooth} $J$-holomorphic curves in these notes, even though the
nonlinear Cauchy-Riemann equation would make sense for maps that are
only differentiable.  The hypotheses can be weakened in various ways,
e.g.~by allowing non-smooth almost complex structures, but we will have
no need to consider this.  The statement is fundametally local,
thus we are free to assume $(M,J) = (\CC^n,J)$.

\begin{prop}
\label{prop:nonlinearReg}
Suppose $J$ is a smooth almost complex structure on $\CC^n$ with $J(0)=i$,
and $u : \DD \to \CC^n$ is a continuous function of class $W^{1,\infty}$
that is a weak solution to the equation
$\p_s u + J(u) \p_t u = 0$ with $u(0)=0$.  Then $u$ is smooth.
\end{prop}
\begin{proof}[Sketch of the proof]
By the Sobolev embedding theorem, it suffices to
prove that $u$ is of class $H^k_\loc$ for every $k \in \NN$.
We prove this by induction on~$k$, and observe first that
at each step of the induction, it will be enough to prove that $u$ is
of class $H^k$ on~$\DD_\rho$ for some $\rho > 0$; indeed, changing
coordinates then produces the same result on sufficiently small
neighborhoods of any point in the domain, so that finitely many
such small neighborhoods can be pieced together to show that $u$
is in $H^k$ on any compact subset of the interior.
 
Another useful observation is that for any constants $R > 0$
and $\rho \in (0,1]$, $u$ is of class $H^k$ on
$\DD_\rho$ if and only if the rescaled map
$$
\widehat{u} : \DD \to \CC^n : z \mapsto R u(\rho z)
$$
is in $H^k(\DD)$.  To make use of this, we rewrite the equation
$\p_s u + J(u) \p_t u = 0$ in the form
$$
\dbar u - Q(u) \p_t u = 0,
$$
where $Q := i - J \in C^\infty(\CC^n,\End_\RR(\CC^n))$.
The rescaled map $\widehat{u}$ then satisfies
\begin{equation}
  \label{eqn:dbarQ2}
\dbar \widehat{u} - \widehat{Q}(\widehat{u}) \p_t \widehat{u} = 0
\end{equation}
if we define $\widehat{J},\widehat{Q} : \CC^n \to \End_\RR(\CC^n)$ by
$$
\widehat{J}(p) := J(p / R), \qquad \widehat{Q} := i - \widehat{J}.
$$
The advantage of these definitions is that $\widehat{Q}$ can be made arbitrarily
$C^\infty$-small by choosing $R > 0$ large, and
since $u : \DD \to \CC^n$ is continuous
with $u(0)=0$, one can subsequently choose $\rho > 0$ small to make
the function
$\widehat{Q} \circ \widehat{u} : \DD \to \End_\RR(\CC^n)$ correspondingly small
so that \eqref{eqn:dbarQ2} becomes a small perturbation of the linear
equation $\dbar\widehat{u} = 0$.  In this context, it will be useful to
note that for any $k \in \NN$ and $p \in (1,\infty)$ with $kp > 2$,
there exist constants $c > 0$ and $\gamma > 0$ such that every
$f \in W^{k,p}(\DD)$ with $f(0)=0$ is related to its rescaled cousin 
$\widehat{f}(z) := R f(\rho z)$ by
\begin{equation}
  \label{eqn:SobolevRescaling}
\| \widehat{f} \|_{W^{k,p}(\DD)} \le c R \rho^{\gamma} \| f \|_{W^{k,p}(\DD)}.
\end{equation}
This can be proved as a corollary of the Sobolev embedding theorem,
and it implies that for each $k \ge 2$, $\|\widehat{u}\|_{H^k}$
can also be made arbitrarily small by choosing $\rho > 0$ small for any
given $R > 0$.
In the following, we always reserve
the right to enlarge $R$ and subsequently shrink $\rho$ whenever convenient.

Arguing by induction, the goal is now to show that if
$\widehat{u} \in H^k(\DD)$ for a given $k \in \NN$, then $\widehat{u}$ is
also of class $H^{k+1}$ on $\DD_r$ for some $r < 1$, where in the case
$k=1$ we impose the extra hypothesis $\widehat{u} \in W^{1,\infty}(\DD)$.
As in Proposition~\ref{prop:regularity2}, the argument uses difference
quotients, e.g.~if one can prove uniform $H^k$-bounds on the difference
quotients $D^h_s\widehat{u}$ with respect to $s$ as $h \to 0$, then
the Banach-Alaoglu theorem implies that $\p_s \widehat{u}$ is in $H^k$.
The assumption $\widehat{u} \in H^k(\DD)$ already implies a uniform
$H^{k-1}$-bound on $D^h_s \widehat{u}$ as $h \to 0$, where in the case $k=1$,
there is an additional $L^\infty$-bound.  Choosing a smooth bump
function $\beta : \DD \to [0,1]$ with compact support in the interior
and $\beta|_{\DD_r} \equiv 1$ for some $r < 1$, it then suffices
to find a uniform bound on $\| \beta D^h_s \widehat{u} \|_{H^k}$ as $h \to 0$.
The usual estimate \eqref{eqn:basicEstimate2} gives
$$
\| \beta D^h_s \widehat{u} \|_{H^k} \le c \| \dbar(\beta D^h_s \widehat{u}) \|_{H^{k-1}}.
$$
To bound the right hand side, one can apply the operator
$D^h_s$ to the equation $\dbar\widehat{u} = (\widehat{Q} \circ \widehat{u}) \p_t\widehat{u}$
from \eqref{eqn:dbarQ2}, giving
$$
\dbar\big( D^h_s \widehat{u}\big) = D^h_s\big( \widehat{Q} \circ \widehat{u}\big) \p_t\widehat{u}
+ (\widehat{Q} \circ \widehat{u}) \p_t\big( D^h_s \widehat{u}\big)
$$
and thus
$$
\dbar\big( \beta D^h_s \widehat{u}\big) =
\beta D^h_s\big( \widehat{Q} \circ \widehat{u}\big) \p_t\widehat{u} +
(\widehat{Q} \circ \widehat{u}) \p_t \big( \beta D^h_s \widehat{u}\big) +
\left( \dbar\beta - (\widehat{Q} \circ \widehat{u}) \p_t\beta\right) D^h_s \widehat{u}.
$$
From this we deduce the estimate
\begin{multline}
  \label{eqn:ugly}
\| \beta D^h_s \widehat{u} \|_{H^k} \le c \left\| \widehat{Q}(\widehat{u}) \p_t
  \big(\beta D^h_s\widehat{u}\big) \right\|_{H^{k-1}} +
c \left\| \beta D^h_s\big( \widehat{Q} \circ \widehat{u}\big) \p_t\widehat{u} \right\|_{H^{k-1}} \\
+ c \left\| \left( \dbar\beta - \widehat{Q}(\widehat{u}) \p_t\beta \right) D^h_s u \right\|_{H^{k-1}}.
\end{multline}
We claim that after suitable adjustments of the rescaling parameters $\rho$ and~$R$,
every term in \eqref{eqn:ugly} is
bounded uniformly as $h \to 0$.

Indeed, since $\widehat{Q}$ can be assumed
arbitrarily $C^k$-small on the image of $\widehat{u}$, we can also
apply \eqref{eqn:SobolevRescaling} if $k \ge 2$ to assume that
the composition $\widehat{Q} \circ \widehat{u}$ is arbitrarily small in~$H^k$,
in which case the continuous product pairing $H^k \times H^{k-1} \to H^{k-1}$
gives a uniform bound on the third term.  This argument does not quite work
in the case $k=1$, as \eqref{eqn:SobolevRescaling} is then not valid
and there is no continuous product pairing $H^1 \times L^2 \to L^2$,
but here one can instead make $\widehat{Q} \circ \widehat{u}$ arbitrarily
$C^0$-small and achieve a uniform $L^2$-bound.

For the first term, if $k \ge 2$ then one can similarly use the continuous
product pairing $H^k \times H^{k-1} \to H^{k-1}$
and make $\|\widehat{Q} \circ \widehat{u}\|_{H^k}$ small via \eqref{eqn:SobolevRescaling},
giving an estimate of the form
$$
\left\| \widehat{Q}(\widehat{u}) \p_t\left( \beta D^h_s\widehat{u}\right) \right\|_{H^{k-1}} \le
 \delta \left\| \p_t\left( \beta D^h_s\widehat{u}\right) \right\|_{H^{k-1}} \le
c \delta \| \beta D^h_s \widehat{u} \|_{H^k},
$$
with a constant $\delta > 0$ that can be made arbitrarily small by suitable adjustments
of $R$ and~$\rho$.  One can therefore absorb this term into the left hand side
of \eqref{eqn:ugly}.  Once again a special argument is required for the case
$k=1$, but here one can instead assume $\widehat{Q} \circ \widehat{u}$ is
$C^0$-small and use the continuous pairing $C^0 \times L^2 \to L^2$ to achieve
the same result.

The second term in \eqref{eqn:ugly} requires some version of the chain rule for the
difference quotient operator~$D^h_s$.  Here one can write
\begin{equation*}
\begin{split}
\widehat{Q}(p + p') 
&= \widehat{Q}(p) + \int_0^1 d\widehat{Q}(p + \tau p') p' \, d\tau \\
&= \widehat{Q}(p) + d\widehat{Q}(p) p' + \left( \int_0^1 \left[ d\widehat{Q}(p + \tau p') - d\widehat{Q}(p)\right] \, d\tau \right) p' \\
&= \widehat{Q}(p) + \left[ d\widehat{Q}(p) + \widehat{R}(p,p') \right] p'
\end{split}
\end{equation*}
for a smooth remainder function $\widehat{R} : \CC^n \times \CC^n \to \Hom_\RR(\CC^n,\End_\RR(\CC^n))$
satisfying $\widehat{R}(\cdot,0) \equiv 0$,
and use this to derive a formula of the form
\begin{equation}
\label{eqn:diffQuotChain}
D^h_s\big(\widehat{Q} \circ \widehat{u}\big) =
\left[ d\widehat{Q} \circ \widehat{u} + \widehat{R} \circ \left( \widehat{u} , h D^h_s \widehat{u}\right) \right] D^h_s \widehat{u},
\end{equation}
valid for all $h \ne 0$ sufficiently close to~$0$.  For $k \ge 2$,
one can use \eqref{eqn:SobolevRescaling} to assume the
terms $\widehat{u}$ and $h D^h_s\widehat{u}$ satisfy an arbitrarily
small $H^k$-bound independent of~$h$, and then use the smoothness
of $\widehat{Q}$ and $\widehat{R}$ and the fact that $\widehat{R}(\cdot,0) \equiv 0$
to assume that the bracketed term in the above expression
is arbitrarily $H^k$-small for all $h$ near~$0$.  Since $H^k$ is a
Banach algebra, this gives rise to an estimate of the form
$$
\left\| \beta D^h_s \left(\widehat{Q} \circ \widehat{u}\right) \right\|_{H^k} \le
\delta \| \beta D^h_s\widehat{u} \|_{H^k}
$$
where the constant $\delta > 0$ can be assumed arbitrarily small
after adjusting $R$ and~$\rho$.  Since $\| \p_t\widehat{u} \|_{H^{k-1}}
\le \|\widehat{u}\|_{H^k}$ can also be assumed small by
\eqref{eqn:SobolevRescaling} and the product pairing
$H^k \times H^{k-1} \to H^{k-1}$ is continuous,
it follows that the second
term in \eqref{eqn:ugly} can also be absorbed into the left hand
side.  In the case $k=1$, we can instead use the uniform $L^\infty$-bound
on $\p_t u$ to put a bound on $\|\p_t\widehat{u}\|_{L^\infty}$ while
making $\rho$ as small as is needed, and then use the
uniform $L^2$-bound on $D^h_s\widehat{u}$ to derive from \eqref{eqn:diffQuotChain}
a uniform $L^2$-bound on $D^h_s (\widehat{Q} \circ \widehat{u})$, which
now direcly implies a uniform bound on the second term
in \eqref{eqn:ugly}.
\end{proof}

\begin{exercise}
  \label{EX:W1infty}
  Show that for any $\varphi \in C_0^\infty(\DD)$ and $f \in C^0(\DD)$ such
  that $f|_{\DD\setminus\{0\}}$ is of class $C^1$ with bounded derivative,
  the usual formula for integration by parts
  $$
\int_\DD \p_j f \cdot \varphi = - \int_\DD f \cdot \p_j\varphi
$$
is valid, and deduce that $f$ belongs to $W^{1,\infty}(\DD)$.
\end{exercise}

\subsection{Local existence of holomorphic sections}

The main engine behind the similarity principle is the following local 
existence result for solutions to linear Cauchy-Riemann type equations.\index{similarity principle|(}

\begin{thm}
\label{thm:linearExistence}
Assume $2 < p \le \infty$ and $A \in L^p(\DD,\End_\RR(\CC^n))$.  
Then for sufficiently small $\rho > 0$, the problem
\begin{equation*}
\begin{split}
\dbar u + A u &= 0 \\
u(0) &= u_0
\end{split}
\end{equation*}
admits a weak solution $u \in C^0(\DD_\rho) \cap H^1(\DD_\rho)$
for every $u_0 \in \CC^n$.
\end{thm}

Notice that by elliptic regularity (Prop.~\ref{prop:regularity2}), the local
solutions $u : \DD_\rho \to \CC^n$ provided by this theorem may be
much nicer than just continuous functions with weak derivatives in~$L^2$, e.g.~they will be
smooth if $A$ is smooth.  One easy consequence is the following fundamental
result in complex geometry, which gives an equivalence between smooth
complex-linear Cauchy-Riemann type operators and holomorphic vector bundle
structures.

\begin{cor}
\label{cor:holStructure}
Suppose $E$ is a complex vector bundle over a Riemann surface~$\Sigma$, and\index{holomorphic vector bundle}\index{Cauchy-Riemann type operator!complex linear}
$\mathbf{D} : \Gamma(E) \to \Gamma(\overline{\Hom}_\CC(T\Sigma,E))$ is a
smooth complex-linear Cauchy-Riemann type operator.  Then $E$ admits a
unique maximal atlas of smooth local complex trivializations whose transition
maps are holomorphic, such that a section $\eta \in \Gamma(E|_\uU)$ defined
on some open domain $\uU \subset \Sigma$ is holomorphic with respect to these
trivializations if and only if $\mathbf{D}\eta = 0$.
\end{cor}
\begin{proof}
For any point $p \in \Sigma$, Theorem~\ref{thm:linearExistence} and
Proposition~\ref{prop:regularity2} together provide
a collection of smooth sections $\eta_1,\ldots,\eta_n$ defined on a neighborhood
of $p$ that all satisfy $\mathbf{D}\eta_i = 0$ and are pointwise complex-linearly
independent at~$p$ (and therefore also in a neighborhood of~$p$).
We define the desired atlas of local trivializations by viewing collections
of this sort as local frames.  The Leibniz rule for complex-linear Cauchy-Riemann
type operators (cf.~Remark~\ref{remark:notComplex}) then implies that transition maps are holomorphic.
\end{proof}

The local existence theorem admits a fairly straightforward proof using
the $p > 2$ case of Prop.~\ref{prop:regularity}. The idea is to multiply
$A$ by the characteristic function $\chi_\rho$ of $\DD_\rho$ for $\rho > 0$,
producing a family of bounded linear operators
$$
\mathbf{D}_\rho := \dbar + \chi_\rho A : W^{1,p}(\DD) \to L^p(\DD),
$$
which converge in the norm topology to $\dbar : W^{1,p}(\DD) \to L^p(\DD)$ as
$\rho \to 0$.  It follows that the operators
$$
\mathbf{L}_\rho : W^{1,p}(\DD) \to L^p(\DD) \times \CC^n : u \mapsto 
(\mathbf{D}_\rho u,u(0))
$$
also converge as $\rho \to 0$ to $\mathbf{L}_0(u) = (\dbar u,u(0))$;
note here that $W^{1,p}(\DD) \to \CC^n : u \mapsto u(0)$ is a well-defined and
continuous linear map due
to the Sobolev embedding theorem.  Since $\dbar : W^{1,p}(\DD) \to L^p(\DD)$
has a bounded right inverse and holomorphic functions on $\DD$ can take
arbitrary values at a point, the operator $\mathbf{L}_0$ also has a bounded
right inverse, and so therefore does $\mathbf{L}_\rho$ for $\rho > 0$
sufficiently small, as the existence of bounded right inverses is an open
condition.  The right inverse of $\mathbf{L}_\rho$ can then be used to
produce functions $u \in W^{1,p}(\DD)$ that have prescribed values at $0$
and satisfy $(\dbar + \chi_\rho A)u = 0$, so in particular they
satisfy $(\dbar + A)u = 0$ on~$\DD_\rho$.  These functions are also
continuous since, by the Sobolev embedding theorem, $W^{1,p}(\DD)$ embeds
continuously into~$C^0(\DD)$ for $p > 2$.

The argument just sketched would not work for $p=2$ because $H^1(\DD) = W^{1,2}(\DD)$
does not embed into~$C^0(\DD)$.
Since we did not prove the $p > 2$ case of Proposition~\ref{prop:regularity},
we will have to do something slightly more roundabout in order to produce
a self-contained proof of local existence.  It is based on
the following lemma, which was suggested by Jean-Claude Sikorav.

\begin{lemma}
\label{lemma:localExistenceL2}
Under the same assumptions as in Theorem~\ref{thm:linearExistence}, suppose
$0 < r \le 1$ and $f_0 : \DD_r \to \CC^n$ is a holomorphic function.
Then for any $\delta > 0$, there exists $\rho \in (0,r]$ and a continuous
function $f : \DD_\rho \to \CC^n$ such that $|f| \le \delta$ and 
$u := f_0 + f : \DD_\rho \to \CC^n$ is a weak solution to the equation
$(\dbar + A) u = 0$.
\end{lemma}
\begin{proof}
We will look for a continuous weak solution $u : \DD \to \CC^n$ to the equation
$$
(\dbar + A_\rho) u = 0,
$$
for some small number $\rho > 0$, where $A_\rho := \chi_\rho A$
and $\chi_\rho : \DD \to [0,1]$ denotes the function that equals~$1$ on
$\DD_\rho$ and $0$ everywhere else.  We claim that each of the operators
$\dbar + A_\rho$ is a bounded linear operator $H^1(\DD) \to L^2(\DD)$, and that
these operators converge in the operator norm to $\dbar$ as $\rho \to 0$.
Recall that $H^1(\DD)$ is a
``Sobolev borderline case,'' so it admits continuous inclusions
$H^1(\DD) \hookrightarrow L^q(\DD)$ for every finite $q \ge 1$
(see \cite{AdamsFournier}).  
Thus if we pick $q > 1$ according to the condition $1/q + 2/p = 1$, then
H\"older's inequality and the continuous inclusion $H^1 \hookrightarrow
L^{2q}$ imply that for any $u \in H^1(\DD)$,
\begin{equation*}
\begin{split}
\| A_\rho u \|_{L^2}^2 &\le \int_{\DD_\rho} |A|^2 |u|^2 \le
\left\| |A|^2 \right\|_{L^{p/2}(\DD_\rho)} \cdot
\left\| |u|^2 \right\|_{L^q(\DD_\rho)}
\le \| A \|_{L^p(\DD_\rho)}^2 \| u \|_{L^{2q}(\DD)}^2 \\
&\le c \| A \|_{L^p(\DD_\rho)}^2 \| u \|_{H^1(\DD)}^2
\end{split}
\end{equation*}
for some constant $c > 0$.
This proves the claim, since $\| A \|_{L^p(\DD_\rho)} \to 0$ as
$\rho \to 0$.

Since $\dbar : H^1(\DD) \to L^2(\DD)$ has a bounded right inverse
$T : L^2(\DD) \to H^1(\DD)$, it follows that $\dbar + A_\rho$ also has a
bounded right inverse
$$
T_\rho : L^2(\DD) \to H^1(\DD)
$$
for all $\rho > 0$ sufficiently small.
It should now at least seem plausible that any solution $f_0 \in H^1(\DD)$
to $\dbar f_0 = 0$ admits an $H^1$-close perturbation $f_0 + f$
satisfying $(\dbar + A_\rho)(f_0 + f) = 0$: indeed, the latter is
equivalent to the equation
$$
(\dbar + A_\rho) f = - A_\rho f_0,
$$
which can be solved by
$$
f := -T_\rho (A_\rho f_0).
$$
This function clearly is $H^1$-small whenever $\rho$ is correspondingly
small since $T_\rho : L^2 \to H^1$ is continuous and
$L^p(\DD)$ embeds continuously into $L^2(\DD)$, hence
\begin{equation}
\label{eqn:AepsSmall}
\| A_\rho f_0 \|_{L^2(\DD)} \le 
c \| A_\rho f_0 \|_{L^p(\DD)} \le c \| A \|_{L^p(\DD_\rho)} \| f_0 \|_{C^0}
\to 0 \quad \text{ as } \quad \rho \to 0.
\end{equation}
We claim in fact that the operator $T_\rho$ can be chosen to make
$f$ continuous and $C^0$-small when
$\rho$ is correspondingly small.  By \eqref{eqn:AepsSmall},
this will be immediate if we can
show that $T_\rho$ restricts to a continuous linear map
$L^p(\DD) \to C^0(\DD)$ for $p > 2$, a fact which we already know is true
of $T : L^2(\DD) \to H^1(\DD)$ by Proposition~\ref{prop:LpToC0}.  Thus to prove
the claim, let us write down a more explicit definition 
of~$T_\rho$.  Notice that
$$
(\dbar + A_\rho) T = \1 + A_\rho T
$$
is a bounded linear operator on $L^2$ and is close to the identity in the
operator norm since $T : L^2 \to H^1$ is continuous and
$A_\rho : H^1 \to L^2$ is small.  But for slightly different reasons,
this operator is also close to the identity in the space of bounded
linear operators on~$L^p$: indeed, $A_\rho : C^0 \to L^p$ is also
continuous and small since
$$
\| A_\rho u \|_{L^p(\DD)} \le \| A_\rho \|_{L^p(\DD)} \| u \|_{C^0(\DD)}
= \| A \|_{L^p(\DD_\rho)} \| u \|_{C^0(\DD)},
$$
so this statement follows from the continuity of $T : L^p(\DD) \to C^0(\DD)$.
Thus for any $\rho > 0$ sufficiently small, $\1 + A_\rho T$ defines
isomorphisms on both $L^2(\DD)$ and $L^p(\DD)$, so that defining
$$
T_\rho := T(\1 + A_\rho T)^{-1}
$$
gives a right inverse of $\dbar + A_\rho$ that is continuous both from
$L^2$ to $H^1$ and from $L^p$ to~$C^0$.
\end{proof}

\begin{proof}[Proof of Theorem~\ref{thm:linearExistence}]
Using Lemma~\ref{lemma:localExistenceL2}, we can construct the columns of
a continuous matrix-valued function $\Phi : \DD_\rho \to \End_\RR(\CC^n)$ for
$\rho > 0$ small such that $\Phi$ weakly satisfies $(\dbar + A)\Phi = 0$ 
and is arbitrarily $C^0$-close to the constant
(and thus holomorphic) function $\Phi_0(z) := \1$.  We can therefore
assume $\Phi$ takes values in $\GL(2n,\RR)$.  Continuous solutions
$u : \DD_\rho \to \CC^n$ to $(\dbar + A) u = 0$ with prescribed values
$u(0) = u_0$ can then be constructed by multiplying $\Phi$ by suitable
constant vectors in~$\CC^n$.
\end{proof}

\subsection{The similarity principle}

We can now prove the main result of the present section.

\begin{thm}
\label{thm:similarity2}
Assume $E$ is a complex vector bundle of class~$C^1$ over a Riemann surface~$\Sigma$,
$\mathbf{D}$ is a linear Cauchy-Riemann type operator on $E$ of class~$L^p$
for some $p \in (2,\infty]$ in the sense of Remark~\ref{remark:CRLp}, 
and $\eta : \Sigma \to E$ is a continuous
section that is a weak solution to the equation $\mathbf{D}\eta = 0$
with $\eta(z_0) = 0$ for some point $z_0 \in \Sigma$.  Then
there exists a continuous local complex trivialization of $E$ near $z_0$
that identifies $\eta$ with a holomorphic function.
Moreover, if $\mathbf{D}$ is smooth and complex linear, then the local
trivialization near $z_0$ can be arranged to be smooth.
\end{thm}
\begin{proof}
The issue is purely local, so assume $A \in L^p(\DD,\End_\RR(\CC^n))$ 
with $p > 2$ and $u : \DD \to \CC^n$ is a continuous weak solution to
$$
(\dbar + A)u = 0
$$
with $u(0)=0$.  We start by replacing $\dbar + A$ by another Cauchy-Riemann
type operator 
that is complex linear but has the same regularity.  Indeed, choose a
measurable function $C : \DD \to \End_\CC(\CC^n)$ such that $|C(z)| \le |A(z)|$
and $C(z) u(z) = A(z) u(z)$ for all $z\in \DD$. Then $C$ is also of class
$L^p(\DD)$ and $u$ also satisfies $\dbar u + Cu = 0$.

Now construct a local frame as in the proof of Corollary~\ref{cor:holStructure},
that is, let $\Phi : \DD_\rho \to \End_\CC(\CC^n)$ be a complex matrix-valued
function whose columns are local weak solutions to $(\dbar + C)\eta = 0$
as provided by Theorem~\ref{thm:linearExistence}, with $\Phi(0) = \1$.
Since $C \in L^p(\DD)$ with $p > 2$, $\Phi$ is in $C^0(\DD_\rho) \cap H^1(\DD_\rho)$,
and it also satisfies $(\dbar + C)\Phi = 0$.  After shrinking $\rho > 0$
if necessary, continuity then implies that we are free to assume $\Phi(z)$
is invertible for all $z \in \DD_\rho$, and we can therefore define
a continuous function $f : \DD_\rho \to \CC^n$ by
$$
f(z) := [\Phi(z)]^{-1} u(z).
$$
To conclude, we need to show that $f$ is a weak solution to $\dbar f = 0$,
in which case Proposition~\ref{prop:regularity2} implies that $f$ is also smooth,
and therefore holomorphic.  If $\Phi$, $u$ and $f$ were all smooth, then
$\dbar f = 0$ would follow from the fact that $\dbar + C$ is complex linear
and annihilates both $\Phi$ and~$u$, as the Leibniz rule
(cf.~Remark~\ref{remark:notComplex}) then implies
$$
0 = (\dbar + C)u = (\dbar + C)(\Phi f) = \left[(\dbar + C)\Phi\right] f + \Phi(\dbar f) =
\Phi(\dbar f).
$$
An additional argument is required in order to justify this conclusion without
knowing whether $\Phi$ and $u$ are smooth.  What we do know is that $u$ is
continuous and $\Phi$ is in both $C^0$ and $H^1$; since $A \in L^p(\DD) \subset
L^2(\DD)$, we also have $- A u \in L^2(\DD)$, so that Prop.~\ref{prop:regularity2}
implies that $u$ is also in $H^1(\DD_\rho)$.  To make use of this, we can
consider the following normed linear space:
$$
X := H^1(\DD_\rho) \cap C^0(\DD_\rho), \qquad
\| \eta \|_X := \| \eta \|_{H^1(\DD_\rho)} + \| \eta \|_{C^0(\DD_\rho)}.
$$
It is a straightforward exercise to prove that $X$ has the following properties:
\begin{itemize}
\item $X$ is complete, i.e.~it is a Banach space.
\item $C^\infty(\DD_\rho) \cap X$ is dense in~$X$.  (Indeed, one can check
that the standard mollification procedure for functions in $H^1(\DD_\rho)$
as in \cite{Evans}*{\S 5.3} works simultaneously for $C^0(\DD_\rho)$.
\item $X$ is a Banach algebra, i.e.~there exists a continuous product pairing
$X \times X \to X : (g,h) \mapsto gh$ for complex-valued functions, and
there is similarly a continuous product pairing $X \times L^2(\DD_\rho) \to L^2(\DD_\rho)$.
(The main tool in both cases is the inequality $\| gh \|_{L^2} \le \|g\|_{C^0} \|h\|_{L^2}$
for $g \in C^0$ and~$h \in L^2$.)
\item If $\Phi \in X$ is a function $\DD \to \End_\CC(\CC^n)$ with image
in $\GL(n,\CC)$, then the function
$\Phi^{-1}(z) := [\Phi(z)]^{-1}$ also belongs to $X$ and depends continuously
on $\Phi \in X$ in the topology of~$X$.  (Recall that $\GL(n,\CC) \to \GL(n,\CC) :
B \mapsto B^{-1}$ is a smooth function.)
\end{itemize}
Notice that for any $g \in X$, we have $\dbar g \in L^2(\DD_\rho)$ since
$g \in H^1(\DD_\rho)$, and similarly, $Cg \in L^2(\DD_\rho)$ since 
$C \in L^p \subset L^2$ and $g$ is continuous.  In fact, $\dbar + C$ defines a
continuous linear operator
$$
(\dbar + C) : X \to L^2(\DD_\rho).
$$
The previous remarks now imply after taking $\rho > 0$ sufficiently small
that $\Phi^{-1}$
and $u$ both belong to~$X$, so by the Banach algebra
property, so does $f = \Phi^{-1} u$.  Now use the density of smooth functions
to find sequences of smooth functions $f_\nu,\Phi_\nu$ converging in $X$
to $f$ and $\Phi$ respectively, so that (using the Banach algebra property
again) $u_\nu := \Phi_\nu f_\nu$ also converges in $X$ to~$u$.  The
Leibniz rule now gives
$$
(\dbar + C) u_\nu = \left[ (\dbar + C) \Phi_\nu \right] f_\nu + \Phi_\nu ( \dbar f_\nu),
$$
in which the left hand side and the first term on the right hand side
both converge in $L^2$ as $\nu \to \infty$ to zero, while 
the last term converges in $L^2$
to $\Phi(\dbar f)$, proving $\dbar f = 0$.

Finally, consider the special case in which $A$ is smooth and complex linear.
Under this assumption, Corollary~\ref{cor:holStructure} implies that
$\mathbf{D}$ defines a holomorphic structure on $E$ in which $\eta$ is a
holomorphic section.  Alternatively, one could instead apply the argument above
after setting $C := A$ in the initial step, so that the function
$\Phi : \DD_\rho \to \End_\CC(\CC^n)$ satisfying $(\dbar + C)\Phi$
is then smooth by elliptic regularity (Corollary~\ref{cor:CRreg}).
\end{proof}

\begin{cor}
\label{cor:similarity}
Under the assumptions of Theorem~\ref{thm:similarity2}, suppose 
$\eta$ is not identically zero near~$z_0$, and choose local holomorphic coordinates and
a local complex trivialization near~$z_0$ to identify $\eta$ with a
function $\DD \to\CC^n$ such that $z_0 = 0 \in \DD$.  Then $\eta$ satisfies
the formula
$$
\eta(z) = z^k C + |z|^k R(z)
$$
for some $k \in \NN$, $C \in \CC^n \setminus \{0\}$ and a function
$R(z) \in \CC^n$ such that $\lim_{z \to 0} R(z) = 0$.
\end{cor}
\begin{proof}
In the chosen coordinates and trivialization, the similarity
principle provides a continuous transition map
$\Phi : \DD_\rho \to \GL(n,\CC)$ for $\rho > 0$ small and a
holomorphic function $f : \DD_\rho \to \CC^n$ such that
$\eta = \Phi f$ and $f(0)=0$.  Since $\eta$ is not identically zero on this neighborhood,
we have $f(z) = z^k g(z)$ for some $k \in \NN$ and a holomorphic function
$g : \DD_\rho \to \CC^n$ with $g(0) \ne 0$.  Then
$$
u(z) = z^k \Phi(0) g(0) + z^k \left[ \Phi(z)g(z) - \Phi(0) g(0) \right],
$$
in which $\Phi(0) g(0) \ne 0$ and the term in brackets is a continuous function
that vanishes at $z=0$.
\end{proof}

\begin{remark}
\label{remark:similarity}
If $\eta$ in the corollary above is smooth, then the result is equivalent
to the statement that the Taylor series of $\eta$ about $z_0$ is nontrivial
and its lowest-order term is holomorphic (i.e.~a polynomial in $z$ with
no dependence on $\bar{z}$).  However, the result remains valid even if
$E$, $\mathbf{D}$ and $\eta$ are not assumed smooth, e.g.~in the proof of
the representation formula in the next section, we will need to consider
examples where $\eta$ is only known to be of class~$C^1$.\index{similarity principle|)}
\end{remark}

\section{The representation formula}
\label{sec:representation}

If $u(z) = v(\zeta)$ is an isolated intersection of two $J$-holomorphic curves
in an almost complex $4$-manifold and at least one of the curves is immersed
at the intersection point, then there is a relatively easy argument via the
similarity principle (cf.~\cite{McDuffSalamon:Jhol2}*{Exercise~2.6.1}) 
to prove that this intersection must count positively.
The same holds without assuming that either curve is immersed, but the proof
requires more work.  One approach,
due to McDuff \cites{McDuff:intersections}, shows that
a $J$-holomorphic curve with critical points always admits a global perturbation
to an \emph{immersed} $J'$-holomorphic curve for some perturbed almost
complex structure~$J'$, thus the general case can be reduced to the
immersed case.  This is an elegant argument, but it gives little insight as to
what is really happening near critical points of holomorphic curves, so we
will instead discuss a purely local approach, using a variation on a result
of Micallef and White \cite{MicallefWhite}.
The following statement is weaker than the actual
Micallef-White theorem but suffices for our purposes, and is
easier to prove.\index{Micallef-White theorem}

\begin{thm}
\label{thm:representation}
Suppose $(M,J)$ is a smooth almost complex manifold of dimension $2n$, and
$u : (\Sigma,j) \to (M,J)$ is a $J$-holomorphic curve that is not constant
in some neighborhood of the point $z_0 \in \Sigma$.  Then there exists
a unique integer $k \in \NN$  and $1$-dimensional complex subspace
$L \subset T_{u(z_0)}M$ such that
one can find a $C^\infty$-smooth coordinate chart on a neighborhood of 
$u(z_0) \in M$ and a $C^1$-smooth coordinate chart on a neighborhood of
$z_0 \in \Sigma$,
identifying these points with the origin in
$\CC^n$ and $\CC$ respectively and identifying $L$ with
$\CC \times \{0\} \subset \CC^n$, so that $u$ in these coordinates near~$z_0$
takes the form
$$
u(z) = (z^{k},\widehat{u}(z)) \in \CC \times \CC^{n-1}
$$
for some $C^1$-smooth function $\widehat{u}(z) \in \CC^{n-1}$ defined near $z=0$
and satisfying $\widehat{u}(z) = O(|z|^{k+1})$.
Moreover, the $C^1$-smooth chart near $z_0$ may be assumed $C^\infty$ at all
points other than~$z_0$, and
$\widehat{u}$ either vanishes identically or satisfies the formula
$$
\widehat{u}(z) = z^{k +\ell_u} C_u + |z|^{k+\ell_u} r_u(z)
$$
for some constants $C_u \in \CC^{n-1} \setminus \{0\}$, $\ell_u \in \NN$, and a function 
$r_u(z) \in \CC^{n-1}$ with $r_u(z) \to 0$ as $z \to 0$.
We will say in this situation that $u$ has \defin{tangent space}\index{holomorphic curve!tangent space at a critical point}\index{tangent space of a holomorphic curve}
$L$ with \defin{critical order}\index{critical order}\index{holomorphic curve!critical order of a critical point}
$k - 1$ at~$z_0$.

Further, if $v : (\Sigma',j') \to (M,J)$ is another nonconstant $J$-holomorphic curve 
with an intersection $u(z_0) = v(\zeta_0)$ at some point $\zeta_0 \in \Sigma'$
where $u$ and $v$ have the same tangent spaces and critical orders, then
the coordinates above can be chosen together with $C^1$-smooth coordinates near
$\zeta_0 \in \Sigma'$ having the same properties, in particular
such that $v$ satisfies a representation formula 
$$
v(z) = (z^k,\widehat{v}(z)),
$$
with either $\widehat{v} \equiv 0$ or 
$$
\widehat{v}(z) = z^{k + \ell_v} C_v + |z|^{k+ \ell_v} r_v(z)
$$
for some $C_v \in \CC^{n-1} \setminus \{0\}$, 
$\ell_v \in \NN$ and function $r_v(z)$ with $r_v(z) \to 0$ as $z \to 0$.

Finally, any two curves written in this way are related to each 
other as follows: either $\widehat{u} \equiv \widehat{v}$, or
\begin{equation}
\label{eqn:relativeuv}
\widehat{v}(z) - \widehat{u}(z) = z^{k+\ell'} C' + |z|^{k+\ell'} r'(z),
\end{equation}
for some constants $C' \in \CC^{n-1} \setminus \{0\}$, $\ell' \in \NN$ 
and a function $r'(z) \in \CC^{n-1}$ with $r'(z) \to 0$ as $z \to 0$.
\end{thm}

\begin{exercise}
Prove Theorem~\ref{thm:representation} for the case $(M,J) = (\CC^n,i)$.
\end{exercise}

\begin{exercise}
\label{EX:tangentSpace}
Use Theorem~\ref{thm:representation} to show that for any $J$-holomorphic curve
$u : (\Sigma,j) \to (M,J)$ with a point $z_0 \in \Sigma$ where $du(z_0)=0$
but $u$ is not constant near~$z_0$, all other points in some neighborhood
of~$z_0$ are immersed points, and moreover, the natural map
$$
z \mapsto \im du(z)
$$
from the immersed points in $\Sigma$ to the bundle of complex $1$-dimensional
subspaces in $(TM,J)$ extends continuously to~$z_0$.
\end{exercise}

The much deeper theorem of Micallef and White \cite{MicallefWhite} applies\index{Micallef-White theorem}
to a more general class of maps than just $J$-holomorphic curves, and it
also provides coordinates in which $\widehat{u}(z)$ and $\widehat{v}(z)$ become
\emph{polynomials}, thus the remainder formulas stated in
Theorem~\ref{thm:representation} become obvious.  The Micallef-White
theorem is discussed in more detail in \cite{McDuffSalamon:Jhol2}*{Appendix~E}
(written with Laurent Lazzarini) and \cite{Sikorav:singularities}.
Our weaker version
is based on ideas due to Hofer, and is essentially a ``non-asymptotic version''
of Siefring's relative asymptotic analysis \cites{Siefring:thesis}
described in Lecture~\ref{sec:3}.

The remainder of \S\ref{sec:representation} will be concerned with the
proof of Theorem~\ref{thm:representation}.

\subsection{The generalized tangent-normal decomposition}
\label{sec:normal}

The first step is to prove a refined version of the corollary that was
observed in Exercise~\ref{EX:tangentSpace}:\index{holomorphic curve!tangent space at a critical point|(}\index{tangent space of a holomorphic curve|(}

\begin{prop}
\label{prop:Tu}
If $u : (\Sigma,j) \to (M,J)$ is a smooth connected $J$-holomorphic
curve that is not constant, then the critical points of $u$ are isolated,
and there exists a unique smooth rank~$1$ complex subbundle
$$
T_u \subset u^*TM
$$
such that $(T_u)_z = \im du(z)$ for all immersed points $z \in \Sigma$ of~$u$.
Moreover, $du$ defines a smooth section of the complex line bundle
$\Hom_\CC(T\Sigma,T_u)$ whose zeroes coincide with the critical points of~$z$,
and these zeroes all have positive order.\index{zeroes of a section!positivity of}
\end{prop}

We shall refer to the subbundle $T_u \subset u^*TM$ in Proposition~\ref{prop:Tu}
as the \defin{generalized tangent bundle}
of the curve $u : (\Sigma,j) \to (M,J)$,
and define the \defin{critical order}\index{critical order}\index{holomorphic curve!critical order of a critical point}
of each critical point of $u$ to be
the order of the corresponding zero of $du \in \Gamma(\Hom_\CC(T\Sigma,T_u))$.
A choice of smooth complex subbundle $N_u \subset u^*TM$ that is complementary to
$T_u$ will then be referred to as the \defin{generalized normal bundle}\index{generalized tangent-normal splitting}\index{holomorphic curve!generalized normal bundle of}
of~$u$,
characterized by the smooth complex-linear splitting
$$
u^*TM = T_u \oplus N_u.
$$
The bundle $N_u$ is non-unique but is clearly unique up to isomorphism, so we
shall typically ignore this detail in our discussion---if you prefer, you
are free to eliminate the ambiguity by assuming
always that $N_u$ is the orthogonal complement of $T_u$
with respect to some fixed choice of $J$-invariant Riemannian metric.

Proposition~\ref{prop:Tu} is an easy consequence of the correspondence given
by Corollary~\ref{cor:holStructure} between complex-linear Cauchy-Riemann
operators and holomorphic bundle structures.  It depends on
the following trick
borrowed from \cite{IvashkovichShevchishin}.  Consider the linearized
Cauchy-Riemann operator
$$
\mathbf{D}_u : \Gamma(u^*TM) \to \Gamma(\overline{\Hom}_\CC(T\Sigma,u^*TM)),
$$
which can be defined via the property that if $\{u_\sigma : \Sigma \to M \}_{\sigma \in (-\epsilon,\epsilon)}$
is any smooth $1$-parameter family of maps satisfying $u_0 = u$ and
$\left.\p_\sigma u_\sigma\right|_{\sigma=0} = \eta \in \Gamma(u^*TM)$, then for
any connection $\nabla$ on $M$ and any $z \in \Sigma$ and $X \in T_z\Sigma$,
$$
(\mathbf{D}_u \eta)(X) = \nabla_\sigma\left[ (\dbar_J u_\sigma)(X)\right]\Big|_{\sigma=0},
$$
where $\dbar_J$ denotes the nonlinear Cauchy-Riemann operator
$$
\dbar_J f := df + J \circ df \circ j \in \Gamma(\overline{\Hom}_\CC(T\Sigma,f^*TM)).
$$
Choosing the connection $\nabla$ to be symmetric, one can derive a more direct
formula for $\mathbf{D}_u$ in the form
$$
\mathbf{D}_u \eta = \nabla \eta + J(u) \circ \nabla \eta \circ j +
(\nabla_\eta J) \circ Tu \circ j,
$$
which shows that $\mathbf{D}_u$ is indeed a smooth linear Cauchy-Riemann type
operator.  In general $\mathbf{D}_u$ is real but not complex linear, because
the connection $\nabla$ need not be complex and $\nabla_{J\eta} J - J \nabla_\eta J$
need not vanish.
On the other hand, it is easy to check that the complex-linear part of
$\mathbf{D}_u$,
\begin{equation*}
\begin{split}
\mathbf{D}_u^\CC : \Gamma(u^*TM) &\to \Gamma(\overline{\Hom}_\CC(T\Sigma,u^*TM)) \\
\eta &\mapsto \frac{1}{2} \left( \mathbf{D}_u \eta - J \mathbf{D}_u (J\eta) \right),
\end{split}
\end{equation*}
also satisfies the required Leibniz rule and is thus a smooth complex-linear
Cauchy-Riemann type operator.  By Corollary~\ref{cor:holStructure},
$\mathbf{D}_u^\CC$ therefore determines a holomorphic vector bundle structure
on~$u^*TM$.

\begin{lemma}
The complex-linear bundle map $du : T\Sigma \to u^*TM$ is holomorphic with
respect to the canonical holomorphic structure of $T\Sigma$ and the
holomorphic structure on $u^*TM$ determined by~$\mathbf{D}_u^\CC$.
\end{lemma}
\begin{proof}
The canonical holomorphic structure of $T\Sigma$ is determined by a
complex-linear Cauchy-Riemann type operator $\mathbf{D}_\Sigma :
\Gamma(T\Sigma) \to \Gamma(\overline{\End}_\CC(T\Sigma))$ which is the
linearization at $\Id : (\Sigma,j) \to (\Sigma,j)$ of the nonlinear operator
$\dbar_j \varphi := d\varphi + j \circ d\varphi \circ j \in
\Gamma(\overline{\Hom}_\CC(T\Sigma,\varphi^*T\Sigma))$ for maps
$\varphi : \Sigma \to \Sigma$.  It follows that a smooth vector field
$X \in \Gamma(T\Sigma)$ is holomorphic near some point $z \in \Sigma$ if
and only if it can be written as
$$
X = \left.\p_\sigma \varphi_\sigma\right|_{\sigma=0}
$$
for a smooth family of maps $\{\varphi_\sigma : \Sigma \to \Sigma \}_{\sigma \in (-\epsilon,\epsilon)}$
which satisfy $\varphi_0 = \Id$ and $\p_\sigma\varphi_\sigma|_{\sigma=0} = X$ and
are holomorphic near~$z$.  In this case, the maps
$$
u_\sigma := u \circ \varphi_\sigma : \Sigma \to M
$$
also satisfy $\dbar_J u_\sigma = 0$ in a neighborhood of~$z$, and the section
$\eta := \p_\sigma u_\sigma|_{\sigma=0} \in \Gamma(u^*TM)$ is related to the
vector field $X$ by $\eta = du(X)$.  This implies that
$\mathbf{D}_u\eta$ also vanishes near~$z$.  Since $jX$ is also a
holomorphic vector field near~$z$, the same argument implies that the
section $du(jX) = J\eta \in \Gamma(u^*TM)$ satisfies $\mathbf{D}_u(J\eta) = 0$
near~$z$.  Both of these facts together prove that
$\mathbf{D}_u^\CC \eta = 0$ vanishes near~$z$.  In summary, we've shown that
$du$ maps any locally defined holomorphic vector field to a locally defined
section of $u^*TM$ that is holomorphic with respect to~$\mathbf{D}_u^\CC$,
and this is equivalent to $du : T\Sigma \to u^*TM$ being a
holomorphic bundle map.
\end{proof}

\begin{proof}[Proof of Proposition~\ref{prop:Tu}]
Since $u$ is not constant and $\Sigma$ is connected,
the holomorphicity of $du \in \Gamma(\Hom_\CC(T\Sigma,u^*TM))$ implies that
zeroes of $du$ and therefore also critical points of $u$ are isolated.
The definition of $T_u \subset u^*TM$ at immersed points of $u$ is
obvious, thus we only need to check that a smooth extension of this subbundle
over the zero-set of $du$ exists.  Given a critical point
$z_0 \in \Sigma$, choose a holomorphic local coordinate for $\Sigma$ and a 
holomorphic trivialization of 
$\Hom_\CC(T\Sigma,u^*TM)$ near $z_0$, so that $du$ is expressed in this
neighborhood as a $\CC^n$-valued holomorphic function of the form
$$
(z - z_0)^k F(z)
$$
for some $k \in \NN$ and a $\CC^n$-valued holomorphic function $F$ with
$F(z_0) \ne 0$.  We can then define $T_u$ at each point $z$ near $z_0$ to be 
the complex line in $T_{u(z)}M$ corresponding to the complex span of
$F(z)$ in the trivialization.  This definition matches the previous
definition at the immersed points $z \ne z_0$ and thus makes $T_u \subset u^*TM$ into a
smooth line bundle on a neighborhood of~$z_0$, with the integer $k > 0$ as
the order of the zero of $du \in \Gamma(\Hom_\CC(T\Sigma,T_u))$ at~$z_0$.\index{holomorphic curve!tangent space at a critical point|)}\index{tangent space of a holomorphic curve|)}
\end{proof}

\subsection{A lemma on normal push-offs}
\label{sec:normalPushoffLemma}

The message of the following result is that whenever
$u : (\Sigma,j) \to (M,J)$ and $v : (\Sigma',j') \to (M,J)$ are two
$J$-holomorphic curves related to each other by
$$
\exp_{u \circ \varphi} \eta = v
$$
for some diffeomorphism $\varphi : \Sigma' \to \Sigma$ and section
$\eta$ of $\varphi^*N_u$, the section $\eta$ is subject to the similarity
principle.  For technical reasons, we will need to allow $\varphi$ and
$\eta$ in the statement to have only finitely-many derivatives, which forces
non-smooth bundles with non-smooth Cauchy-Riemann type operators into the
picture.

For convenience, we shall
denote elements of the bundle $N_u \to \Sigma$ as pairs $(z,w)$ where
$z \in \Sigma$ and $w$ belongs to the fiber $(N_u)_z$ over~$z$.  The zero-section
thus consists of all pairs of the form $(z,0)$, and there are canonical isomorphisms
\begin{equation}
\label{eqn:split}
T_{(z,0)}N_u = T_z \Sigma \oplus (N_u)_z
\end{equation}
due to the natural identification of $\Sigma$ with the zero-section and of
vertical tangent spaces with fibers of~$N_u$.  Given a map
$\varphi : \Sigma' \to \Sigma$, a section $\eta$ of the induced bundle
$\varphi^*N_u \to\Sigma'$ can now be written in the form $\eta(z) = (\varphi(z),f(z)) \in N_u$
with $f(z) \in (N_u)_{\varphi(z)}$ for $z \in \Sigma'$.

\begin{prop}
\label{prop:pushoff}
Suppose $u : (\Sigma,j) \to (M,J)$ and $v : (\Sigma',j') \to (M,J)$ are smooth 
$J$-holomorphic curves, $\varphi : \Sigma' \to \Sigma$
is a diffeomorphism of class $C^k$ for some $k \in \NN \cup \{\infty\}$,
$N_u \subset u^*TM$ is the generalized normal bundle
of $u$ in the sense of \S\ref{sec:normal},
$\oO \subset N_u$ is an open neighborhood of the zero-section, and
$$
\Psi : \oO \to M
$$
is a smooth map that satisfies
$$
\Psi(z,0) = u(z) \quad \text{ and }\quad
d\Psi(z,0) X = X \quad\text{ for all $z \in \Sigma$ and $X \in (N_u)_z$},
$$
where the second condition makes sense due to the canonical splitting \eqref{eqn:split}
and the inclusion $(N_u)_z \subset T_{u(z)}M$.
If $\eta : \Sigma' \to \varphi^*N_u$ is a section of class $C^k$ of
the bundle $\varphi^*N_u \to \Sigma'$ with image in $\varphi^*\oO$ such that
$$
v(z) = \Psi(\varphi(z),\eta(z)) \quad\text{ for all $z \in \Sigma'$},
$$
then $\eta$ satisfies $\mathbf{D}\eta = 0$ for some real-linear
Cauchy-Riemann type operator $\mathbf{D}$ of class $C^{k-1}$ on the
bundle $(\varphi^*N_u,J) \to (\Sigma',j')$.
\end{prop}
\begin{proof}
Choose connections on the bundles $TM$ and $N_u$.  The induced bundle
$\varphi^*N_u$ is of class~$C^k$, and the connection on $N_u$ induces a
connection on $\varphi^*N_u$ of class $C^{k-1}$ (cf.~Remark~\ref{remark:onlyk-1}), 
whose covariant derivative operator we will denote by~$\nabla$.  For $(z,w) \in \oO$, let
$$
P_{(z,w)} : T_{u(z)}M \to T_{\Psi(z,w)}M
$$
denote the isomorphism defined via parallel transport along the path
$[0,1] \to M : \tau \mapsto \Psi(z,\tau w)$.  The connection on $N_u$ also
determines natural isomorphisms
\begin{equation}
\label{eqn:connSplit}
T_{(z,w)}N_u = T_z\Sigma \oplus (N_u)_z
\end{equation}
for each $(z,w) \in N_u$, where the two factors correspond to the horizontal
and vertical subspaces respectively.  We can then associate to each
$(z,w) \in \oO$ the linear map
$$
F(z,w) := P_{(z,w)}^{-1} \circ d\Psi(z,w) \in 
\Hom_\RR\big( T_z \Sigma \oplus (N_u)_z , T_{u(z)}M\big),
$$
which depends smoothly on $(z,w) \in \oO$ and satisfies
$$
F(z,0) = du(z) \oplus \1.
$$
If we fix $z \in \Sigma$, then $F(z,\cdot)$ is a smooth map from $\oO_z := \oO \cap (N_u)_z$
to a fixed vector space of linear maps, and thus satisfies
\begin{equation*}
\begin{split}
F(z,w) &= F(z,0) + \int_0^1 \frac{d}{d\tau} F(z,\tau w)\, d\tau = F(z,0) + 
\left( \int_0^1 d_2 F(z,\tau w)\, d\tau \right) w\\
&= \left( du(z) \oplus \1 \right) + \widetilde{F}(z,w) w,
\end{split}
\end{equation*}
where the integral at the end of the first line is used to define a smooth family
of linear maps
$$
\widetilde{F}(z,w) : (N_u)_z \to \Hom_\RR\big( T_z \Sigma \oplus (N_u)_z , T_{u(z)}M\big)
$$
parametrized by $(z,w) \in \oO$.

Similarly, we associate to each $(z,w) \in \oO$ another linear map
$$
G(z,w) := P_{(z,w)}^{-1} \circ J(\Psi(z,w)) \circ P_{(z,w)} \in
\End_\RR(T_{u(z)}M),
$$
which again depends smoothly on $(z,w) \in \oO$ and has image in a fixed
vector space if $z$ is fixed.  We then have
\begin{equation*}
\begin{split}
G(z,w) &= G(z,0) + \int_0^1 \frac{d}{d\tau} G(z,\tau w)\, d\tau = G(z,0) + 
\left( \int_0^1 d_2 G(z,\tau w)\, d\tau \right) w\\
&= J(u(z)) + \widetilde{G}(z,w) w,
\end{split}
\end{equation*}
where the integral in the first line defines
$$
\widetilde{G}(z,w) : (N_u)_z \to \End_\RR(T_{u(z)}M),
$$
another family of linear maps with smooth dependence on the parameter $(z,w) \in \oO$.

Now suppose $v : (\Sigma',j') \to (M,J)$ is $J$-holomorphic and
$v(z) = \Psi(\varphi(z),\eta(z))$, where $\eta : \Sigma' \to \varphi^*N_u$ is 
a $C^k$-smooth section with image in~$\varphi^*\oO$.  Using the splitting \eqref{eqn:connSplit}
determined by the connection on~$N_u$, we have for each $z \in \Sigma'$,
$$
dv(z) = d\Psi(\varphi(z),\eta(z)) \circ (d\varphi(z),\nabla \eta(z)) =
P_{(\varphi(z),\eta(z))} \circ F(\varphi(z),\eta(z)) \circ (d\varphi(z),\nabla\eta(z)).
$$
The parallel transport isomorphisms $P_{(\varphi(z),\eta(z))} : T_{u(\varphi(z))}M \to T_{v(z)}M$ 
now define a $C^k$-smooth real-linear bundle isomorphism $(u \circ \varphi)^*TM \to v^*TM$, and applying its
inverse to the nonlinear Cauchy-Riemann equation
$dv + J(v) \circ dv \circ j' = 0$ gives an equation for
real-linear bundle maps $T\Sigma \to \varphi^*u^*TM$,
\begin{equation*}
\begin{split}
0 &= F(\varphi,\eta) \circ (d\varphi,\nabla \eta) + G(\varphi,\eta) \circ F(\varphi,\eta) \circ (d\varphi \circ j',\nabla\eta \circ j') \\
&= \left[ (du(\varphi) \oplus \1) + \widetilde{F}(\varphi,\eta) \eta \right] \circ (d\varphi,\nabla\eta) \\
&\qquad +
\left[ J(u(\varphi)) + \widetilde{G}(\varphi,\eta) \eta \right] \circ \left[ (du(\varphi) \oplus \1) + \widetilde{F}(\varphi,\eta) \eta \right] \circ (d\varphi \circ j',\nabla \eta \circ j') \\
&= \left[d(u\circ \varphi) + J(u \circ \varphi) \circ d(u \circ \varphi) \circ j'\right] + \left[\nabla \eta + J(u \circ \varphi) \circ \nabla\eta \circ j'\right] \\
&\qquad + \left[ \widetilde{F}(\varphi,\eta) \eta \right] \circ (d\varphi,\nabla \eta) \\
&\qquad + 
\left[\widetilde{G}(\varphi,\eta) \eta \right] \circ \left[ d(u \circ \varphi) \circ j' + \nabla \eta \circ j' +
\left( \widetilde{F}(\varphi,\eta) \eta \right) \circ (d\varphi \circ j' , \nabla \eta \circ j') \right]\\
&\qquad +
J(u \circ \varphi) \circ \left[ \widetilde{F}(\varphi,\eta) \eta \right] \circ (d\varphi \circ j' , \nabla\eta \circ j') .
\end{split}
\end{equation*}
Since $\eta$ and $\varphi$ are of class~$C^k$, all terms in this expression are at least
$C^{k-1}$-smooth functions of~$z$, and each term in the last three lines can be
understood as a product of $C^{k-1}$-smooth sections of various bundles, at least
one of which is always of the form $B\eta$ for a $C^{k-1}$-smooth linear bundle map
$B$ from $\varphi^*N_u$ to some other bundle, e.g.~we have $B(z) = \widetilde{F}(\varphi(z),\eta(z))$
in the first of these three lines and $B(z) = \widetilde{G}(\varphi(z),\eta(z))$ in the second.
We can therefore abbreviate the last three lines as
$\widehat{A} \eta$ for some $C^{k-1}$-smooth bundle map $\widehat{A} : 
\varphi^*N_u \to \Hom_\RR(T\Sigma,\varphi^*u^*TM)$, 
so that the entire equation becomes
$$
\left[d(u \circ \varphi) + J(u \circ \varphi) \circ d(u \circ \varphi) \circ j'\right] + 
\left[\nabla \eta + J(u \circ \varphi) \circ \nabla\eta \circ j'\right]
 + \widehat{A} \eta = 0.
$$
Finally,
let $\pi_N : u^*TM \to N_u$ denote the smooth bundle map defined by projecting
$u^*TM = T_u \oplus N_u$ along~$T_u$, which induces a $C^k$-smooth bundle map
$$
\pi_N : \varphi^*u^*TM \to \varphi^*N_u.
$$
In terms of the splitting $\varphi^*u^*TM = \varphi^*T_u \oplus \varphi^*N_u$,
the term
$d(u \circ \varphi) + J(u \circ \varphi) \circ d(u \circ \varphi) \circ j'$ 
in the above expression has image in~$\varphi^*T_u$,
while $\nabla \eta + J(u \circ \varphi) \circ \nabla \eta \circ j'$ has image 
in~$\varphi^*N_u$.
Applying $\pi_N$ to the whole equation thus gives rise to 
$$
\nabla \eta + J(u) \circ \nabla\eta \circ j' + 
\pi_N\widehat{A} \eta = 0.
$$
This is not quite yet a Cauchy-Riemann type equation; for this we would need
the target of the bundle map $\pi_N\widehat{A}$ to be the bundle of complex-antilinear maps
$\overline{\Hom}_\CC(T\Sigma,\varphi^*N_u)$, whereas $\pi_N\widehat{A}$ sends
$N_u$ to the larger bundle $\Hom_\RR(T\Sigma,\varphi^*N_u)$.
We can fix this simply by taking the complex-antilinear part, i.e.~we define
\begin{equation*}
\begin{split}
\varphi^*N_u &\stackrel{A}{\to} \overline{\Hom}_\CC\big( (T\Sigma,j') , (\varphi^*N_u,J) \big),\\
A w &:= \frac{1}{2} \left( \pi_N\widehat{A} w + J \circ \pi_N\widehat{A}w \circ j' \right).
\end{split}
\end{equation*}
Since $\pi_N\widehat{A} \eta = - \nabla \eta - J(u \circ \varphi) \circ \nabla\eta \circ j'$ and
the latter is manifestly complex antilinear, we have $A\eta = \pi_N\widehat{A}\eta$,
proving that $\eta$ also satisfies the Cauchy-Riemann type equation
$$
\nabla \eta + J(u \circ \varphi) \circ \nabla\eta \circ j' + A \eta = 0.
$$
\end{proof}

\subsection{Local coordinates}
\label{sec:localCoords}

For the rest of this section, we focus explicitly on the
situation described in the statement of Theorem~\ref{thm:representation}.
Our first objective is to find suitable coordinate charts near $z_0 \in \Sigma$
and $u(z_0) \in M$ so that $u$ near $z_0$ becomes a map of the form
$$
\DD_\rho \to \CC \times \CC^n : z \mapsto (z^k,\widehat{u}(z))
$$
for some $k \in \NN$ with $\widehat{u}(z) = O(|z|^{k+1})$.
In light of our discussion of the
generalized tangent bundle $T_u \subset u^*TM$ in \S\ref{sec:normal}, it should
be clear that the complex subspace $L \subset T_{u(z_0)}M$ mentioned in the 
theorem will be
$$
L = (T_u)_{z_0}.
$$
For the smooth coordinates near $u(z_0)$ on~$M$, we impose the following conditions,
which depend on the point $u(z_0) \in M$ and the subspace $(T_u)_{z_0}$,
but not otherwise on the map $\Sigma \stackrel{u}{\to} M$:
\begin{enumerate}
\item
The point $u(z_0) \in M$ is identified with $0 \in \CC^n$;
\item
The complex subspace $L \subset T_{u(z_0)}M$ is identified with
$\CC \times \{0\} \subset \CC^n$;
\item
The map $u_0(z) := (z,0) \in \CC \times \CC^{n-1}$
is $J$-holomorphic on $\DD_\rho$ for sufficiently small $\rho > 0$,
and $J$ along the image of this map is identified with the standard
complex structure $i$ on~$\CC^n$.
\end{enumerate}
Note that while the first two conditions are easy to achieve,
the third is highly nontrivial.  It is possible due to
the standard local existence result for $J$-holomorphic curves with
a fixed tangent vector---the latter follows from the implicit function
theorem after performing a local rescaling argument to view 
$\dbar_J$ as a small perturbation of the surjective linear 
operator~$\dbar$, see e.g.~\cite{Wendl:lectures}*{Chapter~2}
or \cite{Sikorav}*{Theorem~3.1.1}.  After choosing a suitable $J$-holomorphic
disk $\DD_\rho \hookrightarrow M$ through $u(z_0)$, one can construct
the desired coordinates by exponentiating in complex normal directions from this
disk.  With this understood, for the rest of this section we shall fix
a choice of holomorphic coordinates near $z_0 \in \Sigma$ and smooth coordinates
near $u(z_0) \in M$ as described above in order to assume
$(\Sigma,j) = (\DD_\rho,i)$ with $z_0=0 \in \DD_\rho$
for some $\rho > 0$, while $J$ is a smooth almost
complex structure on $\CC^n = \CC \times \CC^{n-1}$ with $J(z,0)=i$ for all~$z \in\DD_\rho$,
and $u : (\DD_\rho,i) \to (\CC^n,J)$
is a $J$-holomorphic curve with $u(0)=0$ and generalized tangent
space $(T_u)_0 = \CC \times \{0\}$.

We next seek a $C^1$-smooth coordinate
change near the origin on the domain of $u$ so that it becomes a map of the form
$z \mapsto (z^k,O(|z|^{k+1}))$.  We start with the observation
that $u : (\DD_\rho,i) \to (\CC^n,J)$ itself satisfies the smooth complex-linear case of the
similarity principle: indeed, the nonlinear Cauchy-Riemann equation
$$
\p_s u(z) + J(u(z))\, \p_t u(z) = 0
$$
can be interpreted as a smooth complex-linear Cauchy-Riemann type equation $\mathbf{D}u = 0$
on the trivial rank~$n$ complex vector bundle over $\DD_\rho$ with complex
structure $\widebar{J}(z) := J(u(z))$.  As a consequence, 
Theorem~\ref{thm:similarity2} gives
$$
u(z) = \Phi(z) f(z)
$$ 
on $\DD_\rho$ after possibly shrinking $\rho > 0$, where
$\Phi : \DD_\rho \to \GL(2n,\RR)$ is the inverse of a smooth complex
local trivialization and thus satisfies
$\Phi(z) \circ i = J(u(z)) \circ \Phi(z)$, while $f : \DD_\rho \to \CC^n$
is a holomorphic function with $f(0)=0$.  Since $J(0) = i$, we can assume 
without loss of generality that $\Phi(0) = \1$.  The assumption that $u$ is
not constant near $z_0$ implies in turn that $f$ is nontrivial
and thus satisfies
$$
f(z) = z^k g(z)
$$
for some $k \in \NN$ and a holomorphic function $g : \DD_\rho \to \CC^n$
with $g(0) \ne 0$.  By Corollary~\ref{cor:similarity} and Remark~\ref{remark:similarity}, we can identify $k$
as the degree of the lowest-order nontrivial term in the Taylor series of
$u$ at $z=0$; equivalently, $k-1$ is the vanishing order of
$du \in \Gamma(\Hom_\CC(T\Sigma,T_u))$ at $z=0$, also known as the critical
order of~$u$ at this point.
The assumption $L = \CC \times \{0\}$ now implies that after 
a complex-linear
coordinate change on the domain, we may assume $g(0) = (1,0) \in \CC \times
\CC^{n-1}$.  Thus $f(z) = (z^k g_1(z) , z^{k+1} g_2(z))$ on $\DD_\rho$ for some holomorphic
functions $g_1 : \DD_\rho \to \CC$ and $g_2 : \DD_\rho \to \CC^{n-1}$, with
$g_1(0) = 1$.  Let us use the splitting $\CC^n = \CC \times \CC^{n-1}$ to 
write $\Phi(z)$ in block form as
$$
\Phi(z) =
\begin{pmatrix}
\1 + \alpha(z) & \beta(z) \\
\gamma(z) & \1 + \delta(z)
\end{pmatrix},
$$
where the blocks $\alpha(z)$, $\beta(z)$, $\gamma(z)$ and $\delta(z)$ are all
regarded as \emph{real}-linear maps between complex vector spaces, and all of them
vanish at $z=0$ since $\Phi(0) = \1$.  Now $u(z)$ takes the form 
$(u_1(z),u_2(z)) \in \CC \times \CC^{n-1}$, where
\begin{equation}
\label{eqn:u1u2}
\begin{split}
u_1(z) &= z^k g_1(z) + \alpha(z) z^k g_1(z) + \beta(z) z^{k+1} g_2(z),\\
u_2(z) &= \gamma(z) z^k g_1(z) + (\1 + \delta(z)) z^{k+1} g_2(z).
\end{split}
\end{equation}
We claim that after shrinking $\rho > 0$ further if necessary, there
exists a $C^1$-smooth function $\xi : \DD_\rho \to \CC$ such that
$\xi(0) = 0$, $d\xi(0) = \1$ and $[\xi(z)]^k = u_1(z)$.  Indeed,
the desired function can be written for $z \ne 0$ as
$$
\xi(z) = z \sqrt[k]{g_1(z) + \frac{1}{z^k}\alpha(z) z^k g_1(z) + \frac{1}{z^k} \beta(z) z^{k+1} g_2(z)} ,
$$
where the expression under the root lies in a neighborhood of $g_1(0)=1$
for $z$ near~$0$, hence the root is uniquely defined as a continuous function
of $z$ on $\DD_\rho$ if we set $\sqrt[k]{1} := 1$.  It is clear that $\xi$ is also
smooth for $z \ne 0$, and differentiable at $z=0$ with $d\xi(0) = \1$.
Moreover, the fact that $\alpha$ and $\beta$ are smooth functions vanishing
at $z=0$ implies that both are $O(|z|)$, so that the first derivative of the expression under the
root is bounded on $\DD_\rho \setminus \{0\}$.  This is enough information
to prove that $d\xi$ is also continuous at $z=0$, so $\xi$ is of class~$C^1$.

Denote the inverse of the local $C^1$-diffeomorphism $z \mapsto \xi(z)$ by
$$
\varphi := \xi^{-1} : \DD_\rho \to \DD,
$$
where we can again shrink $\rho > 0$ if necessary to make sure that
$\varphi$ is well defined and has image contained in the domain of~$u$.
The composition $u \circ \varphi$ is then well defined and satisfies
$$
u \circ \varphi(z) = (z^k,\widehat{u}(z)),
$$
where $\widehat{u} := u_2 \circ \varphi : \DD_\rho \to \CC^{n-1}$ is a $C^1$-smooth function
satisfying the relation $\widehat{u}(\xi(z)) = u_2(z)$.

\begin{lemma}
\label{lemma:uhato}
The function $\widehat{u} \in C^1(\DD_\rho,\CC^{n-1})$ satisfies
$\widehat{u}(z) = O(|z|^{k+1})$ and $d\widehat{u}(z) = O(|z|^{k})$.
\end{lemma}
\begin{proof}
We have $u_2(z) = O(|z|^{k+1})$ by \eqref{eqn:u1u2} since $\gamma(z)$ is a
smooth function with $\gamma(0)=0$, so in particular the first
nontrivial term in the Taylor series of $u_2$ about $z=0$ has degree at least
$k+1$, implying a similar conclusion for $du_2$ and thus $du_2(z) = O(|z|^{k})$.
The conditions
$\varphi(0) = 0$ and $d\varphi(0) = \1$ imply also that
$\varphi(z) = z + o(|z|)$.  Writing $u_2(z) = |z|^{k+1} B(z)$ 
for a bounded function $B(z)$ near $z=0$ and
$\varphi(z) = z + |z| \cdot r(z)$ for a remainder function with
$\lim_{z \to 0} r(z) = 0$, we find
\begin{equation*}
\begin{split}
\widehat{u}(z) &= u_2(\varphi(z)) = u_2(z + |z| \cdot r(z)) =
\big|z + |z| \cdot r(z)\big|^{k+1} B(z + |z| \cdot r(z)) \\
&= |z|^{k+1} \cdot \left| \frac{z}{|z|} + r(z)\right|^{k+1} B(z + |z| \cdot r(z)) = O(|z|^{k+1}).
\end{split}
\end{equation*}
Similarly, $d\widehat{u}(z) = du_2(\varphi(z)) \circ d\varphi(z)$, where 
$d\varphi$ is continuous and therefore bounded near $z=0$, and the same argument as above gives
$du_2(\varphi(z)) = O(|z|^{k})$ since $du_2(z) = O(|z|^{k})$, so the
result for $d\widehat{u}(z)$ follows.
\end{proof}

Now if $v : (\Sigma',j') \to (\CC^n,J)$ is a second $J$-holomorphic curve with a
point $\zeta_0 \in \Sigma'$ such that $v(\zeta_0) = u(z_0) = 0$,
$(T_v)_{\zeta_0} = (T_u)_{z_0} = \CC \times \{0\}$ and the critical orders at
$v(\zeta_0)$ and $u(z_0)$ match, then we can repeat the same argument to
find a $C^1$-smooth local diffeomorphism $\psi$ from $\DD_\rho$ to a
neighborhood of $\zeta_0$ in $\Sigma'$ sending $0 \mapsto \zeta_0$ such that
$$
v \circ \psi(z) = (z^k,\widehat{v}(z)),
$$
with
$$
\widehat{v} \in C^1(\DD_\rho,\CC^{n-1}) \quad\text{ such that }\quad
\widehat{v}(z) = O(|z|^{k+1}) \text{ and }
d\widehat{v}(z) = O(|z|^{k}).
$$
The main goal for the rest of this section
is to prove that the $C^1$-smooth function
$$
h(z) = (0,\widehat{h}(z)) := (0,\widehat{v}(z) - \widehat{u}(z)) = v \circ \psi(z)
- u \circ \varphi(z)
$$
is either identically zero or
satisfies the formula $h(z) = z^\ell C + o(|z|^\ell)$ for some
$C \in \CC^n \setminus \{0\}$ and $\ell > k$.

\begin{remark}
\label{remark:nonstandardCpx}
It should be emphasized that $\varphi$ and $\psi$ are in general neither
holomorphic nor smooth,
so $u \circ \varphi$ and $v \circ \psi$ are pseudoholomorphic curves of
class $C^1$ with respect to complex structures on $\DD_\rho$ that 
are nonstandard, and continuous but not generally smooth, though since
$d\varphi(0)=d\psi(0)=\1$ and $\varphi$ and $\psi$ are smooth outside the
origin, both complex structures are standard at the origin and smooth elsewhere.
As a special case, however, we could take $v$ to be
$$
v : (\DD_\rho,i) \to (\CC \times \CC^{n-1},J) : z \mapsto (z^k,0),
$$
which is $J$-holomorphic due to the third condition imposed on our local
coordinates in~$M$.  The claim in Theorem~\ref{thm:representation}
that $\widehat{u}$ and $\widehat{v}$ each
satisfy formulas of the form $z^\ell C + o(|z|^\ell)$ will thus
follow as a special case of the general formula for 
$\widehat{v}-\widehat{u}$.
\end{remark}

\subsection{Constructing the normal push-off}

We now define a neighborhood $\oO \subset N_u$ and a map $\Psi : \oO \to M = \CC^n$
as in Proposition~\ref{prop:pushoff}.  Our first task is to specify a concrete
complex subbundle $N_u \subset u^*TM$ complementary to~$T_u$.  Since
$(T_u)_0 = \CC \times \{0\}$, any complex subbundle that matches
$\{0\} \times \CC^{n-1}$ at $0$ will do if we are willing to shrink
$\rho > 0$, as the two subbundles will necessarily be transverse on
$\DD_\rho$ for $\rho$ sufficiently small.  Let
$$
e_1,\ldots,e_n \in \CC^n
$$
denote the standard complex basis of~$\CC^n$, and for each 
$w = (x_2 + iy_2,\ldots,x_n + i y_n) \in \CC^{n-1}$,
define a smooth vector field on $\CC^n$ by
$$
X_w(p) := \sum_{j=2}^n \left( x_j e_j + y_j J(p) e_j \right).
$$
Since $J(z,0) = i$ for $(z,0) \in \DD_\rho \times \CC^{n-1}$, at such points we have $X_w(z,0) = (0,w) \in \CC^n$ 
for all $w \in \CC^{n-1}$.
We shall regard $N_u \to \DD_\rho$ in the following as the pullback
along $u : \DD_\rho \to \CC^n$ of the smooth subbundle of $T\CC^n$
spanned by the vector fields $X_w$ for all $w \in \CC^{n-1}$.  This bundle
comes equipped with a global trivialization
\begin{equation}
\label{eqn:trivNu}
N_u \to \DD_\rho \times \CC^{n-1} : X_w(u(z)) \mapsto (z,w).
\end{equation}

For a constant $\delta > 0$, we define the open set
$$
\oO_\delta := \left\{ (z,w) \ \big|\ |w| < \delta \right\} \subset
\DD_\rho \times \CC^{n-1}
$$
and smooth map
$$
\Psi : \oO_\delta \to \CC^n: (z,w) \mapsto u(z) + X_w(u(z)).
$$
In light of the trivialization \eqref{eqn:trivNu}, we can equivalently regard
$\oO_\delta$ as a neighborhood of the zero-section in~$N_u$ on which $\Psi$ is defined
as in Proposition~\ref{prop:pushoff}.  Notice that $\Psi(\varphi(z),0) = u \circ \varphi(z) =
(z^k,\widehat{u}(z))$.  The goal is to apply Proposition~\ref{prop:pushoff} to the following
construction:

\begin{lemma}
\label{lemma:C1pushoff}
Choosing $\delta > 0$ sufficiently small and then
shrinking $\rho > 0$ further if necessary, there exist $C^1$-smooth
functions $\theta : \DD_\rho \to \CC$ and $\eta : \DD_\rho \to \CC^{n-1}$ such that
$\theta(0) = 0$, $d\theta(0) = \1$, $\eta(z) = O(|z|^{k+1})$, and
$$
v \circ \psi(z) = \Psi(\varphi \circ \theta(z),\eta(z))\quad \text{ for all }\quad z \in \DD_\rho.
$$
\end{lemma}

The proof of this lemma requires some preparation.  We will use the notation
$d_1$ and $d_2$ to denote the
differentials of $\Psi$ or $X_w$ with respect to the first variable $z \in \CC$ or
second variable $w \in \CC^{n-1}$ respectively, e.g.~writing
$$
d_1 \Psi(z,w) \in\Hom_\RR(\CC,\CC^n), \qquad
d_2 \Psi(z,w) \in \Hom_\RR(\CC^{n-1},\CC^n).
$$
Let us also write
$$
\Psi(z,w) =: (\widecheck{\Psi}(z,w),\widehat{\Psi}(z,w)) \in \CC \times \CC^{n-1}
\quad\text{ and }\quad
X_w(p) = (\widecheck{X}_w(p),\widehat{X}_w(p)) \in \CC \times \CC^{n-1},
$$ 
so $\widecheck{\Psi}(z,w) = u_1(z) + \widecheck{X}_w(u(z))$ and
$\widehat{\Psi}(z,w) = u_2(z) + \widehat{X}_w(u(z))$.

\begin{lemma}
\label{lemma:PsiEstimates}
Given a compact region $K \subset \CC \times \CC^{n-1}$, there exists a
constant $C > 0$ such that the following estimates hold for all
$(z,p) \in K$ and all $w \in \CC^{n-1}$:
$$
\left| \widecheck{X}_w(z,p) \right| \le C|p| \cdot |w|,\qquad
\left| \widehat{X}_w(z,p) - w \right| \le C|p| \cdot |w|, \qquad
\left| d_1 X_w(z,p) \right| \le C|p| \cdot |w|.
$$
\end{lemma}
\begin{proof}
For each $(z,p) \in \CC \times \CC^{n-1}$, $w \mapsto \widecheck{X}_w(z,p)$
defines a real-linear map $\widecheck{X}(z,p) : \CC^{n-1} \to \CC$.
Since $J(z,0) = i$ for all~$z$, we have $X_w(z,0) = (0,w)$, thus
$\widecheck{X}(z,0) = 0$, and the smoothness of $\widecheck{X}_w(z,p)$ with
respect to $z$ and $p$ then gives rise to an estimate
$$
| \widecheck{X}(z,p) | \le C|p|
$$
from which the first estimate above follows.  The second estimate follows
in the same manner since $\widehat{X}(z,0)$ is the identity map
$\1 : \CC^{n-1} \to \CC^{n-1}$, hence $|\widehat{X}(z,p) - \widehat{X}(z,0)|
\le C|p|$.  For the third estimate, one observes that
$w \mapsto d_1 X_w(z,p)$ is also a real-linear map $\CC^{n-1}\to
\Hom_\RR(\CC,\CC^n)$ for every $(z,p)$ and $d_1 X_w(z,0) = 0$ since
$X_w(z,0)$ is independent of~$z$, so the same argument applies.
\end{proof}

Due to the coordinate choices made in \S\ref{sec:localCoords}, we also have
$u(z) = (u_1(z),u_2(z)) = (z^k,0) + O(|z|^{k+1})$, thus $|u_1(z)| \ge c |z|^k$
and $|u_2(z)| \le C|z|^{k+1}$ for some constants $c,C>0$, where we are free to
assume $C$ is the same constant as in Lemma~\ref{lemma:PsiEstimates}.
It follows that for $z \ne 0$,
\begin{equation*}
\begin{split}
\big|\widecheck{\Psi}(z,w)\big| &= \big|u_1(z) + \widecheck{X}_w(u(z))\big| \ge
|u_1(z)| - \big|\widecheck{X}_w(u_1(z),u_2(z))\big| \ge c |z|^k - C|u_2(z)| \cdot |w| \\
&\ge c |z|^k - C^2 |z|^{k+1} |w| = |z|^k \left( c - C^2|w|\cdot |z| \right),
\end{split}
\end{equation*}
which is positive if $|w| < c / \rho C^2$.  This proves:

\begin{lemma}
\label{lemma:avoidZero}
If $\delta > 0$ is sufficiently small, then $\Psi$ preserves the subset
$\{(z,w) \in \CC \times \CC^{n-1}\ |\ z \ne 0 \}$.
\qed
\end{lemma}

Now consider the $C^1$-smooth function $\Psi_1 = (\widecheck{\Psi}_1,\widehat{\Psi}_1) : \oO_\delta \to \CC \times \CC^n$
defined by
\begin{equation*}
\begin{split}
\Psi_1(z,w) &= \Psi(\varphi(z),w) = u(\varphi(z)) + X_w(u(\varphi(z)) \\
&= \left(z^k + \widecheck{X}_w(z^k,\widehat{u}(z)) , \widehat{u}(z) + \widehat{X}_w(z^k,\widehat{u}(z)) \right),
\end{split}
\end{equation*}
and extend this to a $C^1$-smooth family of maps 
$\Psi_\varepsilon = (\widecheck{\Psi}_\varepsilon,\widehat{\Psi}_\varepsilon) : \oO_\delta \to \CC \times \CC^{n-1}$ 
for $0 < \varepsilon \le 1$ by
\begin{equation*}
\begin{split}
\Psi_\varepsilon(z,w) &= \left( \frac{\widecheck{\Psi}_1(\varepsilon z,w)}{\varepsilon^k} , \widehat{\Psi}_1(\varepsilon z,w) \right) \\
&= \left( z^k + \frac{\widecheck{X}_w(\varepsilon^k z^k, \widehat{u}(\varepsilon z))}{\varepsilon^k} , \widehat{u}(\varepsilon z) + \widehat{X}_w(\varepsilon^k z^k, \widehat{u}(\varepsilon z)) \right).
\end{split}
\end{equation*}
We would like to understand what happens to $\Psi_\varepsilon$ as $\varepsilon \to 0$,
but from a slightly different vantage point, namely after transforming the
first complex variable in $\CC \times \CC^{n-1}$ to holomorphic cylindrical
coordinates.  Define the biholomorphic map
$$
f : \RR \times S^1 \stackrel{\cong}{\longrightarrow} \CC \setminus \{0\} : (s,t) \mapsto e^{2\pi(s+it)},
$$
let
$$
\dot{\oO}_\delta := \left\{ (s,t,w) \in \RR \times S^1 \times \CC^{n-1}\ \big|\ 
(f(s,t),w) \in \oO_\delta \right\},
$$
and consider the family of $C^1$-smooth maps
$$
\Psi'_\varepsilon := (\widecheck{\Psi}'_\varepsilon,\widehat{\Psi}'_\varepsilon) := (f^{-1} \times \Id) \circ \Psi_\varepsilon \circ (f \times \Id) :
\dot{\oO}_\delta \to \RR \times S^1 \times \CC^{n-1},
$$
which are given by
$$
\widecheck{\Psi}'_\varepsilon(s,t,w) = f^{-1} \circ \widecheck{\Psi}_\varepsilon(e^{2\pi(s+it)},w) \quad\text{ and }\quad
\widehat{\Psi}'_\varepsilon(s,t,w) = \widehat{\Psi}_\varepsilon(e^{2\pi(s+it)},w).
$$
Since $\widehat{\Psi}_\varepsilon(e^{2\pi(s+it)},w) = \widehat{\Psi}_1(\varepsilon e^{2\pi(s+it)},w)$,
the functions $\widehat{\Psi}'_\varepsilon : \dot{\oO}_\delta \to \CC^{n-1}$ converge in $C^1$ to 
$\widehat{\Psi}'_0(s,t,w) := \widehat{\Psi}_1(0,w) = w$ as $\varepsilon \to 0$.
The convergence of $\widecheck{\Psi}'_\varepsilon : \dot{\oO}_\delta \to \RR \times S^1$
as $\varepsilon \to 0$ will be deduced from the next lemma.
To motivate the hypotheses in this statement, notice that the required estimates 
are satisfied automatically by any \emph{smooth} function $g(z,w)$ that is of 
the form $a z^k$ plus terms that are higher order in~$z$; the formulation
below is only more complicated than this because we need to allow functions
that are of class $C^1$ and not smooth.

\begin{lemma}
\label{lemma:IhateThisLemma}
Fix $r \in \RR$, $\rho := e^{2\pi r} > 0$ and an open set $\uU \subset \RR^n$, 
and suppose 
$$
g : \DD_\rho \times \uU \to \CC
$$ 
is a function of class $C^1$
satisfying $g(z,w) \ne 0$ for all $z \ne 0$, along with estimates of the form
\begin{align*}
|g(z,w) - a z^k| &\le C|z|^{k+1}, &
\left| d_2 g(z,w) \right| &\le C|z|^{k+1}, \\
\left|\frac{\p g}{\p z}(z,w) - k a z^{k-1}\right| &\le C|z|^k,
& \left|\frac{\p g}{\p \bar{z}}(z)\right| &\le C|z|^k
\end{align*}
for a constant $C > 0$ independent of $(z,w) \in \DD_\rho \times \uU$,
where $a \in \CC \setminus \{0\}$ and $k \in \NN$ are constants, and $d_2 g(z,w) : \RR^n \to \CC$ 
denotes the differential with respect to the second variable $w \in \uU$.
Using the biholomorphic
map $f : \RR \times S^1 \stackrel{\cong}{\to} \CC \setminus \{0\} : 
(s,t) \mapsto e^{2\pi(s+it)}$, define for each $\varepsilon \in (0,1]$ the maps
$g_\varepsilon : \DD_\rho \times \uU \to \CC$ and
$g_\varepsilon' : (-\infty,r] \times S^1 \times \uU \to \RR \times S^1$ by
$$
g_\varepsilon(z,w) := \frac{g(\varepsilon z,w)}{\varepsilon^k},\quad\text{ and }\quad
g_\varepsilon'(s,t,w) := f^{-1} \circ g_\varepsilon(f(s,t),w).
$$
Then as $\varepsilon \to 0$, the maps $g_\varepsilon'$ are $C^1$-convergent 
on $(-\infty,r] \times S^1 \times \uU$ to
$$
g_0'(s,t,w) := (ks + s_0,kt + t_0),
$$
where $e^{2\pi(s_0 + i t_0)} = a$.
\end{lemma}
\begin{proof}
By assumption, we can write
\begin{align*}
g(z,w) &= a z^k + |z|^{k+1} B(z,w), &
d_2 g(z,w) &= |z|^{k+1} B_w(z,w), \\
\frac{\p g}{\p z}(z) &= k a z^{k-1} + |z|^{k} B_z(z,w),&
\frac{\p g}{\p \bar{z}}(z) &= |z|^k B_{\bar{z}}(z,w)
\end{align*}
for bounded functions $B$, $B_w$, $B_z$ and $B_{\bar{z}}$.  Then
$$
g_\varepsilon(z,w) = a z^k + \varepsilon |z|^{k+1} B(\varepsilon z,w),
$$
and on the punctured domain $\dot{\DD}_\rho \times \uU$ for
$\dot{\DD}_\rho := \DD_\rho \setminus \{0\}$,
we therefore have uniform convergence
$$
\frac{g_\varepsilon(z,w)}{z^k} = a + \varepsilon |z| \frac{|z|^{k}}{z^k} B(\varepsilon z,w) \to a \quad
\text{ as }\quad \varepsilon \to 0.
$$
We claim that this convergence is also in $C^1$ on $\dot{\DD}_\rho \times \uU$.
Indeed, we have
$$
d_2 \left(\frac{g_\varepsilon(z,w)}{z^k}\right) = \frac{1}{\varepsilon^k z^k} d_2 g(\varepsilon z,w) =
\frac{1}{\varepsilon^k z^k} \varepsilon^{k+1} |z|^{k+1} B_w(\varepsilon z,w)
= \varepsilon |z| \frac{|z|^k}{z^k} B_w(\varepsilon z,w),
$$
along with
\begin{equation*}
\begin{split}
\frac{\p}{\p z} &\left(\frac{g_\varepsilon(z,w)}{z^k}\right) =
\frac{1}{z^k} \frac{\p}{\p z} g_\varepsilon(z,w) - \frac{k}{z^{k+1}} g_\varepsilon(z,w) =
\frac{1}{\varepsilon^{k-1} z^k} \frac{\p g}{\p z}(\varepsilon z,w) - 
\frac{k}{\varepsilon^k z^{k+1}} g(\varepsilon z,w) \\
&= \frac{1}{\varepsilon^{k-1} z^k} \left[ \varepsilon^{k-1} k a z^{k-1} + \varepsilon^k |z|^k B_z(\varepsilon z,w) \right]
- \frac{k}{\varepsilon^k z^{k+1}} \left[ \varepsilon^k a z^k + \varepsilon^{k+1} |z|^{k+1} B(\varepsilon z,w) \right] \\
&= \varepsilon \left[ \frac{|z|^k}{z^k} B_z(\varepsilon z,w) - k \frac{|z|^{k+1}}{z^{k+1}} B(\varepsilon z,w) \right],
\end{split}
\end{equation*}
and
$$
\frac{\p}{\p \bar{z}} \left(\frac{g_\varepsilon(z,w)}{z^k} \right) =
\frac{1}{z^k} \frac{\p}{\p \bar{z}} g_\varepsilon(z,w) = \frac{1}{\varepsilon^{k-1} z^k}
\frac{\p g}{\p \bar{z}}(\varepsilon z,w) = \frac{1}{\varepsilon^{k-1} z^k} \varepsilon^k |z|^k B_{\bar{z}}(\varepsilon z,w)
= \varepsilon \frac{|z|^k}{z^k} B_{\bar{z}}(\varepsilon z,w).
$$
All of these converge uniformly to $0$ as $\varepsilon \to 0$.

To relate this to the maps $g_\varepsilon'$, identify $\RR \times S^1$
with $\CC / i\ZZ$ and write $f(\zeta) = e^{2\pi\zeta}$, so
$g_\varepsilon'$ is now determined by $g_\varepsilon$ according to the formula
$e^{2\pi g_\varepsilon'(\zeta,w)} = g_\varepsilon(z,w)$ for 
$z = e^{2\pi\zeta}$, implying
$$
e^{2\pi [ g_\varepsilon'(\zeta,w) - k\zeta ]} = \frac{g_\varepsilon(z,w)}{z^k}.
$$
For $\varepsilon$ small enough, the convergence established above implies that
the right hand side lies in a compact neighborhood of $a \in\CC \setminus \{0\}$
on which the holomorphic logarithm function can be defined, giving rise to the
formula
$$
g_\varepsilon'(\zeta,w) - k\zeta = \frac{1}{2\pi} \log\left( g_\varepsilon(z,w) / z^k\right).
$$
The right hand side is $C^1$-convergent to the constant $\frac{1}{2\pi}\log(a)$ when regarded as a function 
of $(z,w) \in \dot{\DD}_\rho \times \uU$, and composing it with the
transformation $(\zeta,w) \mapsto (e^{2\pi\zeta},w)$ in order
to view it as a function of $(\zeta,w) \in (-\infty,r] \times S^1 \times \uU$ does not change
this result, thus we obtain $C^1$-convergence of $g_\varepsilon'$ to
$k\zeta + \frac{1}{2\pi} \log(a)$.
\end{proof}

We would now like to feed the function
$$
\widecheck{\Psi}_1(z,w) = z^k + \widecheck{X}_w(z^k,\widehat{u}(z))
$$
into Lemma~\ref{lemma:IhateThisLemma}.  Since $\widehat{u}(z) = O(|z|^{k+1})$,
Lemma~\ref{lemma:PsiEstimates} implies
$$
\left|\widecheck{\Psi}_1(z,w) - z^k\right| \le C |w| \cdot |\widehat{u}(z)| \le
C' |w| \cdot |z|^{k+1}
$$
for some constant $C' > 0$ independent of $z$ and~$w$, and another application
of Lemma~\ref{lemma:PsiEstimates} together with the fact that $\widecheck{\Psi}_1(z,w)$
depends linearly on~$w$ gives
$$
\left|d_2 \widecheck{\Psi}_1(z,w) w' \right| = \left| \widecheck{X}_{w'}(z^k,\widehat{u}(z))\right| \le
C |\widehat{u}(z)| \cdot |w'|,
$$
hence
$$
\left|d_2 \widecheck{\Psi}_1(z,w)\right| \le C |\widehat{u}(z)| \le C'|z|^{k+1}.
$$
For the required estimates on derivatives with respect to $z$, it will suffice
to prove that the function
$$
\xi_w(z) := \widecheck{X}_w(z^k,\widehat{u}(z)) \quad\text{ satisfies }\quad
|d\xi_w(z)| \le C|z|^k
$$
for a constant $C > 0$ independent of $(z,w) \in \oO_\delta$.
We have $|d\widehat{u}(z)| = O(|z|^k)$ by Lemma~\ref{lemma:uhato} and can assume
$d_2\widecheck{X}_w(z^k,\widehat{u}(z))$ is bounded for
$(z,w) \in \oO_\delta$, so applying the third estimate in 
Lemma~\ref{lemma:PsiEstimates} gives
\begin{equation*}
\begin{split}
|d\xi_w(z)| &= \left| d_1\widecheck{X}_w(z^k,\widehat{u}(z)) \circ (k z^{k-1}) +
d_2\widecheck{X}_w(z^k,\widehat{u}(z)) \circ d\widehat{u}(z) \right| \\
&\le C |w| \cdot |\widehat{u}(z)| \cdot |z|^{k-1} + C |z|^k \le
C' |z|^{2k} + C|z|^k = O(|z|^k).
\end{split}
\end{equation*}
We can now apply Lemma~\ref{lemma:IhateThisLemma} and conclude:

\begin{lemma}
\label{lemma:thatHurt}
The maps $\Psi'_\varepsilon : \dot{\oO}_\delta \to \RR \times S^1 \times \CC^{n-1}$
are $C^1$-convergent as $\varepsilon \to 0$ to
$$
\Psi'_0(s,t,w) := (ks,kt,w).
$$
\qed
\end{lemma}

A crucial detail in Lemma~\ref{lemma:thatHurt} is that the 
$C^1$-convergence is not just on compact subsets, but remains uniform
(including first derivatives) as $s$ varies on the unbounded
half-interval $(-\infty,r]$.  We conclude from this
in particular that for all $\varepsilon > 0$ sufficiently small,
$\Psi'_\varepsilon$ is a local $C^1$-diffeomorphism whose image contains
the set 
$$
\dot{\oO}'_\delta := \left\{ (s,t,w) \in \dot{\oO}_\delta\ \big|\ |w| < \delta/2 \right\}.
$$
This is finally enough information to prove the main result of this
subsection.

\begin{proof}[Proof of Lemma~\ref{lemma:C1pushoff}]
Denote $v_1(z) = (\widecheck{v}(z),\widehat{v}(z)) := v \circ \psi(z)$, so
$\widecheck{v}(z) = z^k$.
Our objective is to find a suitable local $C^1$-diffeomorphism 
$\theta : \DD_\rho \to \CC$ sending $0 \mapsto 0$ and a
$C^1$-function $\eta : \DD_\rho \to \CC^{n-1}$ such that the relation
\begin{equation}
\label{eqn:Psi1Relation}
\Psi_1(\theta(z),\eta(z)) = v_1(z)
\end{equation}
holds if the disk $\DD_\rho$ is taken to be sufficiently small.
We will do this by applying the same rescaling and cylindrical transformations
to $\theta$, $\eta$ and $v_1$ that were applied above for $\Psi_1$, as the
existence of such functions for $\varepsilon > 0$ sufficiently small
will become obvious in cylindrical coordinates
due to the convergence $\Psi'_\varepsilon \to \Psi'_0$.

Concretely, if maps $\theta$ and $\eta$ as in \eqref{eqn:Psi1Relation} were
already known, then
for $\varepsilon \in (0,1]$, we could define
$\theta_\varepsilon : \DD_\rho \to \CC$, 
$\eta_\varepsilon : \DD_\rho \to\CC^{n-1}$ and
$v_\varepsilon : \DD_\rho \to \CC \times \CC^{n-1}$ by
$$
\theta_\varepsilon(z) := \frac{\theta(\varepsilon z)}{\varepsilon}, \qquad
\eta_\varepsilon(z) := \eta(\varepsilon z), \qquad
v_\varepsilon(z) := \left( \frac{\widecheck{v}(\varepsilon z)}{\varepsilon^k} , \widehat{v}(\varepsilon z)\right) =
(z^k , \widehat{v}(\varepsilon z)),
$$
which must then satisfy the relation
\begin{equation}
\label{eqn:PsiRelation}
\Psi_\varepsilon(\theta_\varepsilon(z),\eta_\varepsilon(z)) = v_\varepsilon(z).
\end{equation}
Transforming one step further, let us again
identify $\RR \times S^1$ with $\CC / i\ZZ$ and write
$f(\zeta) = e^{2\pi \zeta}$, $\rho = 2\pi r$.  If $\theta$ is a local
diffeomorphism sending $0 \mapsto 0$, then we can assume $\theta_\varepsilon(z) \ne 0$
for all $z \ne 0$ and $\varepsilon > 0$ sufficiently small, and can therefore define maps
$\theta'_\varepsilon : (-\infty,r] \times S^1 \to \RR \times S^1$,
$\eta'_\varepsilon : (-\infty,r] \times S^1 \to \CC^{n-1}$ and
$v'_\varepsilon : (-\infty,r]\times S^1 \to \RR \times S^1 \times \CC^{n-1}$ by
$$
\theta'_\varepsilon := f^{-1} \circ \theta_\varepsilon \circ f, \qquad
\eta'_\varepsilon := \eta_\varepsilon \circ f, \qquad
v'_\varepsilon := (f^{-1} \times \Id) \circ v_\varepsilon \circ f.
$$
This last map is of the form
$$
v'_\varepsilon(\zeta) = (k\zeta , \widehat{v}(\varepsilon e^{2\pi \zeta})),
$$
thus for $\varepsilon \to 0$ we have $C^1$-convergence
$v'_\varepsilon \to v'_0$ where
$$
v'_0(\zeta) := (k\zeta,0) = \Psi'_0(\zeta,0).
$$
The cylindrical coordinate version of \eqref{eqn:PsiRelation} is now the relation
\begin{equation}
\label{eqn:PsiRelationCyl}
\Psi'_\varepsilon(\theta'_\varepsilon(\zeta),\eta'_\varepsilon(\zeta)) =
v'_\varepsilon(\zeta),
\end{equation}
which is equivalent to \eqref{eqn:PsiRelation} for each $\varepsilon > 0$.

The discussion of $\theta$ and $\eta$ has been purely hypothetical thus far,
but we are now in a position to find actual maps $\theta'_\varepsilon$ and
$\eta'_\varepsilon$ such that \eqref{eqn:PsiRelationCyl} is satisfied.
Indeed, after shifting the upper boundary of the half-cylinder
$(-\infty,r] \times S^1$ slightly if necessary, the convergence
of local $C^1$-diffeomorphisms $\Psi'_\varepsilon \to \Psi'_0$ together with
the convergence $v'_\varepsilon \to v'_0$ implies that for every
$\zeta \in (-\infty,r] \times S^1$, there exists a unique continuous
family of points 
$(\theta'_\varepsilon(\zeta),\eta'_\varepsilon(\zeta)) \in \dot{\oO}'_\delta$ 
for $\varepsilon \ge 0$ sufficiently small such that
\eqref{eqn:PsiRelationCyl} holds and 
$(\theta_0(\zeta),\eta_0(\zeta)) = (\zeta,0)$; notice that the
$\varepsilon=0$ case of \eqref{eqn:PsiRelationCyl} is then the relation
$v'_0(\zeta) = \Psi'_0(\zeta,0)$ already established.
Since the $\Psi_\varepsilon$ are local $C^1$-diffeomorphisms for
$\varepsilon \ge 0$ small and $v_\varepsilon$ is of class~$C^1$, the maps
$\theta_\varepsilon$ and $\eta_\varepsilon$ defined in this way 
are also of class~$C^1$ and
form a $C^1$-continuous family with respect to the parameter~$\varepsilon$,
implying in particular that we have $C^1$-convergence
$\theta_\varepsilon \to \theta_0$ and $\eta_\varepsilon \to 0$ as
$\varepsilon \to 0$.  To obtain the actual objective, we only need 
fix $\varepsilon > 0$ sufficiently small and observe that
both of the transformations $(\theta,\eta) \mapsto (\theta_\varepsilon,\eta_\varepsilon)$
and $(\theta_\varepsilon,\eta_\varepsilon) \mapsto (\theta'_\varepsilon,\eta'_\varepsilon)$
described above are reversible, at least if we are willing to restrict
the domain of $\theta$ and $\eta$ to a \emph{punctured} disk
$\dot{\DD}_\rho$ whose size is reduced in proportion to the size 
of~$\varepsilon$.  After this reversal, we have a pair of
$C^1$-smooth maps $\theta : \dot{\DD}_\rho \to \dot{\CC}$ and 
$\eta : \dot{\DD}_\rho \to\CC^{n-1}$ that satisfy \eqref{eqn:Psi1Relation}
on the punctured disk~$\dot{\DD}_\rho$.

We claim that both $\theta$ and $\eta$ can be extended over the puncture
to functions of class $C^1$ on~$\DD_\rho$, with
$$
\theta(0)=0, \quad d\theta(0) = \1, \qquad\text{ and }\qquad \eta(0)=0,\quad
d\eta(0)=0.
$$
For $\theta$, we
consider the functions $g_\varepsilon(z) := \frac{\theta_\varepsilon(z)}{z}$
on $\dot{\DD}_\rho$ and observe that since $\theta'_\varepsilon$ converges
in $C^1$ on $(-\infty,r] \times S^1$ to $\theta'_0(\zeta) = \zeta$,
$$
g_\varepsilon \circ f(\zeta) = e^{2\pi \left[\theta'_\varepsilon(\zeta) - \zeta\right]}
$$
is $C^1$-convergent on $(-\infty,r] \times S^1$ to the constant function
with value~$1$.  This implies that $g_\varepsilon$ converges uniformly on
$\dot{\DD}_\rho$ to~$1$,
and writing $\theta_\varepsilon(z) = \theta(\varepsilon z) / \varepsilon$,
we obtain the relation
$$
\theta(\varepsilon z) = \varepsilon z g_\varepsilon(z) =
\varepsilon z + \varepsilon z \left[g_\varepsilon(z) - 1\right].
$$
for all $z \in \dot{\DD}_\rho$.  If we restrict this relation to points
$z$ on the boundary of $\DD_\rho$ and introduce a new variable
$w := \varepsilon z$ living in a neighborhood of $0 \in \DD_\rho$, we can
define a remainder function
$$
R(w) := \frac{w}{|w|} \left[ g_{|w|/\rho}(z) - 1 \right]
$$
that satisfies $\lim_{w \to 0} R(w) = 0$ due to the uniform convergence
of $g_\varepsilon$, and it turns the above relation into
$\theta(w) = w + |w|R(w)$.  Defining $\theta(0) := 0$ therefore makes
$\theta$ continuous and differentiable at $0$, with $d\theta(0) = \1$.

To prove that $d\theta(z)$ is also continuous at $z=0$, we use the
uniform convergence of the first derivatives of $g_\varepsilon \circ f$:
writing $z = f(\zeta) = e^{2\pi\zeta}$, this convergence implies
$$
\frac{\p}{\p\zeta} g_\varepsilon \circ f(\zeta) = \frac{\p g_\varepsilon}{\p z} \frac{\p z}{\p\zeta}
= 2\pi z \frac{\p g_\varepsilon}{\p z} \to 0, \quad\text{ and }\quad
\frac{\p}{\p\bar{\zeta}} g_\varepsilon \circ f(\zeta) = \frac{\p g_\varepsilon}{\p \bar{z}} \frac{\p\bar{z}}{\p\bar{\zeta}}
= 2\pi\bar{z} \frac{\p g_\varepsilon}{\p\bar{z}} \to 0
$$
as $\varepsilon \to 0$.  From the convergence of $\bar{z} \frac{\p g_\varepsilon}{\p\bar{z}}$,
we obtain
$$
\bar{z} \frac{\p}{\p\bar{z}} \left(\frac{\theta_\varepsilon(z)}{z}\right) =
\frac{\bar{z}}{z} \frac{\p}{\p\bar{z}} \theta_\varepsilon(z) =
\frac{\bar{z}}{z} \frac{\p \theta}{\p\bar{z}}(\varepsilon z) \to 0,
$$
implying $\lim_{z \to 0} \frac{\p\theta}{\p\bar{z}}(z) = 0 = \frac{\p\theta}{\p\bar{z}}(0)$.
Similarly, the convergence of $z \frac{\p g_\varepsilon}{\p z}$ implies
$$
z \frac{\p}{\p z}\left(\frac{\theta_\varepsilon(z)}{z}\right) =
z \left( \frac{1}{z} \frac{\p \theta_\varepsilon}{\p z}(z) - \frac{1}{z^2} \theta_\varepsilon(z) \right)
= \frac{\p\theta_\varepsilon}{\p z}(z) - \frac{\theta_\varepsilon(z)}{z} \to 0,
$$
and since $\theta_\varepsilon(z) / z = g_\varepsilon(z) \to 1$ uniformly, it follows that
$\frac{\p\theta_\varepsilon}{\p z}(z) = \frac{\p\theta}{\p z}(\varepsilon z)$
converges as $\varepsilon \to 0$ to $1 = \frac{\p\theta}{\p z}(0)$.
 
Having established that $\theta$ is a $C^1$-smooth function, 
we now take a closer look at the relation
\begin{equation}
\label{eqn:differentiateThis}
\Psi_1(\theta,\eta) = \left( \theta^k + \widecheck{X}_\eta(\theta^k,\widehat{u} \circ \theta),
\widehat{u} \circ \theta + \widehat{X}_\eta(\theta^k , \widehat{u} \circ \theta) \right) =
(z^k,\widehat{v}) = v_1,
\end{equation}
viewed as a function of $z \in \dot{\DD}_\rho$.
Since $\theta(0)=0$ and $d\theta(0)=\1$, the argument of Lemma~\ref{lemma:uhato}
implies that both $\widehat{v}(z)$ and $\widehat{u} \circ \theta(z)$ are
$O(|z|^{k+1})$, so this equation implies
$$
\widehat{X}_{\eta(z)}\left( \left[\theta(z)\right]^k , \widehat{u}\circ \theta(z)\right) = O(|z|^{k+1}).
$$
The second estimate in Lemma~\ref{lemma:PsiEstimates} then gives
$$
\left| \widehat{X}_{\eta(z)}\left( \left[\theta(z)\right]^k , \widehat{u}\circ \theta(z)\right) - \eta(z) \right|
\le C |\widehat{u} \circ \theta(z)| \cdot |\eta(z)| \le C'|z|^{k+1}
$$
for some constant $C' > 0$.
We conclude that $\eta$ extends to a continuous function on $\DD_\rho$
with $\eta(0) = 0$ and $\eta(z) = O(|z|^{k+1})$, and since $k+1 \ge 2$, the
latter implies that $\eta$ is also differentiable at $z=0$ with $d\eta(0)=0$.  

Finally,
differentiating \eqref{eqn:differentiateThis} at $z \ne 0$ gives
$$
d v_1(z) = d_1\Psi_1(\theta(z),\eta(z)) \circ d\theta(z) + d_2\Psi_1(\theta(z),\eta(z)) \circ d\eta(z).
$$
If $k \ge 2$, then in the limit
as $z \to 0$, the first differential of $\Psi_1$ becomes $d_1\Psi_1(0,0) = d(u \circ \varphi)(0) = 0$,
while the second becomes $d_2\Psi_1(0,0) = (0,\1)$ since $\Psi_1(0,w) = (0,w)$,
and since $d v_1(0)$ also vanishes, this proves
$\lim_{z \to 0} d\eta(z) = 0$.  The case $k=1$ is slightly different
since $dv_1(0)$ and $d_1\Psi_1(0,0) \circ d\theta(0) = d_1\Psi_1(0,0) = d(u \circ \varphi)(0)$
do not vanish, but instead they are identical, so we obtain the same conclusion
about $d\eta(z)$.
\end{proof}

\subsection{Conclusion of the proof}

We cannot apply Proposition~\ref{prop:pushoff} directly to the relation
$v \circ \psi(z) = \Psi(\varphi \circ \theta(z),\eta(z))$ because
$v \circ \psi$ is not a smooth map.  However, we can write
$$
\widetilde{\varphi} := \varphi \circ \theta \circ \psi^{-1} \quad \text{ and }\quad
\widetilde{\eta} := \eta \circ \psi^{-1},
$$
and then apply the proposition to the relation
$$
v(z) = \Psi(\widetilde{\varphi}(z),\widetilde{\eta}(z)).
$$
Since $\widetilde{\eta}$ and $\widetilde{\varphi}$ are of class~$C^1$,
it follows that $\widetilde{\eta}$ is a solution to a linear
Cauchy-Riemann type equation of class~$C^0$, and
the similarity principle (Corollary~\ref{cor:similarity}) then implies
that $\widetilde{\eta}$ is either identically zero near $z=0$ or satisfies
$$
\widetilde{\eta}(z) = z^\ell A + o(|z|^\ell)
$$
for some $\ell \in \NN$ and $A \in \CC^{n-1} \setminus \{0\}$.
If $\widetilde{\eta}$ vanishes near~$0$, then so does~$\eta$, and we obtain
$$
(z^k,\widehat{v}(z)) = v \circ \psi(z) = \Psi(\varphi \circ \theta(z),0) = (u \circ \varphi)(\theta(z)) =
\left( \left[\theta(z)\right]^k , \widehat{u} \circ \theta(z)\right).
$$
Given that $\theta$ is of class $C^1$ with $d\theta(0)=\1$, this can only
hold if $\theta$ is the identity map near $z=0$, implying
$\widehat{u} \equiv \widehat{v}$.

If on the other hand $\widetilde{\eta}(z) = z^\ell A + |z|^\ell R(z)$
with $A \ne 0$ and $\lim_{z \to 0}R(z) = 0$,
then since $\psi(0)=0$ and $d\psi(0)=\1$, we can write
$\psi(z) = z + |z| \cdot r(z)$ with $\lim_{z \to 0}r(z) = 0$ and find
$$
\eta(z) = \widetilde{\eta}(\psi(z)) = \left( z + |z|r(z) \right)^\ell A +
\big| z + |z|\cdot r(z)\big|^\ell R\big(z + |z| \cdot r(z)\big) =
z^\ell A + o(|z|^\ell).
$$
Since $\eta(z) = O(|z|^{k+1})$ by Lemma~\ref{lemma:C1pushoff}, 
we deduce from this that $\ell > k$.  It remains to relate this to the
function $h(z) = (0,\widehat{h}(z)) := (0,\widehat{v}(z) - \widehat{u}(z)) =
v \circ \psi(z) - u \circ \varphi(z)$, which can now be expressed as
\begin{equation}
\label{eqn:headache}
\begin{split}
(0,\widehat{h}) &= \Psi(\varphi \circ \theta,\eta) - \Psi(\varphi,0)
= (u \circ \varphi) \circ \theta - u \circ \varphi + X_{\eta}(u \circ \varphi \circ \theta) \\
&= (\theta^k,\widehat{u} \circ \theta) - (z^k,\widehat{u}) +
X_\eta(\theta^k,\widehat{u} \circ \theta) \\
&= \left(\theta^k - z^k + \widecheck{X}_\eta(\theta^k , \widehat{u} \circ \theta) ,
\widehat{u} \circ \theta - \widehat{u} + \widehat{X}_\eta(\theta^k,\widehat{u} \circ \theta)\right).
\end{split}
\end{equation}
Since $\eta(z) = O(|z|^\ell)$ and $\widehat{u} \circ \theta(z) = O(|z|^{k+1})$, 
Lemma~\ref{lemma:PsiEstimates} implies an estimate
$$
\left|\widecheck{X}_{\eta(z)}\big(\left[\theta(z)\right]^k,\widehat{u} \circ \theta(z)\big)\right|
\le C |\eta(z)| \cdot |\widehat{u} \circ \theta(z)| = O(|z|^{\ell+k+1}),
$$
so that \eqref{eqn:headache} then gives $[\theta(z)]^k - z^k = O(|z|^{\ell+k+1})$.
Since $\theta(z) / z$ can be assumed arbitrarily close to $1$ for $|z|$ small,
we then have
\begin{equation*}
\begin{split}
\left|\theta - z\right| &= \left|\frac{\theta^k - z^k}{\theta^{k-1} + \theta^{k-2}z + \ldots + \theta z^{k-2} + z^{k-1}}\right|
= \frac{|\theta^k - z^k|}{|z|^{k-1}} \frac{1}{\left|\left(\frac{\theta}{z}\right)^{k-1} + \ldots + \left(\frac{\theta}{z}\right) + 1\right|}\\
&\le \text{const} \cdot \frac{|z|^{\ell+k+1}}{|z|^{k-1}} = O(|z|^{\ell+2}),
\end{split}
\end{equation*}
which implies an estimate of the form
\begin{equation}
\label{eqn:hatsoff}
\left| \widehat{u}(\theta(z)) - \widehat{u}(z)\right| \le C|\theta(z) - z| = O(|z|^{\ell+2})
\end{equation}
since $\widehat{u}$ is of class~$C^1$.  Finally, the second estimate in
Lemma~\ref{lemma:PsiEstimates} implies
$$
\left| \widehat{X}_{\eta(z)}\big(\left[\theta(z)\right]^k , \widehat{u} \circ \theta(z)\big) - \eta(z)\right|
\le C |\eta(z)| \cdot |\widehat{u} \circ \theta(z)| = O(|z|^{\ell+k+1}),
$$
hence 
$$
\widehat{X}_{\eta(z)}\big(\left[\theta(z)\right]^k , \widehat{u} \circ \theta(z)\big) =
\eta(z) + O(|z|^{\ell+k+1}) = z^\ell A + o(|z|^\ell) + O(|z|^{\ell+k+1}) = z^\ell A + o(|z|^\ell),
$$
and combining this with \eqref{eqn:hatsoff}, we can now derive from
\eqref{eqn:headache} the relation
$$
\widehat{h}(z) = O(|z|^{\ell+2}) + z^\ell A + o(|z|^\ell) = z^\ell A + o(|z|^\ell).
$$
The proof of Theorem~\ref{thm:representation} is now complete.

\section{Counting local intersections and singularities}
\label{sec:intSec}

In this section, we take the local representation formula of
Theorem~\ref{thm:representation} as a black box and
use it deduce the standard results on positivity of intersections.

According to the representation formula, a
nonconstant $J$-holomorphic curve has a well-defined tangent space at
every point, including critical points, with a nonnegative
\emph{critical order} $k \in \ZZ$ that is strictly positive if and only if the
point is critical.  We can now prove local
positivity of intersections (Theorem~\ref{thm:positivity}) by considering
separately the cases where the two curves have matching or non-matching
tangent spaces at their intersection.  Note that when $\dim M = 4$,
the condition that two (complex-linear!) tangent spaces at an intersection point do not
match means simply that they are transverse, and the intersection itself is
then transverse if and only if neither curve is critical at the
intersection point.

\begin{exercise}
\label{EX:nearTangent}
Let $\pi : \CC^n \setminus \{0\} \to \CP^{n-1}$ denote the natural
projection, and consider a map $u : \DD \to \CC^n$ of the form
$u(z) = (z^k , |z|^{k+1} f(z))$ for some $k \ge \NN$ and a bounded function
$f : \DD \to \CC^{n-1}$.  Show that for any neighborhood $\uU$ of
$[1 : 0 : \ldots : 0 ] \in \CP^{n-1}$, one can find $\rho > 0$ such
that the restriction of $\pi \circ u$ to $\DD_\rho \setminus \{0\}$
has image in~$\uU$.
\end{exercise}

\begin{prop}
\label{prop:transverseIntersection}
Suppose $u : (\Sigma,j) \to (M,J)$ and\index{intersections!positivity of|(}\index{positivity of intersections|(}
$v : (\Sigma',j') \to (M,J)$ are two $J$-holomorphic curves with an
intersection $u(z_0) = v(\zeta_0)$ at which $u$ has critical order
$k_u - 1 \ge 0$, $v$ has critical order $k_v - 1 \ge 0$, and their tangent
spaces (in the sense of Theorem~\ref{thm:representation}) are
distinct.  Then the intersection is isolated, and if
$\dim M = 4$, its local intersection index is
$$
\inter(u,z_0 \,;\, v,\zeta_0) = k_u k_v;
$$
in particular, it is positive, and equal to~$1$ if and only if the
intersection is transverse.
\end{prop}
\begin{proof}
By Theorem~\ref{thm:representation}, we can choose $C^1$-smooth coordinates such that
without loss of generality $z_0 = \zeta_0 = 0 \in \DD = \Sigma = \Sigma'$, 
$M = \CC^n$,  $u(z) = (z^{k_u} , |z|^{k_u+1} f(z))$ for some bounded function 
$f : \DD \to \CC^{n-1}$, and $v : \DD \to \CC^n$ satisfies
$v(0) = 0$.  The condition of distinct tangent spaces implies via
Exercise~\ref{EX:nearTangent} that
if $\pi : \CC^n \setminus \{0\} \to \CP^{n-1}$ denotes the natural
projection, we can also assume that the images of the maps
$$
\pi \circ u|_{\DD \setminus\{0\}},\  \pi \circ v|_{\DD \setminus \{0\}} :
\DD \setminus \{0\} \to \CP^{n-1}
$$
lie in arbitrarily small neighborhoods of two distinct points.
The same is also true if we replace $u$ with any of the maps
$$
u_\tau : \DD \to \CC^n : z \mapsto (z^{k_u} , \tau |z|^{k_u+1} f(z)), \qquad
\tau \in [0,1].
$$
The claim that the intersection is isolated follows immediately,
and when $n=2$, we also deduce via Exercise~\ref{EX:interWithBoondary}
that $\inter(u,0 \,;\, v,0) = \inter(u_0,0 \,;\, v,0)$.  After applying the
same homotopy argument in different coordinates adapted to~$v$ and then
choosing new coordinates so that the tangent spaces of $u$ and $v$
match $\CC \times \{0\}$ and $\{0\} \times \CC$ respectively,
we can reduce the problem to a computation of $\inter(u_0,0 \,;\, v_0 , 0)$ for
$$
u_0(z) = (z^{k_u},0), \qquad
v_0(z) = (0, z^{k_v}).
$$
Choose $\epsilon \in \CC\setminus \{0\}$ and perturb these maps to 
$(z^{k_u} + \epsilon,0)$ and $(0, z^{k_v} + \epsilon)$ respectively.
Both are now holomorphic for the standard complex structure on $\CC^2$
and they have exactly $k_u k_v$ intersections, all transverse.
\end{proof}

When both curves have matching tangent spaces where they intersect, we will
need to use the more precise information provided by 
Theorem~\ref{thm:representation}.  Observe that in this case the intersection
can never be transverse.

\begin{exercise}
\label{EX:branched}
Suppose $\dim M = 4$, $u ,v : (\DD,i) \to (M,J)$
are $J$-holomorphic disks and they have an isolated intersection 
$u(0) = v(0)$.  Given $k , \ell \in \NN$, define the $J$-holomorphic 
branched covers $u^k , v^\ell : (\DD,i) \to (M,J)$,
$$
u^k(z) := u(z^k), \qquad v^\ell(z) := v(z^\ell).
$$
Show that $\inter(u^k,0 \,;\, v^\ell,0) = k \ell \cdot \inter(u,0 \,;\, v,0)$.
\end{exercise}

\begin{prop}
\label{prop:sameTangentSpace}
Suppose $u : (\Sigma,j) \to (M,J)$ and
$v : (\Sigma',j') \to (M,J)$ are two $J$-holomorphic curves with an
intersection $u(z_0) = v(\zeta_0)$ at which $u$ has critical order
$k_u - 1 \ge 0$, $v$ has critical order $k_v - 1 \ge 0$, and their tangent
spaces (in the sense of Theorem~\ref{thm:representation}) are identical.
Then either the intersection $u(z_0) = v(\zeta_0)$ is isolated, or there exist
neighborhoods $z_0 \in \uU_{z_0} \subset \Sigma$ and $\zeta_0 \in 
\uU_{\zeta_0} \subset \Sigma'$ such that $u(\uU_{z_0}) = v(\uU_{\zeta_0})$.  
In the former case, if $\dim M = 4$, the local intersection index satisfies
$$
\inter(u,z_0 \,;\, v,\zeta_0) > k_u k_v;
$$
in particular, it is strictly greater than~$1$.
\end{prop}
\begin{proof}
We can choose holomorphic coordinates near 
$z_0 \in \Sigma$ and $\zeta_0 \in \Sigma'$ so that, without loss of
generality, $(\Sigma,j) = (\Sigma',j') = (\DD,i)$ with 
$z_0 = \zeta_0 = 0$.  
Since $k_u$ and $k_v$ may be different, we first replace $u$ and $v$ with
suitable branched covers so that their critical orders become the same:
let
$$
m = k_u k_v \in \NN,
$$
and define $u', v' : (\DD,i) \to (M,J)$ by
$$
u'(z) := u(z^{k_v}), \qquad v'(z) := v(z^{k_u}),
$$
so that in particular $u'$ and $v'$ both have critical order $m-1$
at the intersection $u'(0) = v'(0)$, as well as matching tangent spaces.
Now by Theorem~\ref{thm:representation}, we find new choices of 
$C^1$-smooth local coordinates in $\DD$ near~$0$ and smooth coordinates
in $M$ near $u(0)=v(0)$ such that
$$
u'(z) = (z^m,\widehat{u}(z)), \qquad v'(z) = (z^m,\widehat{v}(z))
$$
for some functions $\widehat{u} , \widehat{v} : \DD \to \CC^{n-1}$ of class $C^1$ 
that are both $O(|z|^{m+1})$.  For each $j=0,\ldots,m-1$, we can also
compose $u'$ with the smooth coordinate change $z \mapsto e^{2\pi i j/m} z$
to produce a new parametrization $v_j' : \DD \to \CC^n$ of the form
$$
v_j'(z) := v'( e^{2\pi i j/m} z) = (z^m, \widehat{v}_j(z)),
\quad \text{ where }\quad \widehat{v}_j(z) = \widehat{v}( e^{2 \pi i j/m} z),
$$
for which the statement of Theorem~\ref{thm:representation} is equally valid.
If $\widehat{u} - \widehat{v}_j$ is identically zero for some $j=0,\ldots,m-1$,
then we have
$$
u'(z) = v'( e^{2\pi i j/m} z) \quad\text{ for all $z \in \DD$},
$$
implying that $u'$ and $v'$ have identical images
on some neighborhood of the intersection, in which case so do $u$ and~$v$.  
If not, then
Theorem~\ref{thm:representation} gives for each $j=0,\ldots,m-1$ the formula
\begin{equation}
\label{eqn:windinguv}
\widehat{u}(z) - \widehat{v}_j(z) = z^{m + \ell_j} C_j + |z|^{m + \ell_j} r_j(z),
\end{equation}
where $C_j \in \CC^{n-1} \setminus \{0\}$, $\ell_j \in \NN$ and 
$r_j(z) \in \CC^{n-1}$ is a function with $r_j(z) \to 0$ as $z \to 0$.  
This expression has an isolated zero at $z=0$, thus the intersection of 
$u'$ and $v'$ (and hence of $u$ and $v$) is isolated.

If $n=2$, we can now compute $\inter(u',0 \,;\, v',0)$ by choosing
$\epsilon \in \CC \setminus \{0\}$ small and defining the perturbation
$$
u'_\epsilon(z) := (z^m , \widehat{u}(z) + \epsilon).
$$
This curve does not intersect $v'$ at $z=0$ since $\epsilon \ne 0$.
If $u'_\epsilon(z) = v'(\zeta)$, then $z^m = \zeta^m$, hence
$\zeta = e^{2\pi i j/m} z$ for some $j=0,\ldots,m-1$, and equality in the
second factor then implies
\begin{equation}
\label{eqn:that}
\widehat{v}_j(z) - \widehat{u}(z) = \epsilon.
\end{equation}
By \eqref{eqn:windinguv}, the zero of $\widehat{v}_j(z) - \widehat{u}(z)$ at $z=0$
has order $m + \ell_j > m$, thus if $\epsilon \in \CC$ is sufficiently small
and chosen generically so that it is a regular value of $\widehat{v}_j - \widehat{u}$,
we conclude that \eqref{eqn:that} has exactly $m + \ell_j$ solutions
near $z=0$, all of them simple positive zeroes of $\widehat{v}_j - \widehat{u} - \epsilon$
and thus corresponding to transverse positive intersections of
$u'_\epsilon$ with $v'$.  Adding these up for all choices of $j=0,\ldots,m-1$, we conclude
$$
\inter(u',0 \,;\, v',0) > m^2 = k_u^2 k_v^2,
$$
so by Exercise~\ref{EX:branched}, $\inter(u,0 \,;\, v,0) > k_u k_v$.
\end{proof}

\begin{exercise}
\label{EX:noUpperBound}
Find examples to show that in the situation of 
Proposition~\ref{prop:sameTangentSpace}, $\inter(u,z_0 \,;\, v,\zeta_0)$
cannot in general be bounded from above.
\end{exercise}

Combining Propositions~\ref{prop:transverseIntersection} 
and~\ref{prop:sameTangentSpace} completes the proof of
Theorem~\ref{thm:positivity}.\index{intersections!positivity of|)}\index{positivity of intersections|)}

We now turn to the proof of Lemma~\ref{lemma:critInj} from
Lecture~\ref{sec:2}, which asserts that any critical point on a simple
$J$-holomorphic curve gives rise to a strictly positive count of
double points after an immersed perturbation.  In the background\index{holomorphic curve!multiply covered|(}\index{multiply covered holomorphic curve|(}
of this statement is the fact that all simple holomorphic curves are\index{simple holomorphic curve!is locally injective|(}
locally injective, which we can now prove using the
representation formula of Theorem~\ref{thm:representation}.

\begin{prop}
\label{prop:branching}
Suppose $u : (\Sigma,j) \to (M,J)$ is a $J$-holomorphic curve
that is nonconstant near a point $z_0 \in \Sigma$ with $du(z_0) = 0$.  Then
there exists a neighborhood $z_0 \in \uU_{z_0} \subset \Sigma$ such that
there is a biholomorphic identification
$$
\varphi : (\DD,i) \stackrel{\cong}{\longrightarrow} (\uU_{z_0},j)
$$
with $\varphi(0) = z_0$,
a number $k \in \NN$, and an \emph{injective} $J$-holomorphic map
$$
v : (\DD,i) \to (M,J)
$$
with 
$$
dv(z) \ne 0 \text{ for $z \in \DD \setminus\{0\}$ and }
u \circ \varphi(z) = v(z^k) \text{ for $z \in \DD$}.
$$
If $u : (\Sigma,j) \to (M,J)$ is a simple curve, then $k=1$.
\end{prop}
\begin{proof}
Theorem~\ref{thm:representation} provides $C^1$-smooth local coordinates near~$z_0 \in \Sigma$
and smooth coordinates near $u(z_0) \in M$ in which $u$ takes the form
$$
u(z) = (z^k , \widehat{u}(z)) \in \CC^n
$$
for a $C^1$-smooth map $\widehat{u} : \DD \to \CC^{n-1}$ 
with $\widehat{u}(z) = O(|z|^{k+1})$, where $k-1 \ge 0$
is the critical order of $u$ at~$z_0$, and all the maps in this picture are of
class $C^\infty$ away from $z_0 \in \Sigma$ or $0 \in \DD$ respectively.
For each $j=1,\ldots,k-1$, we can compose this representation of $u$ with the
smooth reparametrization $\psi_j(z) := e^{2\pi ij/k}z$ and thus use
Theorem~\ref{thm:representation} to compare $u$ with
$$
u_j(z) := u(e^{2\pi i j/k}z) = (z^k , \widehat{u}_j(z)), \quad\text{ where }
\quad \widehat{u}_j(z) := \widehat{u}(e^{2\pi i j/k}z).
$$
The theorem implies that each $\widehat{u} - \widehat{u}_j$
is either identically zero or has an isolated zero at $z=0$.
Self-intersections $u(z) = u(\zeta)$ with $z \ne \zeta$ can now be identified
with pairs $j \in \{1,\ldots,k-1\}$ and $z \in \DD$ for which
$\widehat{u}(z) = \widehat{u}_j(z)$.
Let $m \in \{1,\ldots,k\}$ denote the smallest number for which
$\widehat{u} \equiv \widehat{u}_m$, hence $u(z) = u(e^{2\pi i m/k} z)$ for all~$z$.
Then we also have $\widehat{u} \equiv \widehat{u}_{j m}$ for all $j \in \ZZ$,
so $m$ must divide $k$, and setting $\ell := k/m$, we see that $u : \DD \to \CC^n$ is
invariant with respect to the $\ZZ_\ell$ action on $\DD$ generated by
the rotation $\psi := \psi_m$.  It therefore factors as
$$
u(z) = v(z^\ell)
$$
for a continuous map $v : \DD \to \CC^n$ that is
smooth on $\DD \setminus \{0\}$, and
$v$ is injective near~$0$ since we always have $\widehat{u}(z) \ne \widehat{u}_j(z)$
near $z=0$ for $j=1,\ldots,m-1$.

It remains to show that $v : \DD \to \CC^n$ can be reparametrized near $0 \in \DD$
to become a \emph{smooth} $J$-holomorphic curve.  We shall deduce this from
elliptic regularity, but first, we need to switch back to smooth holomorphic
coordinates on the domain.  Since the parametrization
$u(z) = (z^k,\widehat{u}(z))$ was obtained via a $C^1$-smooth coordinate
chart on the smooth Riemann surface $(\Sigma,j)$, this parametrization is
a pseudoholomorphic map $(\DD,j') \to (M,J)$ for a continuous
complex structure~$j'$
that is smooth on $\dot{\DD} := \DD\setminus \{0\}$ and uniquely determined there by $j' = u^*J$.
It follows that the $\ZZ_\ell$-action on $\DD$ leaving $u$ invariant acts
holomorphically on $(\DD,j')$, and it can therefore be defined as a group of
biholomorphic (and therefore smooth) transformations on the simply connected
neighborhood $\uU \subset \Sigma$ of $z_0$ that is identified with $\DD$
via our $C^1$-coordinates.  Using the Riemann mapping theorem, we can now
choose a holomorphic coordinate chart identifying $(\uU,j)$ with $(\DD,i)$
and $z_0$ with $0 \in \DD$, so that in the new coordinates, $\psi$ generates
a $\ZZ_\ell$-action by biholomorphic transformations on $(\DD,i)$
that fix~$0$.  All such transformations are rotations, thus $\psi$ is
given by the same formula as before in the new coordinates, and we can
define a continuous map $v : \DD \to \CC^n$ as before via the relation
$u(z) = v(z^\ell)$, observing that $v$ is manifestly
smooth and holomorphic on the standard punctured disk~$(\dot{\DD},i)$.
Since $du(z) = O(|z|^{k-1})$, we then deduce from
$u(z) = v(z^\ell)$ and $du(z) = dv(z^\ell) \circ (\ell z^{\ell-1})$
an estimate of the form
$$
\left|dv(z^\ell)\right| \le C \frac{|du(z)|}{|z|^{\ell-1}} \le C' |z|^{k-\ell}
$$
near $z=0$.  This expression is bounded since $\ell \le k$, implying
via Exercise~\ref{EX:W1infty} that
the map $v : \DD \to \CC^n$ is of class~$W^{1,\infty}$.  It is therefore
smooth by Proposition~\ref{prop:nonlinearReg}.\index{holomorphic curve!multiply covered|)}\index{multiply covered holomorphic curve|)}\index{simple holomorphic curve!is locally injective|)}
\end{proof}

The remainder of Lemma~\ref{lemma:critInj} can be restated as follows.\index{local singularity index|(}\index{holomorphic curve!local singularity index at a point|(}

\begin{prop}
\label{prop:critical}
Suppose $\dim M=4$ and $u : (\DD,i) \to (M,J)$ is an injective 
$J$-holomorphic
map with critical order $k-1 \ge 1$ at $z=0$ and no critical points
on $\DD \setminus \{0\}$.  Then there exists an integer
$$
\delta(u,0) \ge \frac{k (k-1)}{2}
$$
depending only on the germ of $u$ near~$0$, such that for any
given neighborhood $\uU \subset \DD$ of~$0$ and symplectic form $\omega_0$ defined
near $u(0)$ taming~$J$, one can find a $C^1$-smooth
map $u_\epsilon :\DD \to M$ satisfying the following conditions:
\begin{enumerate}
\item $u_\epsilon$ is $C^1$-close to $u$ and matches $u$ outside
$\uU$ and at~$0$;
\item $u_\epsilon$ is an immersion with $u_\epsilon^*\omega_0 > 0$;
\item $u_\epsilon$ has finitely many self-intersections and satisfies
\begin{equation}
\label{eqn:deltaLocal}
\frac{1}{2} \sum_{(z,\zeta)} \inter(u_\epsilon,z \,;\, u_\epsilon,\zeta) = \delta(u,0),
\end{equation}
where the sum ranges over all pairs $(z,\zeta) \in \DD \times \DD$
such that $z \ne \zeta$ and $u_\epsilon(z) = u_\epsilon(\zeta)$.
\end{enumerate}
\end{prop}

Our proof will show in fact that the tangent spaces spanned by the
perturbation $u_\epsilon$ can be arranged to be uniformly close 
to $i$-complex subspaces (or equivalently $J$-complex subspaces,
since $J$ and $i$ may also be assumed uniformly close in a small\index{symplectically immersed}
enough neighborhood of~$u(0)$).  This implies that it is a symplectic immersion
without loss of generality for any given $\omega_0$ taming~$J$,
as the condition of being a symplectic subspace
is open.  In practice, the crucial point in applications is
that the complex structure on the bundle $(u_\epsilon^*TM,J)$ 
admits a homotopy supported near~$0$ to 
a new complex structure for which
$\im d u_\epsilon$ becomes a complex subbundle---in this way we can keep
control over the $c_1$ term in the adjunction formula.
The subtlety in the proof is that
the change in tangent subspaces when perturbing from $u$ to $u_\epsilon$ 
cannot be understood as a $C^0$-small perturbation if $du(0)=0$.
Our strategy will be to show that the tangent spaces
spanned by $du_\epsilon$ are in fact $C^0$-close to the tangent spaces
spanned by \emph{another} map which is a holomorphic immersion.
In order to make this notion precise, we need a practical way of measuring
the ``distance'' between two subspaces of a vector space, in particular for
the case when both subspaces arise as images of injective linear maps.

\begin{defn}
\label{defn:distSubspaces}
Fix the standard Euclidean norm on $\RR^n$.  Given two subspaces
$V, W \subset \RR^n$ of the same positive dimension, define
$$
\dist(V,W) := \max_{v \in V, |v|=1} \dist(v,W) :=
\max_{v \in V, |v|=1} \min_{w \in W} |v - w|.
$$
\end{defn}
\begin{defn}
\label{defn:Inj}
The \defin{injectivity modulus}\index{injectivity modulus}
of a linear map $A : \RR^k \to \RR^n$ is
$$
\Inj(A) = \min_{v \in \RR^k \setminus \{0\}} \frac{|Av|}{|v|} \ge 0.
$$
\end{defn}
Clearly $\Inj(A) > 0$ if and only if $A$ is injective.

\begin{lemma}
\label{lemma:Inj}
For any pair of injective linear maps $A , B : \RR^k \to \RR^n$,
$$
\dist\left(\im A, \im B \right) \le \frac{| A - B |}{\Inj(A)}.
$$
\end{lemma}
\begin{proof}
Pick any nonzero vector $v \in \RR^n$.  Then $Av \ne 0$ since $A$ is injective,
and we have
\begin{equation*}
\begin{split}
\dist\left(\frac{Av}{|Av|} , \im B \right) &= 
\min_{w \in \RR^k} \left| A \frac{v}{|Av|} - Bw \right| \le 
\left| A \frac{v}{|Av|} - B \frac{v}{|Av|} \right| \\
&\le | A - B | \frac{|v|}{|Av|} \le \frac{| A - B |}{\Inj(A)}.
\end{split}
\end{equation*}
\end{proof}

\begin{lemma}
\label{lemma:nearlyComplex}
Given a symplectic form $\omega_0$ on $\CC^2$ taming~$i$,
there exists $\epsilon > 0$ such that if $V \subset \CC^2$ is a
complex $1$-dimensional subspace, then all real $2$-dimensional
subspaces $W \subset \CC^2$ satisfying $\dist(V,W) < \epsilon$ are
$\omega_0$-symplectic.\index{symplectically immersed}
\end{lemma}
\begin{exercise}
\label{EX:nearlyComplex}
Prove the lemma.  \textsl{Hint: $\CP^1$ is compact.}
\end{exercise}

\begin{proof}[Proof of Proposition~\ref{prop:critical}]
By Theorem~\ref{thm:representation}, we can assume after choosing suitable 
$C^1$-smooth coordinates
near $0 \in \DD$ and smooth coordinates near $u(0) \in M$ that
$$
u(z) = (z^k, \widehat{u}(z)) \in \CC^2
$$
for some integer $k \ge 2$, where the almost complex structure $J$ matches $i$
at~$0 \in \CC^2$, and $\widehat{u}$ is a map $\DD_\rho \to \CC$ of class $C^1$
on a disk of some radius $\rho > 0$, such that the other branches 
$$
u_j(z) := u(e^{2\pi i j/k}z) = (z^k , \widehat{u}_j(z)), \qquad
\widehat{u}_j(z) := \widehat{u}(e^{2\pi i j/k}z),
$$
for $j=1,\ldots,k-1$ are related by
\begin{equation}
\label{eqn:branchy}
\widehat{u}_j(z) - \widehat{u}(z) = z^{k + \ell_j} C_j + |z|^{k+\ell_j} r_j(z)
\end{equation}
for some $\ell_j \in \NN$, $C_j \in \CC \setminus \{0\}$ and
$r_j : \DD_\rho \to \CC$ with $r_j(z) \to 0$ as $z \to 0$.  Here we've
used the assumption that $u$ is injective in order to conclude that
$\widehat{u}_j - \widehat{u}$ is not identically zero.  By shrinking $\rho > 0$
if necessary, we can also assume $u$ is embedded on $\DD_\rho \setminus \{0\}$,
and that the symplectic form $\omega_0$, which tames $J$ by assumption,
also tames~$i$ on some neighborhood of $u(\DD_\rho)$.
Fix a smooth cutoff
function $\beta : \DD_\rho \to [0,1]$ that equals~$1$ on $\DD_{\rho/2}$ 
and has compact support in the interior.  Then for $\epsilon \in \CC$ sufficiently 
close to~$0$, consider the $C^1$-close perturbation
$$
u_\epsilon(z) := (z^k, \widehat{u}(z) + \epsilon \beta(z) z),
$$
which satisfies $u_\epsilon(0) = 0$ and
is immersed if $\epsilon \ne 0$.  Since $u$ is embedded on
$\DD_\rho \setminus \DD_{\rho/2}$, we may assume
for $|\epsilon|$ sufficiently small that $u_\epsilon$
has no self-intersections outside of the region where $\beta \equiv 1$.
Then a self-intersection $u_\epsilon(z) = u_\epsilon(\zeta)$ with 
$z \ne \zeta$ occurs wherever
$\zeta = e^{2 \pi ij/k}z \ne 0$ for some $j=1,\ldots,k-1$ and
$\widehat{u}(z) + \epsilon z = \widehat{u}_j(z) + \epsilon e^{2\pi i j/k}z$, which by
\eqref{eqn:branchy} means
$$
z^{k + \ell_j} C_j + |z|^{k+\ell_j} r_j(z) + \epsilon \left( e^{2\pi i j/k} - 1\right) z = 0.
$$
Assume $\epsilon \in \CC \setminus \{0\}$ is chosen generically so that the
zeroes of this function are all simple (see Exercise~\ref{EX:trickySard} below).
Then each zero other than the ``trivial'' solution at $z=0$ 
represents a transverse (positive or negative) self-intersection 
of $u_\epsilon$, and the algebraic count of these (discounting the trivial solution) 
for $|\epsilon|$ sufficiently
small is $k + \ell_j - 1 \ge k$.  Adding these up for all $j=1,\ldots,k-1$, we 
obtain 
\begin{equation}
\label{eqn:deltaLocal2}
\delta(u,0) :=
\frac{1}{2}\sum_{(z,\zeta)} \inter(u_\epsilon,z \,;\, u_\epsilon,\zeta) = 
\frac{1}{2} \sum_{j=1}^{k-1} (k + \ell_j - 1) \ge \frac{1}{2} k(k-1).
\end{equation}

It remains to show that $u_\epsilon$ satisfies $u_\epsilon^*\omega_0 > 0$,
which is equivalent to showing that $\im du_\epsilon(z) \subset \CC^2$ is an
$\omega_0$-symplectic subspace for all~$z$.
By Theorem~\ref{thm:representation}, there exist constants
$\ell \in \NN$ and $C \in \CC \setminus \{0\}$ such that
\begin{equation}
\label{eqn:hatu}
\widehat{u}(z) = z^{k+\ell} C + o(|z|^{k+\ell}),
\end{equation}
and we claim that the formula
\begin{equation}
\label{eqn:dhatu}
d\widehat{u}(z) = (k+\ell) z^{k+\ell-1} C + o(|z|^{k+\ell-1})
\end{equation}
also holds.  If $\widehat{u}$ were smooth, this would follow immediately
from \eqref{eqn:hatu} via Taylor's theorem, but we have to work a little bit
harder since $\widehat{u}$ is only of class~$C^1$.  Recall that
$\widehat{u}$ is a composition of the form $\widehat{u} = u_2 \circ \varphi$,
where we can take $\varphi : \DD_\rho \to \DD$ to be a $C^1$-smooth
local diffeomorphism with $\varphi(0)=0$ and $d\varphi(0) = \1$, and
$u_2 : \DD_\epsilon \to \CC$ is a map of class~$C^\infty$, i.e.~the
second coordinate of the original $J$-holomorphic curve before it was
non-smoothly reparametrized.  Since $\varphi(z) = z + o(|z|)$ and
$\varphi^{-1}(z) = z + o(|z|)$, we can write $\varphi^{-1}(z) = z + |z|\cdot r(z)$
with $\lim_{z \to 0}r(z)=0$ and write \eqref{eqn:hatu} as
$\widehat{u}(z) = z^{k+\ell}C + |z|^{k+\ell}R(z)$ with
$\lim_{z \to 0}R(z)=0$, implying
\begin{equation*}
\begin{split}
u_2(z) &= \widehat{u}(z + |z|\cdot r(z)) = \left( z + |z|\cdot r(z)\right)^{k+\ell} C
+ \left| z + |z|\cdot r(z)\right|^{k+\ell} R(z + |z|\cdot r(z)) \\
&= z^{k+\ell} C + o(|z|^{k+\ell}).
\end{split}
\end{equation*}
Now since $u_2$ is smooth, this expression can be interpreted as saying that
$z^{k+\ell}C$ is the lowest-order
nontrivial term in its Taylor series, and we can then draw a similar conclusion
for $du_2$: namely for $C' := (k+\ell)C$
and a function
$R'(z) \in \CC$ with $\lim_{z \to 0} R'(z) = 0$, we have
$$
du_2(z) = z^{k+\ell-1} C' + o(|z|^{k+\ell-1}) =
z^{k+\ell-1} C' + |z|^{k+\ell-1} R'(z).
$$
Finally, reverse the
process: writing $\varphi(z) = z + |z|\cdot r'(z)$ with 
$\lim_{z \to 0} r'(z) = 0$, the relation \eqref{eqn:dhatu} follows from
\begin{equation*}
\begin{split}
d\widehat{u}(z) &= du_2(z + |z|\cdot r'(z)) \circ d\varphi(z) \\
&= \left[ \left( z + |z| \cdot r'(z)\right)^{k+\ell-1} C' +
\left| z + |z| \cdot r'(z)\right|^{k+\ell-1} R'(z)\right] \circ d\varphi(z)\\
&= \left[ \left( z + |z| \cdot r'(z)\right)^{k+\ell-1} C' +
\left| z + |z| \cdot r'(z)\right|^{k+\ell-1} R'(z)\right] \\
&\qquad + \left[\left( z + |z| \cdot r'(z)\right)^{k+\ell-1} C' 
 + \left| z + |z| \cdot r'(z)\right|^{k+\ell-1} R'(z)\right] \circ \left( d\varphi(z) - d\varphi(0) \right) \\
&= z^{k+\ell-1} C' + o(|z|^{k+\ell-1}),
\end{split}
\end{equation*}
where the existence of a suitable remainder function depends on the fact
that $d\varphi(z) - d\varphi(0)$ is a continuous function of $z$ that
vanishes at $z=0$.

We would now like to compare $u_\epsilon$ with the holomorphic polynomial
$$
P_\epsilon : \DD_\rho \to \CC^2 : z \mapsto (z^k, z^{k+\ell} C
+ \epsilon z),
$$
which, due to \eqref{eqn:dhatu}, satisfies
$$
du_\epsilon(z) - dP_\epsilon(z) = |z|^{k+\ell-1} R(z)
$$
for a remainder term $R(z) \in \CC^2$ that satisfies
$\lim_{z \to 0} R(z) = 0$ and does not depend on~$\epsilon$.
Abbreviating $A_\epsilon(z) := dP_\epsilon(z)$ and
$B_\epsilon(z) := du_\epsilon(z)$, this gives an estimate of the form
$$
| A_\epsilon(z) - B_\epsilon(z) | \le c_1 |z|^{k+\ell-1}
$$
for some constant $c_1 > 0$ independent of~$\epsilon$.  
Computing $dP_\epsilon(0)$, we find similarly
a constant $c_2 > 0$ independent of~$\epsilon$ such that
$$
| A_\epsilon(z) v | \ge c_2 |z|^{k-1} |v| \quad\text{ for all $v \in \CC$},
$$
thus $\Inj(A_\epsilon(z)) \ge c_2 |z|^{k-1}$, and
$$
\frac{| A_\epsilon(z) - B_\epsilon(z) |}{\Inj(A_\epsilon(z))} \le c_3 |z|^\ell
$$
for some constant $c_3 > 0$ independent of~$\epsilon$.  
Now since $P_\epsilon$ is holomorphic
(for the standard complex structure) for all $\epsilon$, $\im A_\epsilon(z) \subset \CC^2$
is always complex linear, so the above estimates imply together with
Lemmas~\ref{lemma:Inj} and~\ref{lemma:nearlyComplex} that for a sufficiently small radius
$\rho_0 > 0$, the images of $du_\epsilon(z)$ for all $z \in \DD_{\rho_0} \setminus \{0\}$
and $\epsilon \in \DD_{\rho_0}$ are $\omega_0$-symplectic.  This is also true for
$z=0$ if $\epsilon \ne 0$, since then $du_\epsilon(0) = dP_\epsilon(0)$ is
complex linear.

To conclude, fix $\rho_0 > 0$ as above
and choose $\epsilon \in \CC\setminus\{0\}$ sufficiently close to~$0$ so
that outside of $\DD_{\rho_0}$, $u_\epsilon$ is $C^1$-close enough to $u$ for its
tangent spaces to be $\omega_0$-symplectic (recall that $J$ is also
$\omega_0$-tame).  The previous paragraph then
implies that the tangent spaces of $u_\epsilon$ are $\omega_0$-symplectic
everywhere.
\end{proof}

\begin{exercise}
\label{EX:deltaIndependent}
Verify that the formula obtained in \eqref{eqn:deltaLocal2} for
$\delta(u,0)$ does not depend on any choices.
\end{exercise}

\begin{exercise}
\label{EX:trickySard}
Assume $f : \uU \to \CC$ is a $C^1$-smooth map on a domain
$\uU \subset \CC$ containing~$0$, with $f(0)=0$ and $df(0) = 0$.
Show that for almost every $\epsilon \in \CC$, the map
$f_\epsilon : \uU \to \CC : z \mapsto f(z) + \epsilon z$ has~$0$ as a
regular value.  \textsl{Hint: Use the implicit function theorem to show
that the set}
$$
X := \{ (\epsilon,z) \in \CC \times (\uU \setminus \{0\})\ |\ 
f_\epsilon(z) = 0 \}
$$
\textsl{is a smooth submanifold of $\CC^2$, and a point $(\epsilon,z) \in X$ is
regular for the projection $\pi : X \to \CC : (\epsilon,z) \mapsto \epsilon$
if and only if $z$ is a regular point of~$f_\epsilon$.  Then apply Sard's
theorem to~$\pi$.}\footnote{Note that while Sard's theorem is often stated only for
$C^\infty$-smooth maps, it is valid more generally for continuously differentiable
maps $f : M \to N$ of class $C^{m-n+1}$ for $m := \dim M$ and $n := \dim N$;
see \cite{Sard}.}
\end{exercise}

\begin{exercise}
\label{EX:noUpperBound2}
Find examples to show that the bound $\delta(u,0) \ge \frac{k(k-1)}{2}$ in
Proposition~\ref{prop:critical} is sharp, and that there is no similar upper bound
for $\delta(u,0)$ in terms of~$k$.  \textsl{(Compare Exercise~\ref{EX:noUpperBound}.)}\index{local singularity index|)}\index{holomorphic curve!local singularity index at a point|)}
\end{exercise}

\chapter{A quick survey of Siefring's intersection theory}
\label{app:reference}

\minitoc
\vspace{12pt}

This\CUP{This material will be published by Cambridge University
Press as \textsl{Contact 3-Manifolds, Holomorphic Curves and Intersection Theory}
by Chris Wendl. This pre-publication version is
free to view and download for personal use only. 
Not for re-distribution, re-sale or use in derivative works. \copyright Chris Wendl, 2019.}
appendix is meant in part as a survey and also as a quick reference guide
for the intersection theory of punctured holomorphic curves.  Except where
otherwise noted, the proofs of everything stated below are due to 
Siefring \cite{Siefring:intersection}, and the details (modulo proofs of
the relative asymptotic formulas) can be found in Lectures~3
and~4 of these notes.
Since intersection theory has also played a large role in the development
of Hutchings's embedded contact homology (ECH), we will simultaneously take the 
opportunity to clarify some of the connections between Siefring's theory and 
equivalent notions that often appear (sometimes with very different notation) 
in the ECH literature.  For an important word of caution about notational
differences between these notes and \cite{Siefring:intersection},
see Remark~\ref{remark:notation}.

\section{Preliminaries}

Assume $M$ is a closed oriented $3$-manifold with a stable Hamiltonian structure $(\omega,\lambda)$,\index{stable Hamiltonian structure}
i.e.~a $2$-form $\omega$ and $1$-form $\lambda$ that satisfy $d\omega=0$,
$\lambda \wedge \omega > 0$ and $\ker \omega \subset \ker d\lambda$.  (The
reader unfamiliar with or uninterested in stable Hamiltonian structures is free to assume
$(\omega,\lambda) = (d\alpha,\alpha)$ where $\alpha$ is a contact form.)
This data determines an oriented $2$-plane field
$$
\xi = \ker \lambda \subset TM
$$
and a \defin{Reeb vector field}\index{Reeb vector field!of a stable Hamiltonian structure}
$R$ such that 
$$
\omega(R,\cdot) \equiv 0 \quad \text{ and }\quad
\lambda(R) \equiv 1.
$$
We assume throughout the following that all closed orbits of $R$ are nondegenerate.
As mentioned in the footnote to Theorem~\ref{thm:star}, the major
results continue to hold without serious changes if orbits are Morse-Bott,\index{Reeb orbit!Morse-Bott}\index{Morse-Bott Reeb orbits}
as long as homotopies of asymptotically cylindrical maps are required to fix the
asymptotic orbits in place.  There also exists a generalization of the theory
that lifts the latter condition (see \cite{Wendl:automatic}*{\S 4.1} and \cite{SiefringWendl}).

Suppose $\gamma$ is a closed orbit of $R$ and $\tau$ is a choice of 
trivialization of $\xi$ along~$\gamma$.\index{asymptotic trivializations}
The Conley-Zehnder index of $\gamma$ relative to this trivialization will be denoted by
$$
\muCZ^\tau(\gamma) \in \ZZ.
$$
If $\gamma$ has period $T > 0$, then any choice of $\omega$-compatible complex
structure $J$ on $\xi$ and parametrization 
$\gamma : S^1 := \RR / \ZZ \to M$ satisfying $\lambda(\dot{\gamma}) \equiv T$ gives rise to
an $L^2$-symmetric \index{asymptotic operator}\index{Reeb orbit!asymptotic operator of}
$$
\mathbf{A}_\gamma = -J (\nabla_t - T \nabla R) : \Gamma(\gamma^*\xi) \to \Gamma(\gamma^*\xi),
$$
where $\nabla$ is any symmetric connection on $M$ and $\mathbf{A}_\gamma$ 
does not depend on this choice.  As proved in \cite{HWZ:props2}, 
the nontrivial eigenfunctions of $\mathbf{A}_\gamma$ have winding numbers
(relative to~$\tau$) that depend only on their eigenvalues, defining a
nondecreasing map from the spectrum $\sigma(\mathbf{A}_\gamma) \subset \RR$
to $\ZZ$ that takes every value exactly twice (counting multiplicity of
eigenvalues).  One can therefore define the integers\index{Reeb orbit!extremal winding numbers of}\index{extremal winding numbers of a Reeb orbit}
\begin{equation*}
\begin{split}
\alpha_-^\tau(\gamma) &= \max \left\{ \wind^\tau(e)\ \big|\ 
\text{$\mathbf{A}_\gamma e = \lambda e$ with $\lambda < 0$} \right\},\\
\alpha_+^\tau(\gamma) &= \min \left\{ \wind^\tau(e)\ \big|\ 
\text{$\mathbf{A}_\gamma e = \lambda e$ with $\lambda > 0$} \right\},\\
p(\gamma) &= \alpha_+^\tau(\gamma) - \alpha_-^\tau(\gamma).
\end{split}
\end{equation*}
Since $\gamma$ is nondegenerate, $0$ is not an eigenvalue of $\mathbf{A}_\gamma$,
hence the \defin{parity}\index{Reeb orbit!parity of}\index{parity of a Reeb orbit}
$p(\gamma)$ is either $0$ or~$1$, and
\cite{HWZ:props2} proves the relation
$$
\muCZ^\tau(\gamma) = 2 \alpha_-^\tau(\gamma) + p(\gamma) = 2 \alpha_+^\tau(\gamma) - p(\gamma).
$$
For this reason, the number $\alpha_-^\tau(\gamma)$ sometimes appears in the
literature as $\lfloor \muCZ^\tau(\gamma) / 2 \rfloor$.

Given a closed Reeb orbit $\gamma$, we denote its $k$-fold cover for $k \in \NN$
by~$\gamma^k$.  

\begin{remark}
\label{remark:ellipticHyperbolic}
The parity of Reeb orbits is closely related to the
dichotomy between \emph{elliptic} and \emph{hyperbolic} orbits.  Recall that
since the linearized Reeb flow restricts to an $\omega$-symplectic map on the 
transverse planes $\xi$ along a periodic orbit~$\gamma$, the product of the eigenvalues
of this map is always~$1$.  We call $\gamma$ \defin{elliptic}\index{Reeb orbit!elliptic}\index{elliptic orbit}
if the eigenvalues
are a conjugate pair of non-real numbers on the unit circle, and
\defin{hyperbolic}\index{Reeb orbit!hyperbolic}\index{hyperbolic orbit}
if they are both real but distinct from $\pm 1$.  
(We exclude eigenvalues $\pm 1$ from this dichotomy; in this case either
$\gamma$ or $\gamma^2$ is degenerate.)  If $\gamma$ is an orbit whose covers
are all nondegenerate, then one sees by taking powers of the eigenvalues that
$\gamma$ is elliptic if and only if all of its covers are elliptic.
One can show moreover that $\gamma$ has even parity
if and only if both of the eigenvalues are positive, thus even orbits are
always hyperbolic, and the same applies to all of their covers
(see Exercise~\ref{EX:gcd}).  It follows that elliptic orbits always have
odd parity.  Hyperbolic orbits with odd parity are sometimes also called
\defin{negative hyperbolic}\index{Reeb orbit!negative hyperbolic}\index{negative hyperbolic orbit}
orbits; their even covers have even parity
and are referred to in the literature on symplectic field theory as
\emph{bad orbits}, for reasons having to do with orientations of moduli
spaces (see e.g.~\cite{Wendl:SFT}*{Chapter~11}).
\end{remark}

We say that an almost complex structure $J$ on $\RR \times M$ is
\defin{compatible}\index{almost complex structure!compatible with a stable Hamiltonian structure}\index{compatible almost complex structure}
with the stable Hamiltonian structure
$(\omega,\lambda)$ if
\begin{itemize}
\item $J(\p_r) = R$ for the coordinate vector field $\p_r$ in the $\RR$-direction;
\item $J(\xi) = \xi$ and $J|_\xi$ is compatible with $\omega|_\xi$;
\item $J$ is invariant under the translation action $(r,p) \mapsto (r + c,p)$
for all $c \in \RR$.
\end{itemize}
More generally, we consider almost complex $4$-manifolds $(\widehat{W},J)$
with \emph{cylindrical ends} as in \cite{SFTcompactness}.\index{cylindrical ends!of an almost complex manifold}\index{almost complex manifold with cylindrical ends}
Concretely,
this means $\widehat{W}$ decomposes into the union of a compact subset with
a positive end $[0,\infty) \times M_+$ and 
a negative end $(-\infty,0] \times M_-$, where $M_\pm$ are closed $3$-manifolds 
equipped with
stable Hamiltonian structures $(\omega_\pm,\lambda_\pm)$ and the
restriction of $J$ to each cylindrical end is compatible with these structures.
This will be our standing assumption about $(\widehat{W},J)$ in the
following.  For a punctured Riemann surface $(\dot{\Sigma},j)$, we
consider proper maps $u : \dot{\Sigma} \to \widehat{W}$ that are
\defin{asymptotically cylindrical}\index{asymptotically cylindrical map}\index{holomorphic curve!asymptotically cylindrical}
in the sense that they approximate
trivial cylinders over closed Reeb orbits near each of their (positive or
negative) non-removable punctures; see \S\ref{sec:punctured} for a more
precise definition of this term in the contact case.

\section{The intersection pairing}

Given the almost complex $4$-manifold $(\widehat{W},J)$ with cylindrical ends
as described above, let $\tau$ denote a
choice of trivialization for the complex line bundles 
$\xi_\pm = \ker\lambda_\pm$
along each simply covered closed Reeb orbit in~$M_\pm$.  This induces a
trivialization of $\xi_\pm$ along \emph{every} closed Reeb orbit by pulling
back along multiple covers.  The choice is arbitrary, but it is necessary in 
order to write down most formulas in the intersection theory,
even though none of the important quantities depend on it.
We assume $u : \dot{\Sigma} = \Sigma \setminus \Gamma_u \to \widehat{W}$ is a smooth asymptotically
cylindrical map with positive and/or negative punctures $\Gamma_u = \Gamma_u^+
\cup \Gamma_u^- \subset \Sigma$, and for each puncture $z \in \Gamma_u$,
let $\gamma_z$ denote corresponding asymptotic Reeb orbit.
We also fix a second such map $v : \dot{\Sigma}' \to \widehat{W}$, denote
its punctures by $\Gamma_v = \Gamma_v^+ \cup \Gamma_v^- \subset \Sigma'$
and use the same notation $\{ \gamma_z \}_{z \in \Gamma_v}$ for its
asymptotic orbits.\footnote{Note that each of the orbits $\gamma_z$ may be
multiply covered, and the covering multiplicity is regarded as part of the
data that defines~$\gamma_z$.}

Given any quantity $q_\pm(\gamma)$ which depends on both a Reeb orbit $\gamma$
and a choice of sign $+$ or~$-$, we will use the shorthand notation
$$
\sum_{z \in \Gamma_u^\pm} q_\pm(\gamma_z) := \sum_{z \in \Gamma_u^+} q_+(\gamma_z) 
+ \sum_{z \in \Gamma_u^-} q_-(\gamma_z).
$$
A similar convention applies to summations over pairs of punctures in
$\Gamma_u \times \Gamma_v$ with matching signs, and this will occur
several times in the following.

The intersection product of two asymptotically cylindrical maps $u$ and $v$ is
a symmetric pairing defined by
\begin{equation}
\label{eqn:productRef}
u * v := u \dotrel_\tau v - 
\sum_{(z,\zeta) \in \Gamma_u^\pm \times \Gamma_v^\pm}
\Omega^\tau_\pm(\gamma_z, \gamma_\zeta) \in \ZZ,
\end{equation}
where the individual terms are defined as follows.

The \defin{relative intersection number}\index{relative intersection number}
$$
u \dotrel_\tau v \in \ZZ
$$
is the algebraic count of intersections between $u$ and a generic perturbation
of $v$ that shifts it by an arbitrarily small positive distance in directions 
dictated by the chosen trivializations $\tau$ near infinity, hence the count is 
finite and depends only on the relative homology classes represnted by 
$u$ and $v$ and the homotopy class of the trivializations~$\tau$.  The
relative intersection number also appears in the
ECH literature and is denoted there by $Q_\tau(u,v)$,
cf.~\cites{Hutchings:index,Hutchings:lectures}.  Note that $u \dotrel_\tau u$
is also well defined, and is sometimes denoted by $Q_\tau(u)$ in ECH.

The integers $\Omega_\pm^\tau(\gamma,\gamma')$ are defined for every pair of
Reeb orbits $\gamma,\gamma'$ and also depend on the trivializations~$\tau$.
They satisfy $\Omega_\pm^\tau(\gamma,\gamma') = 0$ whenever $\gamma$ and
$\gamma'$ are not covers of the same orbit, while for any simply covered
orbit $\gamma$ with integers $k,m \in \NN$,
\begin{equation}
\label{eqn:OmegaRef}
\Omega^\tau_\pm(\gamma^k, \gamma^m) :=
\min\left\{ \mp k \alpha^\tau_\mp(\gamma^m), 
\mp m \alpha^\tau_\mp(\gamma^k) \right\}.
\end{equation}
The dependence on $\tau$ in the $\Omega^\tau_\pm$ terms cancels out the
dependence in $u \dotrel_\tau v$, so that $u * v$ is independent of~$\tau$;
it is determined solely by the relative homology classes of $u$ and $v$ and
their sets of asymptotic orbits.  In particular, it is invariant under
homotopies of $u$ and $v$ through families of smooth asymptotically cylindrical
maps with fixed asymptotic orbits.

If $u$ and $v$ are also $J$-holomorphic and are not covers of the same simple
curve, then we can also write
$$
u * v = u \cdot v + \inter_\infty(u,v),
$$
where both terms are nonnegative: the first denotes the actual algebraic
count of intersections between $u$ and $v$ (of which the asymptotic results
in \cite{Siefring:asymptotics} imply there are only finitely many),
and the second is an asymptotic contribution counting the number of
``hidden'' intersections that may emerge from infinity under a generic
perturbation.  A corollary is that if $u * v = 0$, then
$u$ and $v$ are disjoint unless they cover the same simple curve.
The converse of this is false in general, but one can use Fredholm theory
with exponential weights to show that for generic $J$,
$\inter_\infty(u,v) = 0$ for all simple curves $u$ and $v$ belonging to
some open and dense subsets of their respective moduli spaces.

To write down the asymptotic contribution $\inter_\infty(u,v)$ explicitly,
one must first define its relative analogue $\inter_\infty^\tau(u,v)$,
which depends only on the germ of $u$ and $v$ near infinity and on the 
trivializations~$\tau$.  We have
$$
\inter^\tau_\infty(u,v) = \sum_{(z,\zeta) \in \Gamma^\pm_u \times \Gamma^\pm_v}
\inter^\tau_\infty(u,z \,;\, v,\zeta),
$$
where for each pair of punctures $z \in \Gamma^\pm_u$ and $\zeta \in \Gamma^\pm_v$
with the same sign,
$$
\inter^\tau_\infty(u,z \,;\, v,\zeta) \in \ZZ
$$
is the algebraic count of intersections between $u|_{\uU_z}$ and a generic
perturbation of $v|_{\uU_\zeta}$, with $\uU_z$ and $\uU_\zeta$ chosen to be
suitably small neighborhoods of the respective punctures such that $u|_{\uU_z}$ and
$v|_{\uU_\zeta}$ are disjoint, and the perturbation of $v|_{\uU_\zeta}$
chosen to push it a small positive distance in directions dictated by the trivialization
$\tau$ near infinity.  The fact that this number is well defined depends on
the existence of neighborhoods on which $u$ and $v$ are disjoint, hence it
requires them to be geometrically distinct curves, and of course
$\inter^\tau_\infty(u,z \,;\, v,\zeta) = 0$ whenever the asymptotic orbits
$\gamma_z$ and $\gamma_\zeta$ are disjoint.  If on the other hand
$\gamma_z = \gamma^k$ and $\gamma_\zeta = \gamma^m$ for some simply covered
orbit $\gamma$ and integers $k,m \in \NN$, then $\inter^\tau_\infty(u,z \,;\, v,\zeta)$
can be computed in terms of the relative winding of $v$ about $u$ near
infinity; a precise formula is derived in the discussion surrounding
Equation~\eqref{eqn:wind2}.  Combining this formula with the relative 
asymptotic analysis from \cite{Siefring:asymptotics} then yields the bound
$\inter^\tau_\infty(u,z \,;\, v,\zeta) \ge \Omega^\tau_\pm(\gamma_z,\gamma_\zeta)$,
giving rise to the local asymptotic contribution
$$
\inter_\infty(u,z \,;\, v,\zeta) := \inter_\infty^\tau(u,z \,;\, v,\zeta) - 
\Omega^\tau_\pm(\gamma_z,\gamma_\zeta),
$$
which is independent of $\tau$ and is nonnegative, with equality if and
only if all theoretical bounds on the winding of asymptotic eigenfunctions
controlling the approach of $v$ to $u$ at infinity are achieved.  The geometric
interpretation is that $\inter_\infty(u,z \,;\, v,\zeta)$ is the algebraic
count of intersections between $u$ and $v$ that will appear in neighborhoods
of these two punctures if $u$ and $v$ are perturbed to $J'$-holomorphic
curves for some generic perturbation $J'$ of~$J$.  The total number of
hidden intersections is then
$$
\inter_\infty(u,v) = \sum_{(z,\zeta) \in \Gamma^\pm_u \times \Gamma^\pm_v}
\inter_\infty(u,z \,;\, v,\zeta).
$$

\section{The adjunction formula}

The adjunction formula for a closed simple $J$-holomorphic 
curve $u : \Sigma \to W$ can be written as
$$
[u] \cdot [u] = 2\delta(u) + c_N(u),
$$
where $[u] \cdot [u] \in \ZZ$ denotes the homological self-intersection
number of $[u] \in H_2(W)$, $c_N(u) := c_1([u]) - \chi(\Sigma)$ is the
so-called \emph{normal Chern number}, and $\delta(u)$ is the
algebraic count of double points and critical points, cf.~\eqref{eqn:delta}.
For a simple asymptotically cylindrical $J$-holomorphic
curve $u : \dot{\Sigma} \to \widehat{W}$ with punctures~$\Gamma_u$, 
the formula generalizes to\index{self-intersection number!of asymptotically cylindrical maps}\index{adjunction formula!for punctured holomorphic curves}
\begin{equation}
\label{eqn:adjunctionRef}
u * u = 2\left[ \delta(u) + \delta_\infty(u) \right] + c_N(u) + 
\left[ \bar{\sigma}(u) - \#\Gamma_u \right],
\end{equation}
where $u * u$ is the intersection product defined in \eqref{eqn:productRef}
with $u=v$, and the terms on the right hand side will be explained in a moment.
The most important thing to know about \eqref{eqn:adjunctionRef} is
that the terms $u * u$, $c_N(u)$ and $\bar{\sigma}(u)$ are all homotopy
invariant by definition, implying that $\delta(u) + \delta_\infty(u)$ is
also homotopy invariant, 
while $\bar{\sigma}(u) - \#\Gamma_u$, $\delta(u)$ and $\delta_\infty(u)$
are always nonnegative.  Moreover, as in the closed case,
$\delta(u) = 0$ if and only if $u$ is embedded.
It follows that $\delta(u) + \delta_\infty(u) = 0$ gives a homotopy-invariant
condition guaranteeing that $u$ is embedded.  The converse is false, as
$u$ can be embedded and have $\delta_\infty(u) > 0$, but one can again use
Fredholm theory with exponential weights to show that generically the latter
cannot happen for curves in some open and dense subset of the moduli space.

The \defin{normal Chern number}\index{normal Chern number!of a punctured holomorphic curve}\index{holomorphic curve!normal Chern number of}
is defined in the punctured case by
\begin{equation}
\label{eqn:normalChernRef}
c_N(u) := c_1^\tau(u^*T\widehat{W}) - \chi(\dot{\Sigma}) +
\sum_{z \in \Gamma_u^\pm} \pm \alpha^\tau_\mp(\gamma_z),
\end{equation}
and it depends on the relative homology class of $u$ and the topology of
the domain~$\dot{\Sigma}$, but not on the trivializations~$\tau$.
Here $c_1^\tau(u^*T\widehat{W})$ denotes the \defin{relative first Chern number}\index{relative first Chern number}
of the complex vector bundle $u^*T\widehat{W} \to \dot{\Sigma}$ with
respect to the natural trivializations at infinity induced by~$\tau$.
Recall that if $E \to \dot{\Sigma}$ is a complex line bundle equipped with
a preferred trivialization $\tau_E$ near infinity, one can define
$c_1^{\tau_E}(E) \in \ZZ$ as the algebraic count of zeroes of any generic
section of $E$ that is constant and nonzero with respect to $\tau_E$ near
infinity.  The relative first Chern number of higher rank bundles is then
defined via the direct sum property $c_1^{\tau_E \oplus \tau_F}(E \oplus F)
= c_1^{\tau_E}(E) + c_1^{\tau_F}(F)$.  Since $u^*T\widehat{W}$ has a natural
spliting over the positive/negative cylindrical ends into the direct sum of
a trivial complex line bundle with $\xi_\pm = \ker\lambda_\pm$, $\tau$
naturally induces a trivialization of $u^*T\widehat{W}$ over the ends and
we define $c_1^\tau(u^*T\widehat{W})$ accordingly.  (The same quantity is
often denoted by $c_\tau(u)$ in the ECH literature, 
cf.~\cites{Hutchings:index,Hutchings:lectures}.)  The normal Chern number is often most convenient to
calculate via the formula
\begin{equation}
\label{eqn:2cNref}
2 c_N(u) = \ind(u) - 2 + 2g + \#\Gamma\even,
\end{equation}
where $\ind(u)$ denotes the virtual dimension of the moduli space
containing~$u$ (see \eqref{eqn:indexPunctured}), $g$ is the genus of its domain, and $\Gamma\even \subset
\Gamma_u$ is the set of punctures $z \in \Gamma_u$ that satisfy
$p(\gamma_z) = 0$, i.e.~the Conley-Zehnder index of the corresponding
Reeb orbit is even.  This relation is an easy consequence of the
Fredholm index formula and the usual relations between Conley-Zehnder indices
and the winding numbers $\alpha^\tau_\pm(\gamma)$,
cf.~\eqref{eqn:2cN}.  The proper interpretation of $c_N(u)$ is as a
homotopy-invariant algebraic count of zeroes of the normal bundle of an
immersed perturbation of~$u$, including zeroes that are ``hidden at infinity''
but may emerge under small perturbations of~$u$ as a holomorphic curve.

The term $\bar{\sigma}(u)$ is called the \defin{spectral covering number}\index{spectral covering number}\index{holomorphic curve!spectral covering number of}
and is a sum of terms
$$
\bar{\sigma}(u) := \sum_{z \in \Gamma_u^\pm} \bar{\sigma}_\mp(\gamma_z),
$$
each of which is a positive integer that depends only on the orbit $\gamma_z$
and can be greater than $1$ only if $\gamma_z$ is
multiply covered.  Specifically, for any simply covered orbit $\gamma$ and
$k \in \NN$, $\bar{\sigma}_\pm(\gamma^k)$ is the covering multiplicity of\index{spectral covering number}\index{Reeb orbit!spectral covering number of}\index{covering multiplicity!of an asymptotic eigenfunction}
any of the nontrivial asymptotic eigenfunctions $e$ of $\mathbf{A}_{\gamma^k}$
that satisfy
$\wind^\tau(e) = \alpha^\tau_\pm(\gamma^k)$.  It turns out that the dependence
of $\bar{\sigma}_\pm(\gamma^k)$ on the orbit $\gamma$ is fairly mild, as
one can show that
\begin{equation}
\label{eqn:sigmabarRef}
\bar{\sigma}_\pm(\gamma^k) = \gcd(k, \alpha^\tau_\pm(\gamma^k)),
\end{equation}
cf.~Remark~\ref{remark:covInfty}.  Thus $\bar{\sigma}(u) - \#\Gamma$ vanishes,
for instance, whenever all the asymptotic orbits of $u$ are simply covered.

The \defin{singularity index}\index{singularity index of a simple holomorphic curve}\index{holomorphic curve!singularity index of}
$\delta(u)$ is defined just as in the closed
case, as a signed count of double points of $u$ plus positive contributions
for each critical point, interpreted as the count of double points that
appear near each critical point after an immersed perturbation
(cf.~Lemma~\ref{lemma:critInj}).  The only difference from the closed
case is that since $\dot{\Sigma}$ is noncompact, it is less obvious that
$\delta(u)$ is well defined, but the relative asymptotic results of
\cite{Siefring:asymptotics} imply that double points and critical points
of a simple curve cannot occur near infinity, hence $\delta(u)$ is finite.

The term $\delta_\infty(u)$ is an algebraic count of ``hidden'' double
points, i.e.~it is the number of extra contributions to $\delta(u)$
that will emerge from infinity if $u$ is perturbed to a $J'$-holomorphic
curve for a generic perturbation $J'$ of~$J$.  There are two possible
sources of such hidden double points: first, any pair of distinct punctures
$z,\zeta \in \Gamma_u^\pm$ with the same sign such that the corresponding 
asymptotic orbits $\gamma_z$ and $\gamma_\zeta$ are identical up to 
multiplicity contributes $\inter_\infty(u,z \,;\, u,\zeta)$ as in the
definition of $u * v$.  Note that $\inter_\infty(u,z \,;\, u,\zeta)$ is well defined
as long as $u$ is simple and $z \ne \zeta$, since the two punctures then
have neighborhoods $\uU_z$ and $\uU_\zeta$ such that
$u(\uU_z) \cap u(\uU_\zeta) = \emptyset$.  Additional hidden intersections
can emerge from any single puncture $z$ such that $\gamma_z$ is multiply
covered, since $u$ in the neighborhood of such a puncture has multiple
branches that become arbitrarily close to each other near infinity.  Denoting the
contribution from such punctures by $\delta_\infty(u,z)$, we have
$$
\delta_\infty(u) = \frac{1}{2} \sum_{z,\zeta \in \Gamma^\pm_u,\ z \ne \zeta}
\inter_\infty(u,z \,;\, u,\zeta) + \sum_{z \in \Gamma^\pm_u} \delta_\infty(u,z).
$$
In particular, $\delta_\infty(u) = 0$ whenever all asymptotic orbits of
$u$ are distinct and simply covered, though it can also be zero without
this condition.  As with $\inter_\infty(u,z \,;\, v,\zeta)$, writing down a
precise formula for $\delta_\infty(u,z)$ requires first defining a relative
version that depends on the trivialization~$\tau$: we define
$$
\inter_\infty^\tau(u,z) \in \ZZ
$$
as the algebraic count of intersections between $u|_{\uU_z}$ and a generic
small perturbation of itself, where $\uU_z$ is a neighborhood of~$z$ on
which $u$ is embedded, and the perturbation is chosen to shift $u$ a
small positive distance in directions dictated by $\tau$.
As with $\inter_\infty^\tau(u,z \,;\, v,\zeta)$, one can compute 
$\inter_\infty^\tau(u,z)$ in terms of the winding numbers of asymptotic
eigenfunctions that control the relative approach of different branches
of $u|_{\uU_z}$ to each other near infinity, 
cf.~\eqref{eqn:iuz}.  One derives from this the theoretical bound
$\inter_\infty^\tau(u,z) \ge \Omega_\pm^\tau(\gamma_z)$, where for
any simply covered orbit $\gamma$ and $k \in \NN$,
\begin{equation}
\label{eqn:OmegaSelfRef}
\Omega_\pm^\tau(\gamma^k) := \mp (k-1) \alpha^\tau_\mp(\gamma^k) + 
\left[ \bar{\sigma}_\mp(\gamma^k) - 1 \right].
\end{equation}
The precise definition of $\delta_\infty(u,z)$ is then
$$
\delta_\infty(u,z) := \frac{1}{2} \left[ \inter_\infty^\tau(u,z) -
\Omega_\pm^\tau(\gamma_z) \right],
$$
which is a nonnegative integer and is independent of~$\tau$.

As mentioned in Remark~\ref{remark:braids}, the computation of
$\inter_\infty^\tau(u,z)$ in terms of winding numbers also
leads to an alternative interpretation of it
as the writhe of a braid, which we will say more about in \S\ref{sec:ECH}.
Up to issues of bookkeeping, \eqref{eqn:adjunctionRef} is also equivalent to
the so-called \emph{relative} adjunction formula first written down by
Hutchings, see in particular \cite{Hutchings:index}*{Remark~3.2}.  The
innovation of \cite{Siefring:intersection} was to transform this into a
relation between homotopy-invariant quantities that have geometric meanings
independent of any choice of trivializations.

\section{Covering relations}

We now state a few useful results about multiply covered holomorphic curves
that are not mentioned elsewhere in these
notes, but are easy to prove based on the definitions given above.
The results of the present section are due to the author, and complete proofs
may be found in \cite{Wendl:automatic}*{\S 4.2}.

If $u$ and $v$
are two closed $J$-holomorphic curves in a closed almost complex $4$-manifold
and $\widetilde{u}$ is a $d$-fold multiple cover of~$u$, then the relation
$$
[\widetilde{u}] \cdot [v] = d [u] \cdot [v]
$$
is obvious since $[\widetilde{u}] = d[u] \in H_2(W)$.  Things are
less straightforward in the punctured case because $u * v$ depends on more
than just homology, and e.g.~Exercise~\ref{EX:funReeb} exhibits a
specific scenario in which the $*$-product fails to satisfy the obvious 
analogue of the above relation.  One can still however
prove the following:

\begin{prop}
\label{prop:coverRelation}
Suppose $u$, $\widetilde{u}$ and $v$ are asymptotically cylindrical $J$-holomorphic curves
in $\widehat{W}$ such that $\widetilde{u}$ is a $d$-fold cover of $u$ for
$d \in \NN$.  Then
$$
\widetilde{u} * v \ge d (u * v).
$$
\end{prop}

The proof of this inequality is based
on the formula \eqref{eqn:productRef}, in which the relative intersection
numbers are easily seen to satisfy the straightforward relation 
$\widetilde{u} \bullet_\tau v = d(u \bullet_\tau v)$, thus the tricky
part is to understand what happens to the terms
$\Omega^\tau_\pm(\gamma_z,\gamma_\zeta)$ when each of the orbits
$\gamma_z$ is replaced by a collection of covers of $\gamma_z$ whose
multiplicities add up to~$d$.  The answer is a bit intricate if one aims
to write it down precisely, because the winding numbers $\alpha_\pm^\tau(\gamma)$
do not in general behave linearly with respect to iteration of the orbit,
but for the purposes of the inequality in Proposition~\ref{prop:coverRelation},
the information in the following lemma suffices.  This lemma is 
closely related to Exercise~\ref{EX:gcd}, and it can be derived from
the properties of asymptotic eigenfunctions and their winding numbers
proved in \cite{HWZ:props2} (in particular Theorem~\ref{thm:winding}).

\begin{lemma}
\label{lemma:windRelation}
For every closed Reeb orbit $\gamma$ and every $k \in \NN$, there exist
integers $q_\pm(\gamma;k) \in \{0,\ldots,k-1\}$ such that
$$
\alpha_\pm(\gamma^k) = k \alpha_\pm(\gamma) \mp q_\pm(\gamma;k).
$$
\end{lemma}

It is also sometimes useful to have a similar covering relation for the
normal Chern number, since the latter appears in the adjunction formula.
Recall that if $\varphi : (\Sigma',j') \to (\Sigma,j)$ is a $d$-fold
holomorphic branched cover, then the Riemann-Hurwitz formula gives
\begin{equation}
\label{eqn:RiemannHurwitz}
-\chi(\Sigma') + d \chi(\Sigma) = Z(d\varphi),
\end{equation}
where $Z(d\varphi)$ denotes the algebraic count of branch points of~$\varphi$,
$$
Z(d\varphi) := \sum_{z \in d\varphi^{-1}(0)} \ord(d\varphi;z) \ge 0.
$$
One easy proof of this formula views $d\varphi$ as a section of
the line bundle $\Hom_\CC(T\Sigma',\varphi^*T\Sigma)$, whose first Chern
number is the left hand side of \eqref{eqn:RiemannHurwitz}.  If
$u : (\Sigma,j) \to (W,J)$ is a closed $J$-holomorphic curve and
$\widetilde{u} = u \circ \varphi$, this leads to the relation
$$
c_N(\widetilde{u}) = d \cdot c_N(u) + Z(d\varphi).
$$
In the punctured case, one can easily show that \eqref{eqn:RiemannHurwitz}
continues to hold for a branched cover of punctured surfaces, but additional
terms appear in the normal Chern number due to the fact that $\alpha_\pm^\tau(\gamma^d) \ne d \alpha_\pm^\tau(\gamma)$
in general.  As in Proposition~\ref{prop:coverRelation}, the result is then
most easily stated as an inequality.

\begin{prop}
\label{prop:cNrelation}
Suppose $u$ and $\widetilde{u} = u \circ \varphi$ are asymptotically cylindrical
$J$-holomorphic curves in $\widehat{W}$, where $\varphi : \dot{\Sigma}' \to
\dot{\Sigma}$ is a $d$-fold holomorphic branched cover of punctured Riemann
surfaces whose algebraic count of branch points is $Z(d\varphi) \ge 0$.  Then
$$
c_N(\widetilde{u}) \ge d \cdot c_N(u) + Z(d\varphi).
$$
\end{prop}

\begin{remark}
\label{remark:equality}
One can extract from the proofs in \cite{Wendl:automatic}*{\S 4.2} various
conditions to characterize when the inequalities in Propositions~\ref{prop:coverRelation}
and~\ref{prop:cNrelation} are strict or not.  The easiest comes from
the observation that the integers $q_\pm(\gamma;k)$ in Lemma~\ref{lemma:windRelation}
vanish whenever $\gamma$ has even parity.  It follows for instance that
Proposition~\ref{prop:coverRelation} becomes an equality whenever every simple
Reeb orbit that has a cover appearing among the asymptotic orbits of $u$ and $v$
(with the same sign!) is even.  Similarly, Proposition~\ref{prop:cNrelation}
is an equality if all the asymptotic orbits of $u$ are even.
\end{remark}

\section{The intersection product of buildings}

Another topic not mentioned elsewhere in these notes is the extension of
the $*$-pairing to
the \emph{compactified} moduli space $\overline{\mM}_g(\widehat{W},J)$ of
holomorphic buildings defined in \cite{SFTcompactness}.  
Following \cite{Siefring:intersection}, one can
define this in fairly general terms as follows.\index{moduli space!compactification of}\index{compactification of moduli spaces}

If $u \in \overline{\mM}_g(\widehat{W},J)$ and
$v \in \overline{\mM}_{g'}(\widehat{W},J)$ are two nodal $J$-holomorphic
curves in~$\widehat{W}$, i.e.~holomorphic buildings with no upper or lower
levels, then the definition of $u * v$ requires no change from before.
Recall that the domain of a punctured nodal curve $u$ is a possibly disconnected
punctured Riemann surface $\dot{S}$ endowed with a finite set of points
$\Delta \subset \dot{S}$, the \defin{nodal points},\index{nodal points}\index{nodes of a holomorphic curve}\index{holomorphic curve!nodal}
which are grouped into
pairs on which $u$ has matching values (see \S\ref{app:closed}).
A nodal curve then belongs to $\overline{\mM}_g(\widehat{W},J)$
if the surface obtained by performing connected sums on $\dot{S}$ at each
of the nodal pairs is connected with genus~$g$ and every component
of $\dot{S} \setminus \Delta$ on which $u$ is constant has negative Euler
characteristic.  For the present discussion, there is no need to
impose either of these conditions, thus we are free to consider
nodal curves that are non-stable and/or disconnected (even after
gluing together their nodes).  If $u : \dot{S} \to \widehat{W}$ is a nodal
curve and $\dot{S}_0 \subset \dot{S}$ is a connected component of its
domain (ignoring nodes), then let us call the restriction $u|_{\dot{S}_0} :
\dot{S}_0 \to \widehat{W}$ a \defin{connected component}\index{connected components of a nodal holomorphic curve}
of~$u$.
Now it is easy to check that if $u$ and $v$ are nodal curves whose connected components are
$u_1,\ldots,u_m$ and $v_1,\ldots,v_n$ respectively, the
$*$-pairing is additive in the obvious way, namely
$$
u * v = \sum_{i=1}^m \sum_{j=1}^n u_i * v_j.
$$

Things become more interesting if we consider buildings with
multiple levels.
Suppose $(\widehat{W}_0,J_0)$ and
$(\widehat{W}_1,J_1)$ are two almost complex $4$-manifolds with cylindrical
ends such that the positive end of $(\widehat{W}_0,J_0)$ matches the
negative end of $(\widehat{W}_1,J_1)$, meaning that the underlying 
$3$-manifolds and stable Hamiltonian structures are the same, and so are
the restrictions of $J_0$ and $J_1$ to translation-invariant almost complex
structures on the relevant ends.  We will then
refer to the symbol $\widehat{W}_0 \odot \widehat{W}_1$ as the 
\defin{concatenation}\index{concatenation of almost complex manifolds with cylindrical ends}
of $\widehat{W}_0$ with~$\widehat{W}_1$, and say that $u_0 \odot u_1$ is a 
\defin{holomorphic building}\index{holomorphic building!in a concatenation}
in $\widehat{W}_0 \odot \widehat{W}_1$ if $u_0$ and $u_1$ are (possibly disconnected and/or nodal) asymptotically cylindrical
holomorphic curves in $(\widehat{W}_0,J_0)$ and $(\widehat{W}_1,J_1)$
respectively, equipped with the extra structure of a bijection between the positive punctures of
$u_0$ and the negative punctures of $u_1$ that sends each puncture to one that
has the same asymptotic orbit.  We shall refer to $u_0$ and $u_1$ as
the (lower and upper) \defin{levels}\index{holomorphic building!levels}
of $u_0 \odot u_1$ and call the Reeb orbits along which
they connect to each other \defin{breaking orbits}.\index{holomorphic building!breakings orbits of}\index{breaking orbits}\index{Reeb orbit!breaking}
These definitions extend in an obvious way to
allow concatenations with more than two inputs, making $\odot$ an associative
operation.  In this language,\index{holomorphic building!stable}
$\overline{\mM}_g(\widehat{W},J)$ consists of all holomorphic buildings in
$$
(\RR \times M_-) \odot \ldots \odot (\RR \times M_-) \odot \widehat{W} \odot
(\RR \times M_+) \odot \ldots \odot (\RR \times M_+)
$$
that are connected with arithmetic genus $g$ and satisfy the usual
stability condition, where $\RR \times M_\pm$ is an abbreviation for the
symplectization of $M_\pm$ with the same $\RR$-invariant almost complex
structure that appears at the corresponding end of $\widehat{W}$, and
any nonnegative numbers of such symplectization levels are allowed to appear
in the concatenation.  

Recall that the stability
condition on elements of $\overline{\mM}_g(\widehat{W},J)$
precludes (among other things) the existence of any level that lives in an
$\RR$-invariant symplectization and consists of nothing more than a disjoint union of
orbit cylinders with no nodes.  It is necessary to exclude buildings that 
don't satisfy this condition in order for the natural topology on
$\overline{\mM}_g(\widehat{W},J)$ to be Hausdorff, but for our present purposes,
it will be useful
to avoid imposing any such requirement on buildings in concatenations.
We are then free to define the following operation: given a building
$u = u_1 \odot \ldots \odot u_N$ in $\widehat{W}_1 \odot \ldots \odot \widehat{W}_N$
and $k \in \{0,\ldots,N\}$, we construct the building
$$
u' = u_1 \odot \ldots \odot u_{k} \odot v \odot u_{k+1} \odot \ldots \odot u_N \quad\text{ in }\quad
\widehat{W}_1 \odot \ldots \odot \widehat{W}_{k} \odot \widehat{V} \odot \widehat{W}_{k+1} \odot \ldots \odot \widehat{W}_N,
$$
where $\widehat{V}$ is the symplectization corresponding to the positive end of
$\widehat{W}_{k}$ and negative end of $\widehat{W}_{k+1}$, and $v$ is a disjoint
union of orbit cylinders in $\widehat{V}$, one for each of the breaking
orbits that connect $u_{k}$ to~$u_{k+1}$.  Here the cases $k=0$ and $k=N$
are also allowed in order to accommodate adding a trivial level at the very bottom
or top of the building, and one should also keep in mind that $v$ could be
an \emph{empty} curve---this is the case if $u_{k}$ has no positive ends
and $u_{k+1}$ has no negative ends.
Any building obtained from $u$ by
a finite sequence of such operations will be called an \defin{extension}\index{holomorphic building!extension of}\index{extension of a holomorphic building}
of~$u$.

Given two buildings $u = u_1 \odot \ldots \odot u_N$ and
$v = v_1 \odot \ldots \odot v_N$ in a concatenation $\widehat{W}_1 \odot \ldots \odot \widehat{W}_N$,
we make an arbitrary choice of trivializations $\tau$ along all closed Reeb
orbits at the ends of each of $\widehat{W}_1$,\ldots,$\widehat{W}_N$ and
define
\begin{equation}
\label{eqn:starBuilding}
u * v := \sum_{i=1}^N u_i \bullet_\tau v_i - \sum_{(z,\zeta) \in \Gamma_{u_N}^+ \times \Gamma_{v_N}^+} \Omega^\tau_+(\gamma_z,\gamma_\zeta) -
\sum_{(z,\zeta) \in \Gamma_{u_1}^- \times \Gamma_{v_1}^-} \Omega^\tau_-(\gamma_z,\gamma_\zeta).
\end{equation}
Here the dependence on $\tau$ at the positive ends of the top level and
negative ends of the bottom level is cancelled by the $\Omega^\tau_\pm$
terms for the same reasons as in \eqref{eqn:productRef}, while
changing $\tau$ at any of the breaking orbits between levels $k-1$ and
$k$ alters $u_{k-1} \bullet_\tau v_{k-1}$ and $u_k \bullet_\tau v_k$
in ways that cancel out, thus the total expression is independent of choices.
This definition does not yet allow us to define
$u * v$ for an arbitrary pair of stable buildings $u \in \overline{\mM}_g(\widehat{W},J)$ and
$v \in \overline{\mM}_{g'}(\widehat{W},J)$, because these may in general
have differing numbers of levels.  However, one can always add extra
trivial symplectization levels to one or both of them to produce a pair of
\emph{non-stable} buildings that live in the same concatenation of cobordisms.
With this understood, we define
\begin{equation}
\label{eqn:starBuilding2}
u * v := u' * v' \in \ZZ \quad\text{ for }\quad
u \in \overline{\mM}_g(\widehat{W},J) \text{ and }
v \in \overline{\mM}_{g'}(\widehat{W},J),
\end{equation}
where $u'$ and $v'$ are any choices of extensions of $u$ and $v$ that make
$u' * v'$ well defined in the sense of \eqref{eqn:starBuilding}.

\begin{prop}
\label{prop:starBuilding}
The pairing $u * v$ defined in \eqref{eqn:starBuilding2} for stable holomorphic
buildings in $\widehat{W}$ has the following properties:
\begin{enumerate}
\item It is independent of the choices of extensions $u'$ and~$v'$.
\item It is continuous with respect to the natural topologies on
$\overline{\mM}_g(\widehat{W},J)$ and $\overline{\mM}_{g'}(\widehat{W},J)$,
e.g.~if $u_k \in \mM_g(\widehat{W},J)$ is a sequence converging to a building
$u \in \overline{\mM}_g(\widehat{W},J)$ in the sense of \cite{SFTcompactness},
then $u_k * v = u * v$ for large~$k$.
\item It is superadditive with respect to concatenation, i.e.~for any
(not necessarily stable) buildings $u_-,v_-$ and $u_+,v_+$ such that the concatenations
$u_- \odot u_+$ and $v_- \odot v_+$ and the intersection numbers
$u_\pm * v_\pm$ in the sense of \eqref{eqn:starBuilding} are well defined,
one has
$$
(u_- \odot u_+) * (v_- \odot v_+) \ge u_- * v_- + u_+ * v_+,
$$
with equality whenever all the simple orbits with covers appearing as
breaking orbits in both $u_- \odot u_+$ and $v_- \odot v_+$ are even,
and strict inequality if any of these simple orbits is elliptic.
\end{enumerate}
\end{prop}

\begin{remark}
\label{remark:otherCompact}
We have stated the above result with reference to one of the three compactified
moduli spaces of holomorphic buildings defined in \cite{SFTcompactness},
i.e.~for the degeneration of curves in a completed \emph{nontrivial}
symplectic cobordism.  The result can be adapted in obvious ways for
the other two scenarios, namely for degenerations of curves in a symplectization
(so that each level is defined only up to $\RR$-translation and there is no
distinguished ``main level''), and degenerations with respect to neck-stretching.
In the symplectization case, the freedom to choose different extensions
of $u$ and $v$ is more useful than one might at first imagine.

For example, in Lemma~\ref{lemma:int=0}, we considered two curves
$u_P$ and $v$ in a symplectization $\RR \times M$, where $u_P$ was a page
of a holomorphic open book (which has only positive punctures), and $v$ was any 
other curve whose positive ends are all asymptotic to simple orbits in the
binding of the open book.  Having shown in the previous lemma that
$u_P * u_\gamma = 0$ for all of the trivial cylinders $u_\gamma$ over asymptotic
orbits $\gamma$ of~$u_P$, we then used a homotopy of asymptotically cylindrical
maps (Figure~\ref{fig:OBDhomotopy} in Lecture~\ref{sec:5})
to prove that $u * v$ is a sum of such terms, and therefore vanishes.
An alternative argument for the second step is illustrated in Figure~\ref{fig:OBDextension}:
define extensions $u_P'$ of $u_P$ and $v'$ of $v$ as buildings in
$(\RR\times M) \odot (\RR \times M)$, where $u_P':= \emptyset \odot u_P$
has a trivial level added below the original curve (the trivial level is
the empty curve since $u_P$ has no negative ends), and $v' := v \odot v_+$
with $v_+$ as the disjoint union of trivial cylinders over the positive
asymptotic orbits of~$v$.  Instead of exploiting homotopy invariance as we
did in Lemma~\ref{lemma:int=0}, one could now apply Proposition~\ref{prop:starBuilding}
and write
$$
u_P * v = u_P' * v' \ge (\emptyset * v) + (u_P * v_+) = u_P * v_+,
$$
where the last equality follows since the empty curve has zero intersection number with everything else.
The inequality is in this case an equality because every simple orbit that
has a cover appearing as a breaking orbit of both $u_P'$ and $v'$ is even---this
is a statement about the empty set, and is therefore true.  This proves
$u_P * v = u_P * v_+$, and the latter again vanishes due to Lemma~\ref{lemma:orbitCylinder}.
Morally, one can think of the replacement of $v$ with $v_+$ as an
``unbounded homotopy,'' i.e.~it shifts $v$ by $\RR$-translation infinitely
far downward so that $v$ now occupies a lower level.  In this sense,
the independence of $u * v$ with respect to choices of extensions is just
another manifestation of homotopy invariance.
\end{remark}

\begin{figure}
\psfrag{uP}{$u_P$}
\psfrag{v}{$v$}
\psfrag{v+}{$v_+$}
\psfrag{empty}{$\emptyset$}
\includegraphics{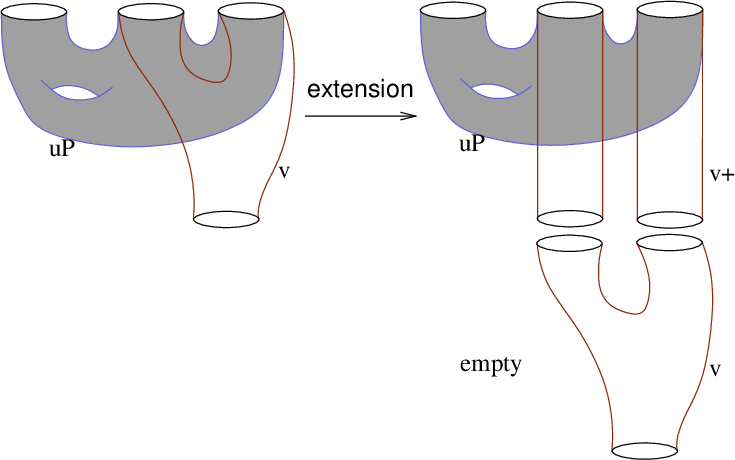}
\caption{\label{fig:OBDextension} An alternative to the homotopy argument
depicted in Figure~\ref{fig:OBDhomotopy} of Lecture~\ref{sec:5}, using the intersection number
between buildings to prove Lemma~\ref{lemma:int=0}.}
\end{figure}

The possibility of strict inequality in the third item of Proposition~\ref{prop:starBuilding} reveals another interesting
``hidden intersection'' phenomenon:\index{intersections!hidden between levels of a building|(}\index{breaking orbits!can hide intersections|(}
intersections between buildings 
can be hidden in the breaking orbits between levels.  Concretely, suppose
$u_k$ and $v_k$ are two sequences of smooth curves in the completed
cobordism $\widehat{W}$ that converge to
two-level buildings $u = u_- \odot u_+$ and $v = v_- \odot v_+$ respectively
in $\widehat{W} \odot (\RR \times M_+)$,
such that $u_\pm$ and $v_\pm$ are disjoint and satisfy $u_\pm * v_\pm = 0$.
Then the curves in each individual level do not have any hidden intersections,
meaning one could make arbitrary small perturbations of the data
on $\widehat{W}$ or $\RR \times M_+$ and 
rely on $u_\pm$ remaining disjoint from~$v_\pm$.  But it is
nonetheless possible that $u_k$ and $v_k$ intersect each other for all
$k$ large, in which case these intersections must escape from every compact
subset of~$\widehat{W}$ as $k \to \infty$, so as not to survive as intersections
of $u_-$ with~$v_-$.  At the same time, the intersections of $u_k$ with $v_k$
cannot congregate as $k \to \infty$ in any
compact subset of $\RR \times M_+$ after shifting the positive cylindrical end
to focus on the convergence of the upper level, as otherwise they would
survive as intersections of $u_+$ with~$v_+$.  Instead, intersections
congregate in the increasingly wide area ``between levels'' as $k \to \infty$,
so that they do not appear at all in the limit.  
Despite this, they
are accounted for by $u * v$, which must in this case be strictly larger
than $u_- * v_- + u_+ * v_+ = 0$.
This phenomenon has sometimes been exploited in applications, 
e.g.~to define a version of contact homology on the complement of a set of
fixed Reeb orbits (see \cite{HryniewiczMominSalomao}).
Relatedly, one can use a ``local'' version of the adjunction formula to
show that the breaking orbits of a single building with embedded
levels can also hide double points of simple curves that
degenerate to them; see \cite{CiobaWendl}.

Looking at the definition \eqref{eqn:starBuilding}, one sees that
in the context of Proposition~\ref{prop:starBuilding},
\begin{equation}
\label{eqn:breakingFormula}
\begin{split}
(u_- \odot u_+) * (v_- \odot v_+) - u_- * v_- - u_+ * v_+ &=
\sum_{(z,\zeta)} \Br(\gamma_z,\gamma_\zeta), \qquad \text{ where }\\
\Br(\gamma,\gamma') &:=
\Omega^\tau_+(\gamma,\gamma') + \Omega^\tau_-(\gamma,\gamma'),
\end{split}
\end{equation}
and the sum is over all pairs of breaking punctures, i.e.~all $(z,\zeta)$ with
$z \in \Gamma_{u_-}^+ \cong \Gamma_{u_+}^-$ and $\zeta \in \Gamma_{v_-}^+ \cong \Gamma_{v_+}^-$.
The so-called \defin{breaking contribution}\index{breaking contribution}\index{intersections!hidden between levels of a building|)}\index{breaking orbits!can hide intersections|)}
$\Br(\gamma_z,\gamma_\zeta) \in \ZZ$ 
is independent of the choice of trivialization~$\tau$ and
vanishes if $\gamma_z$ and $\gamma_\zeta$ are disjoint, whereas if
$\gamma_z = \gamma^k$ and $\gamma_\zeta = \gamma^m$ for some $k,m \in \NN$
and a simple orbit~$\gamma$, one extracts from \eqref{eqn:OmegaRef} the
formula
$$
\Br(\gamma^k,\gamma^m) =
\min \left\{ k \alpha^\tau_+(\gamma^m) , m \alpha^\tau_+(\gamma^k) \right\} -
\max \left\{ k \alpha^\tau_-(\gamma^m) , m \alpha^\tau_-(\gamma^k) \right\}.
$$
Using the relation $p(\gamma) = \alpha^\tau_+(\gamma) - \alpha^\tau_-(\gamma)$,
it is now a straightforward exercise to prove the inequality
\begin{equation}
\label{eqn:rangeBreaking}
\min\left\{ k p(\gamma^m) , m p(\gamma^k) \right\} \le
\Br(\gamma^k,\gamma^m) \le
\max\left\{ k p(\gamma^m) , m p(\gamma^k) \right\}.
\end{equation}
The breaking contributions are thus manifestly nonnegative; moreover, they vanish
whenever $\gamma$ is even and are strictly positive if $\gamma$ is elliptic and
all its covers are nondegenerate,
since the latter guarantees that the covers are also odd (see Remark~\ref{remark:ellipticHyperbolic}).  
This is the reason
for the inequality stated in Proposition~\ref{prop:starBuilding}.

\begin{remark}
If $u_\gamma$ and $u_{\gamma'}$ denote the trivial cylinders over two Reeb
orbits $\gamma$ and $\gamma'$, then taking the same set of trivializations
at positive end negative ends always gives $u_\gamma \bullet_\tau u_{\gamma'} = 0$
since $u_{\gamma'}$ admits a global perturbation that is compatible with 
$\tau$ near infinity and everywhere disjoint from~$u_\gamma$.  Plugging in
the definition of the $*$-pairing, one obtains
from this a geometric interpretation of the breaking contribution
$\Br(\gamma,\gamma')$, namely
\begin{equation}
\label{eqn:breakingCyl}
\Br(\gamma,\gamma') = - u_\gamma * u_{\gamma'},
\end{equation}
along with the useful corollary that $u_\gamma * u_{\gamma'}$ is never
positive (cf.~Exercise~\ref{EX:funReeb}).
\end{remark}

An analogue of Proposition~\ref{prop:starBuilding}
for the normal Chern number is sometimes needed for applications of the
adjunction formula.  For a nodal curve $u : \dot{S} \to \widehat{W}$
with nodal points $\Delta \subset \dot{S}$,
\eqref{eqn:normalChernRef} is not quite the right definition because
$\chi(\dot{S})$ does not generally match the Euler characteristic of
the surface obtained from $\dot{S}$ by performing connected sums at all
nodal pairs.  To achieve this and thus ensure that $c_N$ is continuous
under degenerations from smooth curves to nodal curves, one defines
\begin{equation}
\label{eqn:normalChernNodal}
c_N(u) := c_1^\tau(u^*T\widehat{W}) - \chi(\dot{S} \setminus \Delta) + \sum_{z \in \Gamma_u^+} \alpha^\tau_-(\gamma_z)
 - \sum_{z \in \Gamma_u^-} \alpha^\tau_+(\gamma_z).
\end{equation}
If $u$ has connected components $u_1,\ldots,u_m$, we then have the relation
\begin{equation}
\label{eqn:cNnodal}
c_N(u) = \sum_{i=1}^m c_N(u_i) + 2\left( \#\Delta\right).
\end{equation}
For a building $u_1 \odot \ldots \odot u_N$ in
$\widehat{W}_1 \odot \ldots \odot \widehat{W}_N$, the above definition
now generalizes naturally as
\begin{equation}
\label{eqn:normalChernBuilding}
c_N(u_1 \odot \ldots \odot u_N) := \sum_{k=1}^N \left[ c_1^\tau(u_k^*T\widehat{W}_k) - \chi(\dot{\Sigma}_k \setminus \Delta_k) \right] +
\sum_{z \in \Gamma_{u_N}^+} \alpha^\tau_-(\gamma_z) - \sum_{z \in\Gamma_{u_1}^-} \alpha^\tau_+(\gamma_z),
\end{equation}
where for each $k=1,\ldots,N$, $\dot{\Sigma}_k$ denotes the (possibly disconnected) domain of the level
$u_k$, and $\Delta_k \subset \dot{\Sigma}$ is the set of nodal points in that domain.
This leads to:

\begin{prop}
\label{prop:normalChernBuilding}
The normal Chern number is continuous with respect to the natural topology of
$\overline{\mM}_g(\widehat{W},J)$.  Moreover, it is superadditive with
respect to concatenation: in particular, for any pair of (not necessarily
stable) buildings $u_-$ and $u_+$ for which the concatenation
$u_- \odot u_+$ is well defined, one has
$$
c_N(u_- \odot u_+) = c_N(u_-) + c_N(u_+) + \sum_z p(\gamma_z) \ge
c_N(u_-) + c_N(u_+),
$$
where the sum is over all breaking punctures that connect $u_-$ to~$u_+$,
i.e.~$z \in \Gamma_{u_-}^+ \cong \Gamma_{u_+}^-$.
\end{prop}

\section{Comparison with the ECH literature}
\label{sec:ECH}

Intersection theory plays a major role in Hutchings's theory of embedded
contact homology (ECH), and in fact early developments in ECH
(notably the paper \cite{Hutchings:index}) provided some of the inspiration behind
Siefring's intersection theory of punctured holomorphic curves.  Though
the $*$-pairing and Siefring's adjunction formula do not usually appear in
papers on ECH, the relative intersection numbers and relative adjunction
formula appear quite prominently, with differing notational
conventions, and many of the same winding bounds that underlie Siefring's 
theory also play crucial roles in ECH.  The aim of this section is to
provide a glossary for translating between these two contexts.

Aside from notation, the major difference between the ECH literature and our treatment in these
notes is that ECH expresses all relative asymptotic quantities such as
$\inter_\infty^\tau(u,v)$ and $\inter_\infty^\tau(u)$ in terms of
topological invariants of certain braids.  Concretely, for two asymptotically
cylindrical curves $u$ and $v$ that do not have identical images, we have
written the count of intersections near infinity that appear under small
perturbations moving $v$ in the direction of asymptotic trivializations $\tau$
as
$$
\inter^\tau_\infty(u,v) = \sum_{(z,\zeta) \in \Gamma^\pm_u \times \Gamma^\pm_v}
\inter^\tau_\infty(u,z \,;\, v,\zeta),
$$
where $\inter^\tau_\infty(u,z \,;\, v,\zeta) \in \ZZ$ denotes the contribution
coming from the specific punctures $z \in \Gamma^\pm_u$ and $\zeta \in \Gamma^\pm_v$.
The latter can only be nonzero if $\gamma_z$ and $\gamma_\zeta$ are covers of the
same underlying simple orbit~$\gamma$, thus let us assume this henceforth.
We wrote down Siefring's formula for $\inter^\tau_\infty(u,z \,;\, v,\zeta)$ in terms of
relative winding numbers in \eqref{eqn:wind2}. Hutchings
expresses the same formula as follows.  Writing $u$ and $v$ in holomorphic
cylindrical coordinates $(s,t) \in Z_\pm$ near the punctures $z$ and $\zeta$\index{cylindrical coordinates on a punctured Riemann surface}
respectively,\footnote{Recall that we denote $Z_+ := [0,\infty) \times S^1$
and $Z_- := (-\infty,0] \times S^1$, where the convention is to use the former
near positive punctures and the latter near negative punctures.}
we can fix some $s_0 \gg 0$ and consider the restrictions of $u$ and $v$
to $\{\pm s_0\} \times S^1 \subset Z_\pm$; this defines a disjoint pair of
(possibly multiply covered) oriented loops 
in the positive or negative cylindrical end of~$\widehat{W}$.  Projecting
them to the $3$-manifold $M_\pm$ then gives disjoint oriented loops
$\beta_z , \beta_\zeta : S^1 \to M_\pm$ that live in an arbitrarily small
tubular neighborhood of~$\gamma$.  If we now use the trivialization $\tau$
to identify the neighorhood of $\gamma$ with $S^1 \times \DD$, then
$\beta_z$ and $\beta_\zeta$ become a pair of disjoint braids---strictly
speaking, they are in general ``multiply covered braids,'' but one can
perturb to make each of them embedded and thus view them as honest braids
without changing any essential features of this discussion.
The \defin{linking number}\index{linking number of two braids}\index{braid!linking number}
between these two braids,
$$
\ell_\tau(\beta_z,\beta_\zeta) \in \ZZ,
$$
is defined as one half the signed number of crossings of strands of $\beta_z$
with strands of~$\beta_\zeta$, where the sign convention is that counterclockwise
twists count positively.  (As mentioned in \cite{Hutchings:index}*{\S 3.1},
this convention differs from much of the knot theory literature, but it is
used consistently in papers on ECH and we shall stick with it here as well.)
The precise relation between this linking number and Siefring's relative
asymptotic intersection numbers is then given by
\begin{equation}
\label{eqn:linking}
\inter^\tau_\infty(u,z\,;\,v,\zeta) = \mp \ell_\tau(\beta_z,\beta_\zeta),
\end{equation}
where the sign $\mp$ is opposite the signs $\pm$ of the two punctures.

The relative adjunction formula in \eqref{eqn:relAdjunction} also includes
the term
$$
\inter^\tau_\infty(u) = \sum_{z,\zeta \in \Gamma^\pm,\,
z \ne \zeta} \inter^\tau_\infty(u,z \,;\, u,\zeta) +
\sum_{z \in \Gamma^\pm} \inter^\tau_\infty(u,z),
$$
which is defined only when $u$ is a simple curve; the contribution
$\inter^\tau_\infty(u,z) \in \ZZ$ for each puncture $z$ can only be nonzero
when the orbit $\gamma_z$ is multiply covered, as it is
the count of intersections 
in a neighborhood of $z$ between $u$ and a small
perturbation of itself that is pushed in the direction of the trivialization
$\tau$ near infinity.  Siefring's formula for
$\inter^\tau_\infty(u,z)$ in terms of winding numbers appears in
\eqref{eqn:iuz}, and its topological interpretation is (up to a sign) as
the \defin{writhe}\index{writhe of a braid}\index{braid!writhe}
of the braid $\beta_z$ described in the previous
paragraph,
$$
w_\tau(\beta_z) \in \ZZ.
$$
Here the fact that $u$ is simple guarantees that it is embedded in a
neighborhood of the puncture~$z$, thus the braid $\beta_z$ is automatically
embedded, and its writhe is defined as the signed number of crossings of strands,
using the same sign convention mentioned above.  This is then related to
$\inter^\tau_\infty(u,z)$ by
\begin{equation}
\label{eqn:writhe}
\inter^\tau_\infty(u,z) = \mp w_\tau(\beta_z).
\end{equation}
What Hutchings in \cites{Hutchings:index,Hutchings:lectures} calls the
``total'' writhe $w_\tau(u) \in \ZZ$ of a simple holomorphic
curve $u$ is defined by adding up the writhes at all positive punctures,
plus linking numbers for pairs of distinct punctures that have coinciding asymptotic
orbits (up to multiplicity), and then subtracting all of the corresponding
terms for the negative punctures.  This produces
$$
w_\tau(u) = - \inter^\tau_\infty(u),
$$
thus the relative adjunction formula for a simple curve\index{adjunction formula!relative}\index{relative adjunction formula}
$u : \dot{\Sigma} \to \widehat{W}$ in ECH language takes the form
$$
c_\tau(u) = \chi(u) + Q_\tau(u) + w_\tau(u) - 2 \delta(u),
$$
where:
\begin{itemize}
\item $c_\tau(u) := c_1^\tau(u^*T\widehat{W})$ is the relative first Chern
number;
\item $\chi(u) := \chi(\dot{\Sigma})$ is the Euler characteristic of the domain;
\item $Q_\tau(u) := u \bullet_\tau u$ is the relative self-intersection number;\index{self-intersection number!relative}
\item $w_\tau(u)$ is the total writhe as explained above;
\item $\delta(u)$ is the usual algebraic count of double points and critical points.
\end{itemize}
The reader can now check that this formula is equivalent to \eqref{eqn:relAdjunction}.

\backmatter

\printindex

\end{document}